\documentclass{amsart}
\usepackage{blindtext}

\usepackage{hyperref}

\hypersetup{
    colorlinks=true,
    linkcolor=blue,
    filecolor=magenta,      
    urlcolor=cyan,
    }
    
\usepackage{amssymb}
\usepackage{amsfonts}

\newtheorem{theorem}{Theorem}[subsection]
\newtheorem{lemma}[theorem]{Lemma}
\newtheorem{proposition}[theorem]{Proposition}
\newtheorem{corollary}[theorem]{Corollary}

\theoremstyle{definition}
\newtheorem{definition}[theorem]{Definition}
\newtheorem{example}[theorem]{Example}
\newtheorem{question}[theorem]{Question}
\newtheorem{problem}[theorem]{Problem}
\newtheorem{terminology}[theorem]{Terminology}
\newtheorem{notation}[theorem]{Notation}
\newtheorem{step}[theorem]{Step}
\newtheorem{convention1}[theorem]{Convention}

\theoremstyle{remark}
\newtheorem*{remark}{Remark}

\numberwithin{equation}{section}

\begin{document}
\title{A theory for generalized morphisms and beyond}
\author{Gang Hu}
\subjclass[2020]{Primary 08A99, 08A35; Secondary 12H05, 18A99, 12F10, 13B05, 12F20, 20B25, 20B27} 
\keywords{Generalized morphisms; Galois theory; Galois correspondence; Constructions of morphisms; Solvability of equations; Isomorphism theorem; Transcendental elements}

\begin{abstract}
Some sorts of generalized morphisms are defined from very basic mathematical objects such as sets, functions, and partial functions. A wide range of mathematical notions such as continuous functions between topological spaces, ring homomorphisms, module homomorphisms, group homomorphisms, and covariant functors between categories can be characterized in terms of the generalized morphisms. We show that the inverse of any bijective generalized morphism is also a generalized morphism (of the same kind), and hence a generalized isomorphism can be defined as a bijective generalized morphism. 

Galois correspondences are established and studied, not only for the Galois groups of the generalized automorphisms, but also for the “Galois monoids” of the generalized endomorphisms. 

Ways to construct the generalized morphisms and the generalized isomorphisms are studied.

New interpretations on solvability of polynomials and solvability of homogeneous linear differential equations are introduced, and these ideas are roughly generalized for “general” equation solving in terms of our theory for the generalized morphisms.

Some more results are presented. For example, we generalize the algebraic notions of transcendental elements over a field and purely transcendental field extensions, we obtain an isomorphism theorem that generalizes the first isomorphism theorems (for groups, rings, and modules), and we show that a part of our theory is closely related to dynamical systems. 

\end{abstract}

\maketitle

\tableofcontents

\section{Introduction}
The notions of generalized morphisms introduced in this paper develop from very basic mathematical objects such as sets, functions, and partial functions. As a consequence, many familiar mathematical notions are covered by the generalized morphisms. Specifically, continuous functions between topological spaces, ring homomorphisms, module homomorphisms, group homomorphisms, and covariant functors between categories can be characterized in terms of the generalized morphisms. Indeed, in these examples, the generalized morphisms preserve the structures of the mathematical objects involved. Basically, this is because we define the generalized morphisms to be commutative with the intrinsic operations on the mathematical objects involved. Moreover, we show that the inverse of any bijective generalized morphism is also a generalized morphism (of the same kind), and hence a generalized isomorphism can be defined as a bijective generalized morphism. 

The notion of generalized morphisms in our theory is not the same as the notion of morphisms in category theory. In category theory, a morphism between two objects does not necessarily preserve the structure of objects, while a generalized morphism in our theory does. Nevertheless, since covariant functors preserve the structures of categories, all covariant functors can be characterized in terms of the generalized morphisms. 

One may have encountered concepts similar to the notions of our generalized morphisms, e.g. the notion of homomorphisms between structures in the field of mathematical logic. However, we have not found any research with any content similar to ours.

Besides introducing the new notions, we establish Galois correspondences for the Galois groups of the generalized automorphisms as well as for the “Galois monoids” of the generalized endomorphisms. As is well known, in the Galois theory for infinite algebraic field extensions, Krull topology is put on Galois groups, and in differential Galois theory, differential Galois groups are endowed with Zariski topology. Then the fundamental theorem in each of the two theories tells us that there exists a Galois correspondence between the set of all \emph{closed} subgroups of the Galois group of a Galois field extension (or a Picard-Vessiot extension) and the set of all intermediate (differential) fields. Now for our theory, we first study the lattice structures related to Galois correspondences. Then, we can determine under what conditions we can characterize Galois correspondences with topology (as is done in the above two Galois theories).

Moreover, we study ways to construct the generalized morphisms and the generalized isomorphisms, which we think to be important because the morphisms and isomorphisms preserve the structures of the mathematical objects involved.

In Part I of this paper, which consists of Sections \ref{Basic notions and properties} to \ref{Cons of $T$-mor and theta-mor}, we address the above issues for the case where only single-variable functions are involved, and in Part II, we address the above issues for the case where multivariable functions or partial functions are involved. 

In Part III, new interpretations on solvability of polynomials and solvability of homogeneous linear differential equations are introduced. The key idea is to “decompose” an equation $Q$ into $n(\in {{\mathbb{Z}}^{+}})$ equations ${{Q}_{1}},\cdots ,{{Q}_{n}}$ such that each solution of $Q$ can somehow be expressed in terms of the solutions of ${{Q}_{1}},\cdots ,{{Q}_{n}}$, where for example in the case of a separable polynomial $p(x)$ in $Q:=(p(x)=0)$, the Galois group of each $p_i(x)$ in $Q_i:=(p_i(x)=0)$ is simple. Moreover, this idea is roughly generalized for “general” equation solving in terms of the theory developed in Parts I and II.

Some more results are presented in Part IV. For example, in Section \ref{Transitive}, we obtain analogues of the well-known fact that the Galois group of an irreducible polynomial acts transitively on its roots, and in Section \ref{transcendental}, we generalize the algebraic notions of transcendental elements over a field and purely transcendental field extensions. Moreover, in Section \ref{first iso}, we obtain an isomorphism theorem which generalizes the first isomorphism theorems (for groups, rings, and modules). And in Section \ref{App to dyn}, we show a close relation between a part of our theory and dynamical systems. 

$ $

To help the reader better understand the structure of the paper, we simplify the table of contents as follows.
\begin{center}
    Part I. Theory for the case of unary functions
\end{center}

Section \ref{Basic notions and properties}. Basic notions and properties

Section \ref{I Galois corr}. Galois correspondences

Section \ref{I Lattice structures}. Lattice structures of objects arising in Galois correspondences on a $T$-space $S$

Section \ref{I Topologies employed}. Topologies employed to construct Galois correspondences

Section \ref{theta-mor and theta-iso}. Generalized morphisms and isomorphisms from a $T_1$-space to a $T_2$-space

Section \ref{Cons of $T$-mor and theta-mor}. Constructions of the generalized morphisms and isomorphisms

\begin{center}
    Part II. Theory for the case of multivariable total or partial functions
\end{center}

Section \ref{Basic notions for multivariable}. Basic notions for the case of multivariable (total) functions

Section \ref{Basic notions for partial}. Basic notions for the case of (multivariable) partial functions

Section \ref{12 Basic properties}. Basic properties and more notions

Section \ref{II Gal corr}. Galois correspondences

Section \ref{II latt struc}. Lattice structures of objects arising in Galois correspondences on a $T$-space $S$

Section \ref{II topologies employed}. Topologies employed to construct Galois correspondences

Section \ref{II Cons of $T$-mor and}. Constructions of the generalized morphisms and isomorphisms

\begin{center}
    Part III. Solvability of equations
\end{center}

Section \ref{poly equ}. A solvability of polynomial equations

Section \ref{diff equ}. A solvability of homogeneous linear differential equations

Section \ref{equ sol}. A possible strategy for equation solving

\begin{center}
Part IV. Other topics and future research
\end{center}

Section \ref{Duality}. Dualities of operator semigroups

Section \ref{Transitive}. Fixed sets and transitive actions of $\operatorname{End}_{T}(S)$ and $\operatorname{Aut}_{T}(S)$

Section \ref{Two questions}. Two questions on Galois $T$-extensions and normal subgroups of Galois $T$-groups

Section \ref{transcendental}. $\mathcal{F}$-transcendental elements and $\mathcal{F}$-transcendental subsets

Section \ref{first iso}. A generalized first isomorphism theorem

Section \ref{App to topo}. On topological spaces

Section \ref{App to dyn}. On dynamical systems

Section \ref{Other topics}. Other topics for future research

$ $

The contents of Part I are roughly described section by section as follows.

Let $T$ be a set of functions from a set $D$ to $D$. If the composite of any two elements in $T$ still lies in $T$, then we call $T$ an \emph{operator semigroup} on $D$ (\textit{cf.} Definition \ref{Operator semigroup}). Let $U\subseteq D$. Then we call $\bigcup\nolimits_{f\in T}{\operatorname{Im}(f{{|}_{U}})}$ the \emph{$T$-space} generated by $U$, where $\operatorname{Im}(f{{|}_{U}})$ denotes the image of the restriction of $f$ to $U$ (\textit{cf.} Definition \ref{$T$-space}). $T$-spaces represent mathematical objects whose structures are preserved under generalized morphisms. A critical notion in this theory is \emph{$T$-morphism}, which is defined as a map $\sigma $ from a $T$-space $S$ to a $T$-space such that $\sigma $ commutes with every element of $T$; that is, $\forall a\in S$ and $f\in T$, $\sigma (f(a))=f(\sigma (a))$ (\textit{cf.} Definition \ref{1.3.1}). From these simple definitions much follows. For example, any ring homomorphism between $B[u]$ and $B[v]$ with field $B$ fixed pointwisely can be characterized in terms of a $T$-morphism (\textit{cf.} Proposition \ref{1.3.4}), and a map from a topological space to itself is continuous if and only if the map induces a $T$-morphism (\textit{cf.} Proposition \ref{1.3.7}). In Subsection \ref{2 $T$-isomorphisms}, we show that the inverse of any bijective $T$-morphism from a $T$-space $S_1$ to a $T$-space $S_2$ is a $T$-morphism from $S_2$ to $S_1$, and hence a $T$-isomorphism can be defined as a bijective $T$-morphism.

In Section \ref{I Galois corr}, we establish the Galois correspondence between the Galois groups of a $T$-space $S$ and the fixed subsets of $S$ under the actions of the $T$-automorphisms of $S$ (\textit{cf.} Corollary \ref{2.2.6}). Moreover, we shall find the Galois correspondence between the “Galois monoids” of a $T$-space $S$ and the fixed subsets of $S$ under the actions of the $T$-endomorphisms of $S$ (\textit{cf.} Corollary \ref{2.2.5}). 

To better understand these Galois correspondences, in Section \ref{I Lattice structures}, we study the lattice structures of those objects which arise in Corollaries \ref{2.2.5} and \ref{2.2.6}. Results in Section \ref{I Lattice structures} will be employed in Section \ref{I Topologies employed}.

In the fundamental theorem for infinite algebraic field extensions (resp. for differential field extensions), the Galois correspondences are characterized with Krull topology (resp. Zariski topology) (see e.g. \cite{1,3,4,5,9}). In Section \ref{I Topologies employed}, we shall see when and how we can characterize Galois correspondences with topology. 

Section \ref{theta-mor and theta-iso} introduces the notions of $\theta $-morphism and $\theta $-isomorphism. As a generalization of the notion of $T$-morphism, a \emph{$\theta $-morphism} (\textit{cf.} Definitions \ref{5.3.1} and \ref{5.3.5}) is from a ${{T}_{1}}$-space to a ${{T}_{2}}$-space, where both ${{T}_{1}}$ and ${{T}_{2}}$ are operator semigroups and $\theta \subseteq {{T}_{1}}\times {{T}_{2}}$. We shall find that some more familiar mathematical notions can be characterized by $\theta $-morphisms (\textit{cf.} Propositions \ref{5.3.2}, \ref{5.3.3} and \ref{5.3.6}). In Subsection \ref{6 theta-isomorphisms}, we show that the inverse of any bijective $\theta $-morphism from a $T_1$-space $S_1$ to a $T_2$-space $S_2$ is a $\theta ^{-1}$-morphism from $S_2$ to $S_1$, and thus a $\theta$-isomorphism can be defined as a bijective $\theta$-morphism.

Main results in Section \ref{Cons of $T$-mor and theta-mor} are about constructions of $T$-morphisms and $\theta $-morphisms. Roughly speaking, for a $T$-morphism or $\theta $-morphism from a $T$-space $S$, Subsections \ref{$T$-morphisms from u}, \ref{$T$-morphisms from big u} and \ref{I Constructions of $theta $-morphisms} discuss how to construct the morphism by a map from a set $U$ which generates $S$ (i.e. $S=\bigcup\nolimits_{f\in T}{\operatorname{Im}(f{{|}_{U}})}$). Moreover, in Subsection \ref{I Another construction of $T$-morphisms}, we shall introduce a construction of $T$-morphisms in terms of topology. 

$ $

The contents of Part II are described as follows.

To make our theory more general, in Section \ref{Basic notions for multivariable}, we generalize the notion of operator semigroup to incorporate functions of more than one variable (\textit{cf.} Definition \ref{9.1.2}). And correspondingly, the concepts of $T$-spaces, $T$-morphisms and $\theta $-morphisms are generalized (\textit{cf.} Definitions \ref{9.2.1}, \ref{9.3.1}, and \ref{9.4.2}). Then we shall find that ring homomorphisms, module homomorphisms, and group homomorphisms can be characterized in terms of $T$-morphisms (\textit{cf.} Propositions \ref{9.3.4} to \ref{9.3.7}) or $\theta $-morphisms (\textit{cf.} Propositions \ref{9.4.4} to \ref{9.4.8}).

To further generalize our theory, in Section \ref{Basic notions for partial}, we allow operators in $T$ to be partial functions (\textit{cf.} Definition \ref{10.1.3}). And correspondingly, the notions of $T$-spaces, $T$-morphisms and $\theta $-morphisms are generalized (\textit{cf.} Definitions \ref{10.2.1}, \ref{10.3.1}, and \ref{10.4.2}). Because operators in $T$ are now allowed to be partial functions, we can generate fields, differential fields, and categories as $T$-spaces in a natural way (\textit{cf.} Examples \ref{10.2.2} to \ref{10.2.4}). Then we shall find that ring homomorphisms between fields, differential ring homomorphisms between differential fields, and covariant functors between categories can be characterized in terms of $T$-morphisms (\textit{cf.} Propositions \ref{10.3.4} and \ref{10.3.6}) or $\theta $-morphisms (\textit{cf.} Propositions \ref{10.4.4} and \ref{10.4.5}).

For the case of (partial) functions of more than one variable, results obtained in Sections \ref{Basic notions and properties}, \ref{I Galois corr}, \ref{I Lattice structures}, \ref{I Topologies employed} and \ref{Cons of $T$-mor and theta-mor} are generalized in Sections \ref{12 Basic properties}, \ref{II Gal corr}, \ref{II latt struc}, \ref{II topologies employed} and \ref{II Cons of $T$-mor and}, respectively. (Note that results obtained in Section \ref{theta-mor and theta-iso} are generalized in Sections \ref{Basic notions for multivariable} and \ref{Basic notions for partial}.)
 
$ $

Part III addresses the following issues.

A main goal of the classical Galois theory is to study the solvability by radicals of polynomial equations. In Section \ref{poly equ}, however, we shall introduce another understanding of solvability of polynomial equations, which somehow corresponds to a composition series of the Galois group of the polynomial. And in Subsection \ref{Formulas for the roots} we shall explain why we introduce this notion of solvability. Roughly speaking, the central idea is that a separable polynomial $p(x)$ can be “decomposed” into $n$ polynomials ${{p}_{1}}(x),\cdots ,{{p}_{n}}(x)$ such that the Galois group of each ${{p}_{i}}(x)$ is simple and any root of $p(x)$ can be expressed in terms of the roots of ${{p}_{1}}(x),\cdots ,{{p}_{n}}(x)$ (\textit{cf.} Corollary \ref{17.2.5}).

Analogously, in Section \ref{diff equ}, we shall introduce a new interpretation of solvability of homogeneous linear differential equations, which somehow corresponds to a “composition Zariski-closed series” of the differential Galois group of the equation. In Subsection \ref{Formulas for the solutions} we shall explain why we introduce this notion of solvability. As what we did in Section \ref{poly equ}, our main approach is to “decompose” a homogeneous linear differential equation $L(Y)=0$ into $n$ homogeneous linear differential equations ${{L}_{1}}(Y)=0,\cdots ,{{L}_{n}}(Y)=0$ so that any solution of $L(Y)=0$ can be expressed in terms of the solutions of ${{L}_{1}}(Y)=0,\cdots ,{{L}_{n}}(Y)=0$ (\textit{cf.} Corollary \ref{18.3.5}).

Section \ref{equ sol} generalizes some results in Sections \ref{poly equ} and \ref{diff equ}, and it roughly describes in terms of our theory a possible strategy for equation solving.

$ $

In Part IV we give some more results which we think to be important or deserve deeper research.

Section \ref{Duality} is about dualities of operator semigroups. Let $T$ be an operator semigroup and let $S$ be a $T$-space. We denote by $\operatorname{End}_{T}(S)$ (resp. $\operatorname{Aut}_{T}(S)$) the set of all $T$-endomorphisms of $S$ (resp. the set of all $T$-automorphisms of $S$). Then, roughly speaking, there is a duality between $\operatorname{End}_{T}(S)$ and the maximum operator semigroup which “accommodates” $\operatorname{End}_{T}(S)$ (i.e. $\forall \sigma \in \operatorname{End}_{T}(S)$, $\sigma$ would still be a $T$-endomorphism of $S$ if $T$ were extended to the maximum). Analogously, there is a duality between $\operatorname{Aut}_{T}(S)$ and the maximum operator semigroup which “accommodates” $\operatorname{Aut}_{T}(S)$.

In Section \ref{Transitive}, we shall talk about fixed subsets of a $T$-space $S$ under the actions of $\operatorname{End}_{T}(S)$ and $\operatorname{Aut}_{T}(S)$. And we shall obtain some transitive properties of the actions of $\operatorname{End}_{T}(S)$ and $\operatorname{Aut}_{T}(S)$ on $S$, which are analogues of the well-known fact that the Galois group of an irreducible polynomial acts transitively on its roots.

Section \ref{Two questions} addresses an analogue of the issue which is normally the last part of the fundamental theorem of a Galois theory: the correspondence between the normal subgroups of the Galois group and the Galois extensions of the base field.

In Section \ref{transcendental}, the algebraic notions of transcendental elements over a field and purely transcendental field extensions are generalized.

In Section \ref{first iso}, we shall obtain an isomorphism theorem which generalizes the first isomorphism theorems (for groups, rings, and modules).

In Section \ref{App to topo}, we shall introduce for topological spaces some notions and properties related to our theory. And we shall employ Corollaries \ref{2.2.5} and \ref{2.2.6} to obtain Galois correspondences on topological spaces.

Let $(M,S,\Phi )$ be a dynamical system, where $M$ is a monoid, $S$ is the phase space and $\Phi$ is the evolution function. In Section \ref{App to dyn}, it is shown that $(M,S,\Phi )$ induces in a natural way an operator semigroup $T$ on $S$ such that $T$ contains the identity function and $S$ is a $T$-space. Conversely, for any operator semigroup $T$ which contains the identity function, any $T$-space $S$ induces a dynamical system $(T,S,\Phi )$ in a natural way. Hence we can apply our theory for operator semigroups to dynamical systems. In particular, we shall obtain Galois correspondences on dynamical systems.

Section \ref{Other topics} gives some more suggestions on future research.

$ $

\addcontentsline{toc}{section}{Part I. \textbf{Theory for the case of unary functions}} 

\begin{center}
Part I. THEORY FOR THE CASE OF UNARY FUNCTIONS
\end{center}

In Part I of the article, we introduce the generalized morphisms where all functions involved are unary. The generalized morphisms preserve structures of mathematical objects which we call $T$-spaces, where $T$ stands for a semigroup of functions with composition of functions as the binary operation. 
\section{Basic notions and properties} \label{Basic notions and properties}
\subsection{Operator semigroup $T$} \label{2.1 operator semigrps}
Sets which are closed under some operations are ubiquitous in mathematics. The following notion is related to this observation.
\begin{definition}
\label{Operator semigroup}
Let $D$ be a set and let $T$ be a set of functions from $D$ to $D$. If $\forall f,g\in T$, the composite $f\circ g\in T$, then we call $T$ an \emph{operator semigroup} on $D$, and we call $D$ the \emph{domain} of $T$.
\end{definition}
\begin{remark}
It is clear that $T$, possibly empty, is a semigroup with composition of functions as the binary operation.
\end{remark}
\begin{example}
\label{T on field--1}
Let $B$ be a subfield of a field $F$ and let $B[x]$ be the ring of all polynomials (in one variable) over $B$. Let 
\begin{center}
   $T=\{{f}^*:F\to F  \text{ given by } a\mapsto f(a)\, |\,f(x)\in B[x]\}.$ 
\end{center}
(In particular, $f^*$ is a constant (polynomial) function if $f$ is a constant polynomial.) Then it is not hard to see that $T$ is an operator semigroup on $F$. Note that the map which induces $T$, i.e. $\tau :B[x]\to T$ given by $f\mapsto {{f}^*}$, is surjective, but it is not necessarily injective.
\end{example}
\begin{remark}
    Of course, there are other ways to define an operator semigroup on a field $F$. For example, let $T_1$ consist of only the identity function on $F$ and let $T_2$ be the set of all ring homomorphisms from $F$ to $F$, then both $T_1$ and $T_2$ are operator semigroups on $F$. But unlike $T$ in Example \ref{T on field--1}, neither $T_1$ nor $T_2$ could lead us to any desirable result (in the remaining part of the article). 
    
    In fact and roughly speaking, to obtain desirable results, an operator semigroup should represent the intrinsic operations on the mathematical object involved.
\end{remark}
\begin{notation}
    Throughout the paper, we denote by Id the identity function (on some set which is clear from the context).
\end{notation} 
As Example \ref{T on field--1}, the following two will lead us to desirable results.
\begin{example}
\label{T on topo--1}
Let $X$ be a topological space and let $\mathcal{P}(X)$ be the power set of $X$. Let
\[
T=\{\operatorname{Id}:\mathcal{P}(X)\to \mathcal{P}(X), \operatorname{Cl}:\mathcal{P}(X)\to \mathcal{P}(X) \text{ given by } A\mapsto \overline{A}\},
\]
where $\overline{A}$ is the closure of $A$. Then $T$ has only two elements and it is an operator semigroup on $\mathcal{P}(X)$.
\end{example}
\begin{example}
\label{T on diff ring--1}
    Recall that a differential ring $R$ is a commutative ring with identity endowed with a derivation $\partial $. Then $\{{{\partial }^{n}}\,|\,n\in {{\mathbb{N}}_{0}}\}$, where ${{\partial }^{0}}:=\operatorname{Id}$, is an operator semigroup on $R$.
\end{example}
\begin{definition}
\label{generated operator semigroup}
    Let $D$ be a set and let $G$ be a set of functions from $D$ to $D$. We call the intersection of all operator semigroups (on $D$) containing $G$ the operator semigroup on $D$ \emph{generated by} $G$, and we denote it by $\left\langle G \right\rangle $. If $G=\{g\}$, then we may denote $\left\langle G \right\rangle $ by $\left\langle g \right\rangle $ for brevity.
\end{definition}
\begin{remark}
    $\left\langle G \right\rangle $ is the smallest operator semigroup (on $D$) which contains $G$.
\end{remark}

\subsection{$T$-spaces and quasi-$T$-spaces} \label{$T$-spaces and quasi-$T$-spaces}
Operator semigroups are used to generate $T$-spaces defined as follows.
\begin{definition} \label{$T$-space}
    Let $T$ be an operator semigroup on $D$ and let $U\subseteq D$. We call $S:=\{ f(u)\, |\,f\in T,u\in U $\} the \emph{$T$-space} \emph{generated by} $U$ and denote $S$ by ${{\left\langle U \right\rangle }_{T}}$. If $U=\{u\}$, then we also denote ${{\left\langle \{u\} \right\rangle }_{T}}$ by ${{\left\langle u \right\rangle }_{T}}$.

Moreover, if $B\subseteq S$ is also a $T$-space generated by some subset of $D$, then we say that $B$ is a \emph{$T$-subspace} of $S$ and write $B\le S$ or $S\ge B$.
\end{definition}
\begin{remark}
    \item[(1)] $S$ may have generators other than $U$, and we may just call $S$ a $T$-space.
    \item[(2)] Trivially, ${{\left\langle \emptyset  \right\rangle }_{T}}=\emptyset $ is a $T$-space and ${{\left\langle D \right\rangle }_{\emptyset }}=\emptyset $ is a $\emptyset$-space. 
\end{remark}  
$T$-spaces are mathematical objects whose structures are preserved under the generalized morphisms (to be defined later), and this claim will be shown later on the following examples.
\begin{example}
\label{$T$-space for field--1}
    Let $B$ be a subfield of a field $F$. Let $T$ be the operator semigroup defined in Example \ref{T on field--1}. Then $\forall a\in F$, ${{\left\langle a \right\rangle }_{T}}$ is the ring $B[a]$. In particular, if $F/B$ is an algebraic field extension, then $\forall a\in F$, ${{\left\langle a \right\rangle }_{T}}$ is the field $B(a)$.
\end{example}
\begin{example}
\label{$T$-space for topo--1}
    Let $X$ be a topological space and let $\mathcal{P}(X)$ be the power set of $X$. If $T$ is defined as in Example \ref{T on topo--1}, then ${{\left\langle \mathcal{P}(X) \right\rangle }_{T}}=\mathcal{P}(X)$ is a $T$-space.
\end{example}
\begin{example}
\label{$T$-space for diff ring--1}
    Let $R$ be a differential ring with a derivation $\partial $. Let $T$ be the operator semigroup defined in Example \ref{T on diff ring--1}. Recall that an ideal $I$ of $R$ is a differential ideal if $a\in I$ implies that $\partial (a)\in I$. Thus an ideal $I$ of $R$ is a differential ideal of $R$ if and only if the $T$-space ${{\left\langle I \right\rangle }_{T}}=I$. We will talk more about differential rings in Sections \ref{theta-mor and theta-iso}, \ref{Cons of $T$-mor and theta-mor} and \ref{Basic notions for partial}.
\end{example}
We are going to develop some properties of the new notions. In Sections \ref{Basic notions and properties} to \ref{Cons of $T$-mor and theta-mor}, unless otherwise specified, \textbf{$D$ denotes a set and $T$ denotes an operator semigroup on $D$}. 

First, the following example is used to prove the proposition right after it.
\begin{example}
\label{$T$-space generated by 1}
    Let $T=\{f:{{\mathbb{Z}}^{+}}\to {{\mathbb{Z}}^{+}}$ given by $x\mapsto x+k\, |\,k\in {{\mathbb{Z}}^{+}}\}$. Then $T$ is an operator semigroup on ${{\mathbb{Z}}^{+}}$, and ${{\left\langle 1 \right\rangle }_{T}}=\{2,3,4,\cdots \}.$
\end{example}
\begin{proposition}
\label{4 facts}
\begin{enumerate}
    \item [(a)] A generator set of a $T$-space may have no intersection with the $T$-space.
    \item [(b)] It is possible that no subset of a $T$-space may generate the $T$-space.
    \item [(c)] $D$ is not necessarily a $T$-space.
     \item [(d)] A subset $S$ of $D$ is not necessarily a $T$-space even if it satisfies: 
$f(a)\in S$, $\forall a\in S$ and $f\in T$.
\end{enumerate}
\end{proposition}
\begin{proof}
For (a): In Example \ref{$T$-space generated by 1}, 
\[\{1\}\bigcap {{\left\langle 1 \right\rangle }_{T}}=\{1\}\bigcap \{2,3,4,\cdots \}=\emptyset .\] 

For (b): In Example \ref{$T$-space generated by 1}, no subset of ${{\left\langle 1 \right\rangle }_{T}}$ may generate ${{\left\langle 1 \right\rangle }_{T}}$. 

For (c): In Example \ref{$T$-space generated by 1}, $1\notin {{\left\langle U \right\rangle }_{T}},\forall U\subseteq {\mathbb{Z}}^{+}$, and hence ${{\mathbb{Z}}^{+}}$ is not a $T$-space.

For (d): In Example \ref{$T$-space generated by 1}, ${{\mathbb{Z}}^{+}}$ satisfies: 
$f(a)\in {{\mathbb{Z}}^{+}}$, $\forall a\in {{\mathbb{Z}}^{+}}$ and $f\in T$,	
but ${{\mathbb{Z}}^{+}}$ is not a $T$-space.
\end{proof}
However, any $T$-space $S$ must satisfy: $f(a)\in S$, $\forall a\in S$ and $f\in T$, as shown below.
\begin{proposition}
\label{<S> contained in S}
    Let $S$ be a $T$-space. Then $\forall a\in S$ and $f\in T$, $f(a)\in S$. Hence $\forall A\subseteq S$, ${{\left\langle A \right\rangle }_{T}}\le S$, i.e. ${{\left\langle A \right\rangle }_{T}}$ is $T$-subspace of $S$ (Definition \ref{$T$-space}). In particular, ${{\left\langle S \right\rangle }_{T}}\le S$. 
\end{proposition}   
\begin{proof}
    By Definition \ref{$T$-space}, $\exists U\subseteq D$ such that ${{\left\langle U \right\rangle }_{T}}=S$. Let $a\in S$. Then $\exists g\in T$ and $u\in U$ such that $g(u)=a$. Let $f\in T$. Then by Definition \ref{Operator semigroup}, $f\circ g\in T$, and thus $f(a)=f(g(u))\in {{\left\langle U \right\rangle }_{T}}=S$. Hence $\forall a\in S$ and $f\in T$, $f(a)\in S$. Then $\forall A\subseteq S$, ${{\left\langle A \right\rangle }_{T}}\subseteq S$, and hence by Definition \ref{$T$-space}, ${{\left\langle A \right\rangle }_{T}}\le S$.
\end{proof}
In light of Proposition \ref{<S> contained in S} and (d) in Proposition \ref{4 facts}, we introduce the following concept, whose use will be clear later. 
\begin{definition}
\label{quasi-T-space}
    Let $S\subseteq D$. If ${{\left\langle S \right\rangle }_{T}}\subseteq S$, then we call $S$ a \emph{quasi-$T$-space}. Moreover, if $B\subseteq S$ and both $B$ and $S$ are quasi-$T$-spaces, then we call $B$ a \emph{quasi-$T$-subspace} of $S$ and write $B\le _qS$ or $S\ge _qB$.
\end{definition}
By Proposition \ref{<S> contained in S}, the following is immediate.
\begin{proposition} \label{$T$-space is a quasi-$T$-space}
     Any $T$-space is a quasi-$T$-space. 
\end{proposition}
However, the converse of Proposition \ref{$T$-space is a quasi-$T$-space} is not always true, as is implied by (d) in Proposition \ref{4 facts}, unless we add a condition as follows.
\begin{proposition}
\label{1.2.9}
Let $S$ be a quasi-$T$-space. If $S\subseteq {{\left\langle S \right\rangle }_{T}}$, then $S$ is a $T$-space.
\end{proposition}
\begin{remark}
    If $\operatorname{Id}\in T$, then $S\subseteq {{\left\langle S \right\rangle }_{T}}$.
\end{remark}
\begin{proof}
    If $S\subseteq {{\left\langle S \right\rangle }_{T}}$, then ${{\left\langle S \right\rangle }_{T}}=S$ because ${{\left\langle S \right\rangle }_{T}}\subseteq S$ (by Definition \ref{quasi-T-space}), and hence $S$ is a $T$-space (by Definition \ref{$T$-space}).
\end{proof}
Proposition \ref{1.2.9} implies that, if $\operatorname{Id}\in T$ and $S$ is a quasi-$T$-space, then $S$ is a $T$-space, and hence in this case $T$-space and quasi-$T$-space are actually the same notion. In fact, we shall see that $\operatorname{Id}\in T$ in most of our examples.
\begin{notation}
    In this article, $I$ denotes an index set unless otherwise specified.
\end{notation}
\begin{proposition}
\label{1.2.10}
    The intersection of any family of quasi-$T$-spaces is a quasi-$T$-space.
\end{proposition}
\begin{proof}
    Let $\{{{S}_{i}}\,|\,i\in I\}$ be a family of quasi-$T$-spaces and let $S=\bigcap\nolimits_{i\in I}{{{S}_{i}}}$. Then $\forall i\in I$, ${{\left\langle S \right\rangle }_{T}}\subseteq {{\left\langle {{S}_{i}} \right\rangle }_{T}}$ (since $S\subseteq {{S}_{i}}$). And by Definition \ref{quasi-T-space}, $\forall i\in I$, ${{\left\langle {{S}_{i}} \right\rangle }_{T}}\subseteq {{S}_{i}}$. Hence $\forall i\in I$, ${{\left\langle S \right\rangle }_{T}}\subseteq {{S}_{i}}$. Therefore, ${{\left\langle S \right\rangle }_{T}}\subseteq \bigcap\nolimits_{i\in I}{{{S}_{i}}}=S$.
\end{proof}   
\begin{proposition}
    \label{1.2.11}
    The union of any family of quasi-$T$-spaces is a quasi-$T$-space.
\end{proposition}
\begin{proof}
    Let $\{{{S}_{i}}\,|\,i\in I\}$ be a collection of quasi-$T$-spaces and let $S=\bigcup\nolimits_{i\in I}{{{S}_{i}}}$. Then by Definitions \ref{$T$-space} and \ref{quasi-T-space}, 
    \[{{\left\langle S \right\rangle }_{T}}={{\left\langle \bigcup\nolimits_{i\in I}{{{S}_{i}}} \right\rangle }_{T}}=\bigcup\nolimits_{i\in I}{{{\left\langle {{S}_{i}} \right\rangle }_{T}}}\subseteq \bigcup\nolimits_{i\in I}{{{S}_{i}}}=S.\]   
\end{proof}
 For $T$-spaces, we have
 \begin{proposition}
     \label{1.2.12}
     The union of any family of $T$-spaces is a $T$-space.
 \end{proposition}
\begin{proof}
    Let $\{{{S}_{i}}\,|\,i\in I\}$ be a collection of $T$-spaces and let $S=\bigcup\nolimits_{i\in I}{{{S}_{i}}}$. By Definition \ref{$T$-space}, $\forall i\in I$, we can suppose ${{S}_{i}}={{\left\langle {{U}_{i}} \right\rangle }_{T}}$, where $U_i \subseteq D$. Then $S=\bigcup\nolimits_{i\in I}{{{\left\langle {{U}_{i}} \right\rangle }_{T}}}={{\left\langle \bigcup\nolimits_{i\in I}{{{U}_{i}}} \right\rangle }_{T}}$ is a $T$-space.
\end{proof}
	However, the intersection of a set of $T$-spaces is not necessarily a $T$-space, as shown below.
\begin{example}
    \label{1.2.13}
    Let 
    \begin{center}
     $D=\{{{a}_{0}}+{{a}_{1}}\pi +{{a}_{2}}{{\pi }^{2}}+\cdots +{{a}_{n}}{{\pi }^{n}}\, |\,n\in {{\mathbb{N}}_{0}},{{a}_{i}}\in {{\mathbb{N}}_{0}},\forall i=0,\cdots ,n\}$.
    \end{center}
    Let $f:D\to D$ be given by $x\mapsto \pi +x$, let $g:D\to D$ be given by $x\mapsto \pi x$, let $T=\left\langle \{f,g\} \right\rangle $ (Definition \ref{generated operator semigroup}), and let $A={{\left\langle 0 \right\rangle }_{T}}\bigcap {{\left\langle 1 \right\rangle }_{T}}$. Then each element of $A$ is nonzero (because $0\notin {{\left\langle 1 \right\rangle }_{T}}$) and has the form ${{a}_{1}}\pi +{{a}_{2}}{{\pi }^{2}}+\cdots +{{a}_{n}}{{\pi }^{n}}$ where $n\in {{\mathbb{Z}}^{+}}$ and each ${{a}_{i}}\in {{\mathbb{N}}_{0}}$ (because each element of ${{\left\langle 0 \right\rangle }_{T}}$ has the form). Assume that $A$ is a $T$-space ${{\left\langle U \right\rangle }_{T}}$, where $U\subseteq D$. Because $\pi \in A$, we can tell that either $0\in U$ or $1\in U$. However, it follows that either $0\in A$ or $\pi +1\in A$, a contradiction.
\end{example}
Nevertheless, we have
\begin{proposition}
    \label{1.2.14}
    Let $S$ be the intersection of any family of $T$-spaces. Then $S$ is a quasi-$T$-space. Moreover, if $\operatorname{Id}\in T$ or more generally, $S\subseteq {{\left\langle S \right\rangle }_{T}}$, then $S$ is a $T$-space. 
\end{proposition}
\begin{proof}
    $S$ is a quasi-$T$-space by Propositions \ref{$T$-space is a quasi-$T$-space} and \ref{1.2.10}. If $\operatorname{Id}\in T$ or $S\subseteq {{\left\langle S \right\rangle }_{T}}$, then by Proposition \ref{1.2.9}, $S$ is a $T$-space.
\end{proof}  

\subsection{$T$-morphisms} \label{$T$-morphisms}
A generalized morphism which we call a $T$-morphism between $T$-spaces is only required to satisfy a simple but important condition. 
\begin{definition}
    \label{1.3.1}
    Let $T$ be an operator semigroup and let $\sigma $ be a map from a $T$-space $S$ to a $T$-space. If $\forall f\in T$, $\sigma $ commutes with $f$ on $S$,  i.e. 
$\sigma (f(a))=f(\sigma (a))$, $\forall a\in S$ and $f\in T$,
then $\sigma $ is called a \emph{$T$-morphism}.
\end{definition}
\begin{remark}
\begin{enumerate}
\item By Proposition \ref{<S> contained in S}, $\forall a\in S$ and $f\in T$, $f(a)\in S$, and hence $\sigma (f(a))$ is well-defined.
\item Note that if $S={{\left\langle U \right\rangle }_{T}}$ and $u\in U$, then $\sigma (u)$ is not defined unless $u\in {{\left\langle U \right\rangle }_{T}}$.
\item We regard the empty map $\emptyset \to \emptyset $ as a $T$-morphism from the $T$-space $\emptyset$ to $\emptyset$.
\end{enumerate}
\end{remark}
The following four results show that the notion of $T$-morphism can characterize some familiar notions in mathematics.
\begin{proposition}
\label{1.3.4}
    Let $B$ be a subfield of a field $F$. Let $T$ be the operator semigroup defined in Example \ref{T on field--1}. Then $\forall u,v\in F$, a map $\sigma :{{\left\langle u \right\rangle }_{T}}\to {{\left\langle v \right\rangle }_{T}}$ is a ring homomorphism with $B$ fixed pointwisely if and only if $\sigma $ is a $T$-morphism.
\end{proposition}
\begin{proof}
    Clearly, ${{\left\langle u \right\rangle }_{T}}$ and ${{\left\langle v \right\rangle }_{T}}$ are the rings $B[u]$ and $B[v]$, respectively.
    
1.\emph{ Sufficiency }

Suppose that $\sigma $ is a $T$-morphism. Then by Definition \ref{1.3.1}, $\sigma (f(a))=f(\sigma (a))$, $\forall f\in T$ and $a\in {{\left\langle u \right\rangle }_{T}}=B[u]$. 

Considering the case where $f\in T$ is a constant polynomial function given by $a\mapsto c (\in B),\forall a\in F$, we can tell that $\forall c\in B$, $\sigma (c)(=\sigma (f(a))=f(\sigma (a)))=c$, i.e. $B$ is fixed pointwisely under $\sigma $. 

Let $a,b\in {{\left\langle u \right\rangle }_{T}}=B[u]$. Then $\exists p(x),q(x)\in B[x]$ with $p(u)=a$ and $q(u)=b$. Hence 

$\sigma (a+b)$ 

$=\sigma (p(u)+q(u))$ 

$=\sigma ((p+q)(u))$ 

$=\sigma (\tau (p+q)(u))$ ($\tau :B[x]\to T$ is defined in Example \ref{T on field--1})

$=\tau (p+q)(\sigma (u))$ (since $\sigma $ is a $T$-morphism)

$=(p+q)(\sigma (u))$ 

$=p(\sigma (u))+q(\sigma (u))$ 

$=\tau (p)(\sigma (u))+\tau (q)(\sigma (u))$

$=\sigma (\tau (p)(u))+\sigma (\tau (q)(u))$ (since $\sigma $ is a $T$-morphism)

$=\sigma (p(u))+\sigma (q(u))$ 

$=\sigma (a)+\sigma (b)$. 

And

$\sigma (ab)$

$=\sigma (p(u)q(u))$ 

$=\sigma ((pq)(u))$ 

$=\sigma (\tau (pq)(u))$ 

$=\tau (pq)(\sigma (u))$ (since $\sigma $ is a $T$-morphism)

$=(pq)(\sigma (u))$ 

$=p(\sigma (u))q(\sigma (u))$

$=\tau (p)(\sigma (u))\tau (q)(\sigma (u))$

$=\sigma (\tau (p)(u))\sigma (\tau (q)(u))$ (since $\sigma $ is a $T$-morphism)

$=\sigma (p(u))\sigma (q(u))$ 

$=\sigma (a)\sigma (b)$.

Therefore, $\sigma $ is a ring homomorphism with $B$ fixed pointwisely.

2.\emph{ Necessity }

Suppose that $\sigma $ is a ring homomorphism with $B$ fixed pointwisely. Then 

$\forall c\in B$, $\sigma (c)=c$, and

$\forall a,b\in {{\left\langle u \right\rangle }_{T}}=B[u]$, $\sigma (a+b)=\sigma (a)+\sigma (b)$ and $\sigma (ab)=\sigma (a)\sigma (b)$.

Let $z\in {{\left\langle u \right\rangle }_{T}}=B[u]$ and let $f\in T$. Because $f$ is a polynomial function, it is not hard to see that $\sigma (f(z))=f(\sigma (z))$, as desired.
\end{proof}
\begin{remark}
    Proposition \ref{1.3.4} justifies the definition of $T$ in Example \ref{T on field--1}: we could not obtain such a desirable result if $T$ were defined otherwise.
\end{remark}
 Then for ring isomorphisms, we have
\begin{corollary}
\label{1.3.5}
    Let $B$ be a subfield of a field $F$, let $T$ be the operator semigroup defined in Example \ref{T on field--1}, and let $u,v\in F$ be algebraic over $B$ with $\deg (\min (B,u))=\deg (\min (B,v))$, where $\deg (\min (B,*))$ denotes the degree of the minimal polynomial of $*$ over $B$. Then a map $\sigma :{{\left\langle u \right\rangle }_{T}}\to {{\left\langle v \right\rangle }_{T}}$ is a ring isomorphism with $B$ fixed pointwisely if and only if it is a $T$-morphism.
\end{corollary}
\begin{proof}
    By Proposition \ref{1.3.4}, $\sigma $ is a ring homomorphism with $B$ fixed pointwisely if and only if it is a $T$-morphism. So it suffices to show that the ring homomorphism $\sigma $ is bijective.

    Since $u$ is algebraic over $B$, ${{\left\langle u \right\rangle }_{T}}\,(=B[u]=B(u))$ is a field. It follows that the ring homomorphism $\sigma $ must be injective. Moreover, since $\deg (\min (B,u))=\deg (\min (B,v))$, $\sigma $ must be surjective.
\end{proof}
Recall that every finite field extension is algebraic and every finite separable field extension is simple. So by Corollary \ref{1.3.5}, we immediately have the following, which is for later use.
\begin{corollary}
\label{1.3.6}
    Let $F/B$ be a finite separable field extension. Let $T$ be the operator semigroup defined in Example \ref{T on field--1}. Then $F$ is a $T$-space generated by a single element of $F$. Moreover, a map $\sigma :F\to F$ is a field automorphism with $B$ fixed pointwisely if and only if it is a $T$-morphism.
\end{corollary}
Note that every ring homomorphism discussed above involves only one polynomial ring ($B[x]$ in Example \ref{T on field--1}). We will treat ring homomorphisms which involve two (distinct) polynomial rings in Subsection \ref{I $theta$-morphisms}. Moreover, ring homomorphisms involving polynomial rings in more than one variable will be treated in Sections \ref{Basic notions for multivariable} and \ref{Basic notions for partial}.

	The following may be a little surprising.
\begin{proposition}
\label{1.3.7}
      Let $X$ be a topological space, let $\mathcal{P}(X)$ be the power set of $X$, and let $T=\{\operatorname{Id}$, $\operatorname{Cl}:\mathcal{P}(X)\to \mathcal{P}(X)$ given by $A\mapsto \overline{A}\}$. 
Let a map $p:X\to X$ induce a map ${p}^*:\mathcal{P}(X)\to \mathcal{P}(X)$ as follows. 

$\forall A\in \mathcal{P}(X)$ that is closed in $X$, let ${p}^*(A)=\overline{p(A)}$, where $p(A)$ $($with a slight abuse of notation$)$ denotes the set $\{p(x)\,|\,x\in A\}$. And $\forall A\in \mathcal{P}(X)$ which is not closed in $X$, let ${p}^*(A)=p(A)$. 

Then $p$ is continuous if and only if ${p}^*$ is a $T$-morphism.
\end{proposition}
\begin{remark}
    If $A=\emptyset$, then $p(A)=\{p(x)\,|\,x\in A\}=\emptyset$, and hence ${p}^*(\emptyset )=\overline{p(\emptyset )}=\emptyset$. 
\end{remark}
\begin{proof}
As in Examples \ref{T on topo--1} and \ref{$T$-space for topo--1}, $T$ is an operator semigroup on $\mathcal{P}(X)$ and ${{\left\langle \mathcal{P}(X) \right\rangle }_{T}}=\mathcal{P}(X)$ is a $T$-space.

1.\emph{ Necessity }

Suppose that $p$ is continuous. To prove that ${p}^*$ is a $T$-morphism, it suffices to show that ${p}^*(f(A))=f({p}^*(A))$, $\forall A\in \mathcal{P}(X)$ and $f\in T$. The equation holds when $f=\operatorname{Id}$. So we only need to show that ${p}^*(\operatorname{Cl}(A))=\operatorname{Cl}({p}^*(A))$, $\forall A\in \mathcal{P}(X)$.

If $A=\overline{A}$, by the definitions of ${p}^*$ and $\operatorname{Cl}$, 
\[\operatorname{Cl}({p}^*(A))=\operatorname{Cl}(\overline{p(A)})=\overline{p(A)}={p}^*(A)={p}^*(\operatorname{Cl}(A)),\]
as desired.

Since $p$ is continuous, $p(\overline{A})\subseteq \overline{p(A)}$ (see e.g. \cite{6}). Hence $\overline{p(\overline{A})}\subseteq \overline{p(A)}$. On the other hand, $\overline{p(\overline{A})}\supseteq \overline{p(A)}$ because $p(\overline{A})\supseteq p(A)$. So $\overline{p(\overline{A})}=\overline{p(A)}$. Thus if $A\ne \overline{A}$, then 
\[\operatorname{Cl}({p}^*(A))=\operatorname{Cl}(p(A))=\overline{p(A)}=\overline{p(\overline{A})}={p}^*(\overline{A})={p}^*(\operatorname{Cl}(A)),\]
as desired.

2.\emph{ Sufficiency }

Suppose that ${p}^*$ is a $T$-morphism. To prove that $p$ is continuous, we only need to show that ${{p}^{-1}}(A)=\overline{{{p}^{-1}}(A)}$, $\forall A\in \mathcal{P}(X)$ with $A=\overline{A}$.

Assume $\exists A\in \mathcal{P}(X)$ such that $A=\overline{A}$ and ${{p}^{-1}}(A)\ne \overline{{{p}^{-1}}(A)}$. Let $B={{p}^{-1}}(A)$. Then $B\ne \overline{B}$ and $p(\overline{B})\nsubseteq A$ (because otherwise $B=\overline{B}$). Hence
\[{{p}^*}(\operatorname{Cl}(B))={{p}^*}(\overline{B})=\overline{p(\overline{B})}\supseteq p(\overline{B})\nsubseteq A,\]
but
\[\operatorname{Cl}({p}^*(B))=\operatorname{Cl}(p(B))=\overline{p(B)}\subseteq \overline{A}=A\] because $p(B)\subseteq A.$
Therefore, $\operatorname{Cl}({p}^*(B))\ne {p}^*(\operatorname{Cl}(B))$, which is contrary to the assumption that ${p}^*$ is a $T$-morphism.	
\end{proof}
The map $p$ in Proposition \ref{1.3.7} is from a topological space to itself. We will treat continuous functions between different topological spaces in Subsection \ref{I $theta$-morphisms}.\\

Several properties of $T$-morphisms are as follows. We shall employ them later.
\begin{proposition} \label{1.3.8}
    Let $\sigma $ be a $T$-morphism from ${{S}_{1}}$ to ${{S}_{2}}$. Then $ \operatorname{Im} \sigma  \le _q{S_2}$, i.e. $\operatorname{Im} \sigma$ is a quasi-$T$-subspace of $S_2$ (Definition \ref{quasi-T-space}). Moreover, if $\operatorname{Id}\in T$ or more generally, $\operatorname{Im}\sigma \subseteq {{\left\langle \operatorname{Im}\sigma  \right\rangle }_{T}}$, then $\operatorname{Im}\sigma \le {{S}_{2}}$, i.e. $\operatorname{Im}\sigma$ is a $T$-subspace of $S_2$ (Definition \ref{$T$-space}).
\end{proposition}
 \begin{proof}
     By Definition \ref{1.3.1}, $\forall a\in {{S}_{1}}$ and $f\in T$, $f(\sigma (a))=\sigma (f(a))\in \operatorname{Im}\sigma $, and thus ${{\left\langle \sigma (a) \right\rangle }_{T}}\subseteq \operatorname{Im}\sigma $. Hence ${{\left\langle \operatorname{Im}\sigma  \right\rangle }_{T}}\subseteq \operatorname{Im}\sigma $, and so by Definition \ref{quasi-T-space}, $\operatorname{Im}\sigma $ is a quasi-$T$-space. 
     
Moreover, if $\operatorname{Id}\in T$ or $\operatorname{Im}\sigma \subseteq {{\left\langle \operatorname{Im}\sigma  \right\rangle }_{T}}$, then by Proposition \ref{1.2.9}, $\operatorname{Im}\sigma $ is a $T$-space, and thus $\operatorname{Im}\sigma \le {{S}_{2}}$.
 \end{proof}
\begin{definition} \label{1.3.2}
    Let $S$ be a $T$-space. We call a $T$-morphism from $S$ to $S$ a \emph{$T$-endomorphism} of $S$, and we denote by End$_{T}(S)$ the collection of all $T$-endomorphisms of $S$.
\end{definition}
\begin{proposition}
\label{1.3.10}
    Let $S$ be a $T$-space. Then $\operatorname{End}_{T}(S)$ constitutes a monoid, which we still denote by $\operatorname{End}_{T}(S)$, with composition of functions as the binary operation.
\end{proposition}
\begin{remark}
    We regard the empty function $\emptyset \to \emptyset $ as the identity function Id$_{\emptyset }$. So $\operatorname{End}_{T}(\emptyset)=\{\text{Id}_{\emptyset}\}$.
\end{remark}
\begin{proof}
    The identity map on $S$ lies in $\operatorname{End}_{T}(S)$, and so we only need to show that $\operatorname{End}_{T}(S)$ is a semigroup with composition of functions as the binary operation.
    
Let ${{\sigma }_{1}},{{\sigma }_{2}}\in \operatorname{End}_{T}(S)$, let $f\in T$ and let $a\in S$. Then 
\[{{\sigma }_{1}}\circ {{\sigma }_{2}}(f(a))={{\sigma }_{1}}(f({{\sigma }_{2}}(a)))=f({{\sigma }_{1}}({{\sigma }_{2}}(a))).\]
So ${{\sigma }_{1}}\circ {{\sigma }_{2}}$ commutes with $f$ on $S$, and hence ${{\sigma }_{1}}\circ {{\sigma }_{2}}\in \operatorname{End}_{T}(S)$. 

Composition of functions is associative. Therefore, $\operatorname{End}_{T}(S)$ is a monoid with composition of functions as the binary operation.
\end{proof}

Then the following is obvious.
\begin{proposition}
    \label{1.3.12}
    Let $S$ be a $T$-space. Then the intersection of any family of submonoids of $\operatorname{End}_{T}(S)$ is again a submonoid of $\operatorname{End}_{T}(S)$.
\end{proposition} 

\subsection{$T$-isomorphisms} \label{2 $T$-isomorphisms}
\begin{definition} \label{$T$-isomorphisms}
    Let $S$, ${{S}_{1}}$ and ${{S}_{2}}$ be $T$-spaces. 
    
    Let $\sigma$ be a $T$-morphism from ${{S}_{1}}$ to ${{S}_{2}}$. If $\sigma$ is bijective, then we call $\sigma$ a \emph{$T$-isomorphism} from ${{S}_{1}}$ to ${{S}_{2}}$. We denote by Iso$_{T}({{S}_{1}},{{S}_{2}})$ the family of all $T$-isomorphisms from ${{S}_{1}}$ to ${{S}_{2}}$. 
    
    Moreover, if $\sigma$ is a $T$-isomorphism from $S$ to itself, then we call $\sigma$ a \emph{$T$-automorphism} of $S$. We write Aut$_{T}(S)$ for the collection of all $T$-automorphisms of $S$.
\end{definition}
To justify Definition \ref{$T$-isomorphisms}, we have
\begin{proposition} \label{1.3.a}
    Let ${{S}_{1}}$ and ${{S}_{2}}$ be $T$-spaces, let $\sigma \in \operatorname{Iso}_{T}({{S}_{1}},{{S}_{2}})$ and let ${{\sigma }^{-1}}$ be the inverse map of $\sigma $. Then ${{\sigma }^{-1}}\in \operatorname{Iso}_{T}({{S}_{2}},{{S}_{1}})$.
\end{proposition}
\begin{proof}
  By Definition \ref{1.3.1}, $\forall f\in T$ and $a\in {{S}_{1}}$, $\sigma (f(a))=f(\sigma (a))$, and so $f(a)={{\sigma }^{-1}}(f(\sigma (a)))$. Then $\forall f\in T$ and $b\in {{S}_{2}}$, $f({{\sigma }^{-1}}(b))={{\sigma }^{-1}}(f(b))$. Hence ${\sigma }^{-1}$ is a $T$-morphism from $S_2$ to $S_1$, and so ${{\sigma }^{-1}}\in \operatorname{Iso}_{T}({{S}_{2}},{{S}_{1}})$.
\end{proof}
\begin{proposition}
    \label{1.3.11}
    Let $S$ be a $T$-space. Then $\operatorname{Aut}_{T}(S)$ constitutes a group, which we still denote by $\operatorname{Aut}_{T}(S)$, with composition of functions as the binary operation.
\end{proposition}
\begin{proof}
    Let $\sigma \in \operatorname{Aut}_{T}(S)$ and let ${{\sigma }^{-1}}$ be the inverse map of $\sigma$. By Proposition \ref{1.3.a}, ${{\sigma }^{-1}}\in \operatorname{Aut}_{T}(S)$.
   
  Let ${{\sigma }_{1}},{{\sigma }_{2}}\in \operatorname{Aut}_{T}(S)$. Then by Proposition \ref{1.3.10}, ${{\sigma }_{1}}\circ {{\sigma }_{2}}\in \operatorname{End}_{T}(S)$. Hence ${{\sigma }_{1}}\circ {{\sigma }_{2}}\in \operatorname{Aut}_{T}(S)$ because both ${{\sigma }_{1}}$ and ${{\sigma }_{2}}$ are bijective. 
  
  Moreover, composition of functions is associative and the identity map on $S$ lies in $\operatorname{Aut}_{T}(S)$. Therefore, $\operatorname{Aut}_{T}(S)$ is a group with composition of functions as the binary operation.
\end{proof}
Then the following, which is for later use, is obvious.
\begin{proposition} \label{intersection of Aut}
    Let $S$ be a $T$-space. Then the intersection of any family of subgroups of $\operatorname{Aut}_{T}(S)$ is again a subgroup of $\operatorname{Aut}_{T}(S)$.
\end{proposition} 

\subsection{Galois $T$-extensions} \label{Galois $T$-extensions}
The remaining part of Section \ref{Basic notions and properties} will be used later to study e.g. Galois correspondences. The concept of Galois extensions in the classical Galois theory is generalized as follows.
\begin{definition}
\label{1.4.1}
    Let $S$ be a $T$-space and let $H\subseteq \operatorname{End}_{T}(S)$. Let
\[{{S}^{H}}=\{a\in S\,|\,\forall \sigma \in H,\sigma (a)=a\}.\] 
We call $S$ a \emph{Galois $T$-extension} of ${{S}^{H}}$, and we say that ${{S}^{H}}$ is \emph{fixed $($pointwisely$)$ under the action of} $H$ on $S$. 

More generally, let $K\subseteq S$. Then we denote by ${{K}^{H}}$ the set $\{a\in K\,|\,\forall \sigma \in H,\sigma (a)=a\}$. 

In particular, if $H=\{\sigma \}$, then we use ${{K}^{\sigma }}$ instead of ${{K}^{\{\sigma \}}}$ for brevity.
\end{definition}
\begin{remark}
    ${{K}^{\emptyset }}=K,\forall K\subseteq S$.
\end{remark}
\begin{proposition}
\label{1.4.2}
     Let $S$ be a $T$-space and let $H$ be a subset of $\operatorname{End}_{T}(S)$. Then ${S^H} \le _qS$. Moreover, if $\operatorname{Id}\in T$ or ${{S}^{H}}\subseteq {{\left\langle {{S}^{H}} \right\rangle }_{T}}$, then ${{S}^{H}}\le S$.
\end{proposition}
\begin{proof}
    Let $a\in {{S}^{H}}$ and $f\in T$. Then $\forall \sigma \in H$, $\sigma (f(a))=f(\sigma (a))=f(a)$ (by Definition \ref{1.4.1}), and so $f(a)\in {{S}^{H}}$. Thus ${{\left\langle a \right\rangle }_{T}}\subseteq {{S}^{H}}$, and so ${{\left\langle {{S}^{H}} \right\rangle }_{T}}\subseteq {{S}^{H}}$. Hence ${{S}^{H}}$ is a quasi-$T$-space by Definition \ref{quasi-T-space}, and thus ${S^H} \le _qS$. 

Moreover, if $\operatorname{Id}\in T$ or ${{S}^{H}}\subseteq {{\left\langle {{S}^{H}} \right\rangle }_{T}}$, then by Proposition \ref{1.2.9}, ${{S}^{H}}$ is a $T$-space, and hence ${{S}^{H}}\le S$.	
\end{proof}
However, in Proposition \ref{1.4.2}, if ${{S}^{H}}\nsubseteq {{\left\langle {{S}^{H}} \right\rangle }_{T}}$, then it is possible that ${{S}^{H}}$ is not a $T$-space. We shall prove this claim right after Example \ref{5.1.12}.

The following is obvious.
\begin{proposition}
    \label{1.4.3}
    Let $S$ be a $T$-space and let $H\subseteq G\subseteq \operatorname{End}_{T}(S)$. Then ${{S}^{H}}\supseteq {{S}^{G}}$.
\end{proposition}

\subsection{Galois $T$-monoids and Galois $T$-groups} \label{Galois $T$-monoids and Galois $T$-groups}
	The concept of Galois group in the classical Galois theory is generalized to two notions as follows.
\begin{definition}
    \label{1.5.1}
    Let $S$ be a $T$-space and let $B\subseteq S$. 
    
    The \emph{Galois $T$-monoid} of $S$ over $B$, denoted by $\operatorname{GMn}_{T}(S/B)$, is the set $\{\sigma \in \operatorname{End}_{T}(S)\,|\,\forall b\in B,\sigma (b)=b\}$. 

    The \emph{Galois $T$-group} of $S$ over $B$, denoted by $\operatorname{GGr}_{T}(S/B)$, is the set $\{\sigma \in \operatorname{Aut}_{T}(S)\,|\,\forall b\in B,\sigma (b)=b\}$.
\end{definition}
\begin{remark}
    Trivially, $\operatorname{GMn}_{T}(\emptyset /\emptyset )=\operatorname{GGr}_{T}(\emptyset /\emptyset )=\{\text{I}{{\text{d}}_{\emptyset }}\}$ because we regard the empty function $\emptyset \to \emptyset $ as the identity function $\text{I}{{\text{d}}_{\emptyset }}$.
\end{remark}
To justify Definition \ref{1.5.1}, the following two results are needed, of which the first one generalizes Proposition \ref{1.3.10}.
\begin{proposition}
    \label{1.5.2}  Let $S$ be a $T$-space and let $B\subseteq S$. Then $\operatorname{GMn}_{T}(S/B)$ is a submonoid of $\operatorname{End}_{T}(S)$ with composition of functions as the binary operation.
\end{proposition}
\begin{proof}
    Because $\operatorname{Id}\in \operatorname{GMn}_{T}(S/B)\subseteq \operatorname{End}_{T}(S)$, by Proposition \ref{1.3.10}, we only need to show that $\operatorname{GMn}_{T}(S/B)$ is closed under composition of functions.
	
 Let ${{\sigma }_{1}},{{\sigma }_{2}}\in \operatorname{GMn}_{T}(S/B)$. By Definition \ref{1.5.1}, $\forall b\in B$, ${{\sigma }_{1}}({{\sigma }_{2}}(b))={{\sigma }_{1}}(b)=b$. Hence ${{\sigma }_{1}}\circ {{\sigma }_{2}}\in \operatorname{GMn}_{T}(S/B)$ (since ${{\sigma }_{1}}\circ {{\sigma }_{2}}\in \operatorname{End}_{T}(S)$ by Proposition \ref{1.3.10}), as desired.
\end{proof}
	Analogously, Proposition \ref{1.3.11} is generalized as follows, where the proof is omitted.
\begin{proposition}
\label{1.5.3}
   Let $S$ be a $T$-space and let $B\subseteq S$. Then $\operatorname{GGr}_{T}(S/B)$ is a subgroup of $\operatorname{Aut}_{T}(S)$ with composition of functions as the binary operation.
\end{proposition}
We will employ the following three propositions later, of which the proofs are obvious.
\begin{proposition}
\label{1.5.4}
Let $S$ be a $T$-space. Then $\forall B\subseteq S, $ $B\subseteq {{S}^{\operatorname{GMn}_{T}(S/B)}}\subseteq {{S}^{\operatorname{GGr}_{T}(S/B)}}$.
\end{proposition}
\begin{proposition}
\label{1.5.5}
Let $S$ be a $T$-space. Then $\forall H\subseteq \operatorname{End}_{T}(S)$, $H\subseteq \operatorname{GMn}_{T}(S/{{S}^{H}})$, and $\forall H\subseteq \operatorname{Aut}_{T}(S)$, $H\subseteq \operatorname{GGr}_{T}(S/{{S}^{H}})$.
\end{proposition}
\begin{proposition}
\label{1.5.6}
Let $S$ be a $T$-space. Then $\forall B\subseteq K\subseteq S$,
$\operatorname{GMn}_{T}(S/B)\supseteq \operatorname{GMn}_{T}(S/K)$ and $\operatorname{GGr}_{T}(S/B)\supseteq \operatorname{GGr}_{T}(S/K)$.
\end{proposition}

\subsection{Generated submonoids of $\operatorname{End}_{T}(S)$ and subgroups of $\operatorname{Aut}_{T}(S)$} \label{generated submnd and subgrp}
	We will use the following notions and results later.
\begin{definition}
    \label{1.6.1}
    Let $S$ be a $T$-space.
	
 If $X\subseteq \operatorname{End}_{T}(S)$, then the intersection of all submonoids of $\operatorname{End}_{T}(S)$ containing $X$, denoted by ${{\left\langle X \right\rangle }_{\operatorname{End}(S)}}$, is called the \emph{submonoid of $\operatorname{End}_{T}(S)$ generated by X}. 

If $X\subseteq \operatorname{Aut}_{T}(S)$, then the intersection of all subgroups of $\operatorname{Aut}_{T}(S)$ containing $X$, denoted by ${{\left\langle X \right\rangle }_{\operatorname{Aut}(S)}}$, is called the \emph{subgroup of} $\operatorname{Aut}_{T}(S)$ \emph{generated by $X$}.
\end{definition}
\begin{remark}
    By Proposition \ref{1.3.12}, ${{\left\langle X \right\rangle }_{\operatorname{End}(S)}}$ is the smallest submonoid of $\operatorname{End}_{T}(S)$ which contains $X$. And by Proposition \ref{intersection of Aut}, ${{\left\langle X \right\rangle }_{\operatorname{Aut}(S)}}$ is the smallest subgroup of $\operatorname{Aut}_{T}(S)$ which contains $X$.
\end{remark}
\begin{proposition}
\label{1.6.2}
    Let $S$ be a $T$-space and let $X\subseteq \operatorname{Aut}_{T}(S)$. Suppose that $\forall x\in X$, ${{x}^{-1}}\in X$. Then ${{\left\langle X \right\rangle }_{\operatorname{Aut}(S)}}={{\left\langle X \right\rangle }_{\operatorname{End}(S)}}$.
\end{proposition}
\begin{proof}
    Since $X\subseteq \operatorname{Aut}_{T}(S)$ and every subgroup of $\operatorname{Aut}_{T}(S)$ is a submonoid of $\operatorname{End}_{T}(S)$, the set of all subgroups of $\operatorname{Aut}_{T}(S)$ containing $X$ is a subset of the set of all submonoids of $\operatorname{End}_{T}(S)$ containing $X$. Then by Definition \ref{1.6.1} we can tell that ${{\left\langle X \right\rangle }_{\operatorname{Aut}(S)}}\supseteq {{\left\langle X \right\rangle }_{\operatorname{End}(S)}}$. So it is sufficient to prove that ${{\left\langle X \right\rangle }_{\operatorname{Aut}(S)}}\subseteq {{\left\langle X \right\rangle }_{\operatorname{End}(S)}}$.

    If $X$ is empty, then ${{\left\langle X \right\rangle }_{\operatorname{Aut}(S)}}=\{$Id on $S\}={{\left\langle X \right\rangle }_{\operatorname{End}(S)}}$, as desired.

    Suppose that $X$ is not empty. Let $\sigma \in {{\left\langle X \right\rangle }_{\operatorname{Aut}(S)}}$. Since ${\left\langle X \right\rangle }_{\operatorname{Aut}(S)}$ is the group generated by $X$, $\exists n\in \mathbb{Z}^+$, $x _{1},\cdots, x_{n} \in X$ and $e_1,\cdots,e_n \in \mathbb{Z}$ such that $\sigma =x_{1}^{{{e}_{1}}}\cdots x_{n}^{{{e}_{n}}}$ (which is a product of group elements). Because it is assumed that $\forall x\in X$, ${{x}^{-1}}\in X$, we can tell that any submonoid of $\operatorname{End}_{T}(S)$ containing $X$ also contains $x_{1}^{{{e}_{1}}}\cdots x_{n}^{{{e}_{n}}}=\sigma $. Hence ${{\left\langle X \right\rangle }_{\operatorname{Aut}(S)}}\subseteq {{\left\langle X \right\rangle }_{\operatorname{End}(S)}}$, as desired.
\end{proof}

\section{Galois correspondences} \label{I Galois corr}
 In this section, we shall develop two main results on Galois correspondences. Roughly speaking, the first one shows the Galois correspondence between the Galois $T$-monoids (Definition \ref{1.5.1}) of a $T$-space $S$ and the fixed subsets of $S$ under actions of $T$-endomorphisms of $S$, and the second one shows the Galois correspondence between the Galois $T$-groups (Definition \ref{1.5.1}) of a $T$-space $S$ and the fixed subsets of $S$ under actions of $T$-automorphisms of $S$.
    
    For a $T$-space $S$ and $H\subseteq \operatorname{End}_{T}(S)$, by Proposition \ref{1.4.2}, $S^H$ is a quasi-$T$-subspace of $S$, rather than a $T$-subspace of $S$ unless $\operatorname{Id}\in T$ or ${{S}^{H}}\subseteq {{\left\langle {{S}^{H}} \right\rangle }_{T}}$. This fact implies that, quasi-$T$-subspaces, rather than $T$-subspaces, of a $T$-space may play a major role in establishing Galois correspondences, though Proposition \ref{1.2.9} tells us that in the case such as $\operatorname{Id}\in T$, quasi-$T$-space and $T$-space are actually the same notion. 
\begin{notation}
\label{2.1.1}
    Let $S$ be a $T$-space and let $B\subseteq S$. We denote by $\operatorname{Sub}_T(S)$ the set of all quasi-$T$-subspaces of $S$, and we denote by $\operatorname{Int}_T(S/B)$ the set of all intermediate quasi-$T$-spaces between $B$ and $S$, i.e. the set of all quasi-$T$-spaces $K$ with $S\supseteq K\supseteq B$. 
\end{notation}
\begin{remark} \begin{enumerate}
    \item By Proposition \ref{1.2.9}, both $\operatorname{Sub}_T(S)$ and $\operatorname{Int}_T(S/B)$ are sets of $T$-spaces if $\operatorname{Id}\in T$.
    \item $\operatorname{Sub}_T(\emptyset )=\operatorname{Int}_T(\emptyset /\emptyset )=\{\emptyset \}$.
\end{enumerate}
\end{remark}
\begin{notation}
    \label{2.1.2}
    We denote by $\operatorname{SGr}(G)$ the set of all subgroups of a group $G$, and we denote by $\operatorname{SMn}(M)$ the set of all submonoids of a monoid $M$. 
\end{notation}
\subsection{Two examples} \label{Two examples}
Let $S$ be a $T$-space and let $G$ be a subset of $\operatorname{Aut}_{T}(S)$. Then by Definition \ref{1.4.1}, $S$ is a Galois $T$-extension of ${S^G}$. Let’s check whether there always exists a correspondence between $\operatorname{SGr}(\operatorname{GGr}_{T}(S/{{S}^{G}}))$ and $\operatorname{Int}_T(S/{{S}^{G}})$ as in the classical Galois field theory (see Definition \ref{1.5.1} for the notation $\operatorname{GGr}_{T}(S/{{S}^{G}})$).
\begin{example}
\label{2.1.3}
    Let $T=\{f:\mathbb{R}\to \mathbb{R}$ given by $x\mapsto x+a \,|\, a\in \mathbb{R}\}$. Then $T$ is an operator semigroup on $\mathbb{R}$, $\operatorname{Id}\in T$ and $\forall r\in \mathbb{R},{{\left\langle r \right\rangle }_{T}}=\mathbb{R}$. Therefore, $\mathbb{R}$ is a $T$-space and the only (quasi-)$T$-subspaces of $\mathbb{R}$ are $\emptyset $ and $\mathbb{R}$. 

    Let $\sigma \in \operatorname{End}_{T}(\mathbb{R})$. Then $\forall f\in T$, $\sigma $ commutes with $f$ on $\mathbb{R}$. Thus for $a\in \mathbb{R}$ and $f:\mathbb{R}\to \mathbb{R}$ given by $x\mapsto x+a$,
    \[\sigma (x)+a=f(\sigma (x))=\sigma (f(x))=\sigma (x+a), \forall x\in \mathbb{R},\]
    which implies that $\sigma $ is a translation along the real axis. Hence $\sigma \in T$ and $\operatorname{End}_{T}\mathbb{R}\subseteq T$.  On the other hand, we can tell $T\subseteq \operatorname{Aut}_{T}(\mathbb{R})\subseteq \operatorname{End}_{T}(\mathbb{R})$. Thus, $\operatorname{End}_{T}(\mathbb{R})=\operatorname{Aut}_{T}(\mathbb{R})=T$. Then $\mathbb{R}$ is a Galois $T$-extension of ${{\mathbb{R}}^{T}}=\emptyset $ and $\operatorname{Int}_T(\mathbb{R}/{{\mathbb{R}}^{T}})=\{\emptyset ,\mathbb{R}\}$ (since the only (quasi-)$T$-subspaces of $\mathbb{R}$ are $\emptyset $ and $\mathbb{R}$).

    However, $\operatorname{GGr}_{T}(\mathbb{R}/{{\mathbb{R}}^{T}})(=\operatorname{GGr}_{T}(\mathbb{R}/\emptyset)=\operatorname{Aut}_{T}(\mathbb{R})=T)$ has an infinite number of subgroups, so
    \[\left| \operatorname{Int}_T(\mathbb{R}/{{\mathbb{R}}^{T}}) \right|<\left| \operatorname{SGr}(\operatorname{GGr}_{T}(\mathbb{R}/{{\mathbb{R}}^{T}})) \right|.\]
\end{example}
The preceding example shows that the number of all intermediate (quasi-)$T$-spaces of a Galois $T$-extension may be strictly smaller than that of all subgroups of the corresponding Galois $T$-group.
The converse is also possible, as shown below.
\begin{example}
    \label{2.1.4}
    Let $T=\{$Id$\}$. Then $T$ is an operator semigroup on any set $D$ and any subset of $D$ is a $T$-space. Obviously any map from $D$ to $D$ belongs to $\operatorname{End}_{T}(D)$, and so any bijective map from $D$ to $D$ lies in $\operatorname{Aut}_{T}(D)$.
	
    Suppose $\left| D \right|=2$. Then $\operatorname{Aut}_{T}(D)={{S}_2}$, which is the symmetric group on $2$ letters. Then $D$ is a Galois $T$-extension of ${D^{S_2}}=\emptyset $, $\operatorname{Int}_T(D/{{D}^{S_2}})$ is the power set of $D$, and
    \[\operatorname{SGr}(\operatorname{GGr}_{T}(D/{{D}^{S_2}}))=\operatorname{SGr}(\operatorname{Aut}_{T}(D))=\operatorname{SGr}({{S}_{2}}).\]
    Then 
    \[\left| \operatorname{Int}_T(D/{{D}^{S_2}}) \right|=4>2=\left| \operatorname{SGr}(\operatorname{GGr}_{T}(D/{{D}^{S_2}})) \right|.\]
\end{example}
    Example \ref{2.1.4} shows that the number of all intermediate (quasi-)$T$-spaces of a Galois $T$-extension may be strictly larger than that of all subgroups of the corresponding Galois $T$-group.

    However, if we focus on the fixed subsets of a $T$-space $S$ under actions of $T$-automorphisms of $S$ and on the subgroups (of $\operatorname{Aut}_{T}(S)$) which are Galois $T$-groups, then we will see a Galois correspondence. Moreover, we shall find a Galois correspondence for Galois $T$-monoids as well.

\subsection{Galois correspondences} \label{I subsection Galois correspondences}
First, we need to define the fixed subsets of a $T$-space $S$ under actions of $T$-endomorphisms (or $T$-automorphisms) of $S$ as follows.
\begin{definition}
    \label{2.2.1}
    Let $S$ be a $T$-space and let $B\subseteq S$. Then
    \[\operatorname{Int}_{T}^{\operatorname{End}}(S/B):=\{{{S}^{H}}\,|\,B\subseteq {{S}^{H}},H\subseteq \operatorname{End}_{T}(S)\},\] and
	\[\operatorname{Int}_{T}^{\operatorname{Aut}}(S/B):=\{{{S}^{H}}\,|\,B\subseteq {{S}^{H}},H\subseteq \operatorname{Aut}_{T}(S)\}.\]
\end{definition}
\begin{remark}
    \begin{enumerate}
        \item  $\forall K\in \operatorname{Int}_{T}^{\operatorname{End}}(S/B)$, $S$ is a Galois $T$-extension of $K$ (by Definition \ref{1.4.1}).
         \item Since $\operatorname{Aut}_{T}(S) \subseteq \operatorname{End}_{T}(S)$, $\operatorname{Int}_{T}^{\operatorname{Aut}}(S/B)\subseteq \operatorname{Int}_{T}^{\operatorname{End}}(S/B)\subseteq \operatorname{Int}_T(S/B)$ (by Proposition \ref{1.4.2}). 
     \item $\operatorname{Int}_{T}^{\operatorname{Aut}}(\emptyset /\emptyset )=\operatorname{Int}_{T}^{\operatorname{End}}(\emptyset /\emptyset )=\{\emptyset \}$. 
    \end{enumerate}
\end{remark}
The following characterizes $\operatorname{Int}_{T}^{\operatorname{End}}(S/B)$ and $\operatorname{Int}_{T}^{\operatorname{Aut}}(S/B)$.
\begin{lemma}
\label{2.2.2}
Let $S$ be a $T$-space and let $B\subseteq K\subseteq S$. Then
\[K\in \operatorname{Int}_{T}^{\operatorname{End}}(S/B)\Leftrightarrow K={{S}^{\operatorname{GMn}_{T}(S/K)}}\]
and \[K\in \operatorname{Int}_{T}^{\operatorname{Aut}}(S/B)\Leftrightarrow K={{S}^{\operatorname{GGr}_{T}(S/K)}}.\]
\end{lemma}
\begin{proof}
    1. \emph{Sufficiency}
    
If $K={{S}^{\operatorname{GMn}_{T}(S/K)}}$, then $K\in \operatorname{Int}_{T}^{\operatorname{End}}(S/B)$ by Definition \ref{2.2.1}.

If $K={{S}^{\operatorname{GGr}_{T}(S/K)}}$, then $K\in \operatorname{Int}_{T}^{\operatorname{Aut}}(S/B)$ by Definition \ref{2.2.1}.

2. \emph{Necessity}

    Suppose $K\in \operatorname{Int}_{T}^{\operatorname{End}}(S/B)$. Then $\exists H\subseteq \operatorname{End}_{T}(S)$ such that $K={{S}^{H}}$, and so by Proposition \ref{1.5.5}, $H\subseteq \operatorname{GMn}_{T}(S/{{S}^{H}})=\operatorname{GMn}_{T}(S/K)$. Hence $K={{S}^{H}}\supseteq {{S}^{\operatorname{GMn}_{T}(S/K)}}$ by Proposition \ref{1.4.3}.

    On the other hand, $K\subseteq {{S}^{\operatorname{GMn}_{T}(S/K)}}$ by Proposition \ref{1.5.4}.

    Therefore, $K={{S}^{\operatorname{GMn}_{T}(S/K)}}$.

    If $K\in \operatorname{Int}_{T}^{\operatorname{Aut}}(S/B)$, analogously we can show $K={{S}^{\operatorname{GGr}_{T}(S/K)}}$.
\end{proof}
Second, we define 
\begin{definition}
    \label{2.2.3}
    Let $S$ be a $T$-space and let $B\subseteq S$. Then
\[\operatorname{GSMn}_{T}(S/B):=\{\operatorname{GMn}_{T}(S/K)\,|\,B\subseteq K\subseteq S\},\] and 
\[\operatorname{GSGr}_{T}(S/B):=\{\operatorname{GGr}_{T}(S/K)\,|\,B\subseteq K\subseteq S\}.\]
\end{definition}
\begin{remark} \begin{enumerate}
    \item GSMn stands for “Galois submonoids” and GSGr stands for “Galois subgroups”.
    \item Clearly,
    \[\operatorname{GSMn}_{T}(S/B)\subseteq \operatorname{SMn}(\operatorname{GMn}_{T}(S/B))\] (by Notation \ref{2.1.2} and Proposition \ref{1.5.2}), 
    \[\operatorname{GSGr}_{T}(S/B)\subseteq \operatorname{SGr}(\operatorname{GGr}_{T}(S/B))\] (by Notation \ref{2.1.2} and Proposition \ref{1.5.3}),
    \[\operatorname{GSGr}_{T}(S/B)\subseteq \operatorname{GSMn}_{T}(S/B),\] and 
    \[\operatorname{GSGr}_{T}(\emptyset /\emptyset )=\operatorname{GSMn}_{T}(\emptyset /\emptyset )=\{\{\text{I}{{\text{d}}_{\emptyset }}\}\}.\]
\end{enumerate}
\end{remark}
	The following characterizes $\operatorname{GSMn}_{T}(S/B)$ and $\operatorname{GSGr}_{T}(S/B)$.
\begin{lemma}
    \label{2.2.4}
    Let $S$ be a $T$-space, let $H\subseteq \operatorname{End}_{T}(S)$, and let $B\subseteq {{S}^{H}}\subseteq S$. Then
\[H\in \operatorname{GSMn}_{T}(S/B)\Leftrightarrow \operatorname{GMn}_{T}(S/{{S}^{H}})=H\] and 
\[H\in \operatorname{GSGr}_{T}(S/B)\Leftrightarrow \operatorname{GGr}_{T}(S/{{S}^{H}})=H.\]
\end{lemma}
\begin{proof}
    1. \emph{Sufficiency}
    
	If $\operatorname{GMn}_{T}(S/{{S}^{H}})=H$, then $H\in \operatorname{GSMn}_{T}(S/B)$ by Definition \ref{2.2.3}.
 
	If $\operatorname{GGr}_{T}(S/{{S}^{H}})=H$, then $H\in \operatorname{GSGr}_{T}(S/B)$ by Definition \ref{2.2.3}.
 
2. \emph{Necessity}

Suppose $H\in \operatorname{GSMn}_{T}(S/B)$. Then $\exists K$ such that $B\subseteq K\subseteq S$ and $H=\operatorname{GMn}_{T}(S/K)$, and thus by Proposition \ref{1.5.4}, $K\subseteq {{S}^{\operatorname{GMn}_{T}(S/K)}}={{S}^{H}}$. Hence
$H=\operatorname{GMn}_{T}(S/K)\supseteq \operatorname{GMn}_{T}(S/{{S}^{H}})$ by Proposition \ref{1.5.6}.

On the other hand, $H\subseteq \operatorname{GMn}_{T}(S/{{S}^{H}})$ by Proposition \ref{1.5.5}.

Therefore, $H=\operatorname{GMn}_{T}(S/{{S}^{H}})$.

Analogously, if $H\in \operatorname{GSGr}_{T}(S/B)$, we can show $H=\operatorname{GGr}_{T}(S/{{S}^{H}})$.
\end{proof}
Now two main results of this section follow. Roughly speaking, the first one shows the Galois correspondence between the Galois $T$-monoids of a $T$-space $S$ and the fixed subsets of $S$ under actions of $T$-endomorphisms of $S$.
\begin{corollary}
     \label{2.2.5}
    Let $S$ be a $T$-space and let $B\subseteq S$. Then the correspondences
\begin{center}
    $\gamma :H\mapsto {{S}^{H}}$ and $\delta :K\mapsto \operatorname{GMn}_{T}(S/K)$
\end{center}
define inclusion-inverting mutually inverse bijective maps $($i.e. $\gamma $ is an inclusion-reversing bijective map whose inverse $\delta $ is also an inclusion-reversing bijective map$)$ between $\operatorname{GSMn}_{T}(S/B)$ and $\operatorname{Int}_{T}^{\operatorname{End}}(S/B)$. 
\end{corollary}
\begin{proof}
    $\gamma $ is inclusion-reversing by Proposition \ref{1.4.3}, and $\delta $ is inclusion-reversing by Proposition \ref{1.5.6}. 
	
 Lemmas \ref{2.2.2} and \ref{2.2.4} imply that $\gamma \circ \delta $ and $\delta \circ \gamma $ are identity functions, respectively. Hence both $\delta $ and $\gamma $ are bijective maps and ${{\delta }^{-1}}=\gamma $.	
\end{proof}
Analogously, we have the following, which, roughly speaking, shows the Galois correspondence between the Galois $T$-groups of a $T$-space $S$ and the fixed subsets of $S$ under actions of $T$-automorphisms of $S$.
\begin{corollary}
    \label{2.2.6}
    Let $S$ be a $T$-space and let $B\subseteq S$. Then the correspondences
\begin{center}
    $\gamma :H\mapsto {{S}^{H}}$ and $\delta :K\mapsto \operatorname{GGr}_{T}(S/K)$
\end{center}
define inclusion-inverting mutually inverse bijective maps between $\operatorname{GSGr}_{T}(S/B)$ and $\operatorname{Int}_{T}^{\operatorname{Aut}}(S/B)$.
\end{corollary}
\begin{proof}
    The same as that of Corollary \ref{2.2.5}.
\end{proof}
As a matter of fact, Corollaries \ref{2.2.5} and \ref{2.2.6} tell us, essentially, what Galois correspondences are: They tell us where to find Galois correspondences in any specific case. In fact, with examples we gave (e.g. Proposition \ref{1.3.7}) or will give for $T$-morphisms, we can tell from Corollaries \ref{2.2.5} and \ref{2.2.6} that Galois correspondences are ubiquitous. 

 In Section \ref{II Gal corr}, we shall prove that Corollaries \ref{2.2.5} and \ref{2.2.6} still hold for the case where elements of $T$ are allowed to be functions in more than one variable and/or even partial.

As is well known, in the Galois theory for infinite algebraic field extensions, Krull topology is put on Galois groups, and in differential Galois theory, differential Galois groups are endowed with Zariski topology. Then the fundamental theorem in each of the two theories tells us that there exists a Galois correspondence between the set of all closed subgroups of the Galois group of a Galois field extension (or a Picard-Vessiot extension) and the set of all intermediate (differential) fields. 

Now for Corollary \ref{2.2.5}, can we define a topology on $S$ and a topology on $\operatorname{End}_{T}(S)$ such that $\operatorname{Int}_{T}^{\operatorname{End}}(S/B)$ is the set of all closed quasi-$T$-subspaces (of $S$) containing $B$ and $\operatorname{GSMn}_{T}(S/B)$ is the set of all closed submonoids of $\operatorname{GMn}_{T}(S/B)$? If such topologies exist, then we can characterize the Galois correspondence between $\operatorname{Int}_{T}^{\operatorname{End}}(S/B)$ and $\operatorname{GSMn}_{T}(S/B)$ in terms of closed quasi-$T$-subspaces of $S$ and closed submonoids of $\operatorname{GMn}_{T}(S/B)$. 

And for Corollary \ref{2.2.6}, can we define a topology on $S$ and a topology on $\operatorname{Aut}_{T}(S)$ with $\operatorname{Int}_{T}^{\operatorname{Aut}}(S/B)$ being the set of all closed quasi-$T$-subspaces (of $S$) containing $B$ and $\operatorname{GSGr}_{T}(S/B)$ being the set of all closed subgroups of $\operatorname{GGr}_{T}(S/B)$? If such topologies exist, then we can characterize the Galois correspondence between $\operatorname{Int}_{T}^{\operatorname{Aut}}(S/B)$ and $\operatorname{GSGr}_{T}(S/B)$ in terms of closed quasi-$T$-subspaces of $S$ and closed subgroups of $\operatorname{GGr}_{T}(S/B)$.

Before we can answer the above questions in Section \ref{I Topologies employed}, we need to address a related issue in Section \ref{I Lattice structures}, which is about the lattice structures of the sets involved in Corollaries \ref{2.2.5} and \ref{2.2.6}.

There is one more thing which we have not talked about for Galois correspondences. It is normally the last part of the fundamental theorem of a Galois theory, i.e. the correspondence between the normal subgroups of a Galois group and the Galois extensions of the base field. We will address this issue in Section \ref{Two questions}.

\section{Lattice structures of objects arising in Galois correspondences on a $T$-space $S$} \label{I Lattice structures}
In this section, we shall study lattice structures of objects which arise in Corollaries \ref{2.2.5} and \ref{2.2.6}. The value of this section, though being a preparation for the next section, asserts itself.

	Recall that a lattice is a (nonempty) partially ordered set $L$ in which every pair of elements $a,b\in L$ has a meet $a\wedge b$ (greatest lower bound) and a join $a\vee b$ (least upper bound), and $L$ is a complete lattice if any subset of it has a greatest lower bound and a least upper bound. Note that if lattice $L$ is complete, then $\wedge \emptyset =\vee L$ and $\vee \emptyset =\wedge L$ (see e.g. \cite{2}).
\begin{notation}
     To abbreviate our notations, by $\{{{A}_{i}}\}$, $\bigcup{{{A}_{i}}}$ and $\bigcap{{{A}_{i}}}$ we mean $\{{{A}_{i}}|i\in I\}$, $\bigcup\nolimits_{i\in I}{{{A}_{i}}}$ and $\bigcap\nolimits_{i\in I}{{{A}_{i}}}$, respectively. 
 \end{notation}
\subsection{$\operatorname{Sub}_T(S)$, $\operatorname{SMn}(\operatorname{End}_{T}(S))$ and $\operatorname{SGr}(\operatorname{Aut}_{T}(S))$}
\begin{proposition}
    \label{3.1.1}
    Let $S$ be a $T$-space. Then $\operatorname{Sub}_T(S)$ $($Notation \ref{2.1.1}$)$ is a lattice with inclusion as the binary relation if $\forall {{S}_{1}},{{S}_{2}}\in \operatorname{Sub}_T(S)$, ${{S}_{1}}\wedge {{S}_{2}}:={{S}_{1}}\bigcap {{S}_{2}}$ and ${{S}_{1}}\vee {{S}_{2}}:={{S}_{1}}\bigcup {{S}_{2}}$. Moreover, $\operatorname{Sub}_T(S)$ is a complete lattice if let $\wedge A=\bigcap{{{S}_{i}}}$ and let $\vee A=\bigcup{{{S}_{i}}}$, $\forall A:=\{{{S}_{i}}\}\subseteq \operatorname{Sub}_T(S)$.
\end{proposition}
\begin{proof}
    Straightforward consequences of Propositions \ref{1.2.10} and \ref{1.2.11}. Note that $\operatorname{Sub}_T(\emptyset )=\{\emptyset \}$ is a complete lattice.	
\end{proof}
\begin{proposition}
    \label{3.1.2}
    Let $S$ be a $T$-space. Then $\operatorname{SMn}(\operatorname{End}_{T}(S))$ $($Notation \ref{2.1.2}$)$ is a lattice with inclusion as the binary relation if ${{M}_{1}}\wedge {{M}_{2}}:={{M}_{1}}\bigcap {{M}_{2}}$ and ${{M}_{1}}\vee {{M}_{2}}:={{\left\langle {{M}_{1}}\bigcup {{M}_{2}} \right\rangle }_{\operatorname{End}(S)}}$ $($Definition \ref{1.6.1}$)$, $\forall {{M}_{1}},{{M}_{2}}\in \operatorname{SMn}(\operatorname{End}_{T}(S))$. Moreover, if let $\wedge A=\bigcap{{{M}_{i}}}$ and let $\vee A={{\left\langle \bigcup{{{M}_{i}}} \right\rangle }_{\operatorname{End}(S)}}$, $\forall A:=\{{{M}_{i}}\}\subseteq \operatorname{SMn}(\operatorname{End}_{T}(S))$, then $\operatorname{SMn}(\operatorname{End}_{T}(S))$ is a complete lattice.
\end{proposition}
\begin{proof}
    Obvious (\textit{cf.} Proposition \ref{1.3.12} and the remark right after Definition \ref{1.6.1}).
\end{proof}
	For $\operatorname{SGr}(\operatorname{Aut}_{T}(S))$ (Notation \ref{2.1.2}), we also have
\begin{proposition}
    \label{3.1.3}
    Let $S$ be a $T$-space. Then $\operatorname{SGr}(\operatorname{Aut}_{T}(S))$ is a lattice with inclusion as the binary relation if ${{G}_{1}}\wedge {{G}_{2}}:={{G}_{1}}\bigcap {{G}_{2}}$ and ${{G}_{1}}\vee {{G}_{2}}:={{\left\langle {{G}_{1}}\bigcup {{G}_{2}} \right\rangle }_{\operatorname{Aut}(S)}}$ $($Definition \ref{1.6.1}$)$, $\forall {{G}_{1}},{{G}_{2}}\in \operatorname{SGr}(\operatorname{Aut}_{T}(S))$. Moreover, if let $\wedge A=\bigcap{{{G}_{i}}}$ and let $\vee A={{\left\langle \bigcup{{{G}_{i}}} \right\rangle }_{\operatorname{Aut}(S)}}$,  $\forall A:=\{{{G}_{i}}\}\subseteq \operatorname{SGr}(\operatorname{Aut}_{T}(S))$, then $\operatorname{SGr}(\operatorname{Aut}_{T}(S))$ is a complete lattice.
\end{proposition}
\begin{proof}
    Obvious by Proposition \ref{intersection of Aut} and the remark right after Definition \ref{1.6.1}.	
\end{proof}
	Moreover, we have
\begin{proposition}
    \label{3.1.4}
    Let $S$ be a $T$-space. Then $\operatorname{SGr}(\operatorname{Aut}_{T}(S))$ is a complete sublattice of  the complete lattice $\operatorname{SMn}(\operatorname{End}_{T}(S))$ defined in Proposition \ref{3.1.2}.
\end{proposition}
\begin{proof}
    It suffices to show that for any $A:=\{{{G}_{i}}\}\subseteq \operatorname{SGr}(\operatorname{Aut}_{T}(S))$,
				\[\vee A={{\left\langle \bigcup{{{G}_{i}}} \right\rangle }_{\operatorname{Aut}(S)}}={{\left\langle \bigcup{{{G}_{i}}} \right\rangle }_{\operatorname{End}(S)}}.\] 
From Proposition \ref{1.6.2}, we can tell that the equations hold.
\end{proof}

\subsection{$\operatorname{Int}_{T}^{\operatorname{End}}(S/B)$ and $\operatorname{Int}_{T}^{\operatorname{Aut}}(S/B)$}
We are interested in the two sets because they arise in Corollary \ref{2.2.5} or \ref{2.2.6}.
\begin{lemma}
\label{3.2.1}
    Let $S$ be a $T$-space and let $\{{{H}_{i}}\}$ be a set of subsets of $\operatorname{End}_{T}(S)$. Then $\bigcap{{{S}^{{{H}_{i}}}}}={{S}^{\bigcup{{{H}_{i}}}}}$.
\end{lemma}
\begin{proof} 
\begin{align*}
      \bigcap{{{S}^{{{H}_{i}}}}}&=\bigcap{\{a\in S\,|\,\forall \sigma \in {{H}_{i}},\sigma (a)=a\}}\\    
    &=\{a\in S\,|\,\forall \sigma \in \bigcup{{{H}_{i}}},\sigma (a)=a\}\\    
    &={{S}^{\bigcup{{{H}_{i}}}}}.  
\end{align*} 
\end{proof}
	Lemma \ref{3.2.1} implies 
 \begin{proposition}
     \label{3.2.2}
     $\operatorname{Int}_{T}^{\operatorname{End}}(S/B)$ is a complete $ \wedge $-sublattice of the complete lattice $\operatorname{Sub}_T(S)$ defined in Proposition \ref{3.1.1}.
 \end{proposition}
\begin{proof}
    By Definition \ref{2.2.1}, each element of $\operatorname{Int}_{T}^{\operatorname{End}}(S/B)$ has the form ${{S}^{{{H}_{i}}}}(\supseteq B)$, where ${{H}_{i}}\subseteq \operatorname{End}_{T}(S)$. Then by Lemma \ref{3.2.1}, $\bigcap{{{S}^{{{H}_{i}}}}}={{S}^{\bigcup{{{H}_{i}}}}}\in \operatorname{Int}_{T}^{\operatorname{End}}(S/B)$. 

    Moreover, by Proposition \ref{1.4.2}, $\operatorname{Int}_{T}^{\operatorname{End}}(S/B)\subseteq \operatorname{Sub}_T(S)$.

    Therefore, $\operatorname{Int}_{T}^{\operatorname{End}}(S/B)$ is a complete $ \wedge $-sublattice of the complete lattice $\operatorname{Sub}_T(S)$ defined in Proposition \ref{3.1.1}.	
\end{proof}
However, $\operatorname{Int}_{T}^{\operatorname{End}}(S/B)$ is not necessarily a $ \vee $-sublattice of the lattice $\operatorname{Sub}_T(S)$ defined in Proposition \ref{3.1.1}. In fact, it is possible that $\operatorname{Int}_{T}^{\operatorname{End}}(S/B)$ is not even a $ \vee $-semilattice if $\forall {{S}_{1}},{{S}_{2}}\in \operatorname{Int}_{T}^{\operatorname{End}}(S/B)$, ${{S}_{1}}\vee {{S}_{2}}:={{S}_{1}}\bigcup {{S}_{2}}$, as shown in the following example.
\begin{example}
    \label{3.2.3}
    Let $T$ be the operator semigroup on $\mathbb{C}$ induced by $\mathbb{Q}[x]$ as defined in Example \ref{T on field--1}. That is, let 
    \[T=\{{{f}^*}:\mathbb{C}\to \mathbb{C}  \text{ given by } a\mapsto f(a)\,|\,f(x)\in \mathbb{Q}[x]\}.\]
    Let a $T$-space $S={{\left\langle \sqrt{2}+\sqrt{3} \right\rangle }_{T}}=\mathbb{Q}(\sqrt{2}+\sqrt{3})$. Then $S=\mathbb{Q}(\sqrt{2},\sqrt{3})$  and $S/\mathbb{Q}$ is a finite Galois field extension (see e.g. \cite{7}).

    By Corollary \ref{1.3.6}, a map $\sigma :S\to S$ is a $T$-morphism if and only if it is a field automorphism with $\mathbb{Q}$ fixed pointwisely, and so by Definition \ref{2.2.1} and the fundamental theorem of the classical Galois theory,
    \[\operatorname{Int}_{T}^{\operatorname{End}}(S/\mathbb{Q})=\operatorname{Int}_{T}^{\operatorname{Aut}}(S/\mathbb{Q})=\{\mathbb{Q}\subseteq F\subseteq S \,|\, F\text{ is a field}\},\] 
    and it contains both $\mathbb{Q}(\sqrt{2})$ and $\mathbb{Q}(\sqrt{3})$.

    Let ${S}'=\mathbb{Q}(\sqrt{2})\bigcup \mathbb{Q}(\sqrt{3})$. Then ${S}'\notin \operatorname{Int}_{T}^{\operatorname{End}}(S/\mathbb{Q})$ because ${S}'$ is not a field. But ${S}'=\mathbb{Q}(\sqrt{2})\vee \mathbb{Q}(\sqrt{3})$ as defined in Proposition \ref{3.1.1}. Therefore, $\operatorname{Int}_{T}^{\operatorname{End}}(S/\mathbb{Q})$ is not a $ \vee $-sublattice of the lattice $\operatorname{Sub}_T(S)$ defined in Proposition \ref{3.1.1}. In fact, $\operatorname{Int}_{T}^{\operatorname{End}}(S/\mathbb{Q})$ is not even a $ \vee $-semilattice if $\forall {{S}_{1}},{{S}_{2}}\in \operatorname{Int}_{T}^{\operatorname{End}}(S/\mathbb{Q})$, ${{S}_{1}}\vee {{S}_{2}}:={{S}_{1}}\bigcup {{S}_{2}}$.	
\end{example}
By the way, note that in the preceding example, ${S}'(\in \operatorname{Sub}_T(S))$ is not a subfield of $S$. This shows that the notion of $T$-space is not equivalent to the notion of field in this case. Basically this is because all elements of $T$ defined in Example \ref{T on field--1} are unary functions. We will solve this problem in Section \ref{Basic notions for partial} (\textit{cf.} Example \ref{10.2.2}).

Because $\operatorname{Aut}_{T}(S)\subseteq \operatorname{End}_{T}(S)$, the following is a straightforward consequence of Lemma \ref{3.2.1}.
\begin{lemma}
    \label{3.2.4} Let $S$ be a $T$-space and let $\{{{H}_{i}}\}$ be a set of subsets of $\operatorname{Aut}_{T}(S)$. Then $\bigcap{{{S}^{{{H}_{i}}}}}={{S}^{\bigcup{{{H}_{i}}}}}$.
\end{lemma}
	Analogously, Lemma \ref{3.2.4} implies
 \begin{proposition}
     \label{3.2.5} $\operatorname{Int}_{T}^{\operatorname{Aut}}(S/B)$ is a complete $\wedge$-sublattice of the complete lattice $\operatorname{Sub}_T(S)$ defined in Proposition \ref{3.1.1}. 
 \end{proposition}
\begin{proof}
    By Definition \ref{2.2.1}, each element of $\operatorname{Int}_{T}^{\operatorname{Aut}}(S/B)$ has the form ${{S}^{{{H}_{i}}}}(\supseteq B)$, where ${{H}_{i}}\subseteq \operatorname{Aut}_{T}(S)$. Then by Lemma \ref{3.2.4}, $\bigcap{{{S}^{{{H}_{i}}}}}={{S}^{\bigcup{{{H}_{i}}}}}\in \operatorname{Int}_{T}^{\operatorname{Aut}}(S/B)$. 

Moreover, by Proposition \ref{1.4.2}, $\operatorname{Int}_{T}^{\operatorname{Aut}}(S/B)\subseteq \operatorname{Sub}_T(S)$.

Therefore, $\operatorname{Int}_{T}^{\operatorname{Aut}}(S/B)$ is a complete $\wedge$-sublattice of the complete lattice $\operatorname{Sub}_T(S)$ defined in Proposition \ref{3.1.1}.	
\end{proof}
	However, recall that in Example \ref{3.2.3}, $\operatorname{Int}_{T}^{\operatorname{Aut}}(S/\mathbb{Q})=\operatorname{Int}_{T}^{\operatorname{End}}(S/\mathbb{Q})$. So Example \ref{3.2.3} also tells us that $\operatorname{Int}_{T}^{\operatorname{Aut}}(S/B)$ is not always a $\vee$-sublattice of the lattice $\operatorname{Sub}_T(S)$ defined in Proposition \ref{3.1.1}.

\subsection{$\operatorname{GSMn}_{T}(S/B)$ and $\operatorname{GSGr}_{T}(S/B)$}
	$\operatorname{GSMn}_{T}(S/B)$ and $\operatorname{GSGr}_{T}(S/B)$ are important because they arise in Corollary \ref{2.2.5} or \ref{2.2.6}. 
\begin{lemma}
    \label{3.3.1}
    Let $S$ be a $T$-space and let $\{{{S}_{i}}\}$ be a set of subsets of $S$. Then
\[\bigcap{\operatorname{GMn}_{T}(S/{{S}_{i}})}=\operatorname{GMn}_{T}(S/\bigcup{{{S}_{i}}}).\]
\end{lemma}
\begin{proof}
 By Definition \ref{1.5.1},
\begin{align*}
        \bigcap{\operatorname{GMn}_{T}(S/{{S}_{i}})}&=\bigcap{\{\sigma \in \operatorname{End}_{T}(S)\,|\,\sigma (a)=a,\forall a\in {{S}_{i}}\}}\\
    &=\{\sigma \in \operatorname{End}_{T}(S)\,|\,\sigma (a)=a,\forall a\in \bigcup{{{S}_{i}}}\}\\
    &=\operatorname{GMn}_{T}(S/\bigcup{{{S}_{i}}}).	
\end{align*}
\end{proof}
	Lemma \ref{3.3.1} implies
\begin{proposition}
    \label{3.3.2}
    $\operatorname{GSMn}_{T}(S/B)$ is a complete $\wedge$-sublattice of the complete lattice $\operatorname{SMn}(\operatorname{End}_{T}(S))$ defined in Proposition \ref{3.1.2}. 
\end{proposition} 
\begin{proof}
    By Definition \ref{2.2.3}, each element of $\operatorname{GSMn}_{T}(S/B)$ has the form $\operatorname{GMn}_{T}(S/{{S}_{i}})$, where $B\subseteq {{S}_{i}}\subseteq S$. Then by Lemma \ref{3.3.1},
\[\bigcap{\operatorname{GMn}_{T}(S/{{S}_{i}})}=\operatorname{GMn}_{T}(S/\bigcup{{{S}_{i}}})\in \operatorname{GSMn}_{T}(S/B).\]
Moreover, by Proposition \ref{1.5.2}, $\operatorname{GSMn}_{T}(S/B)\subseteq \operatorname{SMn}(\operatorname{End}_{T}(S))$. 

Therefore, $\operatorname{GSMn}_{T}(S/B)$ is a complete $\wedge$-sublattice of the complete lattice $\operatorname{SMn}(\operatorname{End}_{T}(S))$ defined in Proposition \ref{3.1.2}.
\end{proof}
	Nevertheless, the following example tells us that $\operatorname{GSMn}_{T}(S/B)$ is not necessarily a $\vee$-sublattice of the lattice $\operatorname{SMn}(\operatorname{End}_{T}(S))$ defined in Proposition \ref{3.1.2}. Indeed, if $\forall {{M}_{1}},{{M}_{2}}\in \operatorname{GSMn}_{T}(S/B)$, ${{M}_{1}}\vee {{M}_{2}}:={{\left\langle {{M}_{1}}\bigcup {{M}_{2}} \right\rangle }_{\operatorname{End}(S)}}$, then it is possible that $\operatorname{GSMn}_{T}(S/B)$ is not even a $\vee$-semilattice.
\begin{example}
    \label{3.3.3}  Let $D=\{a,b\}$ and let $T=$\{Id on $D$\}. Then $D$ is a $T$-space and $\operatorname{End}_{T}(D)$ is the set of all maps from $D$ to $D$. Let 
\[{{M}_{1}}=\{\text{all maps from $D$ to $D$ with $a$ fixed}\}\] and  
\[{{M}_{2}}=\{\text{all maps from $D$ to $D$ with $b$ fixed}\}.\] 

Then as monoids, ${{M}_{1}},{{M}_{2}}\in \operatorname{GSMn}_{T}(D/\emptyset )$ by Definition \ref{2.2.3}. And it is not hard to see that ${{M}_{1}}\bigcup {{M}_{2}}\in \operatorname{SMn}(\operatorname{End}_{T}(D))$. 

Thus by Proposition \ref{3.1.2}, ${{M}_{1}}\vee {{M}_{2}}={{\left\langle {{M}_{1}}\bigcup {{M}_{2}} \right\rangle }_{End(D)}}={{M}_{1}}\bigcup {{M}_{2}}$.

However, since the transposition of $a$ and $b$ does not lie in ${{M}_{1}}\bigcup {{M}_{2}}$, 
\[\operatorname{GMn}_{T}(D/{{D}^{{{M}_{1}}\vee {{M}_{2}}}})=\operatorname{GMn}_{T}(D/\emptyset )=\{\text{all maps from $D$ to $D$}\}\ne {{M}_{1}}\vee {{M}_{2}}.\] 

Hence by Lemma \ref{2.2.4}, ${{M}_{1}}\vee {{M}_{2}}\notin \operatorname{GSMn}_{T}(D/\emptyset )$. 

Therefore, $\operatorname{GSMn}_{T}(D/\emptyset )$ is not a $\vee$-sublattice of the lattice $\operatorname{SMn}(\operatorname{End}_{T}(D))$ defined in Proposition \ref{3.1.2}. In fact, $\operatorname{GSMn}_{T}(D/\emptyset )$ is not even a $\vee$-semilattice if $\forall {{M}_{1}},{{M}_{2}}\in \operatorname{GSMn}_{T}(D/\emptyset )$, ${{M}_{1}}\vee {{M}_{2}}:={{\left\langle {{M}_{1}}\bigcup {{M}_{2}} \right\rangle }_{End(D)}}$.		
\end{example}
By a similar argument as in the proof of Lemma \ref{3.3.1}, we can show the following, whose proof is omitted.
\begin{lemma}
    \label{3.3.4}
    Let $S$ be a $T$-space and let $\{{{S}_{i}}\}$ be a set of subsets of $S$. Then
$\bigcap{\operatorname{GGr}_{T}(S/{{S}_{i}})}=\operatorname{GGr}_{T}(S/\bigcup{{{S}_{i}}})$.
\end{lemma}  
\begin{proposition}
    \label{3.3.5}
    $\operatorname{GSGr}_{T}(S/B)$ is a complete $\wedge$-sublattice of the complete lattice $\operatorname{SGr}(\operatorname{Aut}_{T}(S))$ defined in Proposition \ref{3.1.3}.
\end{proposition}
\begin{proof}
    By Definition \ref{2.2.3}, each element of $\operatorname{GSGr}_{T}(S/B)$ has the form $\operatorname{GGr}_{T}(S/{{S}_{i}})$, where $B\subseteq {{S}_{i}}\subseteq S$. Then by Lemma \ref{3.3.4},
\[\bigcap{\operatorname{GGr}_{T}(S/{{S}_{i}})}=\operatorname{GGr}_{T}(S/\bigcup{{{S}_{i}}})\in \operatorname{GSGr}_{T}(S/B).\]
Moreover, by Proposition \ref{1.5.3}, $\operatorname{GSGr}_{T}(S/B)\subseteq \operatorname{SGr}(\operatorname{Aut}_{T}(S))$. 

Therefore, $\operatorname{GSGr}_{T}(S/B)$ is a complete $\wedge$-sublattice of the complete lattice $\operatorname{SGr}(\operatorname{Aut}_{T}(S))$ defined in Proposition \ref{3.1.3}.
\end{proof}
However, the following example shows that $\operatorname{GSGr}_{T}(S/B)$ is not necessarily a $\vee$-sublattice of the lattice $\operatorname{SGr}(\operatorname{Aut}_{T}(S))$ defined in Proposition \ref{3.1.3}. Indeed, if $\forall {{G}_{1}},{{G}_{2}}\in \operatorname{GSGr}_{T}(S/B)$, ${{G}_{1}}\vee {{G}_{2}}:={{\left\langle {{G}_{1}}\bigcup {{G}_{2}} \right\rangle }_{\operatorname{Aut}(S)}}$, then it is possible that $\operatorname{GSGr}_{T}(S/B)$ is not even a $\vee$-semilattice.
\begin{example}
    \label{3.3.6}
 Let $D=\{1,2,3,4\}$ and let $T=\{$Id on $D\}$. Then $D$ is a $T$-space and $\operatorname{Aut}_{T}(D)$ is the set of all bijective maps from $D$ to $D$. Let 
\[{{G}_{1}}=\{\text{Id on $D$, the transposition of 1 and 2 (with 3 and 4 fixed)}\}\] and
\[{{G}_{2}}=\{\text{Id on $D$, the transposition of 3 and 4 (with 1 and 2 fixed)}\}.\]
Then we can tell that $\operatorname{GGr}_{T}(D/{{D}^{{{G}_{1}}}})={{G}_{1}}$ and $\operatorname{GGr}_{T}(D/{{D}^{{{G}_{2}}}})={{G}_{2}}$, and so by Lemma \ref{2.2.4}, ${{G}_{1}},{{G}_{2}}\in \operatorname{GSGr}_{T}(D/\emptyset )$. 

By Proposition \ref{3.1.3}, ${{G}_{1}}\vee {{G}_{2}}={{\left\langle {{G}_{1}}\bigcup {{G}_{2}} \right\rangle }_{\operatorname{Aut}(D)}}={{G}_{1}}\times {{G}_{2}}$. 

However, $\operatorname{GGr}_{T}(D/{{D}^{{{G}_{1}}\times {{G}_{2}}}})=\operatorname{GGr}_{T}(D/\emptyset )={{S}_{4}}$, which is the symmetric group on four letters. Thus $\operatorname{GGr}_{T}(D/{{D}^{{{G}_{1}}\vee {{G}_{2}}}})\ne {{G}_{1}}\vee {{G}_{2}}$, and so by Lemma \ref{2.2.4}, ${{G}_{1}}\vee {{G}_{2}}\notin \operatorname{GSGr}_{T}(D/\emptyset )$. Hence $\operatorname{GSGr}_{T}(D/\emptyset )$ is not a $\vee$-sublattice of the lattice $\operatorname{SGr}(\operatorname{Aut}_{T}(D))$ defined in Proposition \ref{3.1.3}. Indeed, $\operatorname{GSGr}_{T}(D/\emptyset )$ is not even a $\vee$-semilattice if $\forall {{G}_{1}},{{G}_{2}}\in \operatorname{GSGr}_{T}(D/\emptyset )$, ${{G}_{1}}\vee {{G}_{2}}:={{\left\langle {{G}_{1}}\bigcup {{G}_{2}} \right\rangle }_{\operatorname{Aut}(D)}}$.
\end{example}

\section{Topologies employed to construct Galois correspondences} \label{I Topologies employed}
As was promised near the end of Section \ref{I Galois corr}, now we are going to answer the following questions. 

To construct the Galois correspondence described in Corollary \ref{2.2.5} without resort to Definition \ref{2.2.1} or \ref{2.2.3}, can we define a topology on $S$ and a topology on $\operatorname{End}_{T}(S)$ such that $\operatorname{Int}_{T}^{\operatorname{End}}(S/B)$ is the set of all closed quasi-$T$-subspaces (of $S$) containing $B$ and $\operatorname{GSMn}_{T}(S/B)$ is the set of all closed submonoids of $\operatorname{GMn}_{T}(S/B)$? This question will be answered in Subsections \ref{I Topologies on $S$ for Int} and \ref{I Topologies on End}.

And for Corollary \ref{2.2.6}, can we define a topology on $S$ and a topology on $\operatorname{Aut}_{T}(S)$ with $\operatorname{Int}_{T}^{\operatorname{Aut}}(S/B)$ being the set of all closed quasi-$T$-subspaces (of $S$) containing $B$ and $\operatorname{GSGr}_{T}(S/B)$ being the set of all closed subgroups of $\operatorname{GGr}_{T}(S/B)$? This question will be answered in Subsections \ref{I Topologies on $S$ for IntAut} and \ref{I Topologies on Aut}.

\subsection{Topologies on a $T$-space $S$ for $\operatorname{Int}_T^{\operatorname{End}}(S/B)$} \label{I Topologies on $S$ for Int}
\begin{lemma}
    \label{4.1.1}
    Let $S$ be a $T$-space with $B\subseteq S$. Then the topology ${{\mathcal{T}}_{1}}$ on $S$ generated by a subbasis ${{\beta }_{1}}:=\{S\backslash A\,|\,A\in \operatorname{Int}_{T}^{\operatorname{End}}(S/B)\}\bigcup \{S\}$ is the smallest $($coarsest$)$ topology on $S$ such that the collection of all closed sets contains $\operatorname{Int}_{T}^{\operatorname{End}}(S/B)$.
\end{lemma}
\begin{proof}
    Since the union of all elements of ${{\beta }_{1}}$ is $S$, ${{\beta }_{1}}$ qualifies as a subbasis for a topology on $S$. Hence the topology ${{\mathcal{T}}_{1}}$ generated by ${{\beta }_{1}}$ equals the intersection of all topologies on $S$ that contain ${{\beta }_{1}}$ (see e.g. \cite{6}), and so ${{\mathcal{T}}_{1}}$ is the smallest topology on $S$ containing ${{\beta }_{1}}$. Thus by the definition of ${{\beta }_{1}}$, ${{\mathcal{T}}_{1}}$ is the smallest topology on $S$ such that the collection of all closed sets contains $\operatorname{Int}_{T}^{\operatorname{End}}(S/B)$.
\end{proof}
For a $T$-space $S$ with $B \subseteq S$, to characterize the Galois correspondence depicted in Corollary \ref{2.2.5} in terms of closed quasi-$T$-subspaces of $S$, i.e. to replace $\operatorname{Int}_T^{\operatorname{End}}(S/B)$ in Corollary \ref{2.2.5} with the set of all closed quasi-$T$-subspaces (of $S$) containing $B$, we need a topology on $S$ such that the following equation is satisfied.
\begin{equation}
\label{5.1}
    \operatorname{Int}_T^{\operatorname{End}}(S/B)\backslash \{ \emptyset \}  = \{B \subseteq S' \le _qS\,|\,S'\text{ is closed and nonempty}\}
\end{equation}
\begin{remark}
    The reason why we require the above equality rather than the following one is explained as follows.
    \begin{equation}
    \label{5.2}
        \operatorname{Int}_{T}^{\operatorname{End}}(S/B)=\{B\subseteq {S}'\le _qS\,|\,{S}'\text{ is closed\}}
    \end{equation}

    Suppose $B={S}'=\emptyset $. Then $B \subseteq S' \le _qS$ and ${S}'$ is closed. Hence $\emptyset$ lies in the right side of (5.2). But it is possible that ${S}'(=\emptyset )\notin \operatorname{Int}_{T}^{\operatorname{End}}(S/B)$ (e.g. when $S \ne \emptyset$ and $\operatorname{End}_{T}(S)=\{\operatorname{Id}\}$), and thus in this case we cannot replace $\operatorname{Int}_{T}^{\operatorname{End}}(S/B)$ with the right side of (5.2).
\end{remark}
The following is to determine whether there exists a topology on $S$ such that Equation $($\ref{5.1}$)$ is satisfied: If $P_1$ defined as follows is nonempty, then we can replace $\operatorname{Int}_T^{\operatorname{End}}(S/B)\backslash \{ \emptyset \}$ with the set of all nonempty closed quasi-$T$-subspaces (of $S$) containing $B$.
\begin{theorem}
    \label{4.1.2}
    Let $S$ be a $T$-space and let $B\subseteq S$. Let
	\[{{P}_{1}}=\{\text{all topologies on $S$ such that Equation $($\ref{5.1}$)$ is satisfied}\}\] and let
	\[{{Q}_{1}}=\{\text{all topologies on $S$ which are finer $($larger$)$ than ${{\mathcal{T}}_{1}}$}\},\] 
where ${{\mathcal{T}}_{1}}$ is defined in Lemma \ref{4.1.1}. Then ${{P}_{1}}\subseteq {{Q}_{1}}$ and the following statements are equivalent:
\begin{enumerate}
    \item[(i)] ${P_1} \ne \emptyset $.
    \item[(ii)] $\forall {{S}_{1}},{{S}_{2}}\in \operatorname{Int}_{T}^{\operatorname{End}}(S/B)$, ${{S}_{1}}\bigcup {{S}_{2}}\in \operatorname{Int}_{T}^{\operatorname{End}}(S/B)$. $($Or equivalently, ${{\beta }_{1}}$ defined in Lemma \ref{4.1.1} is a basis for ${{\mathcal{T}}_{1}}\ ($see e.g. \cite{6}$))$.
    \item[(iii)]	$\operatorname{Int}_{T}^{\operatorname{End}}(S/B)$ is a $\vee$-sublattice of the lattice $\operatorname{Sub}_T(S)$ defined in Proposition \ref{3.1.1}.
    \item[(iv)] ${{\mathcal{T}}_{1}}\in {{P}_{1}}(\subseteq {{Q}_{1}})$; that is, ${{\mathcal{T}}_{1}}$ is the coarsest $($smallest$)$ topology on $S$ such that Equation $($\ref{5.1}$)$ is satisfied.
\end{enumerate}
\end{theorem}
\begin{proof}
    By Lemma \ref{4.1.1}, ${\mathcal{T}}_{1}$ is the smallest topology on $S$ such that the collection of all closed sets contains $\operatorname{Int}_{T}^{\operatorname{End}}(S/B)$. Suppose $\mathcal{T} \in {{P}_{1}}$. Then all elements of $\operatorname{Int}_{T}^{\operatorname{End}}(S/B)$ are closed in $\mathcal{T} $, and so ${\mathcal{T}}_{1} \subseteq \mathcal{T}$. Hence $\mathcal{T} \in {{Q}_{1}}$. Thus ${{P}_{1}}\subseteq {{Q}_{1}}$. 
    
    (i) $\Rightarrow$ (ii): Let $\mathcal{T} \in {{P}_{1}}$. 
    To prove (ii), we only need to consider the case where ${{S}_{1}},{{S}_{2}}\in \operatorname{Int}_{T}^{\operatorname{End}}(S/B)\backslash \{\emptyset \}$. Then by the definition of ${P_1}$, 
    \[{{S}_{1}},{{S}_{2}}\in \{B\subseteq {S}' \le _q S\,|\,{S}'\text{ is nonempty and closed in }\mathcal{T} \}.\]
    By Proposition \ref{1.2.11}, ${{S}_{1}}\bigcup {{S}_{2}} \le _q S$. Hence 
    \[{{S}_{1}}\bigcup {{S}_{2}}\in \{B\subseteq {S}' \le _q S\,|\,{S}'\text{ is nonempty and closed in }\mathcal{T} \},\]
    which equals $\operatorname{Int}_{T}^{\operatorname{End}}(S/B)\backslash \{\emptyset \}$ by (5.1). Therefore, ${{S}_{1}}\bigcup {{S}_{2}}\in \operatorname{Int}_{T}^{\operatorname{End}}(S/B)$, and so (ii) is true.

    (ii) $\Rightarrow$ (iii): By Definition \ref{2.2.1} and Proposition \ref{1.4.2}, $\operatorname{Int}_{T}^{\operatorname{End}}(S/B)\subseteq \operatorname{Sub}_T(S)$. Then from Proposition \ref{3.1.1}, we can tell that (iii) is true if (ii) is true.

    (iii) $\Rightarrow$ (iv): By the definition of ${{\mathcal{T}}_{1}}$ in Lemma \ref{4.1.1}, all elements of $\operatorname{Int}_{T}^{\operatorname{End}}(S/B)$ are closed in ${{\mathcal{T}}_{1}}$, and so by Definition \ref{2.2.1} and Proposition \ref{1.4.2},
    \[\operatorname{Int}_T^{\operatorname{End}}(S/B)\backslash \{ \emptyset \}  \subseteq  \{B \subseteq S' \le _qS\,|\,S'\text{ is nonempty and closed in }{{\mathcal{T}}_{1}}\}.\]
    On the other hand, we show
    \[\operatorname{Int}_T^{\operatorname{End}}(S/B)\backslash \{ \emptyset \}  \supseteq  \{B \subseteq S' \le _qS\,|\,S'\text{ is nonempty and closed in }{{\mathcal{T}}_{1}}\}\]
    as follows.

    Let $B\subseteq {S}'\le _qS$ such that ${S}'$ is nonempty and closed in ${{\mathcal{T}}_{1}}$. Then we need to show ${S}'\in \operatorname{Int}_{T}^{\operatorname{End}}(S/B)$.
    
    Since ${{\mathcal{T}}_{1}}$ is the collection of all unions of finite intersections of elements of ${{\beta }_{1}}=\{S\backslash A\,|\,A\in \operatorname{Int}_{T}^{\operatorname{End}}(S/B)\}\bigcup \{S\}$, by DeMorgan’s Laws, the collection of all closed sets in ${{\mathcal{T}}_{1}}$ is the family of all intersections of finite unions of elements of $\operatorname{Int}_{T}^{\operatorname{End}}(S/B)\bigcup \{\emptyset \}$. Then ${S}'$ is an intersection of finite unions of elements of $\operatorname{Int}_{T}^{\operatorname{End}}(S/B)$. Because (iii) is supposed to be true, any finite union of elements of $\operatorname{Int}_{T}^{\operatorname{End}}(S/B)$ belongs to $\operatorname{Int}_{T}^{\operatorname{End}}(S/B)$. So ${S}'$ is an intersection of elements of $\operatorname{Int}_{T}^{\operatorname{End}}(S/B)$. Then by Lemma \ref{3.2.1}, ${S}'\in \operatorname{Int}_{T}^{\operatorname{End}}(S/B)$, as desired. Hence 
    \[\operatorname{Int}_T^{\operatorname{End}}(S/B)\backslash \{ \emptyset \}  \supseteq  \{B \subseteq S' \le _qS\,|\,S'\text{ is nonempty and closed in }{{\mathcal{T}}_{1}}\}.\]

    Therefore,
\[\operatorname{Int}_T^{\operatorname{End}}(S/B)\backslash \{ \emptyset \} = \{B \subseteq S' \le _qS\,|\,S'\text{ is nonempty and closed in }{{\mathcal{T}}_{1}}\}.\]
Thus ${{\mathcal{T}}_{1}}\in {{P}_{1}}$. We showed ${{P}_{1}}\subseteq {{Q}_{1}}$ at the beginning of our proof, and so by the definition of ${{Q}_{1}}$, ${{\mathcal{T}}_{1}}$ is the coarsest topology in $P_1$. Then by the definition of ${{P}_{1}}$, ${{\mathcal{T}}_{1}}$ is the coarsest topology on $S$ such that Equation $($\ref{5.1}$)$ is satisfied.

(iv) $\Rightarrow$ (i): Obvious.
\end{proof}
	By Theorem \ref{4.1.2}, when ${{P}_{1}}\ne \emptyset $, we can put a topology such as ${{\mathcal{T}}_{1}}$ on $S$ such that Equation (\ref{5.1}) is satisfied. Then $\operatorname{Int}_{T}^{\operatorname{End}}(S/B)$ in Corollary \ref{2.2.5} can be replaced by \{$B\subseteq {S}'\le _qS\,|\,{S}'$ is closed and nonempty\} if $\emptyset \notin \operatorname{Int}_{T}^{\operatorname{End}}(S/B)$, and it can be replaced by \{$B\subseteq {S}'\le _qS\,|\,{S}'$ is closed\} if $\emptyset \in \operatorname{Int}_{T}^{\operatorname{End}}(S/B)$.
 
	However, Example \ref{3.2.3} tells us that statement (ii) (and (iii)) in Theorem \ref{4.1.2} is not always true. Hence it is possible that ${{P}_{1}}=\emptyset $. When this occurs, we cannot endow $S$ with any topology such that Equation (\ref{5.1}) is satisfied. Then in this case, there is no topology on $S$ to characterize the Galois correspondence depicted in Corollary \ref{2.2.5} in terms of closed subsets of $S$.

\subsection{Topologies on a $T$-space $S$ for $\operatorname{Int}_T^{\operatorname{Aut}}(S/B)$} \label{I Topologies on $S$ for IntAut}
	To replace $\operatorname{Int}_{T}^{\operatorname{Aut}}(S/B)$ in Corollary \ref{2.2.6}, we have
\begin{lemma}
    \label{4.2.1}
    Let $S$ be a $T$-space with $B\subseteq S$. Then the topology ${{\mathcal{T}}_{2}}$ on $S$ generated by a subbasis ${{\beta }_{2}}:=\{S\backslash A\,|\,A\in \operatorname{Int}_{T}^{\operatorname{Aut}}(S/B)\}\bigcup \{S\}$ is the smallest topology on $S$ such that the collection of all closed sets contains $\operatorname{Int}_{T}^{\operatorname{Aut}}(S/B)$.
\end{lemma}   
\begin{proof}
    Almost the same as the proof of Lemma \ref{4.1.1} except that $\operatorname{Int}_{T}^{\operatorname{End}}(S/B)$, ${{\mathcal{T}}_{1}}$ and ${{\beta }_{1}}$ are replaced by $\operatorname{Int}_{T}^{\operatorname{Aut}}(S/B)$, ${{\mathcal{T}}_{2}}$ and ${{\beta }_{2}}$, respectively.
\end{proof}
For a $T$-space $S$ with $B \subseteq S$, to characterize the Galois correspondence depicted in Corollary \ref{2.2.6} in terms of closed quasi-$T$-subspaces of $S$, we need a topology on $S$ such that the following equation is satisfied.
\begin{equation}
\label{5.3}
    \operatorname{Int}_T^{\operatorname{Aut}}(S/B)\backslash \{ \emptyset \}  = \{B \subseteq S' \le _qS\,|\,S'\text{ is closed and nonempty}\}
\end{equation}
\begin{theorem}
    \label{4.2.2}
    Let $S$ be a $T$-space and let $B\subseteq S$. Let
	\[{{P}_{2}}=\{\text{all topologies on $S$ such that Equation $($\ref{5.3}$)$ is satisfied}\}\] and let
	\[{{Q}_{2}}=\{\text{all topologies on $S$ which are finer than ${{\mathcal{T}}_{2}}$}\},\] 
where ${{\mathcal{T}}_{2}}$ is defined in Lemma \ref{4.2.1}. Then ${{P}_{2}}\subseteq {{Q}_{2}}$ and the following statements are equivalent:
\begin{enumerate}
    \item[(i)] ${P_2} \ne \emptyset $.
    \item[(ii)] $\forall {{S}_{1}},{{S}_{2}}\in \operatorname{Int}_{T}^{\operatorname{Aut}}(S/B)$, ${{S}_{1}}\bigcup {{S}_{2}}\in \operatorname{Int}_{T}^{\operatorname{Aut}}(S/B)$. $($Or equivalently, ${{\beta }_{2}}$ defined in Lemma \ref{4.2.1} is a basis for ${{\mathcal{T}}_{2}})$.
    \item[(iii)]	$\operatorname{Int}_{T}^{\operatorname{Aut}}(S/B)$ is a $\vee$-sublattice of the lattice $\operatorname{Sub}_T(S)$ defined in Proposition \ref{3.1.1}.
    \item[(iv)] ${{\mathcal{T}}_{2}}\in {{P}_{2}}(\subseteq {{Q}_{2}})$; that is, ${{\mathcal{T}}_{2}}$ is the coarsest topology on $S$ such that Equation $($\ref{5.3}$)$ is satisfied.
\end{enumerate}
\end{theorem}
\begin{proof}
    Almost the same as the proof of Theorem \ref{4.1.2} except that $\operatorname{Int}_{T}^{\operatorname{End}}(S/B)$, ${{\mathcal{T}}_{1}}$, ${{\beta }_{1}}$, ${{P}_{1}}$ and ${{Q}_{1}}$ are replaced by $\operatorname{Int}_{T}^{\operatorname{Aut}}(S/B)$, ${{\mathcal{T}}_{2}}$, ${{\beta }_{2}}$, ${{P}_{2}}$ and ${{Q}_{2}}$, respectively, and accordingly, Lemmas \ref{3.2.1} and \ref{4.1.1} are replaced by Lemmas \ref{3.2.4} and \ref{4.2.1}, respectively.
\end{proof}
    It follows that when ${{P}_{2}}\ne \emptyset $, we can give a topology such as ${{\mathcal{T}}_{2}}$ on $S$ such that Equation (\ref{5.3}) is satisfied. Then $\operatorname{Int}_{T}^{\operatorname{Aut}}(S/B)$ in Corollary \ref{2.2.6} can be substituted with \{$B\subseteq {S}'\le _qS\,|\,{S}'$ is closed and nonempty\} if $\emptyset \notin \operatorname{Int}_{T}^{\operatorname{Aut}}(S/B)$, and it can be substituted with \{$B\subseteq {S}'\le _qS\,|\,{S}'$ is closed\} if $\emptyset \in \operatorname{Int}_{T}^{\operatorname{Aut}}(S/B)$.

	Recall that at the end of Subsection \ref{I Topologies on $S$ for Int}, we use Example \ref{3.2.3} to show that statement (ii) (and (iii)) in Theorem \ref{4.1.2} is not always true. Because in Example \ref{3.2.3}, $\operatorname{Int}_{T}^{\operatorname{End}}(S/\mathbb{Q})=\operatorname{Int}_{T}^{\operatorname{Aut}}(S/\mathbb{Q})$, statement (ii) (and (iii)) in Theorem \ref{4.2.2} is not necessarily true, either. Hence it is possible that ${{P}_{2}}=\emptyset$. When this occurs, we cannot give any topology on $S$ such that Equation (\ref{5.3}) is satisfied. Then in this case, there does not exist any topology on $S$ such that the Galois correspondence depicted in Corollary \ref{2.2.6} can be characterized in terms of closed subsets of $S$.

\subsection{Topologies on $\operatorname{End}_{T}(S)$} \label{I Topologies on End}
To replace $\operatorname{GSMn}_{T}(S/B)$ in Corollary \ref{2.2.5}, we need
\begin{lemma}
    \label{4.3.1} Let $S$ be a $T$-space with $B\subseteq S$. Then the topology ${{\mathcal{T}}_{3}}$ on $\operatorname{End}_{T}(S)$ generated by a subbasis \[{{\beta }_{3}}:=\{\operatorname{End}_{T}(S)\backslash H\,|\,H\in \operatorname{GSMn}_{T}(S/B)\}\bigcup \{\operatorname{End}_{T}(S)\}\] is the smallest topology on $\operatorname{End}_{T}(S)$ such that the collection of all closed sets contains $\operatorname{GSMn}_{T}(S/B)$.
\end{lemma}
\begin{proof}
    Since the union of all elements of ${{\beta }_{3}}$ is $\operatorname{End}_{T}(S)$, ${{\beta }_{3}}$ qualifies as a subbasis for a topology on $\operatorname{End}_{T}(S)$. Hence the topology ${{\mathcal{T}}_{3}}$ generated by ${{\beta }_{3}}$ equals the intersection of all topologies on $\operatorname{End}_{T}(S)$ that contain ${{\beta }_{3}}$ (see e.g. \cite{6}), and so ${{\mathcal{T}}_{3}}$ is the smallest topology on $\operatorname{End}_{T}(S)$ containing ${{\beta }_{3}}$. Thus by the definition of ${{\beta }_{3}}$, ${{\mathcal{T}}_{3}}$ is the smallest topology on $\operatorname{End}_{T}(S)$ such that the collection of all closed sets contains $\operatorname{GSMn}_{T}(S/B)$.	
\end{proof}
For a $T$-space $S$ with $B \subseteq S$, to characterize the Galois correspondence depicted in Corollary \ref{2.2.5} in terms of closed submonoids of $\operatorname{GMn}_T(S/B)$, we need a topology on $\operatorname{End}_{T}(S)$ such that the following equation is satisfied.
\begin{equation}
    \label{5.4}
    \operatorname{GSMn}_T(S/B) = \{M \in \operatorname{SMn}(\operatorname{GMn}_T(S/B))\,|\,M \text{ is closed} \}
\end{equation}
\begin{remark}
    See Notation \ref{2.1.2} for $\operatorname{SMn}$.
\end{remark}
\begin{theorem}
    \label{4.3.2}
    Let $S$ be a $T$-space and let $B\subseteq S$. Let
   \[{{P}_{3}}=\{\text{all topologies on $\operatorname{End}_{T}(S)$ such that Equation $($\ref{5.4}$)$ is satisfied\}}\] and let
\[\text{${{Q}_{3}}=$\{all topologies on $\operatorname{End}_{T}(S)$ which are finer than }{{\mathcal{T}}_{3}}\},\]
where ${{\mathcal{T}}_{3}}$ is defined in Lemma \ref{4.3.1}. Then ${{P}_{3}}\subseteq {{Q}_{3}}$ and the following statements are equivalent:
\begin{enumerate}
    \item [(i)] ${P_3} \ne \emptyset $.
    \item[(ii)] For any intersection $M$ of finite unions of elements of $\operatorname{GSMn}_{T}(S/B)$, $M\in \operatorname{SMn}(\operatorname{End}_{T}(S))$ implies $M\in \operatorname{GSMn}_{T}(S/B)$.
\item[(iii)] ${{\mathcal{T}}_{3}}\in {{P}_{3}}(\subseteq {{Q}_{3}})$; that is, ${{\mathcal{T}}_{3}}$ is the coarsest topology on $\operatorname{End}_{T}(S)$ such that Equation $($\ref{5.4}$)$ is satisfied.
\end{enumerate}

Moreover, if $\operatorname{GSMn}_{T}(S/B)$ is a complete $\vee$-sublattice of the complete lattice $\operatorname{SMn}(\operatorname{End}_{T}(S))$ defined in Proposition \ref{3.1.2}, then the above three statements are true.
\end{theorem}
\begin{proof}
    By Lemma \ref{4.3.1}, ${{\mathcal{T}}_{3}}$ is the smallest topology on $\operatorname{End}_{T}(S)$ such that the collection of all closed sets contains $\operatorname{GSMn}_{T}(S/B)$. Suppose $ \mathcal{T} \in {{P}_{3}}$. Then all elements of $\operatorname{GSMn}_{T}(S/B)$ are closed in $\mathcal{T} $, and so ${\mathcal{T}}_{3} \subseteq \mathcal{T}$. Hence $\mathcal{T} \in {{Q}_{3}}$. Thus ${{P}_{3}}\subseteq {{Q}_{3}}$. 
    
(i)$\Rightarrow $(ii): For any intersection $M$ given in (ii), from Definitions \ref{2.2.3} and \ref{1.5.1}, we can tell $M\subseteq \operatorname{GMn}_{T}(S/B)$. Suppose $M\in \operatorname{SMn}(\operatorname{End}_{T}(S))$. Then $M\in \operatorname{SMn}(\operatorname{GMn}_{T}(S/B))$. Since ${P_3} \ne \emptyset $, let $\mathcal{T} \in {{P}_{3}}$. By the definition of ${{P}_{3}}$, every element of $\operatorname{GSMn}_{T}(S/B)$ is closed in $\mathcal{T} $, and so is the intersection $M$ (of finite unions of elements of $\operatorname{GSMn}_{T}(S/B)$). Again by the definition of ${{P}_{3}}$, $M\in \operatorname{GSMn}_{T}(S/B)$.

(ii)$\Rightarrow $(iii): 

By the definition of ${{\mathcal{T}}_{3}}$ in Lemma \ref{4.3.1}, all elements of $\operatorname{GSMn}_{T}(S/B)$ are closed in ${{\mathcal{T}}_{3}}$, and so by Definition \ref{2.2.3} and Proposition \ref{1.5.2},
\[\operatorname{GSMn}_T(S/B) \subseteq \{M \in \operatorname{SMn}(\operatorname{GMn}_T(S/B)) \,|\, M \text{ is closed in } {{\mathcal T}_3}\}.\]

On the other hand, we prove
\[\operatorname{GSMn}_T(S/B) \supseteq \{M \in \operatorname{SMn}(\operatorname{GMn}_T(S/B)) \,|\, M \text{ is closed in } {{\mathcal T}_3}\}\]
as follows.

Let $M\in \operatorname{SMn}(\operatorname{GMn}_{T}(S/B))$ such that $M$ is closed in ${{\mathcal{T}}_{3}}$. We need to show $M\in \operatorname{GSMn}_{T}(S/B)$.

Because ${{\mathcal{T}}_{3}}$ is generated by subbasis ${{\beta }_{3}}$, ${{\mathcal{T}}_{3}}$ is the collection of all unions of finite intersections of elements of 
\[{{\beta }_{3}}=\{\operatorname{End}_{T}(S)\backslash H\,|\,H\in \operatorname{GSMn}_{T}(S/B)\}\bigcup \{\operatorname{End}_{T}(S)\}.\]
Since $M$ is closed in ${{\mathcal{T}}_{3}}$, $\operatorname{End}_{T}(S)\backslash M$ is a union of finite intersections of elements of ${{\beta }_{3}}$. It follows from DeMorgan’s Laws that $M$ is an intersection of finite unions of elements of $\operatorname{GSMn}_{T}(S/B)\bigcup \{\emptyset \}$. Because $\text{(Id}\in )M\ne \emptyset $, $M$ is an intersection of finite unions of elements of $\operatorname{GSMn}_{T}(S/B)$. Since (ii) is assumed to be true and $M\in \operatorname{SMn}(\operatorname{GMn}_{T}(S/B))\subseteq \operatorname{SMn}(\operatorname{End}_{T}(S))$, $M\in \operatorname{GSMn}_{T}(S/B)$, as desired. Hence
\[\operatorname{GSMn}_T(S/B) \supseteq \{M \in \operatorname{SMn}(\operatorname{GMn}_T(S/B)) \,|\, M \text{ is closed in } {{\mathcal T}_3}\}\]
Therefore, 
\[\operatorname{GSMn}_T(S/B) = \{M \in \operatorname{SMn}(\operatorname{GMn}_T(S/B)) \,|\, M \text{ is closed in } {{\mathcal T}_3}\}\]

Thus ${{\mathcal{T}}_{3}}\in {{P}_{3}}$. We showed ${{P}_{3}}\subseteq {{Q}_{3}}$ at the beginning of our proof. So by the definition of ${{Q}_{3}}$, ${{\mathcal{T}}_{3}}$ is the coarsest topology in $P_3$. Then by the definition of ${{P}_{3}}$, ${{\mathcal{T}}_{3}}$ is the coarsest topology on $\operatorname{End}_{T}(S)$ such that Equation $($\ref{5.4}$)$ is satisfied.

(iii)$\Rightarrow $(i): Obvious.

By Proposition \ref{3.3.2}, any intersection of elements of $\operatorname{GSMn}_{T}(S/B)$ still lies in $\operatorname{GSMn}_{T}(S/B)$. Then as a result of the first distributive law of set operations, any intersection $M$ of finite unions of elements of $\operatorname{GSMn}_{T}(S/B)$ is a union of elements of $\operatorname{GSMn}_{T}(S/B)$. Suppose that $\operatorname{GSMn}_{T}(S/B)$ is a complete $\vee$-sublattice of the complete lattice $\operatorname{SMn}(\operatorname{End}_{T}(S))$ defined in Proposition \ref{3.1.2}. Then $M\in \operatorname{SMn}(\operatorname{End}_{T}(S))$ implies $M\in \operatorname{GSMn}_{T}(S/B)$, and so (ii) is true (in this case).
\end{proof}
When ${{P}_{3}}\ne \emptyset $, by Theorem \ref{4.3.2}, we can endow $\operatorname{End}_{T}(S)$ with a topology such as ${{\mathcal{T}}_{3}}$ such that Equation $($\ref{5.4}$)$ is satisfied. Then $\operatorname{GSMn}_{T}(S/B)$ in Corollary \ref{2.2.5} can be replaced by the set of all closed submonoids of $\operatorname{GMn}_T(S/B)$.
	
 Nevertheless, it is possible that ${{P}_{3}}=\emptyset $. Recall that in Example \ref{3.3.3},
${{M}_{1}},{{M}_{2}}\in \operatorname{GSMn}_{T}(D/\emptyset )$ and ${{M}_{1}}\bigcup {{M}_{2}}\in \operatorname{SMn}(\operatorname{End}_{T}(D))$, but
${{M}_{1}}\bigcup {{M}_{2}}\notin \operatorname{GSMn}_{T}(D/\emptyset )$. 
Thus for Example \ref{3.3.3}, statement (ii) in Theorem \ref{4.3.2} is not true, and hence ${{P}_{3}}=\emptyset $. 
Therefore, it is possible that we cannot endow $\operatorname{End}_{T}(S)$ with any topology such that the Galois correspondence depicted in Corollary \ref{2.2.5} can be characterized in terms of closed submonoids of $\operatorname{GMn}_T(S/B)$.
\subsection{Topologies on $\operatorname{Aut}_{T}(S)$} \label{I Topologies on Aut}
To replace $\operatorname{GSGr}_{T}(S/B)$ in Corollary \ref{2.2.6}, we need
\begin{lemma}
    \label{4.4.1}
    Let $S$ be a $T$-space and let $B\subseteq S$. Then the topology ${{\mathcal{T}}_{4}}$ on $\operatorname{Aut}_{T}(S)$ generated by a subbasis 
    \[{{\beta }_{4}}:=\{\operatorname{Aut}_{T}(S)\backslash H\,|\,H\in \operatorname{GSGr}_{T}(S/B)\}\bigcup \{\operatorname{Aut}_{T}(S)\}\]
    is the smallest topology on $\operatorname{Aut}_{T}(S)$ such that the collection of all closed sets contains $\operatorname{GSGr}_{T}(S/B)$.
\end{lemma} 
\begin{proof}
    Almost the same as the proof of Lemma \ref{4.3.1} except that $\operatorname{GSMn}_{T}(S/B)$, $\operatorname{End}_{T}(S)$, ${{\mathcal{T}}_{3}}$ and ${{\beta }_{3}}$ are replaced by $\operatorname{GSGr}_{T}(S/B)$, $\operatorname{Aut}_{T}(S)$, ${{\mathcal{T}}_{4}}$ and ${{\beta }_{4}}$, respectively.
\end{proof}
For a $T$-space $S$ with $B \subseteq S$, to characterize the Galois correspondence depicted in Corollary \ref{2.2.6} in terms of closed subgroups of $\operatorname{GGr}_T(S/B)$, we need a topology on $\operatorname{Aut}_{T}(S)$ such that the following equation is satisfied.
\begin{equation}
    \label{5.5}
    \operatorname{GSGr}_T(S/B) = \{M \in \operatorname{SGr}(\operatorname{GGr}_T(S/B))\,|\,M \text{ is closed} \}
\end{equation}
\begin{remark}
    See Notation \ref{2.1.2} for $\operatorname{SGr}$.
\end{remark}
\begin{theorem}
    \label{4.4.2}
    Let $S$ be a $T$-space and let $B\subseteq S$. Let
   \[{{P}_{4}}=\{\text{all topologies on $\operatorname{Aut}_{T}(S)$ such that Equation $($\ref{5.5}$)$ is satisfied\}}\] and let
\[\text{${{Q}_{4}}=$\{all topologies on $\operatorname{Aut}_{T}(S)$ which are finer than }{{\mathcal{T}}_{4}}\},\]
where ${{\mathcal{T}}_{4}}$ is defined in Lemma \ref{4.4.1}. Then ${{P}_{4}}\subseteq {{Q}_{4}}$ and the following statements are equivalent:
\begin{enumerate}
    \item [(i)] ${P_4} \ne \emptyset $.
    \item[(ii)] For any intersection $G$ of finite unions of elements of $\operatorname{GSGr}_{T}(S/B)$, $G\in \operatorname{SGr}(\operatorname{Aut}_{T}(S))$ implies $G\in \operatorname{GSGr}_{T}(S/B)$.
\item[(iii)] ${{\mathcal{T}}_{4}}\in {{P}_{4}}(\subseteq {{Q}_{4}})$; that is, ${{\mathcal{T}}_{4}}$ is the coarsest topology on $\operatorname{Aut}_{T}(S)$ such that Equation $($\ref{5.5}$)$ is satisfied.
\end{enumerate}

Moreover, if $\operatorname{GSGr}_{T}(S/B)$ is a complete $\vee$-sublattice of the complete lattice $\operatorname{SGr}(\operatorname{Aut}_{T}(S))$ defined in Proposition \ref{3.1.3}, then the above three statements are true.
\end{theorem}
\begin{proof}
    Almost the same as the proof of Theorem \ref{4.3.2} except that $\operatorname{End}_{T}(S)$, $\operatorname{GSMn}_{T}(S/B)$, ${{\mathcal{T}}_{3}}$, ${{\beta }_{3}}$, ${{P}_{3}}$, ${{Q}_{3}}$, and $M$ are replaced by $\operatorname{Aut}_{T}(S)$, $\operatorname{GSGr}_{T}(S/B)$, ${{\mathcal{T}}_{4}}$, ${{\beta }_{4}}$, ${{P}_{4}}$, ${{Q}_{4}}$, and $G$, respectively, and accordingly, Lemma \ref{4.3.1} and Propositions \ref{1.5.2}, \ref{3.1.2} and \ref{3.3.2} are replaced by Lemma \ref{4.4.1} and Propositions \ref{1.5.3}, \ref{3.1.3} and \ref{3.3.5}, respectively.
\end{proof}
	Nevertheless, we do not know whether the three equivalent statements in Theorem \ref{4.4.2} are always true, and hence we do not know whether there always exists a topology on $\operatorname{Aut}_{T}(S)$ such that Equation $($\ref{5.5}$)$ is satisfied. In other words, for any $T$-space $S$ with $B\subseteq S$, we do not know whether we can always endow $\operatorname{Aut}_{T}(S)$ with a topology such that $\operatorname{GSGr}_T(S/B)$ in Corollary \ref{2.2.6} can be replaced by the set of all closed subgroups of $\operatorname{GGr}_T(S/B)$.

\section{Generalized morphisms and isomorphisms from a $T_1$-space to a $T_2$-space} \label{theta-mor and theta-iso}
In this section, we shall define generalized morphisms and isomorphisms from a $T_1$-space to a $T_2$-space, where both $T_1$ and $T_2$ are operator semigroups.
\subsection{$\theta$-morphisms} \label{I $theta$-morphisms}
    In algebra, ring homomorphisms from the ring $B[x]$ to the ring ${B}'[x]$, where both $B$ and ${B}'$ are commutative rings, are defined. However, Definition \ref{1.3.1} does not cover this kind of morphisms because, by Example \ref{T on field--1}, $B[x]$ and ${B}'[x]$ may induce different operator semigroups. So we generalize Definition \ref{1.3.1} as follows.
\begin{definition} \label{5.3.1}
    Let $T_1$ and $T_2$ be operator semigroups, let $\theta $ be a map from $T_1$ to $T_2$, and let $\phi $ be a map from a $T_1$-space $S$ to a $T_2$-space. If $\forall f\in {{T}_{1}}$ and $a\in S$, $\phi (f(a))=\theta (f)(\phi (a))$, then $\phi $ is called a \emph{$\theta $-morphism}.
\end{definition}
\begin{remark}
     If ${{T}_{1}}={{T}_{2}}$ and $\theta $ is the identity map, then $\phi $ is also a ${{T}_{1}}$-morphism defined by Definition \ref{1.3.1}.
\end{remark}
Let us see two examples as follows. 
\begin{proposition} \label{5.3.2}
    Let $A$ and $R$ be differential rings, let ${{T}_{A}}$ $($resp. ${{T}_{R}})$ be the operator semigroup on $A$ $($resp. $R)$ defined as in Example \ref{T on diff ring--1}, and let a map $\theta :{{T}_{A}}\to {{T}_{R}}$ be given by $\partial_A^n\mapsto \partial _{R}^{n}$, $\forall n\in {{\mathbb{N}}_{0}}$, where ${{\partial }_{A}}$ and ${{\partial }_{R}}$ are derivations of $A$ and $R$, respectively. Let $\phi :A\to R$ be a ring homomorphism. Then $\phi $ is a differential homomorphism if and only if $\phi $ is a $\theta $-morphism.
\end{proposition}
\begin{proof}
    Recall that a ring homomorphism $\phi :A\to R$ is a differential homomorphism if and only if it satisfies $\phi ({{\partial }_{A}}(a))={{\partial }_{R}}(\phi (a)),\forall a\in A$ (see e.g. \cite{1}). Thus we only need to show that $\phi$ is a $\theta $-morphism if and only if $\phi$ satisfies $\phi ({{\partial }_{A}}(a))={{\partial }_{R}}(\phi (a)),\forall a\in A$. Since $\forall n\in {{\mathbb{N}}_{0}}$, $\theta (\partial _{A}^{n})=\partial _{R}^{n}$, the latter equivalence is obvious by Definition \ref{5.3.1}.	
\end{proof}
	Now Proposition \ref{1.3.7} can be generalized to continuous functions from a topological space to another.
\begin{proposition} \label{5.3.3}
    Let $X$and $Y$ be topological spaces. Let $\mathcal{P}(X)$ and $\mathcal{P}(Y)$ denote the power sets of $X$ and $Y$, respectively. Let 
\[{{T}_{X}}=\{{\operatorname{Id}_{X}}, {\operatorname{Cl}_{X}}:\mathcal{P}(X)\to \mathcal{P}(X)\text{ given by }{{A}_{X}}\mapsto \overline{{{A}_{X}}}\},\] and 
\[{{T}_{Y}}=\{{\operatorname{Id}_{Y}}, {\operatorname{Cl}_{Y}}:\mathcal{P}(Y)\to \mathcal{P}(Y)\text{ given by }{{A}_{Y}}\mapsto \overline{{{A}_{Y}}}\},\] 
where ${\operatorname{Id}_{X}}$ and ${\operatorname{Id}_{Y}}$ denote the identity functions on $\mathcal{P}(X)$ and $\mathcal{P}(Y)$, respectively. Let a map $p:X\to Y$ induce a map $p^*:\mathcal{P}(X)\to \mathcal{P}(Y)$ as follows.

$\forall {{A}_{X}}\in \mathcal{P}(X)$ that is closed in $X$, let $p^*({{A}_{X}})=\overline{p({{A}_{X}})}$, where $p({{A}_{X}})$ denotes the set $\{p(x)\,|\,x\in {{A}_{X}}\}$ for convenience. And $\forall {{A}_{X}}\in \mathcal{P}(X)$ which is not closed in $X$, let $p^*({{A}_{X}})=p({{A}_{X}})$. 

Let $\theta $ be the map from ${{T}_{X}}$ to ${{T}_{Y}}$ given by $ \theta ({\operatorname{Id}_{X}})={\operatorname{Id}_{Y}}$ and $\theta ({\operatorname{Cl}_{X}})={\operatorname{Cl}_{Y}}$. Then $p$ is continuous if and only if $p^*$ is a $\theta $-morphism.
\end{proposition}
 \begin{proof}
     Our proof is analogous to that of Proposition \ref{1.3.7}.
     
Both ${{T}_{X}}$ and ${{T}_{Y}}$ are operator semigroups. ${{\left\langle \mathcal{P}(X) \right\rangle }_{{{T}_{X}}}}=\mathcal{P}(X)$ and ${{\left\langle \mathcal{P}(Y) \right\rangle }_{{{T}_{Y}}}}=\mathcal{P}(Y)$ are ${{T}_{X}}$-space and ${{T}_{Y}}$-space, respectively.

1. \emph{Necessity} Suppose that $p$ is continuous. To prove that $p^*$ is a $\theta \text{-}$morphism, it is sufficient to show that $\forall {{A}_{X}}\in \mathcal{P}(X)$ and $f\in {{T}_{X}}$, $p^*(f({{A}_{X}}))=\theta (f)(p^*({{A}_{X}}))$. The equation holds when $f=\text{I}{{\text{d}}_{X}}$. So it suffices to show that
\begin{center}
    $\forall {{A}_{X}}\in \mathcal{P}(X)$, $p^*(\text{C}{{\text{l}}_{X}}({{A}_{X}}))=\text{C}{{\text{l}}_{Y}}(p^*({{A}_{X}}))(=\theta (\text{C}{{\text{l}}_{X}})(p^*({{A}_{X}})))$.
\end{center}

If ${{A}_{X}}=\overline{{{A}_{X}}}$, by the definitions of $p^*$, $\text{C}{{\text{l}}_{X}}$ and $\text{C}{{\text{l}}_{Y}}$, 
\begin{center}
    $\text{C}{{\text{l}}_{Y}}(p^*({{A}_{X}}))=\text{C}{{\text{l}}_{Y}}(\overline{p({{A}_{X}})})=\overline{p({{A}_{X}})}=p^*({{A}_{X}})=p^*(\text{C}{{\text{l}}_{X}}({{A}_{X}}))$, 
\end{center}
as desired.

Since $p$ is continuous, $p(\overline{{{A}_{X}}})\subseteq \overline{p({{A}_{X}})}$ (see e.g. \cite{6}). Hence $\overline{p(\overline{{{A}_{X}}})}\subseteq \overline{p({{A}_{X}})}$. On the other hand, $\overline{p(\overline{{{A}_{X}}})}\supseteq \overline{p({{A}_{X}})}$ because $p(\overline{{{A}_{X}}})\supseteq p({{A}_{X}})$. So $\overline{p(\overline{{{A}_{X}}})}=\overline{p({{A}_{X}})}$. Thus if ${{A}_{X}}\ne \overline{{{A}_{X}}}$, then
\begin{center}
    $\text{C}{{\text{l}}_{Y}}(p^*({{A}_{X}}))=\text{C}{{\text{l}}_{Y}}(p({{A}_{X}}))=\overline{p({{A}_{X}})}=\overline{p(\overline{{{A}_{X}}})}=p^*(\overline{{{A}_{X}}})=p^*(\text{C}{{\text{l}}_{X}}({{A}_{X}}))$, 
\end{center}
as desired.

2. \emph{Sufficiency} Suppose that $p^*$ is a $\theta $-morphism. To prove that $p$ is continuous, we only need to show that ${{p}^{-1}}({{A}_{Y}})=\overline{{{p}^{-1}}({{A}_{Y}})}$, $\forall {{A}_{Y}}\in \mathcal{P}(Y)$ with ${{A}_{Y}}=\overline{{{A}_{Y}}}$.

Assume $\exists {{A}_{Y}}\in \mathcal{P}(Y)$ with ${{A}_{Y}}=\overline{{{A}_{Y}}}$ and ${{p}^{-1}}({{A}_{Y}})\ne \overline{{{p}^{-1}}({{A}_{Y}})}$.	Let $B={{p}^{-1}}({{A}_{Y}})$. Then $B\ne \overline{B}$ and $p(\overline{B})\nsubseteq {{A}_{Y}}$ (because otherwise $B=\overline{B}$). Hence
\begin{center}
    $p^*(\text{C}{{\text{l}}_{X}}(B))=p^*(\overline{B})=\overline{p(\overline{B})}\supseteq p(\overline{B})\nsubseteq {{A}_{Y}}$,
\end{center}
 but
\begin{center}
    $\text{C}{{\text{l}}_{Y}}(p^*(B))=\text{C}{{\text{l}}_{Y}}(p(B))=\overline{p(B)}\subseteq \overline{{{A}_{Y}}}={{A}_{Y}}$ 
\end{center}
because $p(B)\subseteq {{A}_{Y}}$.
Therefore, $p^*(\text{C}{{\text{l}}_{X}}(B))\ne \text{C}{{\text{l}}_{Y}}(p^*(B))(=\theta (\text{C}{{\text{l}}_{X}})(p^*(B)))$, which is contrary to the assumption that $p^*$ is a $\theta $-morphism.
 \end{proof}

 	However, when we tried to generalize Proposition \ref{1.3.4} to the case of $\theta $-morphisms, we found that in Definition \ref{5.3.1}, the condition $\theta $ being a map is too strong for ring homomorphisms. Basically this is because in Example \ref{T on field--1}, the map $\tau $ which induces $T$ is not necessarily injective (and so $\theta $ defined in Proposition \ref{5.3.6} is not necessarily a map). Hence we generalize Definition \ref{5.3.1} to Definition \ref{5.3.5}, where the condition $\theta $ being a map is replaced by a weaker one. 

First, we allow $\theta$ to be a binary relation, not necessarily a function.

\begin{notation} \label{5.3.4}
    Let $D_1$ (resp. $D_2$) be a set, let $M_1$ (resp. $M_2$) be a set of functions from $D_1$ to $D_1$ (resp. a set of functions from $D_2$ to $D_2$), let binary relation $\theta \subseteq {{M}_{1}}\times {{M}_{2}}$, and let $A\subseteq {{D}_{2}}$. We denote by $\theta {{|}_{A}}$ the binary relation $\{(f,g{{|}_{A}})\,|\,(f,g)\in \theta \}$, where $g{{|}_{A}}$ is the restriction of $g$ to $A$.
\end{notation}

\begin{definition} \label{5.3.a}
    Let $\theta$ and $A$ be defined as in Notation \ref{5.3.4}. Let $f\in \operatorname{Dom}\theta (\subseteq M_1)$ and $a \in A$. 
    If $\forall(f,g_1),(f,g_2)\in \theta$, $g_{1}(a)=g_2(a)$, then we say that $\theta (f)(a)$ is \emph{well-defined} and let $\theta (f)(a)=g_{1}(a)$.
\end{definition}

\begin{proposition} \label{5.3.b}
  Let $\theta$ and $A$ be defined as in Notation \ref{5.3.4}. Then the following statements are equivalent: \begin{enumerate}
        \item [(i)] $\theta {{|}_{A}}$ is a map.
         \item [(ii)] $\theta (f)(a)$ is well-defined, $\forall f\in \operatorname{Dom}\theta $ and $a\in A$.
    \end{enumerate}
\end{proposition}
\begin{proof}
    $\theta {{|}_{A}}$ is a map;
    
\hspace{2mm}$ \Leftrightarrow $ $\forall(f,g_1),(f,g_2)\in \theta$, $g_{1}{{|}_{A}}=g_2{{|}_{A}}$;

\hspace{2mm}$ \Leftrightarrow $ $\forall(f,g_1),(f,g_2)\in \theta$ and $a\in A$, $g_{1}(a)=g_2(a)$;

\hspace{2mm}$ \Leftrightarrow $ $\forall f\in \operatorname{Dom}\theta $ and $a\in A$, $\theta (f)(a)$ is well-defined. 
\end{proof}

Moreover, to simplify our descriptions, we need a convention as follows.
\begin{convention1} \label{convention}
    In this article, unless otherwise specified, by an equation we always imply that both sides of the equation are well-defined.
\end{convention1}

Then Definition \ref{5.3.1} is generalized to
\begin{definition} \label{5.3.5}
    Let $T_1$ and $T_2$ be operator semigroups, let $\theta \subseteq {{T}_{1}}\times {{T}_{2}}$, and let $\phi $ be a map from a $T_1$-space $S$ to a $T_2$-space. If $\forall f\in \operatorname{Dom}\theta $ and $a\in S$, $\phi (f(a))=\theta (f)(\phi (a))$, then $\phi $ is called a \emph{$\theta $-morphism}.
\end{definition}
\begin{corollary} \label{5.3.c}
Let $\phi$ be a map, not necessarily a $\theta$-morphism, given as in Definition \ref{5.3.5}. If $\phi $ is a $\theta$-morphism, then $\theta {{|}_{\operatorname{Im}\phi }}$ is a map. Conversely, if $\theta {{|}_{\operatorname{Im}\phi }}$ is a map, then $\forall f\in \operatorname{Dom}\theta $ and $a\in S$, $\theta (f)(\phi (a))$ is well-defined.
\end{corollary}
\begin{proof} 
  If $\phi $ is a $\theta$-morphism, then by Definition \ref{5.3.5} and Convention \ref{convention}, $\forall f\in \operatorname{Dom}\theta $ and $a\in S$, $\theta (f)(\phi (a))$ is well-defined, and thus by Proposition \ref{5.3.b}, $\theta {{|}_{\operatorname{Im}\phi }}$ is a map. The converse also follows from Proposition \ref{5.3.b}.
\end{proof}
\begin{remark}
    By Corollary \ref{5.3.c}, to define a $\theta$-morphism, we may dismiss Convention \ref{convention} and put the condition $\theta {{|}_{\operatorname{Im}\phi }}$ being a map into Definition \ref{5.3.5}. However, later we shall find that it is convenient to use Convention \ref{convention} when things become complicated.
\end{remark}
 From Corollary \ref{5.3.c}, we can tell that the condition $\theta $ being a map in Definition \ref{5.3.1} is replaced by a weaker one (i.e. $\theta {{|}_{\operatorname{Im}\phi }}$ being a map) in Definition \ref{5.3.5}. So Definition \ref{5.3.1} is a special case of Definition \ref{5.3.5}.

In the remaining part of Section \ref{theta-mor and theta-iso} and Section \ref{Cons of $T$-mor and theta-mor}, \textbf{a $\theta$-morphism is always defined by Definition \ref{5.3.5}} unless otherwise specified.

Now Proposition \ref{1.3.4} can be generalized to
\begin{proposition} \label{5.3.6}
    Let $F/B$ $($resp. ${F}'/{B}'$$)$ be a field extension, let $T$ $($resp. ${T}'$$)$ be the corresponding operator semigroup defined as in Example \ref{T on field--1}, let $\varphi :B\to {B}'$ be a field isomorphism, and let $\vartheta :B[x]\to {B}'[x]$ be the map given by $f\mapsto {f}'$ where
\begin{center}
     ${f}'(x)=\varphi ({{a}_{0}})+\varphi ({{a}_{1}})x+\cdots +\varphi ({{a}_{n}}){{x}^{n}}$ for $f(x)={{a}_{0}}+{{a}_{1}}x+\cdots +{{a}_{n}}{{x}^{n}}$. 
\end{center}
Let $\theta =\{(\tau (f),{\tau }'(\vartheta (f))\,|\,f\in B[x]\}$, where $\tau :B[x]\to T$ and $\tau' :B'[x]\to T'$ are the maps which induce 
$T$ and $T'$, respectively, defined as in Example \ref{T on field--1}. Then $\forall u\in F$ and ${u}'\in {F}'$, a map $\phi :{{\left\langle u \right\rangle }_{T}}\to {{\left\langle {{u}'} \right\rangle }_{{{T}'}}}$ is a ring homomorphism extending $\varphi $ if and only if $\phi $ is a $\theta$-morphism. 
\end{proposition}
\begin{proof}
    Obviously, $\theta \subseteq T\times {T}'$ with $\operatorname{Dom}\theta =T$, and ${{\left\langle u \right\rangle }_{T}}$ and ${{\left\langle {{u}'} \right\rangle }_{{{T}'}}}$ are the rings $B[u]$ and ${B}'[{u}']$, respectively.

1. \emph{Sufficiency} Suppose that $\phi $ is a $\theta$-morphism. By Definition \ref{5.3.5}, $\phi (g(a))=\theta (g)(\phi (a))$, $\forall g\in \operatorname{Dom}\theta $ and $a\in {{\left\langle u \right\rangle }_{T}}$. 

Let $g\in \operatorname{Dom}\theta(=T)$ be a constant polynomial function given by $a\mapsto c (\in B),\forall a\in F$. Then by the definitions of $\theta$ and $\vartheta$, we can tell that $\theta(g)$ is the constant polynomial function given by $a'\mapsto \varphi (c)(\in B'),\forall a'\in F'$. Thus, $\forall c\in B$, $\phi (c)(=\phi (g(a))=\theta (g)(\phi (a)))=\varphi (c)$, that is, $\phi$ extends $\varphi$.

Let $a,b\in {{\left\langle u \right\rangle }_{T}}$. Then $\exists p(x),q(x)\in B[x]$ with $p(u)=a$ and $q(u)=b$. Hence

$\phi (a+b)$

$=\phi (p(u)+q(u))$ 

$=\phi ((p+q)(u))$ 

$=\phi (\tau (p+q)(u))$ 

$=\theta (\tau (p+q))(\phi (u))$ (since $\phi $ is a $\theta$-morphism)

$={\tau }'(\vartheta (p+q))(\phi (u))$ (by the definition of $\theta $)

$=\vartheta (p+q)(\phi (u))$  

$=\vartheta (p)(\phi (u))+\vartheta (q)(\phi (u))$ (since $\varphi :B\to {B}'$ is a field isomorphism)

$={\tau }'(\vartheta (p))(\phi (u))+{\tau }'(\vartheta (q))(\phi (u))$ 

$=\theta (\tau (p))(\phi (u))+\theta (\tau (q))(\phi (u))$ (since by Corollary \ref{5.3.c}, $\theta {{|}_{\operatorname{Im}\phi }}$ is a function)

$=\phi (\tau (p)(u))+\phi (\tau (q)(u))$ (since $\phi $ is a $\theta$-morphism)

$=\phi (p(u))+\phi (q(u))$ 

$=\phi (a)+\phi (b)$. 

And

$\phi (ab)$ 

$=\phi (p(u)q(u))$ 

$=\phi ((pq)(u))$ 

$=\phi (\tau (pq)(u))$

$=\theta (\tau (pq))(\phi (u))$ (since $\phi $ is a $\theta$-morphism)

$={\tau }'(\vartheta (pq))(\phi (u))$ (by the definition of $\theta $)

$=\vartheta (pq)(\phi (u))$ 

$=\vartheta (p)(\phi (u))\vartheta (q)(\phi (u))$ (since $\varphi :B\to {B}'$ is a field isomorphism)

$={\tau }'(\vartheta (p))(\phi (u)){\tau }'(\vartheta (q))(\phi (u))$ 

$=\theta (\tau (p))(\phi (u))\theta (\tau (q))(\phi (u))$ (since by Corollary \ref{5.3.c}, $\theta {{|}_{\operatorname{Im}\phi }}$ is a function)

$=\phi (\tau (p)(u))\phi (\tau (q)(u))$ (since $\phi $ is a $\theta$-morphism)

$=\phi (p(u))\phi (q(u))$ 

$=\phi (a)\phi (b)$.

Therefore $\phi $ is a ring homomorphism extending $\varphi $.

2.\emph{Necessity} 
Suppose that $\phi $ is a ring homomorphism extending $\varphi $. Then $\forall c\in B$, $\phi (c)=\varphi (c)$, and
$\forall a,b\in {{\left\langle u \right\rangle }_{T}}$, $\phi (a+b)=\phi (a)+\phi (b)$ and $\phi (ab)=\phi (a)\phi (b)$.

Let $z\in {{\left\langle u \right\rangle }_{T}}$ and let $f={{a}_{0}}+{{a}_{1}}x+\cdots +{{a}_{n}}{{x}^{n}} \in B[x]$. Then

$\phi (f(z))$ 

$= \phi ({a_0} + {a_1}z +  \cdots  + {a_n}{z^n}) $

$= \varphi ({a_0}) + \varphi ({a_1})\phi (z) +  \cdots  + \varphi ({a_n}){(\phi (z))^n} $

$= \vartheta (f)(\phi (z)), $

that is 
\begin{equation} \label{6.9}
    \phi (f(z))=\vartheta (f)(\phi (z))
\end{equation}

To show that $\phi $ is a $\theta$-morphism, by Corollary \ref{5.3.c}, we need to show that $\theta {{|}_{\operatorname{Im}\phi }}$ is a map. Let $(h,g{{|}_{\operatorname{Im}\phi }}),(h,{g}'{{|}_{\operatorname{Im}\phi }})\in \theta {{|}_{\operatorname{Im}\phi }}$ with $(h,g),(h,{g}')\in \theta $. By the definition of $\theta $, $\exists l,m\in B[x]$ such that 
$(h,g)=(\tau (l),{\tau }'(\vartheta (l)))$ and $(h,{g}')=(\tau (m),{\tau }'(\vartheta (m)))$, 
and so $\tau (l)=\tau (m)$. 

Assume $g{{|}_{\operatorname{Im}\phi }}\ne {g}'{{|}_{\operatorname{Im}\phi }}$; that is, ${\tau }'(\vartheta (l)){{|}_{\operatorname{Im}\phi }}\ne {\tau }'(\vartheta (m)){{|}_{\operatorname{Im}\phi }}$. Then $\exists z\in {{\left\langle u \right\rangle }_{T}}$ such that $\vartheta (l)(\phi (z))\ne \vartheta (m)(\phi (z))$. However, since $\tau (l)=\tau (m)$, $\phi (l(z))=\phi (m(z))$, and so from Equation $($\ref{6.9}$)$, we can tell that $\vartheta (l)(\phi (z))=\phi (l(z))=\phi (m(z))=\vartheta (m)(\phi (z))$, a contradiction. 

Hence $\theta {{|}_{\operatorname{Im}\phi }}$ must be a map. Then $\forall z\in {{\left\langle u \right\rangle }_{T}}$ and $f\in B[x]$, 

$\phi (\tau (f)(z))$

$=\phi (f(z))$ 

$=\vartheta (f)(\phi (z))$ (by Equation $($\ref{6.9}$)$)

$={\tau }'(\vartheta (f))(\phi (z))$ 

$=\theta (\tau (f))(\phi (z))$ (by the definition of $\theta $ and the fact that $\theta {{|}_{\operatorname{Im}\phi }}$ is a map).

So by Definition \ref{5.3.5}, $\phi $ is a $\theta$-morphism.
\end{proof}

	Note that up to now all polynomial rings involved are in one variable. We shall discuss ring homomorphisms involving polynomial rings in more than one variable in Sections \ref{Basic notions for multivariable} and \ref{Basic notions for partial}.\\

For $\operatorname{Im}\phi$ of $\theta$-morphisms $\phi$, we generalize Proposition \ref{1.3.8} as follows.
\begin{proposition} \label{image theta space}
    Let $\phi $ be a $\theta $-morphism from a $T_1$-space $S_1$ to a $T_2$-space $S_2$. Suppose $\operatorname{Im}\theta = T_2$. Then $ \operatorname{Im} \phi  \le _q{S_2}$, i.e. $\operatorname{Im} \phi$ is a quasi-$T_2$-subspace of $S_2$ (Definition \ref{quasi-T-space}). Moreover, if $\operatorname{Id}\in T_2$ or more generally, $\operatorname{Im}\phi \subseteq {{\left\langle \operatorname{Im}\phi  \right\rangle }_{T_2}}$, then $\operatorname{Im}\phi \le {{S}_{2}}$, i.e. $\operatorname{Im}\phi$ is a $T_2$-subspace of $S_2$ (Definition \ref{$T$-space}).
\end{proposition}
 \begin{proof}
     By Definition \ref{5.3.5}, $\forall a\in {{S}_{1}}$ and $f\in \operatorname{Dom}\theta$, $\theta(f)(\phi (a))=\phi (f(a))\in \operatorname{Im}\phi $, and thus ${{\left\langle \phi (a) \right\rangle }_{T_2}}\subseteq \operatorname{Im}\phi $ because $\operatorname{Im}\theta = T_2$. Hence ${{\left\langle \operatorname{Im}\phi  \right\rangle }_{T_2}}\subseteq \operatorname{Im}\phi $, and so by Definition \ref{quasi-T-space}, $\operatorname{Im}\phi $ is a quasi-$T_2$-subspace of $S_2$. 
     
Moreover, if $\operatorname{Id}\in T_2$ or $\operatorname{Im}\phi \subseteq {{\left\langle \operatorname{Im}\phi  \right\rangle }_{T_2}}$, then by Proposition \ref{1.2.9}, $\operatorname{Im}\phi $ is a $T_2$-space and hence $\operatorname{Im}\phi \le {{S}_{2}}$.
 \end{proof}
\subsection{$\theta$-isomorphisms} \label{6 theta-isomorphisms}
The notion of $T$-isomorphism is generalized as follows.
\begin{definition} \label{5.4.a}
Let $\phi $ be a $\theta $-morphism from a $T_1$-space $S_1$ to a $T_2$-space $S_2$. If $\phi $ is bijective, then we call $\phi $ a \emph{$\theta$-isomorphism}. Moreover, we denote by $\operatorname{Iso}_{\theta}({{S}_{1}},{{S}_{2}})$ the set of all $\theta$-isomorphisms from ${{S}_{1}}$ to ${{S}_{2}}$.
\end{definition} 
To justify the definition, we need to show that $\phi^{-1}$ is a $\theta^{-1}$-morphism from $S_2$ to $S_1$. For this purpose, we need
\begin{lemma} \label{L6.2.2}
    Let $\phi $ be a $\theta $-isomorphism from a $T_1$-space $S_1$ to a $T_2$-space $S_2$. Then $\theta^{-1} {{|}_{S_1}}$ is a map, where $\theta^{-1}:=\{(g,f)\,|\,(f,g) \in \theta\}$.
\end{lemma}
\begin{proof}
    Let $(g,f_1),(g,f_2)\in \theta^{-1}$ and $a \in S_1$. 
    Then to show that $\theta^{-1} {{|}_{S_1}}$ is a map, by Proposition \ref{5.3.b} and Definition \ref{5.3.a}, it suffices to show
    $f_1(a)=f_2(a)$. Moreover, since $\phi $ is injective, we only need to show $\phi(f_1(a))=\phi(f_2(a))$: 
    \begin{align*}
        \phi(f_1(a))&=\theta(f_1)(\phi (a)) \text{ (by Definition } \ref{5.3.5})\\
        &=g(\phi (a)) \text{ (because }(f_1,g)\in \theta \text{ and }\theta {|}_{\operatorname{Im}\phi} \text{ is a map by Corollary } \ref{5.3.c})\\ 
        &=\theta(f_2)(\phi (a)) \text{ (because }(f_2,g)\in \theta \text{ and }\theta {|}_{\operatorname{Im}\phi} \text{ is a map})\\
        &=\phi(f_2(a)) \text{ (by Definition } \ref{5.3.5}).
    \end{align*}
\end{proof}

\begin{proposition} \label{5.4.b}
    Let $S_1$ and $S_2$ be a $T_1$-space and a $T_2$-space, respectively, let $\phi \in \operatorname{Iso}_{\theta}({{S}_{1}},{{S}_{2}})$ and let ${{\phi }^{-1}}$ be the inverse map of $\phi $. Then ${{\phi }^{-1}}\in \operatorname{Iso}_{\theta^{-1}}({{S}_{2}},{{S}_{1}})$.
\end{proposition}
\begin{proof}
   By Definition \ref{5.3.5}, $\forall f\in \operatorname{Dom}\theta$ and $a\in {{S}_{1}}$, $\phi(f(a))=\theta(f)(\phi (a))$, and so $f(a)={{\phi }^{-1}}(\theta(f)(\phi(a)))$. Hence 
\begin{equation} \label{6.a}
    f({{\phi }^{-1}}(b))={{\phi }^{-1}}(\theta(f)(b)),\forall f\in \operatorname{Dom}\theta \text{ and } b \in {{S}_{2}}.
\end{equation}
 
    By Lemma \ref{L6.2.2}, $\theta^{-1} {{|}_{S_1}}$ is a map. Then by Proposition \ref{5.3.b}, $\theta^{-1}(g)(a)$ is well-defined, $\forall g\in \operatorname{Dom}\theta^{-1}(=\operatorname{Im}\theta) $ and $a\in S_1$. Thus it follows from (\ref{6.a}) that
    \[\theta^{-1}(g)({{\phi }^{-1}}(b))={{\phi }^{-1}}(g(b)),\forall g\in \operatorname{Dom}\theta^{-1} \text{ and } b \in {{S}_{2}}.\]
So by Definition \ref{5.3.5}, $\phi^{-1} $ is a $\theta^{-1} $-morphism from $S_2$ to $S_1$.

Moreover, since ${\phi }^{-1}$ is bijective, by Definition \ref{5.4.a}, ${{\phi }^{-1}}\in \operatorname{Iso}_{\theta^{-1}}({{S}_{2}},{{S}_{1}})$.
\end{proof}

\section{Constructions of the generalized morphisms and isomorphisms} \label{Cons of $T$-mor and theta-mor}
For a $T$-morphism or a $\theta $-morphism from a $T$-space $S$, Subsections \ref{$T$-morphisms from u}, \ref{$T$-morphisms from big u} and \ref{I Constructions of $theta $-morphisms} discuss how to construct it by a map from a set which generates $S$. Specifically, Subsection \ref{$T$-morphisms from u} is about a construction of a $T$-morphism from ${{\left\langle u \right\rangle }_{T}}$, where $u\in D$, by a map from $\{u\}$. And Subsection \ref{$T$-morphisms from big u} generalizes the construction to a $T$-morphism from ${{\left\langle U \right\rangle }_{T}}$, where $U \subseteq D$, by a map from $U$. Then in Subsection \ref{I Constructions of $theta $-morphisms}, the construction is further generalized to $\theta$-morphisms. In Subsection \ref{I Another construction of $T$-morphisms}, we shall introduce a way of constructing $T$-morphisms in terms of topology.

Moreover, some notions and results in algebra are generalized in Section \ref{Cons of $T$-mor and theta-mor} as “byproducts” (see e.g. Propositions \ref{5.1.2}, \ref{5.1.3} and \ref{5.1.4}).
\subsection{A construction of $T$-morphisms from ${{\left\langle u \right\rangle }_{T}}(u \in D)$} \label{$T$-morphisms from u}
Let $F/B$ be a field extension and let $u,v\in F$ be algebraic over $B$. If $u$ and $v$ share the same minimal polynomial over $B$, then 
\[p(u)=q(u)\Leftrightarrow p(v)=q(v),\forall p(x),q(x)\in B[x].\]
In light of this observation, we introduce the following.
\begin{notation}
    \label{5.1.1}Let $D$ be a set, let $M$ be a set of maps from $D$ to $D$, and let $u,v\in D$. Then by $u\xrightarrow{M}v$ we mean that
\begin{center}
    $\forall f,g\in M$, if $f(u)=g(u)$, then $f(v)=g(v)$. 
\end{center}

Moreover, by $u\overset{M}{\longleftrightarrow}v$ we mean that
\begin{center}
    $u\xrightarrow{M}v$ and $v\xrightarrow{M}u$,
\end{center} 
or equivalently, by $u\overset{M}{\longleftrightarrow}v$ we mean that
\begin{center}
 $f(u)=g(u)\Leftrightarrow f(v)=g(v)$, $\forall f,g\in M$.    
\end{center}

Besides, we denote by ${{[u)}_{M}}$ and ${[u]_M}$ the set $\{w\in D\,|\,u\xrightarrow{M}w\}$ and the set $\{w\in D\,|\,u\overset{M}{\longleftrightarrow}w\}$, respectively.
\end{notation}
\begin{remark}
    Apparently, $u\overset{M}{\longleftrightarrow}v$ defines an equivalence relation on $D$.
\end{remark}

The following serves as an example.

\begin{proposition}
    \label{5.1.2} Let $B$ be a subfield of a field $F$. Let $T$ be the operator semigroup defined in Example \ref{T on field--1}.
\begin{enumerate}
\item [(a)] If $u\in F$ is algebraic over $B$, then 
\begin{center}
    ${{[u]}_{T}}={{[u)}_{T}}=$\{all roots of $\min (B,u)$\},
\end{center} 
where $\min (B,u)$ is the minimal polynomial of $u$ over $B$.
 \item [(b)] If $u\in F$ is transcendental over $B$, then 
 \begin{center}
     ${{[u)}_{T}}=F$ and ${{[u]}_{T}}=$\{$v\in F\, |\, v $ is transcendental over $B$\}.
 \end{center}
\end{enumerate}
\end{proposition}  
\begin{proof}
For (a): Let $u\in F$ be algebraic over $B$. 

Let $w\in {{[u)}_{T}}$. By Notation \ref{5.1.1}, $\forall f^*,g^*\in T$ (\textit{cf.} Example \ref{T on field--1}), $f^*(u)=g^*(u)$ implies $f^*(w)=g^*(w)$. So $\forall f(x),g(x)\in B[x]$, $f(u)=g(u)$ implies $f(w)=g(w)$. Thus $w$ is a root of $\min (B,u)$. Hence 
\begin{center}
    ${{[u)}_{T}}\subseteq $\{all roots of $\min (B,u)$\}.
\end{center}

Moreover, let $v \in F$ be a root of $\min (B,u)$. Then $\forall f(x),g(x)\in B[x]$, $f(u)=g(u)\Leftrightarrow f(v)=g(v)$. It follows that $\forall f^*,g^*\in T$, $f^*(u)=g^*(u)\Leftrightarrow f^*(v)=g^*(v)$; that is, $u\overset{T}{\longleftrightarrow}v$, and so $v\in {{[u]}_{T}}$. Hence
\begin{center}
    $({{[u)}_{T}}\subseteq )$\{all roots of $\min (B,u)$\}$\subseteq {{[u]}_{T}}$. 
\end{center}

On the other hand, ${{[u)}_{T}}\supseteq {{[u]}_{T}}$ by Notation \ref{5.1.1}. Therefore, 
\begin{center}
    ${{[u)}_{T}}=$\{all roots of $\min (B,u)$\}$={{[u]}_{T}}$.
\end{center}
For (b): Let $u\in F$ be transcendental over $B$. 

Then $\forall f(x),g(x)\in B[x]$, $f(u)=g(u)$ implies $f=g$. So $\forall f^*,g^*\in T$, $f^*(u)=g^*(u)$ implies $f^*=g^*$. Then $\forall v\in F$, $u\xrightarrow{T}v$, and so ${{[u)}_{T}}=F$.

Hence ${{[u]}_{T}}=\{v\in F\,|\,v\overset{T}{\longleftrightarrow}u\}=$\{$v\in F\,|\,v\xrightarrow{T}u$\}.

$\forall v\in F$, if $v$ is also transcendental over $B$, then as was just shown for $u$, $v\xrightarrow{T}u$. So 
\begin{center}
    ${{[u]}_{T}}=\{v\in F\,|\,v\xrightarrow{T}u\}\supseteq $\{$v\in F\,|\,v$ is transcendental over $B$\}.
\end{center}

On the other hand, if $v\in F$ is algebraic over $B$, then by (a), $u\notin {{[v)}_{T}}$, i.e. $v\xrightarrow{T}u$ is false; That is, if $v\xrightarrow{T}u$ is true, then $v$ is transcendental over $B$. So
\begin{center}
    ${{[u]}_{T}}=\{v\in F\,|\,v\xrightarrow{T}u\}\subseteq $\{$v\in F\,|\,v$ is transcendental over $B$\}.
\end{center}

Therefore, ${{[u]}_{T}}=$\{$v\in F\,|\,v$ is transcendental over $B$\}.	
\end{proof}
We will talk more about the notion of transcendental elements in Section \ref{transcendental}.

Recall that in the classical Galois theory, each element of the Galois group of a polynomial permutes its roots. From Proposition \ref{5.1.2}, we can tell that the following two results generalize this well-known fact. 

\begin{proposition}
    \label{5.1.3} Let $\sigma$ be a $T$-morphism from a $T$-space $S$ (to a $T$-space). Then $\forall a\in S$, $a\xrightarrow{T}\sigma (a)$ $($and so $\sigma (a)\in {{[a)}_{T}})$. 
\end{proposition}
\begin{proof}
    For $a\in S$, let $f,g\in T$ such that $f(a)=g(a)$. Then by Definition \ref{1.3.1},
$f(\sigma (a))=\sigma (f(a))=\sigma (g(a))=g(\sigma (a))$. 
Hence by Notation \ref{5.1.1}, $a\xrightarrow{T}\sigma (a)$.	
\end{proof}
\begin{proposition}
    \label{5.1.4} Let $\sigma$ be a $T$-morphism from a $T$-space $S$ (to a $T$-space) such that $\sigma$ is injective. Then $\forall a\in S$, $a\overset{T}{\longleftrightarrow}\sigma (a)$ $($and so $\sigma (a)\in {{[a]}_{T}})$.
\end{proposition}
\begin{proof}
    Let $a\in S$. By Proposition \ref{5.1.3}, we only need to show $\sigma (a)\xrightarrow{T}a$. Let $f,g\in T$ such that $f(\sigma (a))=g(\sigma (a))$. Then by Definition \ref{1.3.1},
$\sigma (f(a))=f(\sigma (a))=g(\sigma (a))=\sigma (g(a))$. 
Hence $f(a)=g(a)$ because $\sigma $ is injective, as desired.
\end{proof}
The converse of Proposition \ref{5.1.3} may imply a potential way to construct $T$-morphisms. However, the converse of Proposition \ref{5.1.3} is not true. That is, for a map $\sigma$ from a $T$-space $S$, the condition “$\forall a\in S$, $a\xrightarrow{T}\sigma (a)$" is not sufficient for $\sigma$ to be a $T$-morphism, as shown below.
\begin{example}
    \label{5.1.5} Let $R$ be a differential ring with a derivation $\partial $ such that $\{\partial (a)\,|\,a\in R\}=R$ (it is easy to find such differential ring), let $T$ be the operator semigroup on $R$ given by $\{{{\partial }^{n}}\,|\,n\in {{\mathbb{Z}}^{+}}\}$ (note that unlike the operator semigroup defined in Example \ref{T on diff ring--1}, now $\operatorname{Id}\notin T$), let $C$ be the constant ring of $R$, and let a map $\sigma :R\to R$ be given by $a\mapsto a+c$, where $c\in C$ and $c\ne 0$. Then $\forall f,g\in T$ and $a\in R$, $f(a)=g(a)$ implies $f(\sigma (a))=g(\sigma (a))$. Hence $\forall a\in R$, $a\xrightarrow{T}\sigma (a)$. But obviously $\exists f\in T$ and $a\in R$ such that $f(\sigma (a))\ne \sigma (f(a))$. So $\sigma $, a map from the $T$-space ${{\left\langle R \right\rangle }_{T}}=R$ to itself, is not a $T$-morphism.
\end{example}
	Thus, to make a map $\sigma$ from a $T$-space $S$ be a $T$-morphism, we need to impose a condition on $\sigma $ which is stronger than “$\forall a\in S$, $a\xrightarrow{T}\sigma (a)$". Before we do this by Proposition \ref{5.1.8}, let’s see a fact as follows.
\begin{proposition}
    \label{5.1.6} Let $u,v\in D$. Then $u\xrightarrow{T}v$ implies $f(u)\xrightarrow{T}f(v)$, $\forall f\in T$. 
\end{proposition}
\begin{remark}
    So, if $u\xrightarrow{T}v$ and there exists a map $\sigma :{{\left\langle u \right\rangle }_{T}}\to {{\left\langle v \right\rangle }_{T}}$ given by $f(u)\mapsto f(v),\forall f\in T$, then $\forall a\in {{\left\langle u \right\rangle }_{T}}$, $a\xrightarrow{T}\sigma (a)$ (since $f(u)\xrightarrow{T}f(v)$, $\forall f\in T$). We shall see in Proposition \ref{5.1.8} that the map $\sigma$ is a $T$-morphism.
\end{remark}
 \begin{proof}
     Suppose $u\xrightarrow{T}v$. Let $f,g,h\in T$ such that $g(f(u))=h(f(u))$. Then by Notation \ref{5.1.1}, it suffices to show $g(f(v))=h(f(v))$.

    By Definition \ref{Operator semigroup}, $g\circ f\in T$ and $h\circ f\in T$. Since $u\xrightarrow{T}v$ and $g(f(u))=h(f(u))$, by Notation \ref{5.1.1}, we have $g(f(v))=h(f(v))$, as desired.
 \end{proof}

\begin{proposition}
    \label{5.1.8} Let $u,v\in D$. Then 
\begin{enumerate}
    \item [(a)] $u\xrightarrow{T}v$ if and only if there exists a map $\sigma :{{\left\langle u \right\rangle }_{T}}\to {{\left\langle v \right\rangle }_{T}}$ given by $f(u)\mapsto f(v),\forall f\in T$.
    \item[(b)] The map $\sigma $ given in \emph{(a)} is a $T$-morphism.
\end{enumerate}
\end{proposition}
\begin{proof}
For (a):  

$u\xrightarrow{T}v$; 

$\Leftrightarrow$ $\forall f,g\in T$, $f(u)=g(u)$ implies $f(v)=g(v)$;

$\Leftrightarrow$ $\sigma :{{\left\langle u \right\rangle }_{T}}\to {{\left\langle v \right\rangle }_{T}}$ given by $f(u)\mapsto f(v),\forall f\in T$, is a well-defined map.

For (b): Let $a\in {{\left\langle u \right\rangle }_{T}}$. Then $\exists f\in T$ such that $f(u)=a$. Let $g\in T$. Then $\sigma (g(a))=\sigma (g(f(u)))$. By Definition \ref{Operator semigroup}, $g\circ f\in T$, and so by the definition of $\sigma $ in (a), $\sigma (g(f(u)))=g(f(v))$. Besides, $g(\sigma (a))=g(\sigma (f(u)))=g(f(v))$ also by the definition of $\sigma $. Combining the above equations, we have $g(\sigma (a))=\sigma (g(a))$. Therefore, $\sigma $ is a $T$-morphism.
\end{proof}
For $T$-isomorphisms (Definition \ref{$T$-isomorphisms}), we have the following straightforward result of Proposition \ref{5.1.8}.
\begin{corollary}
    \label{5.1.9} Let $u,v\in D$.  
\begin{enumerate}
    \item [(a)] $u\overset{T}{\longleftrightarrow}v$ if and only if there exists a bijective map $\sigma :{{\left\langle u \right\rangle }_{T}}\to {{\left\langle v \right\rangle }_{T}}$ given by $f(u)\mapsto f(v),\forall f\in T$.
    \item[(b)] The bijective map $\sigma $ given in \emph{(a)} is a $T$-isomorphism.
\end{enumerate}
\end{corollary}
	The converse of statement (b) in Proposition \ref{5.1.8} raises a question as follows. 
\begin{question}
\label{5.1.11} For $u\in D$ and a $T$-morphism $\sigma $ from ${{\left\langle u \right\rangle }_{T}}$, does there exist any $v\in D$ such that $\sigma (f(u))=f(v),\forall f\in T$? 
\end{question}
 The significance of this question is that, in the cases where the answer is yes, we can construct a $T$-morphism $\sigma $ by letting $\sigma (f(u))=f(v),\forall f\in T$. Proposition \ref{5.1.8} implies that there are cases where the answer to the question is positive, and we shall talk about these cases in Subsection \ref{$T$-morphisms from big u}. It seems not easy to find a case which gives negative answer to the question, but we found one as follows.
\begin{example}
    \label{5.1.12} Define $h:\mathbb{R}\to \mathbb{R}$ by $x\mapsto x+1$ and $g:\mathbb{R}\to \mathbb{R}$ by $x\mapsto \pi x$. Let $T=\left\langle \{h,g\} \right\rangle $ (Definition \ref{generated operator semigroup}) and $S={{\left\langle 1 \right\rangle }_{T}}$. 
    
    Then each $f\in T$ can be expressed as a finite composite of alternating “powers” of $h$ or $g$ as follows, where each ${{h}^{{{k}_{i}}}}$ or ${{g}^{{{k}_{i}}}}$ denotes the composite of ${{k}_{i}}$ copies of $h$ or $g$, respectively.
\begin{equation}
    \label{1a} f = {h^{{k_n}}}{g^{{k_{n - 1}}}} \cdots {h^{{k_2}}}{g^{{k_1}}},
\end{equation}
\begin{equation}
    \label{2a} f = {g^{{k_n}}}{h^{{k_{n - 1}}}} \cdots {h^{{k_2}}}{g^{{k_1}}},
\end{equation}
\begin{equation}
    \label{3a} f = {h^{{k_n}}}g^{{k_{n - 1}}} \cdots {g^{{k_2}}}h^{{k_1}},
\end{equation}
or
\begin{equation}
    \label{4a} f = {g^{{k_n}}}h_{}^{{k_{n - 1}}} \cdots {g^{{k_2}}}h_{}^{{k_1}},
\end{equation}
where $n$ is an even number for $($\ref{1a}$)$ and $($\ref{4a}$)$, $n$ is an odd number for $($\ref{2a}$)$ and $($\ref{3a}$)$, and ${{k}_{i}}\in {{\mathbb{Z}}^{+}},\forall i=1,\cdots ,n$. In particular, $h$ has the form (\ref{3a}), where $n=1$, and $g$ has the form (\ref{2a}), where $n=1$. 

Correspondingly, $\forall a\in S={{\left\langle 1 \right\rangle }_{T}}$, $a=f(1)$ has the form
\begin{equation} \label{1b}
    a = ( \cdots ({\pi ^{{k_1}}} + {k_2}){\pi ^{{k_3}}} +  \cdots  + {k_{n - 2}}){\pi ^{{k_{n - 1}}}} + {k_n},
\end{equation}
\begin{equation} \label{2b}
   a = ( \cdots ({\pi ^{{k_1}}} + {k_2}){\pi ^{{k_3}}} +  \cdots  + {k_{n - 1}}){\pi ^{{k_n}}},
\end{equation}
\begin{equation} \label{3b}
    a = ( \cdots ((1 + {k_1}){\pi ^{{k_2}}} + {k_3}){\pi ^{{k_4}}} +  \cdots  + {k_{n - 2}}){\pi ^{{k_{n - 1}}}} + {k_n},
\end{equation}
or
\begin{equation} \label{4b}
   a = ( \cdots ((1 + {k_1}){\pi ^{{k_2}}} + {k_3}){\pi ^{{k_4}}} +  \cdots  + {k_{n - 1}}){\pi ^{{k_n}}}, 
\end{equation}
where $n$ is an even number for $($\ref{1b}$)$ and $($\ref{4b}$)$, $n$ is an odd number for $($\ref{2b}$)$ and $($\ref{3b}$)$, and ${{k}_{i}}\in {{\mathbb{Z}}^{+}},\forall i=1,\cdots ,n$.

Let $a\in S$. Since $\pi $ is transcendental over $\mathbb{Q}$, the expression for $a$ in the above forms is unique. It follows that there is only one $f \in T$ in the form of $($\ref{1a}$)$, $($\ref{2a}$)$, $($\ref{3a}$)$, or $($\ref{4a}$)$ such that $a=f(1)$. Hence we can define a map $\sigma :S\to S$ as follows.

For all $a\in S$, if $f \in T$ has the form $($\ref{1a}$)$ or $($\ref{2a}$)$ such that $a=f(1)$, let $\sigma (a)=f(\pi )$, which clearly belongs to $S$, and if $f \in T$ has the form $($\ref{3a}$)$ or $($\ref{4a}$)$ such that $a=f(1)$, let $\sigma (a)=a$. 

Now let’s show $\sigma \in \operatorname{End}_{T}(S)$. Let $l\in T$. 

Suppose that $a=f(1)$ has the form $($\ref{1b}$)$ or $($\ref{2b}$)$. Then $f$ has the form $($\ref{1a}$)$ or $($\ref{2a}$)$, and so $l\circ f$ also has the form $($\ref{1a}$)$ or $($\ref{2a}$)$. Hence by the definition of $\sigma $, $\sigma (l(a))=\sigma (l(f(1)))=l(f(\pi ))=l(\sigma (a))$, as desired. 

Suppose that $a=f(1)$ has the form $($\ref{3b}$)$ or $($\ref{4b}$)$. Then $f$ has the form $($\ref{3a}$)$ or $($\ref{4a}$)$, and so $l\circ f$ also has the form $($\ref{3a}$)$ or $($\ref{4a}$)$. Hence by the definition of $\sigma $, $\sigma (l(a))=\sigma (l(f(1)))=l(f(1))=l(a)=l(\sigma (a))$, as desired. 

Therefore, $\sigma (l(a))=l(\sigma (a))$, $\forall l\in T$ and $a\in S$. So $\sigma \in \operatorname{End}_{T}(S)$.

For 1 and $\sigma $, assume that the answer to Question \ref{5.1.11} is positive, that is, there exists $v\in \mathbb{R}$ such that $\sigma (f(1))=f(v),\forall f\in T$. Recall the definitions of $h$ and $g$ at the beginning of the example. Since $h$ has the form (\ref{3a}), $h(v)=\sigma (h(1))=h(1)=2$, and so $v=1$; but since $g$ has the form (\ref{2a}), $g(v)=\sigma (g(1))=g(\pi )={{\pi }^{2}}$, and so $v=\pi $, a contradiction.

Therefore, for 1 and $\sigma $, the answer to Question \ref{5.1.11} is negative.
\end{example}

Recall that right after Proposition \ref{1.4.2}, we claimed that for a $T$-space $S$ and $H\subseteq \operatorname{End}_{T}(S)$, it is possible that ${{S}^{H}}$ is not a $T$-space. Now we show this claim as follows.
\begin{proof}
Let $T$, $S$ and $\sigma$ be the same as in Example \ref{5.1.12} and let $D=S\bigcup \{1\}$. Let us restrict the domain of $T$ to $D$.  We can tell that the “new” $T$, which we still denote by $T$ for convenience, is an operator semigroup on $D$, and all the argument in Example \ref{5.1.12} still holds. In particular, $\sigma \in \operatorname{End}_{T}(S)$, and $2\in {{S}^{\sigma }}$ because 2 has the form $($\ref{3b}$)$. Suppose that ${{S}^{\sigma }}$ is a $T$-space. Then $\exists U\subseteq D$ such that ${{\left\langle U \right\rangle }_{T}}={{S}^{\sigma }}$. Since $2\in {{S}^{\sigma }}$, $\exists u\in U \subseteq D=S\bigcup \{1\}$ such that $2\in {{\left\langle u \right\rangle }_{T}}$. It follows that $u$ must be 1, and so ${{S}^{\sigma }}\supseteq {{\left\langle 1 \right\rangle }_{T}}=S$. Then ${S}^{\sigma }=S$ because ${S}^{\sigma }\subseteq S$, and hence every element of $S$ is fixed under the action of $\sigma$. However, $\pi \in S$ and $\sigma (\pi )=\sigma (g(1))=g(\pi )={{\pi }^{2}}\ne \pi $, a contradiction.
\end{proof}

\subsection{A construction of $T$-morphisms from ${{\left\langle U \right\rangle }_{T}}(U \subseteq D)$}  \label{$T$-morphisms from big u}
$T$-morphisms discussed in Subsection \ref{$T$-morphisms from u} are from $T$-spaces generated by a single element of $D$. To study $T$-morphisms from other $T$-spaces, we generalize Notation \ref{5.1.1} as follows.
\begin{notation} \label{5.2.1}
    Let $D$ be a set, let $M$ be a set of maps from $D$ to $D$ and let $\alpha :U(\subseteq D)\to V(\subseteq D)$ be a map. Then by $U\xrightarrow{M,\alpha }V $ we mean that
\begin{center}
    $\forall f,g\in M$ and $u,w\in U$, if $f(u)=g(w)$, then $f(\alpha (u))=g(\alpha (w))$. 
\end{center}

Moreover, by $U\overset{M,\alpha }{\longleftrightarrow}V $ we mean that 
\begin{center}
    $\alpha $ is bijective, $U\xrightarrow{M,\alpha }V $ and $V \xrightarrow{M,{{\alpha }^{-1}}}U$,
\end{center}
where $\alpha ^{-1}:V \to U$ denotes the inverse of $\alpha$;

or equivalently, by $U\overset{M,\alpha }{\longleftrightarrow}V $ we mean that
\begin{center}
    $\alpha $ is bijective, and $\forall f,g\in M$ and $u,w\in U$, $f(u)=g(w)\Leftrightarrow f(\alpha (u))=g(\alpha (w))$. 
\end{center}
\end{notation}
	Now Proposition \ref{5.1.3} can be generalized to
\begin{proposition} \label{5.2.2}
    Let $\sigma$ be a $T$-morphism from a $T$-space $S_1$ to a $T$-space $S_2$. Then $\forall A\subseteq {{S}_{1}}$, $A\xrightarrow{T,\alpha }{S}_{2} $, where $\alpha :=\sigma {{|}_{A}}$. In particular, ${{S}_{1}}\xrightarrow{T,\sigma }{S}_{2}$. 
\end{proposition}
\begin{proof}
    Let $f,g\in T$ and let $a,b\in A\subseteq {{S}_{1}}$. If $f(a)=g(b)$, then
\[f(\alpha (a))=f(\sigma (a))=\sigma (f(a))=\sigma (g(b))=g(\sigma (b))=g(\alpha (b)).\] 

Hence by Notation \ref{5.2.1}, $A\xrightarrow{T,\alpha }{S}_{2}$.
\end{proof}
	However, the converse of Proposition \ref{5.2.2} is not true, as shown below. 
\begin{example} \label{5.2.3}
    Let $R$, $T$ and $\sigma $ be defined as in Example \ref{5.1.5}. Let $A\subseteq R$ and let $\alpha =\sigma {{|}_{A}}$. Then $\forall a,b\in A$ and $f,g\in T$, $f(a)=g(b)$ implies $f(\alpha (a))=g(\alpha (b))$. Hence $A\xrightarrow{T,\alpha }R$. But $\sigma $ is not a $T$-morphism, as explained in Example \ref{5.1.5}.
\end{example}
	To construct $T$-morphisms, we generalize Proposition \ref{5.1.8} as follows.
\begin{proposition} \label{5.2.4}
    Let $\alpha :U(\subseteq D)\to V(\subseteq D)$ be a map. 
    \begin{enumerate}
        \item [(a)] $U\xrightarrow{T,\alpha }V$ if and only if there exists a map $\sigma :{{\left\langle U \right\rangle }_{T}}\to {{\left\langle V  \right\rangle }_{T}}$ given by $f(u)\mapsto f(\alpha (u))$, $\forall u\in U$ and $f\in T$.
        \item[(b)] The map $\sigma $ given in \emph{(a)} is a $T$-morphism.
    \end{enumerate}
\end{proposition}
\begin{proof}
For (a): 

$U\xrightarrow{T,\alpha }V $; 

$\Leftrightarrow$ $f(u)=g(w)$ implies $f(\alpha (u))=g(\alpha (w))$, $\forall f,g\in T$ and $u,w\in U$; (by Notation \ref{5.2.1})

$\Leftrightarrow$ $\sigma :{{\left\langle U \right\rangle }_{T}}\to {{\left\langle V  \right\rangle }_{T}}$ given by $f(u)\mapsto f(\alpha (u))$, $\forall u\in U$ and $f\in T$, is a well-defined map.

For (b): Let $a\in {{\left\langle U \right\rangle }_{T}}$. Then $\exists u\in U$ and $g\in T$ such that $a=g(u)$. Let $f\in T$. Then $\sigma (f(a))=\sigma ((f\circ g)(u))$. By Definition \ref{Operator semigroup}, $f\circ g\in T$, and so by the definition of $\sigma $, $\sigma ((f\circ g)(u))=(f\circ g)(\alpha (u))$. Besides, $g(\alpha (u))=\sigma (g(u))$ also by the definition of $\sigma $. Combining the above equations, we have
\[\sigma (f(a)) = f(g(\alpha (u))) = f(\sigma (g(u))) = f(\sigma (a)).\]

Thus $\sigma $ is a $T$-morphism.	
\end{proof}
To construct $T$-isomorphisms, Corollary \ref{5.1.9} is generalized as follows.
\begin{proposition} \label{5.2.5}
    Let $\alpha :U(\subseteq D)\to V(\subseteq D)$ be a map. 
\begin{enumerate}
    \item [(a)] The following statements are equivalent:
    \begin{enumerate}
        \item [(i)] $U\overset{T,\alpha }{\longleftrightarrow}V $.
        \item[(ii)] $\alpha $ is bijective and there exists a bijective map $\sigma :{{\left\langle U \right\rangle }_{T}}\to {{\left\langle V  \right\rangle }_{T}}$ given by $f(u)\mapsto f(\alpha (u))$, $\forall u\in U$ and $f\in T$.
    \end{enumerate}  
    \item[(b)] The map $\sigma $ given in \emph{(ii)} is a $T$-isomorphism.
\end{enumerate}
\end{proposition}
\begin{proof}
    For (a): 
    
    $U\overset{T,\alpha }{\longleftrightarrow}V $;

    $\Leftrightarrow$ $\alpha $ is bijective, and $\forall f,g\in T$ and $u,w\in U$, $f(u)=g(w)$ if and only if $f(\alpha (u))=g(\alpha (w))$; (by Notation \ref{5.2.1})

    $\Leftrightarrow$ $\alpha $ is bijective and $\sigma :{{\left\langle U \right\rangle }_{T}}\to {{\left\langle V  \right\rangle }_{T}}$ given by $f(u)\mapsto f(\alpha (u))$, $\forall u\in U$ and $f\in T$, is a well-defined bijective map.
    
    For (b): By Proposition \ref{5.2.4}, $\sigma $ is a $T$-morphism. Since $\sigma $ is bijective, by Definition \ref{$T$-isomorphisms}, it is a $T$-isomorphism.
\end{proof}
	The converse of statement (b) in Proposition \ref{5.2.4} raises a question as follows, which generalizes Question \ref{5.1.11}. 
\begin{question} \label{5.2.6}
    For $U\subseteq D$ and a $T$-morphism $\sigma $ from ${{\left\langle U \right\rangle }_{T}}$, is there any map $\alpha :U\to D$ such that $\sigma (f(u))=f(\alpha (u))$, $\forall u\in U$ and $f\in T$?
\end{question}
	The significance of this question is that, in the cases where the answer is yes, we can construct a $T$-morphism $\sigma $ by letting $\sigma (f(u))=f(\alpha (u))$, $\forall u\in U$ and $f\in T$. Example \ref{5.1.12} shows a case where the answer to Question \ref{5.2.6} is negative, but soon we shall see that there are cases where the answer is positive. Thus the property in question deserves a name.
\begin{definition} \label{5.2.7}
    Let $T$ be an operator semigroup on $D$, let $U\subseteq D$, and let $\sigma $ be a $T$-morphism from ${{\left\langle U \right\rangle }_{T}}$. If there exists a map $\alpha :U\to D$ such that $\sigma (f(u))=f(\alpha (u))$, $\forall u\in U$ and $f\in T$, then we say that $\sigma $ is \emph{constructible by} $\alpha $, or just say that $\sigma $ is \emph{constructible} for brevity. 
\end{definition}
Recall that it is possible that ${{\left\langle U \right\rangle }_{T}}$ does not contain $U$ (Proposition \ref{4 facts}). But if the identity function lies in $T$, which is true in the cases such as Examples \ref{T on field--1}, \ref{T on topo--1} and \ref{T on diff ring--1}, then $U\subseteq {{\left\langle U \right\rangle }_{T}}$, where we can get a positive answer to Question \ref{5.2.6} as shown below.
\begin{corollary} \label{5.2.9}
    Let $U\subseteq D$. If $U\subseteq {{\left\langle U \right\rangle }_{T}}$, then any $T$-morphism $\sigma $ from ${{\left\langle U \right\rangle }_{T}}$ is constructible by $\alpha :=\sigma {{|}_{U}}$.
\end{corollary}
\begin{proof}
    Since $U\subseteq {{\left\langle U \right\rangle }_{T}}$, by Definition \ref{1.3.1}, 
$\sigma (f(u))=f(\sigma (u))=f(\alpha (u))$, $\forall u\in U$ and $f\in T$.  
Thus $\sigma $  is constructible by $\alpha $.
\end{proof}
	Moreover, if the operator semigroup $T$ can be generated by a single element, then we also have a positive answer to Question \ref{5.2.6} as shown below.
\begin{theorem} \label{5.2.10}
Let $\sigma $ be a $T$-morphism from ${{\left\langle U \right\rangle }_{T}}$ to ${{\left\langle V \right\rangle }_{T}}$. If $T=\left\langle g \right\rangle $ (Definition \ref{generated operator semigroup}), where $g$ is a function from $D$ to $D$, then there exists a map $\alpha :U\to V\bigcup {{\left\langle V \right\rangle }_{T}}$ such that $\sigma $ is constructible by $\alpha $.
\end{theorem}
\begin{proof}
    If $\sigma $ is the empty map $\emptyset \to \emptyset $, then by Definition \ref{5.2.7}, $\sigma $ is constructible by the empty map trivially. Hence it suffices to show the case where $U(\subseteq D)$ and $V(\subseteq D)$ are nonempty. 

    Let $A=V\bigcup \left\langle V \right\rangle _T$. $\forall x\in D$ and $f\in T$, set 
    \begin{equation} \label{6.10a}
    {f_A^{-1}}(x)=\{a\in A\,|\, f(a)=x\}.
    \end{equation}

    Let $u\in U$. Let ${{A}_{u}}=\bigcap\nolimits_{f\in T}{f_A^{-1}}(\sigma (f(u)))$. Suppose ${{A}_{u}}\ne \emptyset $. Then by Equation $($\ref{6.10a}$)$, $f(a)=\sigma (f(u))$, $\forall a\in {{A}_{u}}$ and $f\in T$. To define a map $\alpha :U\to V\bigcup {{\left\langle V \right\rangle }_{T}}$ such that $\sigma $ is constructible by $\alpha $, let $\alpha (u)$ be any $a\in A_u (\subseteq A=V\bigcup {\left\langle V \right\rangle }_T)$. Then $\forall f\in T$, $\sigma (f(u))=f(a)=f(\alpha (u))$, as desired for $\sigma $ to be constructible by $\alpha$. 

Therefore, it suffices to show ${{A}_{u}}\ne \emptyset $.

To show ${{A}_{u}}\ne \emptyset $, we shall first show ${g_A^{-1}}(\sigma (g(u)))\subseteq A_u$. Then we shall prove ${g_A^{-1}}(\sigma (g(u)))\ne \emptyset $.

Since $T=\left\langle g \right\rangle $, ${{A}_{u}}=\bigcap\nolimits_{f\in T}{f_A^{-1}}(\sigma (f(u)))=\bigcap\nolimits_{k\in {{\mathbb{Z}}^{+}}}{{(g^k)_A^{-1}}(\sigma ({{g}^{k}}(u))})$, where ${{g}^{k}}$ denotes the composite of $k$ copies of $g$.

Let $k(\in {{\mathbb{Z}}^{+}})>1$. Because $g(u)\in {{\left\langle U \right\rangle }_T}$, by Definition \ref{1.3.1}, 
\[\sigma ({{g}^{k}}(u))=\sigma ({{g}^{k-1}}(g(u)))=g^{k-1}(\sigma (g(u))).\]
Applying ${(g^k)}_A^{-1}$ to the left and right sides of the above equations, we have
\begin{equation} \label{6.11a}
    {(g^k)}_A^{-1}(\sigma ({{g}^{k}}(u)))={(g^k)}_A^{-1}({g^{k-1}}(\sigma (g(u)))).	
\end{equation}

    To show ${g_A^{-1}}(\sigma (g(u)))\subseteq A_u$, suppose $\exists a\in {g_A^{-1}}(\sigma (g(u)))$. 
    
    Then by $($\ref{6.10a}$)$ (with $f:=g$ and $x:=\sigma (g(u))$), $g(a)=\sigma (g(u))$. Hence ${g^{k}}(a)={g^{k-1}}(g(a))={g^{k-1}}(\sigma (g(u)))$. Then by $($\ref{6.10a}$)$ (with $f:=g^k$ and $x:={g^{k-1}}(\sigma (g(u)))$), $a\in {{({g}^{k})}_A^{-1}}({g^{k-1}}(\sigma (g(u))))$. Thus,
    \begin{equation} \label{6.12a}
    {g_A^{-1}}(\sigma (g(u))) \subseteq {{({g}^{k})}_A^{-1}}({g^{k-1}}(\sigma (g(u)))).
    \end{equation}
   	
    Combining $($\ref{6.11a}$)$ with $($\ref{6.12a}$)$, we have
    \[{(g^k)}_A^{-1}(\sigma ({{g}^{k}}(u))) \supseteq {g_A^{-1}}(\sigma (g(u))).\]

    Hence $\forall k \ge 1, {(g^k)}_A^{-1}(\sigma ({{g}^{k}}(u))) \supseteq {g_A^{-1}}(\sigma (g(u)))$, and so
    \[{{A}_{u}}(=\bigcap\nolimits_{k\in {{\mathbb{Z}}^{+}}}{{(g^k)_A^{-1}}(\sigma ({{g}^{k}}(u))}))\supseteq {g_A^{-1}}(\sigma (g(u))).\] 
    
    Therefore, to show ${{A}_{u}}\ne \emptyset $, it suffices to prove ${g_A^{-1}}(\sigma (g(u)))\ne \emptyset $.

	Since $\sigma (g(u))\in {\left\langle V \right\rangle }_T$, $\exists v\in V$ and $h\in {{\mathbb{Z}}^{+}}$ such that $\sigma (g(u))= g^h(v)$. If $h=1$, then by $(\ref{6.10a})$ (with $f:=g$ and $x:=\sigma (g(u))$), $v\in {g_A^{-1}}(\sigma (g(u)))$. And if $h>1$, then $\sigma (g(u))=g(g^{h-1}(v))$, and thus by $(\ref{6.10a})$ (with $f:=g$, $x:=\sigma (g(u))$ and the fact that $g^{h-1}(v)\in {{\left\langle V \right\rangle }_T}\subseteq A$), $g^{h-1}(v)\in g_A^{-1}(\sigma (g(u)))$. In both cases of $h$, $g_A^{-1}(\sigma (g(u)))\ne \emptyset $, as desired. Therefore, ${{A}_{u}}\ne \emptyset $, as desired.

In summary, to define a map $\alpha :U\to V\bigcup {{\left\langle V \right\rangle }_T}$, for each $u\in U$, let $\alpha (u)$ be any element of ${A}_{u}$. Then $\alpha$ satisfies $\sigma (f(u))=f(\alpha (u))$, $\forall u\in U$ and $f\in T$. Hence by Definition \ref{5.2.7}, $\sigma$ is constructible by $\alpha$.
\end{proof}
\subsection{A construction of $\theta $-morphisms} \label{I Constructions of $theta $-morphisms}
For $\theta $-morphisms, we generalize Notation \ref{5.2.1} to 
 \begin{notation} \label{5.4.1}
     Let $D_1$ (resp. $D_2$) be a set, let $M_1$ (resp. $M_2$) be a set of maps from $D_1$ to $D_1$ (resp. a set of maps from $D_2$ to $D_2$), let $\theta \subseteq {{M}_{1}}\times {{M}_{2}}$, and let $\alpha :U(\subseteq {D_1})\to V(\subseteq {{D}_{2}})$ be a map. 
     
     By $U\xrightarrow{\theta ,\alpha }V $ we mean that $\forall f,g\in \operatorname{Dom}\theta $ and $u,w\in U$, 
\begin{center}
    $f(u)=g(w)$ implies $\theta (f)(\alpha (u))=\theta (g)(\alpha (w))$. 
\end{center}

Moreover, by $U\overset{\theta ,\alpha }{\longleftrightarrow}V $ we mean that 
\begin{center}
   $\alpha $ is bijective, $U\xrightarrow{\theta ,\alpha }V $ and $ V \xrightarrow{{{\theta }^{-1}},{{\alpha }^{-1}}}U$, 
\end{center}
 where ${{\theta }^{-1}}:=\{(h,f)\,|\,(f,h)\in \theta \}$.  
 \end{notation}
\begin{remark} \begin{enumerate}
\item $\forall f\in \operatorname{Dom}\theta $ and $u\in U$, $f(u)=f(u)$, and thus $U\xrightarrow{\theta ,\alpha }V $ implies that $\forall f\in \operatorname{Dom}\theta $ and $u\in U$, $\theta (f)(\alpha (u))=\theta (f)(\alpha (u))$. Hence by Convention \ref{convention}, $U\xrightarrow{\theta ,\alpha }V $ implies that $\forall f\in \operatorname{Dom}\theta $ and $u\in U$, $\theta (f)(\alpha (u))$ is well-defined. Thus by Proposition \ref{5.3.b}, $U\xrightarrow{\theta ,\alpha }V $ implies that $\theta{{|}_{\operatorname{Im}\alpha }}$ is a map. 
\item If ${{M}_{1}}={{M}_{2}}=:M$ and $\theta $ is the identity map on $M$, then $U\xrightarrow{\theta ,\alpha }V $ is equivalent to $U\xrightarrow{M,\alpha }V $ (defined in Notation \ref{5.2.1}).
\end{enumerate}
\end{remark}
The following generalizes Proposition \ref{5.2.2}.
\begin{proposition} \label{5.4.2}
    Let $\phi $ be a $\theta $-morphism from a $T_1$-space $S_1$ to a $T_2$-space $S_2$. Then $\forall A\subseteq S_1$, $A\xrightarrow{\theta ,\alpha }S_2$, where $\alpha :=\phi {{|}_{A}}$. In particular, $S_1\xrightarrow{\theta ,\phi }S_2$.
\end{proposition}  
\begin{proof}
    By Corollary \ref{5.3.c}, $\theta {{|}_{\operatorname{Im}\phi }}$ is a map, and so is $\theta {{|}_{\operatorname{Im}\alpha }}$ (since $\operatorname{Im}\alpha \subseteq \operatorname{Im}\phi$ and by Proposition \ref{5.3.b}).  

Let $f,g\in \operatorname{Dom}\theta $ and let $a,b\in A\subseteq S_1$. If $f(a)=g(b)$, then by Definition \ref{5.3.5},
\begin{center}
    $\theta (f)(\alpha (a))=\theta (f)(\phi (a))=\phi (f(a))=\phi (g(b))=\theta (g)(\phi (b))=\theta (g)(\alpha (b))$. 
\end{center}

Hence by Notation \ref{5.4.1}, $A\xrightarrow{\theta ,\alpha }S_2 $.
\end{proof}
For some reasons which will be clear soon, we are more interested when $\theta $ has the following property.
\begin{definition} \label{5.3.9}
    Let $T_1$ and $T_2$ be operator semigroups on $D_1$ and $D_2$, respectively, let $\theta \subseteq {{T}_{1}}\times {{T}_{2}}$, and let $A\subseteq {{D}_{2}}$. If $\theta (f\circ g)(a)=(\theta (f)\circ \theta (g))(a)$, $\forall f,g\in \operatorname{Dom}\theta $ and $a\in A$, then we say that $\theta $ is \emph{distributive over} $A$.
\end{definition}
\begin{remark}
    By Convention \ref{convention}, the condition
\begin{center}
    “$\theta (f\circ g)(a)=(\theta (f)\circ \theta (g))(a)$, $\forall f,g\in \operatorname{Dom}\theta $ and $a\in A$”
\end{center}
 implies that $\forall f,g\in \operatorname{Dom}\theta $ and $a\in A$, both $\theta (f\circ g)(a)$ and $(\theta (f)\circ \theta (g))(a)$ are well-defined, and hence the condition implies that $\operatorname{Dom}\theta $ is an operator semigroup on $D_1$ and $\theta {{|}_{A}}$ is a map (by Proposition \ref{5.3.b}).
\end{remark}
	Then to construct $\theta $-morphisms, we generalize Proposition \ref{5.2.4} as follows.
\begin{proposition} \label{5.4.3}
    Let $T_1$ and $T_2$ be operator semigroups on $D_1$ and $D_2$ respectively, let $\theta \subseteq {{T}_{1}}\times {{T}_{2}}$ with $\operatorname{Dom}\theta ={{T}_{1}}$, and let $\alpha :U(\subseteq {D_1})\to V(\subseteq{{D}_{2}})$ be a map. 
\begin{enumerate}
    \item [(a)] $U\xrightarrow{\theta ,\alpha }V $ if and only if there exists a map $\phi :{{\left\langle U \right\rangle }_{{{T}_{1}}}}\to {{\left\langle V  \right\rangle }_{{{T}_{2}}}}$ given by $f(u)\mapsto \theta (f)(\alpha (u))$, $\forall u\in U$ and $f\in {{T}_{1}}$.
    \item[(b)] The map $\phi $ given in \emph{(a)} is a $\theta $-morphism if $\theta $ is distributive over $V $. 
\end{enumerate}
\end{proposition}
\begin{proof}
 For (a): 
 
$U\xrightarrow{\theta ,\alpha }V $;

$\Leftrightarrow$ $\forall f,g\in \operatorname{Dom}\theta (={{T}_{1}})$ and $a,b\in U$, $f(a)=g(b)$ implies that $\theta (f)(\alpha (a))=\theta (g)(\alpha (b))$; (by Notation \ref{5.4.1})
 
$\Leftrightarrow$ $\phi :{{\left\langle U \right\rangle }_{{{T}_{1}}}}\to {{\left\langle V  \right\rangle }_{{{T}_{2}}}}$ given by $f(u)\mapsto \theta (f)(\alpha (u))$, $\forall u\in U$ and $f\in {{T}_{1}}$, is a well-defined map.

For (b): Suppose that $\theta $ is distributive over $V $. 

Let $a\in {{\left\langle U \right\rangle }_{{{T}_{1}}}}$. Then $\exists u\in U$ and $g\in {{T}_{1}}$ such that $a=g(u)$. Let $f\in {{T}_{1}}$. Then $f\circ g\in {{T}_{1}}$ by Definition \ref{Operator semigroup}. So $\phi (f(a))=\phi ((f\circ g)(u))=\theta (f\circ g)(\alpha (u))$ by the definition of $\phi $. Then because $\theta $ is distributive over $V $, by Definition \ref{5.3.9}, $\theta (f\circ g)(\alpha (u))=\theta (f)(\theta (g)(\alpha (u)))$. And $\theta (g)(\alpha (u))=\phi (g(u))$ by the definition of $\phi $. Combining the above equations, we have
\[\phi (f(a))=\theta (f)(\phi (g(u)))=\theta (f)(\phi (a)).\]
	Thus by Definition \ref{5.3.5}, $\phi $ is a $\theta $-morphism from ${{\left\langle U \right\rangle }_{{{T}_{1}}}}$ to ${{\left\langle V  \right\rangle }_{{{T}_{2}}}}$.
\end{proof}

Then Proposition \ref{5.2.5} is generalized as follows.
\begin{proposition} \label{5.4.4}
    Let $T_1$, $T_2$, $U$ and $\alpha $ be defined as in Proposition \ref{5.4.3}. Let $\theta \subseteq {{T}_{1}}\times {{T}_{2}}$ such that $\operatorname{Dom}\theta ={{T}_{1}}$ and $\operatorname{Im}\theta ={{T}_{2}}$. Then 
\begin{enumerate}
    \item [(a)] The following statements are equivalent:
\begin{enumerate}
    \item [(i)] $U\overset{\theta ,\alpha }{\longleftrightarrow}V $.
    \item[(ii)] $\alpha $ is bijective, there exists a bijective map $\phi :{{\left\langle U \right\rangle }_{{{T}_{1}}}}\to {{\left\langle V  \right\rangle }_{{{T}_{2}}}}$ given by $f(u)\mapsto \theta (f)(\alpha (u))$, $\forall u\in U$ and $f\in {{T}_{1}}$, and its inverse $\phi^{-1}:{{\left\langle V  \right\rangle }_{{{T}_{2}}}}\to {{\left\langle U \right\rangle }_{{{T}_{1}}}}$ can be given by $g(v)\mapsto \theta^{-1} (g)(\alpha^{-1} (v))$, $\forall v\in V$ and $g\in {{T}_{2}}$.
\end{enumerate}  
\item[(b)] Suppose that $\theta $ is distributive over $V $. Then the bijective map $\phi $ given in \emph{(ii)} is a $\theta $-isomorphism.
 \end{enumerate}
\end{proposition}  
\begin{proof}
  For (a): 
  
$U\overset{\theta ,\alpha }{\longleftrightarrow}V $;

$\Leftrightarrow$ $\alpha $ is bijective, $U\xrightarrow{\theta ,\alpha }V $ and $ V \xrightarrow{{{\theta }^{-1}},{{\alpha }^{-1}}}U$ (by Notation \ref{5.4.1});

$\Leftrightarrow$ $\alpha $ is bijective, $\phi :{{\left\langle U \right\rangle }_{{{T}_{1}}}}\to {{\left\langle V  \right\rangle }_{{{T}_{2}}}}$ given by $f(u)\mapsto \theta (f)(\alpha (u))$, $\forall u\in U$ and $f\in {{T}_{1}}$, is a well-defined map, and $\phi' :{{\left\langle V  \right\rangle }_{{{T}_{2}}}}\to {{\left\langle U \right\rangle }_{{{T}_{1}}}}$ given by $g(v)\mapsto \theta^{-1} (g)(\alpha^{-1} (v))$, $\forall v\in V$ and $g\in {{T}_{2}}$ is also a well-defined map (by Proposition \ref{5.4.3});

$\Leftrightarrow$ $\alpha $ is bijective, $\phi :{{\left\langle U \right\rangle }_{{{T}_{1}}}}\to {{\left\langle V  \right\rangle }_{{{T}_{2}}}}$ given by $f(u)\mapsto \theta (f)(\alpha (u))$, $\forall u\in U$ and $f\in {{T}_{1}}$, is a well-defined bijective map, and its inverse $\phi^{-1}:{{\left\langle V  \right\rangle }_{{{T}_{2}}}}\to {{\left\langle U \right\rangle }_{{{T}_{1}}}}$ can be given by $g(v)\mapsto \theta^{-1} (g)(\alpha^{-1} (v))$, $\forall v\in V$ and $g\in {{T}_{2}}$.

   The third equivalence relation is explained as follows. The sufficiency ($\Leftarrow$) is obvious, so we only show the necessity ($\Rightarrow$). For this purpose, we show that both $\phi'\circ\phi$ and $\phi\circ\phi'$ are the identity map.
   
   Let $(f,g)\in \theta$ and $u\in U$. Then 
   
   $\phi'(\phi(f(u)))$
   
   $=\phi'(\theta(f)(\alpha (u)))$ (by the definition of $\phi$) 

   $=\phi'(g(\alpha (u)))$ (by Definition \ref{5.3.a})
   
   $=\theta^{-1}(g)(u)$ (by the definition of $\phi'$) 
   
   $=f(u)$ (by (the $\theta^{-1}$ version of) Definition \ref{5.3.a}).
   
   Then $\phi'\circ\phi$ is the identity map on ${{\left\langle U  \right\rangle }_{{{T}_{1}}}}$ (because $\operatorname{Dom}\theta ={{T}_{1}}$).
   
   Analogously, we can show that $\phi\circ\phi'$ is the identity map on ${{\left\langle V  \right\rangle }_{{{T}_{2}}}}$. Therefore, both $\phi$ and $\phi'$ are bijective and $\phi'$ is the inverse of $\phi$.

For (b): 

By (b) in Proposition \ref{5.4.3}, the map $\phi $ given in (ii) is a $\theta$-morphism. Since $\phi $ is bijective, by Definition \ref{5.4.a}, $\phi$ is a $\theta$-isomorphism.
\end{proof}
	Proposition \ref{5.4.3} raises a question as follows, which generalizes Question \ref{5.2.6}. 
\begin{question} \label{5.4.5}
    Let $T_1$ and $T_2$ be operator semigroups on $D_1$ and $D_2$, respectively, let $\theta \subseteq {{T}_{1}}\times {{T}_{2}}$, let $U\subseteq D_1$ and let $\phi$ be a $\theta$-morphism from ${{\left\langle U \right\rangle }_{{{T}_{1}}}}$ (to a $T_2$-space). If $\operatorname{Dom}\theta ={{T}_{1}}$, is there any map $\alpha :U\to {{D}_{2}}$ such that $\phi (f(u))=\theta (f)(\alpha (u))$, $\forall u\in U$ and $f\in {{T}_{1}}$?
\end{question}
The significance of this question is that, in the cases where the answer is yes, we can construct a $\theta$-morphism $\phi $ by letting $\phi (f(u))=\theta (f)(\alpha (u))$, $\forall u\in U$ and $f\in {{T}_{1}}$. The property in question deserves a name as follows, which generalizes Definition \ref{5.2.7}.
\begin{definition} \label{5.4.6}
    Let $T_1$ and $T_2$ be operator semigroups on $D_1$ and $D_2$, respectively, let $\theta \subseteq {{T}_{1}}\times {{T}_{2}}$, let $U\subseteq D_1$ and let $\phi$ be a $\theta$-morphism from ${{\left\langle U \right\rangle }_{{{T}_{1}}}}$ (to a $T_2$-space). If $\operatorname{Dom}\theta ={{T}_{1}}$ and there exists a map $\alpha :U\to {{D}_{2}}$ such that $\phi (f(u))=\theta (f)(\alpha (u))$, $\forall u\in U$ and $f\in {{T}_{1}}$, then we say that $\phi $ is \emph{constructible by} $\alpha $, or just say that $\phi $ is \emph{constructible} for brevity.
\end{definition}
\begin{remark}
    By Convention \ref{convention}, the condition “$\phi (f(u))=\theta (f)(\alpha (u))$, $\forall u\in U$ and $f\in {{T}_{1}}$” implies that $\theta {{|}_{\operatorname{Im}\alpha}}$ is a map.
\end{remark}
	Since $T$-morphisms are special cases of $\theta $-morphisms, Example \ref{5.1.12} shows that it is possible that a $\theta $-morphism $\phi$ from ${{\left\langle U \right\rangle }_{{{T}_{1}}}}$ is not constructible by any map $\alpha :U\to {{D}_{2}}$ even if $\theta $ is distributive over ${{D}_{2}}$. Now let’s show some cases where $\theta $-morphisms are constructible.
The following generalizes Corollary \ref{5.2.9}. 
\begin{corollary} \label{5.4.8}
    Let $T_1$ and $T_2$ be operator semigroups on $D_1$ and $D_2$, respectively, let $\theta \subseteq {{T}_{1}}\times {{T}_{2}}$ and let $U\subseteq {D_1}$. If $\operatorname{Dom}\theta ={{T}_{1}}$ and $U\subseteq {{\left\langle U \right\rangle }_{{{T}_{1}}}}$, then any $\theta $-morphism $\phi $ from ${{\left\langle U \right\rangle }_{{{T}_{1}}}}$ is constructible by $\alpha :=\phi {{|}_{U}}$.
\end{corollary}
\begin{proof}
    Since $\operatorname{Dom}\theta ={{T}_{1}}$ and $U\subseteq {{\left\langle U \right\rangle }_{{{T}_{1}}}}$, by Definition \ref{5.3.5}, $\forall u\in U$ and $f\in {{T}_{1}}$, $\phi (f(u))=\theta (f)(\phi (u))=\theta (f)(\alpha (u))$. 
\end{proof}
	Theorem \ref{5.2.10} is generalized for $\theta $-morphisms as follows.
\begin{theorem} \label{5.4.9}
    Let $T_1$ and $T_2$ be operator semigroups on $D_1$ and $D_2$, respectively, let $\theta \subseteq {{T}_{1}}\times {{T}_{2}}$, and let $\phi $ be a $\theta$-morphism from ${{\left\langle U \right\rangle }_{{{T}_{1}}}}$ to ${{\left\langle V \right\rangle }_{{{T}_{2}}}}$. If $\operatorname{Dom}\theta ={{T}_{1}}=\left\langle g \right\rangle $, $\operatorname{Im}\theta ={{T}_{2}}$, and $\theta $ is distributive over $V\bigcup {{\left\langle V \right\rangle }_{{{T}_{2}}}}$, then there exists a map $\alpha :U\to V\bigcup {{\left\langle V \right\rangle }_{{{T}_{2}}}}$ such that $\phi $ is constructible by $\alpha $.
\end{theorem}
\begin{proof}
    The proof is analogous to the one for Theorem \ref{5.2.10}.
    
    If $\phi $ is the empty map $\emptyset \to \emptyset $, then by Definition \ref{5.4.6}, $\phi $ is constructible by the empty map trivially. Hence we only need to prove the case where $U(\subseteq {D_1})$ and $V(\subseteq {{D}_{2}})$ are nonempty. 

    By (the remark under) Definition \ref{5.3.9}, $\theta {{|}_{A}}$ is a map, where $A:=V\bigcup {{\left\langle V \right\rangle }_{{{T}_{2}}}}$, and so is $\theta {{|}_{\operatorname{Im}\phi }}$ because $\operatorname{Im}\phi \subseteq {{\left\langle V \right\rangle }_{{{T}_{2}}}}\subseteq A$.

    First, $\forall x\in {{D}_{2}}$ and $f\in {{T}_{1}}$, set 
    \begin{equation} \label{6.10}
    {{(\theta (f))}_A^{-1}}(x)=\{a\in A\,|\,\theta (f)(a)=x\}.
    \end{equation}
     It is well-defined because $\theta {{|}_{A}}$ is a map. 

    Let $u\in U$. Let ${{A}_{u}}=\bigcap\nolimits_{f\in {{T}_{1}}}{{{(\theta (f))}_A^{-1}}(\phi (f(u)))}$. Suppose ${{A}_{u}}\ne \emptyset $. Then by Equation $($\ref{6.10}$)$, $\theta (f)(a)=\phi (f(u))$, $\forall a\in {{A}_{u}}$ and $f\in {{T}_{1}}$. To define a map $\alpha :U\to V\bigcup {{\left\langle V \right\rangle }_{{{T}_{2}}}}$ such that $\phi $ is constructible by $\alpha $, let $\alpha (u)$ be any $a\in {{A}_{u}}(\subseteq A=V\bigcup {{\left\langle V \right\rangle }_{{{T}_{2}}}})$. Then $\forall f\in {{T}_{1}}$, $\phi (f(u))=\theta (f)(a)=\theta (f)(\alpha (u))$, as desired for $\phi $ to be constructible by $\alpha$. 

Therefore, it suffices to show ${{A}_{u}}\ne \emptyset $.

To show ${{A}_{u}}\ne \emptyset $, we shall first show ${{(\theta (g))}_A^{-1}}(\phi (g(u)))\subseteq A_u$. Then we shall prove ${{(\theta (g))}_A^{-1}}(\phi (g(u)))\ne \emptyset $.

Since ${{T}_{1}}=\left\langle g \right\rangle $, ${{A}_{u}}=\bigcap\nolimits_{f\in {{T}_{1}}}{{{(\theta (f))}_A^{-1}}(\phi (f(u)))}=\bigcap\nolimits_{k\in {{\mathbb{Z}}^{+}}}{{{(\theta ({{g}^{k}}))}_A^{-1}}(\phi ({{g}^{k}}(u))})$, where ${{g}^{k}}$ denotes the composite of $k$ copies of $g$.

Let $k(\in {{\mathbb{Z}}^{+}})>1$. Because $g(u)\in {{\left\langle U \right\rangle }_{{{T}_{1}}}}$, by Definition \ref{5.3.5}, 
\[\phi ({{g}^{k}}(u))=\phi ({{g}^{k-1}}(g(u)))=\theta ({{g}^{k-1}})(\phi (g(u)))={{(\theta (g))}^{k-1}}(\phi (g(u)))\]
because $\theta $ is assumed to be distributive over $V\bigcup {{\left\langle V \right\rangle }_{{{T}_{2}}}}(\supseteq \operatorname{Im}\phi )$, and so
\[\phi ({{g}^{k}}(u))={{(\theta (g))}^{k-1}}(\phi (g(u))).\]
Applying ${{(\theta ({{g}^{k}}))}_A^{-1}}$ to both sides of the above equation, we have
\begin{equation} \label{6.11}
    {{(\theta ({{g}^{k}}))}_A^{-1}}(\phi ({{g}^{k}}(u)))={{(\theta ({{g}^{k}}))}_A^{-1}}({{(\theta (g))}^{k-1}}(\phi (g(u)))).	
\end{equation}

    To show ${{(\theta (g))}_A^{-1}}(\phi (g(u)))\subseteq A_u$, suppose $\exists a\in {{(\theta (g))}_A^{-1}}(\phi (g(u)))$. Then by $($\ref{6.10}$)$ (with $f:=g$ and $x:=\phi (g(u))$), $\theta (g)(a)=\phi (g(u))$. Hence 
    \[{{(\theta (g))}^{k-1}}(\theta (g)(a))={{(\theta (g))}^{k-1}}(\phi (g(u))).\]
    Since $\theta $ is distributive over $A$,
    \[{{(\theta (g))}^{k-1}}(\theta (g)(a))={{(\theta (g))}^{k}}(a)=\theta ({{g}^{k}})(a).\] 

    Combining the above three equations, we have $\theta ({{g}^{k}})(a)={{(\theta (g))}^{k-1}}(\phi (g(u)))$. Then by $($\ref{6.10}$)$ (with $f:=g^k$ and $x:={{(\theta (g))}^{k-1}}(\phi (g(u)))$), 
    \[a\in {{(\theta ({{g}^{k}}))}_A^{-1}}({{(\theta (g))}^{k-1}}(\phi (g(u)))).\] 
    Since $a$ is assumed to be any element of ${{(\theta (g))}_A^{-1}}(\phi (g(u)))$,
    \begin{equation} \label{6.12}
    {{(\theta (g))}_A^{-1}}(\phi (g(u)))\subseteq {{(\theta ({{g}^{k}}))}_A^{-1}}({{(\theta (g))}^{k-1}}(\phi (g(u)))).
    \end{equation}
   	
    Combining $($\ref{6.11}$)$ with $($\ref{6.12}$)$, we have
    \[{{(\theta ({{g}^{k}}))}_A^{-1}}(\phi ({{g}^{k}}(u)))\supseteq {{(\theta (g))}_A^{-1}}(\phi (g(u))).\]

    Hence $\forall k \ge 1, {{(\theta ({{g}^{k}}))}_A^{-1}}(\phi ({{g}^{k}}(u)))\supseteq {{(\theta (g))}_A^{-1}}(\phi (g(u)))$, and so
    \[{{A}_{u}}(=\bigcap\nolimits_{k\in {{\mathbb{Z}}^{+}}}{{{(\theta ({{g}^{k}}))}_A^{-1}}(\phi ({{g}^{k}}(u))}))\supseteq (\theta (g))_A^{-1}(\phi (g(u)).\] 
    
    Therefore, to show ${{A}_{u}}\ne \emptyset $, it suffices to prove ${{(\theta (g))}_A^{-1}}(\phi (g(u)))\ne \emptyset $.

	Since $\phi (g(u))\in {{\left\langle V \right\rangle }_{{{T}_{2}}}}$ and $\operatorname{Im}\theta ={{T}_{2}}$, $\exists v\in V$ and $h\in {{\mathbb{Z}}^{+}}$ such that $\phi (g(u))=\theta ({{g}^{h}})(v)$. If $h=1$, then by $($\ref{6.10}$)$ (with $f:=g$ and $x:=\phi (g(u))$), $v\in {{(\theta (g))}_A^{-1}}(\phi (g(u)))$. And if $h>1$, then because $\theta $ is distributive over $V\bigcup {{\left\langle V \right\rangle }_{{{T}_{2}}}}$, $\phi (g(u))=\theta ({{g}^{h}})(v)=\theta (g)(\theta ({{g}^{h-1}})(v))$, and hence by $($\ref{6.10}$)$ (with $f:=g$, $x:=\phi (g(u))$ and the fact that $\theta ({{g}^{h-1}})(v)\in {{\left\langle V \right\rangle }_{{{T}_{2}}}}\subseteq A$), $\theta ({{g}^{h-1}})(v)\in {{(\theta (g))}_A^{-1}}(\phi (g(u)))$. In both cases of $h$, ${{(\theta (g))}_A^{-1}}(\phi (g(u)))\ne \emptyset $, as desired. Therefore, ${{A}_{u}}\ne \emptyset $, as desired.

In summary, to define a map $\alpha :U\to V\bigcup {{\left\langle V \right\rangle }_{{{T}_{2}}}}$, for each $u\in U$, let $\alpha (u)$ be any element of ${A}_{u}$. Then $\alpha $ satisfies $\phi (f(u))=\theta (f)(\alpha (u))$, $\forall u\in U$ and $f\in {{T}_{1}}$. Hence by Definition \ref{5.4.6}, $\phi$ is constructible by $\alpha$.
\end{proof}

\subsection{Another construction of $T$-morphisms} \label{I Another construction of $T$-morphisms}
 	Let $\sigma $ be a $T$-morphism from a $T$-space $S$. If $\sigma (a)=b$, then $\forall f\in T$, $\sigma (f(a))=f(\sigma (a))=f(b)$. In other words, if an ordered pair $(a,b)\in \sigma $, where $\sigma $ is regarded as a binary relation, then $\forall f\in T$, $(f(a),f(b))\in \sigma $. This observation implies that there exists some sort of basic parts of $T$-morphisms. 

To facilitate our descriptions, we give the following notations only for this subsection.
\begin{notation} \label{5.5.1}
    Let $T^*=T\bigcup \{$Id on $D$\} and let ${{R}_{T}}(x,y)=\{(f(x),f(y))\,|\,f\in T^*\}$, where $(x,y)\in D\times D$. Besides, both ${{S}_{a}}$ and ${{S}_{b}}$ denote $T$-spaces.
\end{notation}
\begin{remark}
    Clearly, ${{R}_{T}}(x,y)\subseteq D\times D$ is a binary relation in $D$.
\end{remark}
  \begin{proposition} \label{5.5.2}
Let $R\subseteq {{S}_{a}}\times {{S}_{b}}$ such that $\operatorname{Dom}R$ is a $T$-subspace of $S_a$. Then the following statements are equivalent:
\begin{enumerate}
    \item [(i)] $R$ is a $T$-morphism from $\operatorname{Dom}R$ to ${{S}_{b}}$.
    \item[(ii)] There exists a binary relation $\bigcup\nolimits_{i}{\{({{a}_{i}},{{b}_{i}})\}}\subseteq {{S}_{a}}\times {{S}_{b}}$ such that $R=\bigcup\nolimits_{i}{{{R}_{T}}({{a}_{i}},{{b}_{i}})}$ and $R$ is a function. 
\end{enumerate}
  \end{proposition}
\begin{proof}
(i)$ \Rightarrow $(ii): Suppose $R=\bigcup\nolimits_{i}{\{({{a}_{i}},{{b}_{i}})}\}$. By Definition \ref{1.3.1}, $\forall f\in T^*$ and $({{a}_{i}},{{b}_{i}})\in R$, $R(f({{a}_{i}}))=f(R({{a}_{i}}))=f({{b}_{i}})$, and hence $(f({{a}_{i}}),f({{b}_{i}}))\in R$. 

Therefore, $R=\bigcup\nolimits_{i}{\{(f({{a}_{i}}),f({{b}_{i}}))\,|\,f\in T^*\}}=\bigcup\nolimits_{i}{{{R}_{T}}({{a}_{i}},{{b}_{i}})}$. Hence (ii) is true.

(ii)$ \Rightarrow $(i): Let $g\in T^*$. Suppose $(u,v)\in R$. 

Since $R=(\bigcup\nolimits_{i}{{{R}_{T}}({{a}_{i}},{{b}_{i}})}=)\bigcup\nolimits_{i}{\{(f({{a}_{i}}),f({{b}_{i}}))}\,|\,f\in T^*\}$ and $T$ is a semigroup, $(g(u),g(v))\in R$. Because $R$ is a function, $R(g(u))=g(v)=g(R(u))$. Hence by Definition \ref{1.3.1}, $R$ is a $T$-morphism from $\operatorname{Dom}R$ to ${{S}_{b}}$.	
\end{proof}
Proposition \ref{5.5.2} implies that ${{R}_{T}}(x,y)$, where $(x,y)\in {{S}_{a}}\times {{S}_{b}}$, may be regarded as some sort of basic part of a $T$-morphism from a $T$-subspace of ${{S}_{a}}$ to ${{S}_{b}}$. This observation inspires us to define a topology on ${{S}_{a}}\times {{S}_{b}}$ as follows.
\begin{definition} \label{5.5.3}
    Let ${B_T}({S_a} \times {S_b}) = \{ {R_T}(x,y)\,|\,(x,y) \in {S_a} \times {S_b}\} $. Let ${{\mathcal{T}}_{T}}({{S}_{a}}\times {{S}_{b}})$ be the topology on ${{S}_{a}}\times {{S}_{b}}$ generated by ${{B}_{T}}({{S}_{a}}\times {{S}_{b}})$ as a subbasis; that is, let ${{\mathcal{T}}_{T}}({{S}_{a}}\times {{S}_{b}})$ be the collection of all unions of finite intersections of elements of ${{B}_{T}}({{S}_{a}}\times {{S}_{b}})$. 
\end{definition}
\begin{proposition} \label{5.5.4}
    Let $\bigcap\nolimits_{i}{B_i}$ be an arbitrary intersection of elements $B_i$ of ${{B}_{T}}({{S}_{a}}\times {{S}_{b}})$. Then $\bigcap\nolimits_{i}{B_i}$ is also a union of elements of ${{B}_{T}}({{S}_{a}}\times {{S}_{b}})$.
\end{proposition}
\begin{proof}
    If $\bigcap\nolimits_{i}{B_i}=\emptyset $, then by convention, $\bigcap\nolimits_{i}{B_i}$ may be regarded as the union of an empty family of elements of ${{B}_{T}}({{S}_{a}}\times {{S}_{b}})$. So we only need to consider the case where $\bigcap\nolimits_{i}{B_i}\ne \emptyset $.

     Any element $B_i$ of ${{B}_{T}}({{S}_{a}}\times {{S}_{b}})$ is a set ${R_T}(x,y)=\{(f(x),f(y))\,|\,f\in T^*\text{ }\!\!\}\!\!\text{ }$, where $(x,y)\in {{S}_{a}}\times {{S}_{b}}$. So for any $(u,v)\in B_i$ and $g\in T^*$, since $T$ is a semigroup, we can tell $(g(u),g(v))\in B_i$. Then $\forall g\in T^*$ and $(u,v)\in \bigcap\nolimits_{i}{B_i}$, $(g(u),g(v))\in \bigcap\nolimits_{i}{B_i}$.
    
    Therefore, 
    
    $\bigcap\nolimits_{i}{B_i}$
    
    $=\bigcup\nolimits_{(u,v)\in \bigcap\nolimits_{i}{B_i}}{\{(g(u),g(v))\,|\,g\in T^*\text{ }\!\!\}\!\!\text{ }}$
    
    $=\bigcup\nolimits_{(u,v)\in \bigcap\nolimits_{i}{B_i}}{{{R}_{T}}(u,v)}.$
\end{proof}
Thus immediately we have
\begin{corollary} \label{5.5.5}
    ${{B}_{T}}({{S}_{a}}\times {{S}_{b}})$ is a basis of ${{\mathcal{T}}_{T}}({{S}_{a}}\times {{S}_{b}})$. Thus, ${{\mathcal{T}}_{T}}({{S}_{a}}\times {{S}_{b}})$ is the collection of all unions of elements of ${{B}_{T}}({{S}_{a}}\times {{S}_{b}})$.
\end{corollary}
	Combining Proposition \ref{5.5.2} with the above corollary, we immediately obtain the main theorem of this subsection as follows.
 \begin{theorem} \label{5.5.6}
     Let $R\subseteq {{S}_{a}}\times {{S}_{b}}$ such that $\operatorname{Dom}R$ is a $T$-subspace of $S_a$. Then $R$ is a $T$-morphism from $\operatorname{Dom}R$ to ${{S}_{b}}$ if and only if $R\in {{\mathcal{T}}_{T}}({{S}_{a}}\times {{S}_{b}})$ and $R$ is a function. 
 \end{theorem}

$ $

\addcontentsline{toc}{section}{Part II. \textbf{Theory for the case of multivariable total or partial functions}}

\begin{center}
Part II. THEORY FOR THE CASE OF MULTIVARIABLE TOTAL OR PARTIAL FUNCTIONS 
\end{center}

To make our theory more general, in Section \ref{Basic notions for multivariable}, we shall generalize the notion of operator semigroup to incorporate functions of more than one variable. Then we shall find that some familiar concepts such as ring homomorphisms, module homomorphisms, and group homomorphisms can be characterized by $T$-morphisms or $\theta$-morphisms. Moreover, in Section \ref{Basic notions for partial}, we shall allow elements in $T$ to be partial functions. Then we will find that some notions such as covariant functors between categories can be characterized in terms of $T$-morphisms or $\theta$-morphisms.

\section{Basic notions for the case of multivariable (total) functions} \label{Basic notions for multivariable}
	Subsection \ref{Operator generalized-semigroups} generalizes the notion of operator semigroup to incorporate multivariable functions. Accordingly, Subsections \ref{10 $T$-spaces}, \ref{10 $T$-morphisms}, \ref{10 $T$-isomorphisms}, \ref{10 $theta$-morphisms} and \ref{10 $theta$-isomorphisms} generalize the notions of $T$-spaces, $T$-morphisms, $T$-isomorphisms, $\theta$-morphisms and $\theta$-isomorphisms, respectively.
\subsection{Operator generalized-semigroup $T$} \label{Operator generalized-semigroups}
\begin{definition} \label{9.1.1}
    Let $D$ be a set, let $n\in {{\mathbb{Z}}^{+}}$, and let $f$ be a function from the cartesian product ${{D}^{n}}$ to $D$. $\forall i=1, \cdots, n$, let $n_i\in {{\mathbb{Z}}^{+}}$ and let ${{g}_{i}}$ be a function from the cartesian product ${{D}^{{{n}_{i}}}}$ to $D$. Let $m=\sum\nolimits_{i}{{{n}_{i}}}$. Then the \emph{composite} $f\circ ({{g}_{1}},\cdots ,{{g}_{n}})$ is the function from ${{D}^{m}}$ to $D$ given by 
    \[({{x}_{1}},\cdots ,{{x}_{m}})\mapsto f({{g}_{1}}({{\text{v}}_{1}}),\cdots ,{{g}_{n}}({{\text{v}}_{n}})),\]
    where ${{\text{v}}_{i}}:=({{x}_{{{m}_{i}}+1}},\cdots ,{{x}_{{{m}_{i}}+{{n}_{i}}}})\in {{D}^{{{n}_{i}}}},\forall i=1,\cdots ,n$, ${{m}_{1}}:=0$ and ${{m}_{i}}:=\sum\nolimits_{k=1}^{i-1}{{{n}_{k}}},\forall i=2,\cdots ,n$.
\end{definition}
\begin{remark}
    If $n=1$, we can tell that the $f\circ g_1$ defined above is the usual composite of functions $f$ and $g_1$.
\end{remark}
	Now Definition \ref{Operator semigroup} is generalized to
\begin{definition} \label{9.1.2}
   Let $D$ be a set and let $T$ be a subset of 
   \begin{center}
       \{all functions from the cartesian product ${{D}^{n}}$ to $D\, |\,n\in {{\mathbb{Z}}^{+}}$\}.
   \end{center}
If $\forall n\in {{\mathbb{Z}}^{+}}$, $f\in T$ of $n$ variables, and $({{g}_{1}},\cdots ,{{g}_{n}})\in {{T}^{n}}$ (cartesian product), the composite $f\circ ({{g}_{1}},\cdots ,{{g}_{n}})\in T$, then we call $T$ an \emph{operator generalized-semigroup}, abbreviated as \emph{operator gen-semigroup}, on $D$ and $D$ is called the \emph{domain} of $T$.
\end{definition}  
\begin{remark}
\begin{enumerate}
    \item The terminology “generalized-semigroup” is informal and just for convenience because we have not defined it.
    \item Obviously, any operator semigroup is an operator gen-semigroup.
\end{enumerate}
\end{remark}
The following generalizes Example \ref{T on field--1}.
\begin{example} \label{9.1.3}
    Let $B$ be a subfield of a field $F$. Let $T$ be the set
\begin{center}
    $\{f^*:{{F}^{n}}\to F$ given by $(a_1,\cdots ,a_n)\mapsto f(a_1,\cdots ,a_n)\,|\,f\in B[{{x}_{1}},\cdots ,{{x}_{n}}],n\in {{\mathbb{Z}}^{+}}\}$,
\end{center} 
    where $B[{{x}_{1}},\cdots ,{{x}_{n}}]$ is the ring of polynomials over $B$ in $n$ variables.

	Then by Definition \ref{9.1.1}, $\forall f^*,{{g}_{1}}^*,\cdots ,{{g}_{n}}^*\in T$ such that $f\in B[{{x}_{1}},\cdots ,{{x}_{n}}]$ and ${{g}_{i}}\in B[{{x}_{1}},\cdots ,{{x}_{{{n}_{i}}}}]$, $\forall i=1,\cdots ,n$, the composite $f^*\circ ({{g}_{1}}^*,\cdots ,{{g}_{n}}^*)$ can be given by
 \begin{center}
     $({{x}_{1}},\cdots ,{{x}_{m}})\mapsto f({{g}_{1}}({{\text{v}}_{1}}),\cdots ,{{g}_{n}}({{\text{v}}_{n}}))$, 
 \end{center}
where $m:=\sum\nolimits_{i}{{{n}_{i}}}$, $({{x}_{1}},\cdots ,{{x}_{m}}) \in {{F}^{m}}$, ${{\text{v}}_{i}}:=({{x}_{{{m}_{i}}+1}},\cdots ,{{x}_{{{m}_{i}}+{{n}_{i}}}})\in {{F}^{{{n}_{i}}}},\forall i=1,\cdots ,n$, ${{m}_{1}}:=0$ and ${{m}_{i}}:=\sum\nolimits_{k=1}^{i-1}{{{n}_{k}}},\forall i=2,\cdots ,n$. 

For example, let $f={{g}_{1}}={{g}_{2}}={{x}_{1}}{{x}_{2}}$. Then $f^*\circ ({{g}_{1}}^*,{{g}_{2}}^*)$ can be induced by (with a slight abuse of notation) $f({{g}_{1}},{{g}_{2}})=({{x}_{1}}{{x}_{2}})({{x}_{3}}{{x}_{4}})={{x}_{1}}{{x}_{2}}{{x}_{3}}{{x}_{4}}$. That is, $f^*\circ ({{g}_{1}}^*,{{g}_{2}}^*):{{F}^{4}}\to F$ is given by $({{a}_{1}},{{a}_{2}},{{a}_{3}},{{a}_{4}})\mapsto {{a}_{1}}{{a}_{2}}{{a}_{3}}{{a}_{4}},\forall ({{a}_{1}},{{a}_{2}},{{a}_{3}},{{a}_{4}})\in {{F}^{4}}$. 

Then from Definition \ref{9.1.2}, we can tell that $T$ is an operator gen-semigroup on $F$. 

Note that the map which induces $T$, i.e.
					\[\tau :\bigcup\nolimits_{n\in {{\mathbb{Z}}^{+}}}{B[{{x}_{1}},\cdots ,{{x}_{n}}]}\to T\] 
given by $f\mapsto f^*$ , is surjective, but it is not necessarily injective.
\end{example}
The following several examples will also be employed.
\begin{example} \label{9.1.4}
Let $R$ be a ring with identity and let ${{\mathbb{F}}_{2}}\left\langle {{x}_{1}},\cdots ,{{x}_{n}} \right\rangle $ be the ring of polynomials over the prime field of order 2 in $n$ noncommuting variables. Let
\begin{center}
    $T=\{f^*:{{R}^{n}}\to R$ given by $(a_1,\cdots ,a_n)\mapsto f(a_1,\cdots ,a_n)\,|\,f\in {{\mathbb{F}}_{2}}\left\langle {{x}_{1}},\cdots ,{{x}_{n}} \right\rangle ,n\in {{\mathbb{Z}}^{+}}\}$. 
\end{center}
Then we can tell that $T$ is an operator gen-semigroup on $R$. 

The map which induces $T$, i.e.
\[\tau :\bigcup\nolimits_{n\in {{\mathbb{Z}}^{+}}}{{{\mathbb{F}}_{2}}\left\langle {{x}_{1}},\cdots ,{{x}_{n}} \right\rangle }\to T\] 
is given by $f\mapsto f^*$.
\end{example}
\begin{example} \label{9.1.5}
    Let $M$ be a module over a commutative ring $R$ with identity. Let 
\begin{center}
    $T=\{f^*:{{M}^{n}}\to M$ given by $(a_1,\cdots ,a_n)\mapsto f(a_1,\cdots ,a_n)\,|\,f\in \{\sum\nolimits_{i=1}^{n}{{{r}_{i}}{{x}_{i}}}\,|\,{{r}_{i}}\in R\},n\in {{\mathbb{Z}}^{+}}\}$,
\end{center}
(where $\sum\nolimits_{i=1}^{n}{{{r}_{i}}{{x}_{i}}}\in R[{{x}_{1}},\cdots ,{{x}_{n}}]$). Then $T$ is an operator gen-semigroup on $M$.

 Let $\tau :\bigcup\nolimits_{n\in {{\mathbb{Z}}^{+}}}{\{\sum\nolimits_{i=1}^{n}{{{r}_{i}}{{x}_{i}}}\,|\,{{r}_{i}}\in R\}}\to T$ be the map inducing $T$ given by $f\mapsto f^*$.
\end{example}
\begin{example} \label{9.1.6}
    Let $G$ be an abelian group. Let
 \begin{center}
     $T=\bigcup\nolimits_{n\in {{\mathbb{Z}}^{+}}}{\{f^*:{{G}^{n}}\to G}$ given by $(a_1,\cdots ,a_n)\mapsto f(a_1,\cdots ,a_n)\,|\,f\in \{\prod\nolimits_{i=1}^{n}{x_{i}^{{{p}_{i}}}}\,|\,{{p}_{i}}\in \mathbb{Z},\forall
     i=1,\cdots,n\}$\},
 \end{center}
(where $\prod\nolimits_{i=1}^{n}{x_{i}^{{{p}_{i}}}}\in \operatorname{Frac}({{\mathbb{F}}_{2}}[{{x}_{1}},\cdots ,{{x}_{n}}])$). Then $T$ is an operator gen-semigroup on $G$. 

Let $\tau :\bigcup\nolimits_{n\in {{\mathbb{Z}}^{+}}}{\{\prod\nolimits_{i=1}^{n}{x_{i}^{{{p}_{i}}}}\,|\,{{p}_{i}}\in \mathbb{Z}\}}\to T$ be the map inducing $T$ given by $f\mapsto f^*$.
\end{example}
For not necessarily abelian groups, we have
\begin{example} \label{9.1.7}
    Let $G$ be a group. Recall that a word is a finite string of symbols, where repetition is allowed. Let $\left\langle {{x}_{1}},\cdots ,{{x}_{n}},x_{1}^{-1},\cdots ,x_{n}^{-1} \right\rangle $ denote the set of all words on $\{{{x}_{1}},\cdots ,{{x}_{n}},x_{1}^{-1},\cdots ,x_{n}^{-1}\}$. Let
\begin{center}
    $T=\bigcup\nolimits_{n\in {{\mathbb{Z}}^{+}}}{\{f^*:{{G}^{n}}\to G}$ given by $(a_1,\cdots ,a_n)\mapsto f(a_1,\cdots ,a_n)\,|\,f\in \left\langle {{x}_{1}},\cdots ,{{x}_{n}},x_{1}^{-1},\cdots ,x_{n}^{-1} \right\rangle \}$. 
\end{center}

Then we can tell that $T$ is an operator gen-semigroup on $G$. 

Let	$\tau :\bigcup\nolimits_{n\in {{\mathbb{Z}}^{+}}}{\left\langle {{x}_{1}},\cdots ,{{x}_{n}},x_{1}^{-1},\cdots ,x_{n}^{-1} \right\rangle }\to T$ be the map inducing $T$ given by $f\mapsto f^*$.
\end{example}

\subsection{$T$-spaces} \label{10 $T$-spaces}
\begin{definition} \label{9.2.1}
    Let $T$ be an operator gen-semigroup on $D$ and let $U\subseteq D$. 
    Then we call 
    \begin{center}
    $S:=\{f(u_1,\cdots ,u_n)\,|\,n\in {{\mathbb{Z}}^{+}}$, $f\in T$ has $n$ variables, and $(u_1,\cdots ,u_n)\in U^n\}$
    \end{center}
   the \emph{$T$-space} generated by $U$ and denote $S$ by ${{\left\langle U \right\rangle }_{T}}$.
\end{definition}
\begin{example} \label{9.2.2}
    Let $B$ be a subfield of a field $F$. Let $T$ be the operator gen-semigroup on $F$ defined in Example \ref{9.1.3}. Then $\forall (\emptyset \ne )U\subseteq F$, ${{\left\langle U \right\rangle }_{T}}$ is the ring $B[U]$.
\end{example}
\begin{example} \label{9.2.3}
    Let $R$ be a ring with identity. Let $T$ be the operator gen-semigroup defined in Example \ref{9.1.4}. Then $\forall (\emptyset \ne )U\subseteq R$, ${{\left\langle U \right\rangle }_{T}}$ is a subring of $R$. Specifically, ${{\left\langle R \right\rangle }_{T}}=R$ because $\operatorname{Id}\in T$. 
\end{example}
\begin{example} \label{9.2.4}
    Let $M$ be a module over a commutative ring $R$ with identity. Let $T$ be the operator gen-semigroup defined in Example \ref{9.1.5}. If $X$ is a nonempty subset of $M$, then ${{\left\langle X \right\rangle }_{T}}$ is the submodule of $M$ generated by $X$. 
\end{example}
\begin{example} \label{9.2.5}
    Let $G$ be an abelian group. Let $T$ be the operator gen-semigroup defined in Example \ref{9.1.6}. If $X$ is a nonempty subset of $G$, then ${{\left\langle X \right\rangle }_{T}}$ is the subgroup of $G$ generated by $X$.
\end{example}
\begin{example} \label{9.2.6}
   Let $G$ be a group. Let $T$ be the operator gen-semigroup defined in Example \ref{9.1.7}. If $X$ is a nonempty subset of $G$, then ${{\left\langle X \right\rangle }_{T}}$ is the subgroup of $G$ generated by $X$.
\end{example}
\subsection{$T$-morphisms} \label{10 $T$-morphisms}
\begin{definition} \label{9.3.1}
    Let $T$ be an operator gen-semigroup and let $\sigma$ be a map from a $T$-space $S$ to a $T$-space. If $\forall n\in {{\mathbb{Z}}^{+}}$, $(a_1,\cdots ,a_n)\in {{S}^{n}}$ and $f\in T$ of $n$ variables, 
\[\sigma (f(a_1,\cdots ,a_n))=f(\sigma ({{a}_{1}}),\cdots ,\sigma ({{a}_{n}})),\] 
then we call $\sigma $ a \emph{$T$-morphism}.
\end{definition}
\begin{remark}
    From Definitions \ref{9.1.2} and \ref{9.2.1} (or from Proposition \ref{11.1.1}, which generalizes Proposition \ref{<S> contained in S}), we can tell $f(a_1,\cdots ,a_n)\in S$. Hence $\sigma (f(a_1,\cdots ,a_n))$ is well-defined.
\end{remark}
	Note that in Examples \ref{9.1.3} to \ref{9.1.7}, all operator gen-semigroups are induced from polynomials, rational functions, or words on a set. Therefore, we introduce an (informal) notion as follows, which will facilitate our study on $T$-morphisms and $\theta$-morphisms (defined later).
\begin{definition} \label{9.3.2}
    Let $f$ be an expression associated with a rule $R$. Let $n\in \mathbb{Z}^+$.
    
    If all of $x_1,\cdots,x_n$ arise as symbols in $f$, and there exist sets $C$ and $D$ such that, by the rule $R$, $f$ yields a single element, which we denote by $f(a_1,\cdots ,a_n)$, in $C$ whenever every $x_1,\cdots,x_n$ in $f$ is replaced by any  $a_1,\cdots,a_n\in D$, respectively, then we call the couple $(f,R)$ a (C-valued) \emph{formal function} of $n$ variables $x_1,\cdots,x_n$ and $D$ is called a \emph{domain} of $(f,R)$. 
    
    Moreover, we may call $f$ a \emph{formal function} if its associated rule $R$ is clear from the context.
\end{definition}
\begin{remark} \begin{enumerate}
    \item For example, $1+0{{x}_{1}}+0{{x}_{2}}\in \mathbb{Q}[{{x}_{1}},{{x}_{2}}]$ is a formal function of two variables with the associated rule being the addition and multiplication of field elements. In Examples \ref{9.1.3} to \ref{9.1.7}, all operator gen-semigroups are induced from formal functions.
\item	By the definition, a formal function $(f,R)$ may have more than one domain. For example, any field which contains $\mathbb{Q}$ is a domain of any formal function $f\in \mathbb{Q}[{{x}_{1}},\cdots ,{{x}_{n}}]$.
\item	Any function $f$ of $n$ variables $x_1,\cdots,x_n$ may be regarded as the formal function $f(x_1,\cdots,x_n)$, which is an expression with its associated rule being the mapping defined by the function $f$. The converse is not true, however, because a formal function may have more than one domain, and different domains of a formal function may induce different functions. Therefore, the notion of formal function generalizes the notion of function.
\end{enumerate} \end{remark}

When operator gen-semigroups are induced from formal functions, as is the case in Examples \ref{9.1.3} to \ref{9.1.7}, the following lemma will be useful. 
\begin{lemma} \label{9.3.3}
    Let $\mathcal{F}$ be a set of D-valued formal functions on a domain $D$ and let $T$ be an operator gen-semigroup on $D$ induced by $\mathcal{F}$; that is, there is a surjective map $\tau :\mathcal{F}\to T$ given by $f\mapsto f^*$, where $f^*:{{D}^{n}}\to D$ is given by $(a_1,\cdots ,a_n)\mapsto f(a_1,\cdots ,a_n)$ for $f$ of $n$ variables, $\forall n\in {{\mathbb{Z}}^{+}}$. Let $\sigma $ be a map from a $T$-space $S$ to a $T$-space. Then the following statements are equivalent:
\begin{enumerate}
\item [(i)]	$\sigma $ is a $T$-morphism.
\item [(ii)] $\forall n\in {{\mathbb{Z}}^{+}}$, $(a_1,\cdots ,a_n)\in {{S}^{n}}$ and $f\in \mathcal{F}$ of $n$ variables,
\begin{equation} \label{10.1}
    \sigma (f(a_1,\cdots ,a_n))=f(\sigma ({{a}_{1}}),\cdots ,\sigma ({{a}_{n}})).
\end{equation}
\end{enumerate} \end{lemma}
\begin{proof} 
(i)$\Rightarrow $(ii): By Definition \ref{9.3.1}, $\forall n\in {{\mathbb{Z}}^{+}}$, $(a_1,\cdots ,a_n)\in {{S}^{n}}$ and $f\in \mathcal{F}$ of $n$ variables, 
\[\sigma (\tau (f)(a_1,\cdots ,a_n))=\tau (f)(\sigma ({{a}_{1}}),\cdots ,\sigma ({{a}_{n}})).\]
	Then by the definition of $\tau $, (ii) must be true.

(ii)$\Rightarrow $(i): 	Since $\tau $ is surjective, $\forall f^*\in T$, $\exists f\in \mathcal{F}$ such that $f^*=\tau (f)$. Then from statement (ii) and the definition of $\tau $, we can tell that $\forall n\in {{\mathbb{Z}}^{+}}$, $(a_1,\cdots ,a_n)\in {{S}^{n}}$ and $f^*\in T$ of $n$ variables, $\sigma (f^*(a_1,\cdots ,a_n))=f^*(\sigma ({{a}_{1}}),\cdots ,\sigma ({{a}_{n}}))$.	Hence by Definition \ref{9.3.1}, $\sigma $ must be a $T$-morphism.										\end{proof}		
	Lemma \ref{9.3.3} characterizes a $T$-morphism by the formal functions which induce $T$. When an operator gen-semigroup is induced from formal functions, the criterion for a $T$-morphism by Lemma \ref{9.3.3} is normally more convenient than that by Definition \ref{9.3.1}. Thus Lemma \ref{9.3.3} simplifies our proofs of some of our results as follows.
	
 The following several results serve as examples of $T$-morphisms. 

Firstly, Proposition \ref{1.3.4} is generalized to
\begin{proposition} \label{9.3.4}
    Let $B$ be a subfield of a field $F$, let $T$ be the operator gen-semigroup on $F$ defined in Example \ref{9.1.3}, and let nonempty $U,V\subseteq F$. Then a map $\sigma :{{\left\langle U \right\rangle }_{T}}\to {{\left\langle V \right\rangle }_{T}}$ is a ring homomorphism with $B$ fixed pointwisely if and only if it is a $T$-morphism.
\end{proposition}   
\begin{proof}
    We may give a proof which is analogous to that of Proposition \ref{1.3.4}. However, now we employ Lemma \ref{9.3.3} instead.
By Example \ref{9.1.3},
\begin{center}
    $T=\{f^*:{{F}^{n}}\to F$ given by $(a_1,\cdots ,a_n)\mapsto f(a_1,\cdots ,a_n)\,|\,f\in B[{{x}_{1}},\cdots ,{{x}_{n}}],n\in {{\mathbb{Z}}^{+}}\}$.
\end{center}
    Thus ${{\left\langle U \right\rangle }_{T}}$ and ${{\left\langle V \right\rangle }_{T}}$ are the rings $B[U]$ and $B[V]$, respectively. 
    
    Let $\mathcal{F}=\{f\in B[{{x}_{1}},\cdots ,{{x}_{n}}]\,|\,n\in {{\mathbb{Z}}^{+}}\}$. Then Lemma \ref{9.3.3} applies. Hence we only need to show that $\sigma :{{\left\langle U \right\rangle }_{T}}\to {{\left\langle V \right\rangle }_{T}}$ is a ring homomorphism with $B$ fixed pointwisely if and only if statement (ii) in Lemma \ref{9.3.3} is true in this case. Specifically, it suffices to show that the following statements are equivalent:
    \begin{enumerate}
    \item $\forall c\in B$, $\sigma (c)=c$, and $\forall a,b\in {{\left\langle U \right\rangle }_{T}}(=B[U])$,
    $\sigma (a+b)=\sigma (a)+\sigma (b)$ and $\sigma (ab)=\sigma (a)\sigma (b)$.
    \item $\forall n\in {{\mathbb{Z}}^{+}}$, $(a_1,\cdots ,a_n)\in \left\langle U \right\rangle _{T}^{n}$ and $f\in B[{{x}_{1}},\cdots ,{{x}_{n}}]$,
    \begin{equation} \label{10.2}
    \sigma (f(a_1,\cdots ,a_n))=f(\sigma ({{a}_{1}}),\cdots ,\sigma ({{a}_{n}})).
    \end{equation}
    \end{enumerate}

(1)$\Rightarrow $(2): Since $f\in B[{{x}_{1}},\cdots ,{{x}_{n}}]$ is a polynomial, it is not hard to see that statement (2) is true.

(2)$\Rightarrow $(1): Considering the case where $f\in B[{{x}_{1}},\cdots ,{{x}_{n}}]$ is a constant polynomial, we can tell from statement (2) that $\forall c\in B$, $\sigma (c)=c$. 

Let $a,b\in {{\left\langle U \right\rangle }_{T}}$. Let $g={{x}_{1}}+{{x}_{2}}$ and $h={{x}_{1}}{{x}_{2}}$. Then

$\sigma (a+b)$

$=\sigma (g(a,b))$

$=g(\sigma (a),\sigma (b))$ (by Equation $($\ref{10.2}$)$)

$=\sigma (a)+\sigma (b)$. 

And 

$\sigma (ab)$

$=\sigma (h(a,b))$

$=h(\sigma (a),\sigma (b))$ (by Equation $($\ref{10.2}$)$)

$=\sigma (a)\sigma (b)$.											
\end{proof}
However, ring homomorphism in Proposition \ref{9.3.4} is between extensions over the same field. We will deal with ring homomorphisms between extensions over different fields in Subsection \ref{10 $theta$-morphisms}.
\begin{proposition} \label{9.3.5}
    Let $M$ be a module over a commutative ring $R$ with identity. Let $T$ be the operator gen-semigroup defined in Example \ref{9.1.5}. If $X$ and $Y$ are nonempty subsets of $M$, then a map $\sigma :{{\left\langle X \right\rangle }_{T}}\to {{\left\langle Y \right\rangle }_{T}}$ is an $R$-homomorphism of $R$-modules if and only if it is a $T$-morphism.
\end{proposition}   
\begin{proof}
    By Example \ref{9.1.5},
    \begin{center}
      $T=\{f^*:{{M}^{n}}\to M$ given by $(a_1,\cdots ,a_n)\mapsto f(a_1,\cdots ,a_n)\,|\,f\in \{\sum\nolimits_{i=1}^{n}{{{r}_{i}}{{x}_{i}}}\,|\,{{r}_{i}}\in R\},n\in {{\mathbb{Z}}^{+}}\}$.  
    \end{center}
So ${{\left\langle X \right\rangle }_{T}}$ and ${{\left\langle Y \right\rangle }_{T}}$ are the submodules of $M$ generated by $X$ and $Y$, respectively. 

Let $\mathcal{F}=\{f\in \{\sum\nolimits_{i=1}^{n}{{{r}_{i}}{{x}_{i}}}\,|\,{{r}_{i}}\in R\}\ |\ n\in {{\mathbb{Z}}^{+}}\}$. Then Lemma \ref{9.3.3} applies. Thus we only need to show that $\sigma :{{\left\langle X \right\rangle }_{T}}\to {{\left\langle Y \right\rangle }_{T}}$ is an $R$-homomorphism of $R$-modules if and only if statement (ii) in Lemma \ref{9.3.3} is true in this case. Specifically, it suffices to show that the following statements are equivalent:
\begin{enumerate}
    \item $\forall a,b\in {{\left\langle X \right\rangle }_{T}}$ and $r\in R$, $\sigma (a+b)=\sigma (a)+\sigma (b)$ and $\sigma (ra)=r\sigma (a)$.
    \item $\forall n\in {{\mathbb{Z}}^{+}}$, $(a_1,\cdots ,a_n)\in \left\langle X \right\rangle _{T}^{n}$ and $f\in \{\sum\nolimits_{i=1}^{n}{{{r}_{i}}{{x}_{i}}}\,|\,{{r}_{i}}\in R\}$,
    \begin{equation} \label{10.3}
        \sigma (f(a_1,\cdots ,a_n))=f(\sigma ({{a}_{1}}),\cdots ,\sigma ({{a}_{n}})).
    \end{equation}
\end{enumerate}

(1)$\Rightarrow $(2): Because $f\in \{\sum\nolimits_{i=1}^{n}{{{r}_{i}}{{x}_{i}}}\,|\,{{r}_{i}}\in R\}$ is a linear polynomial whose constant term is zero, it is not hard to see that statement (2) is true.

(2)$\Rightarrow $(1): Let $a,b\in {{\left\langle X \right\rangle }_{T}}$ and $r\in R$. Let $g={{x}_{1}}+{{x}_{2}}$ and $h=r{{x}_{1}}$. Then 

$\sigma (a+b)$

$=\sigma (g(a,b))$

$=g(\sigma (a),\sigma (b))$ (by Equation $($\ref{10.3}$)$)

$=\sigma (a)+\sigma (b)$. 

And 

$\sigma (ra)$

$=\sigma (h(a))$

$=h(\sigma (a))$ (by Equation $($\ref{10.3}$)$)

$=r\sigma (a)$.													
\end{proof}
	$R$-homomorphism in Proposition \ref{9.3.5} is between submodules of the same $R$-module. We will deal with general $R$-homomorphisms of $R$-modules in Subsection \ref{10 $theta$-morphisms}.
\begin{proposition} \label{9.3.6}
    Let $G$ be an abelian group and let $T$ be the operator gen-semigroup on $G$ defined in Example \ref{9.1.6}. If $X$ and $Y$ are nonempty subsets of $G$, then a map $\sigma :{{\left\langle X \right\rangle }_{T}}\to {{\left\langle Y \right\rangle }_{T}}$ is a group homomorphism if and only if it is a $T$-morphism.
\end{proposition}
\begin{proof}
    By Example \ref{9.1.6}, 
    \begin{center}
        $T=\bigcup\nolimits_{n\in {{\mathbb{Z}}^{+}}}{\{f^*:{{G}^{n}}\to G}$ given by $(a_1,\cdots ,a_n)\mapsto f(a_1,\cdots ,a_n)\,|\,f\in \{\prod\nolimits_{i=1}^{n}{x_{i}^{{{p}_{i}}}}\,|\,{{p}_{i}}\in \mathbb{Z}\}$\}.
    \end{center}

Obviously, ${{\left\langle X \right\rangle }_{T}}$ and ${{\left\langle Y \right\rangle }_{T}}$ are the subgroups of $G$ generated by $X$ and $Y$, respectively. 

We may apply Lemma \ref{9.3.3} as well in this proof. However, our argument as follows seems more straightforward.

Suppose that $\sigma :{{\left\langle X \right\rangle }_{T}}\to {{\left\langle Y \right\rangle }_{T}}$ is a $T$-morphism. Then by Definition \ref{9.3.1}, $\forall n\in {{\mathbb{Z}}^{+}}$, $(a_1,\cdots ,a_n)\in \left\langle X \right\rangle _{T}^{n}$ and $f\in T$ of $n$ variables,
\[\sigma (f(a_1,\cdots ,a_n))=f(\sigma ({{a}_{1}}),\cdots ,\sigma ({{a}_{n}})).\] 
In words, $\sigma $ commutes with the (multiplicative) operation equipped by $G$. Hence $\sigma $ is a group homomorphism.

Apparently the converse of the above argument is also true. That is, if $\sigma $ is a group homomorphism, then $\sigma $ commutes with the operation equipped by $G$, and hence it is a $T$-morphism.																
\end{proof}
	Group homomorphism in Proposition \ref{9.3.6} is between subgroups of the same abelian group. We will deal with group homomorphisms between different abelian groups in Subsection \ref{10 $theta$-morphisms}.

 For not necessarily abelian groups, we have
\begin{proposition} \label{9.3.7}
    Let $G$ be a group and let $T$ be the operator gen-semigroup on $G$ defined in Example \ref{9.1.7}. If $X$ and $Y$ are nonempty subsets of $G$, then a map $\sigma :{{\left\langle X \right\rangle }_{T}}\to {{\left\langle Y \right\rangle }_{T}}$ is a group homomorphism if and only if it is a $T$-morphism.
\end{proposition}
\begin{proof}
    The proof is comparable with that of Proposition \ref{9.3.6}.

By Example \ref{9.1.7}, 
\begin{center}
    $T=\bigcup\nolimits_{n\in {{\mathbb{Z}}^{+}}}{\{f^*:{{G}^{n}}\to G}$ given by $(a_1,\cdots ,a_n)\mapsto f(a_1,\cdots ,a_n)\,|\,f\in \left\langle {{x}_{1}},\cdots ,{{x}_{n}},x_{1}^{-1},\cdots ,x_{n}^{-1} \right\rangle \}$, 
\end{center}

Obviously, ${{\left\langle X \right\rangle }_{T}}$ and ${{\left\langle Y \right\rangle }_{T}}$ are the subgroups of $G$ generated by $X$ and $Y$, respectively. 

Suppose that $\sigma :{{\left\langle X \right\rangle }_{T}}\to {{\left\langle Y \right\rangle }_{T}}$ is a $T$-morphism. Then by Definition \ref{9.3.1}, $\forall n\in {{\mathbb{Z}}^{+}}$, $(a_1,\cdots ,a_n)\in \left\langle X \right\rangle _{T}^{n}$ and $f\in T$ of $n$ variables,
\[\sigma (f(a_1,\cdots ,a_n))=f(\sigma ({{a}_{1}}),\cdots ,\sigma ({{a}_{n}})).\] 
In words, $\sigma $ commutes with the multiplication equipped by $G$. Hence $\sigma $ is a group homomorphism.

It is not hard to see that the converse of the above argument is also true. Specifically, if $\sigma $ is a group homomorphism, then $\sigma $ commutes with the multiplication equipped by $G$, and hence it is a $T$-morphism.													
\end{proof}
	Group homomorphism in Proposition \ref{9.3.7} is between subgroups of the same group. We will deal with general group homomorphisms in Subsection \ref{10 $theta$-morphisms}.

\subsection{$T$-isomorphisms} \label{10 $T$-isomorphisms}
Definition \ref{$T$-isomorphisms} still applies: if a $T$-morphism is bijective, then we call it a \emph{$T$-isomorphism}. To justify the definition, we need to show Proposition \ref{1.3.a} with $T$ being an operator gen-semigroup as follows.
 \begin{proposition} \emph{(Proposition \ref{1.3.a})}
    Let $\sigma$ be a $T$-isomorphism from a $T$-space $S_1$ to a $T$-space $S_2$ and let ${{\sigma }^{-1}}$ be the inverse map of $\sigma $. Then ${{\sigma }^{-1}}\in \operatorname{Iso}_{T}({{S}_{2}},{{S}_{1}})$.
\end{proposition}
\begin{proof}
    By Definition \ref{9.3.1}, $\forall n\in {{\mathbb{Z}}^{+}}$, $f\in T$ of $n$ variables and $(a_1,\cdots ,a_n)\in {S_1^n}$,  \[f(a_1,\cdots ,a_n)={\sigma }^{-1}(f(\sigma ({{a}_{1}}),\cdots ,\sigma ({{a}_{n}}))).\] 
    So $\forall n\in {{\mathbb{Z}}^{+}}$, $f\in T$ of $n$ variables and $(b_1,\cdots ,b_n)\in {S_2^n}$,  
    \[f({\sigma }^{-1}({b}_{1}),\cdots ,{\sigma }^{-1}({b}_{n}))={\sigma }^{-1}(f(b_1,\cdots ,b_n)).\] 
    Then by Definition \ref{9.3.1}, ${\sigma }^{-1}$ is a $T$-morphism from $S_2$ to $S_1$, and hence ${{\sigma }^{-1}}\in \operatorname{Iso}_{T}({{S}_{2}},{{S}_{1}})$ by Definition \ref{$T$-isomorphisms}.
\end{proof}

\subsection{$\theta$-morphisms} \label{10 $theta$-morphisms}
To generalize Definition \ref{5.3.5}, we first generalize Notation \ref{5.3.4} to
\begin{notation} \label{9.4.1}
    Let $D_1$ (resp. $D_2$) be a set, let $M_1$ (resp. $M_2$) be a subset of \{all functions from the cartesian product $D_{1}^{n}$ to ${D_1}\,|\,n\in {{\mathbb{Z}}^{+}}$\} (resp. a subset of \{all functions from $D_{2}^{n}$ to ${{D}_{2}}\,|\,n\in {{\mathbb{Z}}^{+}}$\}), let $\theta \subseteq {{M}_{1}}\times {{M}_{2}}$, and let $A\subseteq {{D}_{2}}$. We denote by $\theta {{|}_{A}}$ the binary relation $\{(f,g{{|}_{A}})\,|\,(f,g)\in \theta \}$, where $g{{|}_{A}}$ denotes the function obtained by restricting every variable of $g$ to $A$.
\end{notation}
Definition \ref{5.3.a} is generalized to
\begin{definition} \label{9.4.a}
    Let $\theta$ and $A$ be defined as in Notation \ref{9.4.1}. Let $f\in \operatorname{Dom}\theta $, $n\in {{\mathbb{Z}}^{+}}$ and $(a_1, \cdots, a_n) \in A^n$. 
    If $\forall(f,g_1),(f,g_2)\in \theta$, $g_{1}(a_1, \cdots, a_n)=g_2(a_1, \cdots, a_n)$, then we say that $\theta (f)(a_1, \cdots, a_n)$ is \emph{well-defined} and let 
    \[\theta (f)(a_1, \cdots, a_n)=g_{1}(a_1, \cdots, a_n).\]
\end{definition}
\begin{remark}
    By Convention \ref{convention}, 
    \[\forall(f,g_1),(f,g_2)\in \theta, g_{1}(a_1, \cdots, a_n)=g_2(a_1, \cdots, a_n)\] 
    implies that $\forall(f,g)\in \theta$, $g$ has $n$ variables.
\end{remark}
Hence Proposition \ref{5.3.b} is generalized to
\begin{proposition} \label{9.4.b}
  Let $\theta$ and $A$ be defined as in Notation \ref{9.4.1}. Then the following statements are equivalent: \begin{enumerate}
        \item [(i)] $\theta {{|}_{A}}$ is a map.
         \item [(ii)] $\forall f\in \operatorname{Dom}\theta $, $\exists n\in {{\mathbb{Z}}^{+}}$ such that $\forall (a_1, \cdots, a_n) \in A^n$, $\theta (f)(a_1, \cdots, a_n)$ is well-defined.
    \end{enumerate}
\end{proposition}
\begin{proof}
    $\theta {{|}_{A}}$ is a map;
    
$ \Leftrightarrow $ $\forall(f,g_1),(f,g_2)\in \theta$, $g_{1}{{|}_{A}}=g_2{{|}_{A}}$;

$ \Leftrightarrow $ $\forall(f,g_1),(f,g_2)\in \theta$, $g_1$ and $g_2$ have the same number of variables, assumed $n$, and $\forall (a_1, \cdots, a_n) \in A^n$, $g_{1}(a_1, \cdots, a_n)=g_2(a_1, \cdots, a_n)$;

$ \Leftrightarrow $ $\forall f\in \operatorname{Dom}\theta $, $\exists n\in {{\mathbb{Z}}^{+}}$ such that $\forall (a_1, \cdots, a_n) \in A^n$, $\theta (f)(a_1, \cdots, a_n)$ is well-defined. 
\end{proof}
Then Definition \ref{5.3.5} is generalized to
\begin{definition} \label{9.4.2}
Let $T_1$ and $T_2$ be operator gen-semigroups and let $\theta \subseteq {{T}_{1}}\times {{T}_{2}}$. Let $\phi $ be a map from a $T_1$-space $S$ to a $T_2$-space. If 
 $\forall n\in {{\mathbb{Z}}^{+}}$, $(a_1,\cdots ,a_n)\in {{S}^{n}}$ and $f\in \operatorname{Dom}\theta $ of $n$ variables,
\[\phi (f(a_1,\cdots ,a_n))=\theta (f)(\phi ({{a}_{1}}),\cdots ,\phi ({{a}_{n}})),\]  
then $\phi $ is called a \emph{$\theta $-morphism}.
\end{definition}
Then the first half of Corollary \ref{5.3.c} is generalized to
\begin{corollary} \label{9.4.c}
Let $\phi $ be a $\theta$-morphism from a $T_1$-space $S$. Then $\forall (f,g)\in \theta $, $f$ and $g$ have the same number of variables, and $\theta {{|}_{\operatorname{Im}\phi }}$ is a map.
\end{corollary}
\begin{proof} 
    Suppose $n$-variable $f\in \operatorname{Dom}\theta $ and $(a_1,\cdots ,a_n)\in {{S}^{n}}$.  Then by Definition \ref{9.4.2} and Convention \ref{convention}, $\theta (f)(\phi ({{a}_{1}}),\cdots ,\phi ({{a}_{n}}))$ is well-defined. Hence by Definition \ref{9.4.a}, $\forall (f,g)\in \theta $, $g$ has $n$ variables.
    Moreover, by Proposition \ref{9.4.b}, $\theta {{|}_{\operatorname{Im}\phi }}$ is a map.
\end{proof}
When operator gen-semigroups are induced from formal functions, the following, which generalizes Lemma \ref{9.3.3}, will facilitate study of $\theta$-morphisms. 
\begin{lemma} \label{9.4.3}
    Let $\mathcal{F}$ $($resp. ${\mathcal{F}}')$ be a set of D-valued $($resp. ${D}'$-valued$)$ formal functions on $D$ $($resp. ${D}')$. Let $T$ $($resp. ${T}')$ be an operator gen-semigroup on $D$ $($resp. ${D}'$$)$ such that there is a map $\tau :\mathcal{F}\to T$ being given by $f\mapsto f^*$, where $f^*:{{D}^{n}}\to D$ is given by $(a_1,\cdots ,a_n)\mapsto f(a_1,\cdots ,a_n)$ for $f$ of $n$ variables, $\forall n\in {{\mathbb{Z}}^{+}}$ $($resp. such that there is a map ${\tau }':{\mathcal{F}}'\to {T}'$ being given analogously$)$. Let $\vartheta :\mathcal{F}\to {\mathcal{F}}'$ be a map. Let $\theta =\{(\tau (f),{\tau }'(\vartheta (f)))\,|\,f\in \mathcal{F}\}$ and let $\phi $ be a map from a $T$-space $S$ to a ${T}'$-space. 

Then the following statements are equivalent:
\begin{enumerate}
    \item [(i)] $\phi $ is a $\theta$-morphism.
    \item [(ii)] $\forall n\in {{\mathbb{Z}}^{+}}$, $(a_1,\cdots ,a_n)\in {{S}^{n}}$ and $f\in \mathcal{F}$ of n variables,
    \begin{equation} \label{10.4}
        \phi (f(a_1,\cdots ,a_n))=\vartheta (f)(\phi ({{a}_{1}}),\cdots ,\phi ({{a}_{n}})).
    \end{equation}
\end{enumerate}
\end{lemma}
\begin{proof}
Obviously, $\theta \subseteq T \times T'$.

(i) $\Rightarrow$ (ii):
Assume $n$-variable $f\in \mathcal{F}$ and $(a_1,\cdots ,a_n)\in {{S}^{n}}$. Then

$\phi (f(a_1,\cdots ,a_n))$

$=\phi (\tau (f)(a_1,\cdots ,a_n))$  

$=\theta (\tau (f))(\phi ({{a}_{1}}),\cdots ,\phi ({{a}_{n}}))$ (because $\tau (f)\in \operatorname{Dom}\theta $ and by Definition \ref{9.4.2})

$={\tau }'(\vartheta (f))(\phi ({{a}_{1}}),\cdots ,\phi ({{a}_{n}}))$ (because $(\tau (f),{\tau }'(\vartheta (f)))\in \theta $ and by Corollary \ref{9.4.c}, $\theta {{|}_{\operatorname{Im}\phi }}$ is a map)

$=\vartheta (f)(\phi ({{a}_{1}}),\cdots ,\phi ({{a}_{n}}))$, as desired.

(ii) $\Rightarrow$ (i):
To show that $\phi $ is a $\theta$-morphism, we first need to show that $\theta {{|}_{\operatorname{Im}\phi }}$ is a map. 

Let $(h,g{{|}_{\operatorname{Im}\phi }}),(h,{g}'{{|}_{\operatorname{Im}\phi }})\in \theta {{|}_{\operatorname{Im}\phi }}$ with $(h,g),(h,{g}')\in \theta $. By the definition of $\theta $, $\exists l,m\in \mathcal{F}$ such that $(h,g)=(\tau (l),{\tau }'(\vartheta (l)))$ and $(h,{g}')=(\tau (m),{\tau }'(\vartheta (m)))$, and hence $\tau (l)=\tau (m)$. Thus $l$ and $m$ have the same number of variables, which we assume to be $n$. Then by Equation (\ref{10.4}) and Convention \ref{convention}, both $\vartheta (l)$ and $\vartheta (m)$ have $n$ variables.

Assume $g{{|}_{\operatorname{Im}\phi }}\ne {g}'{{|}_{\operatorname{Im}\phi }}$; that is, ${\tau }'(\vartheta (l)){{|}_{\operatorname{Im}\phi }}\ne {\tau }'(\vartheta (m)){{|}_{\operatorname{Im}\phi }}$. Then\\
$\exists ({{z}_{1}},\cdots ,{{z}_{n}})\in {{S}^{n}}$ such that
\[\vartheta (l)(\phi ({{z}_{1}}),\cdots ,\phi ({{z}_{n}}))\ne \vartheta (m)(\phi ({{z}_{1}}),\cdots ,\phi ({{z}_{n}})).\] 
However, $\phi (l({{z}_{1}},\cdots ,{{z}_{n}}))=\phi (m({{z}_{1}},\cdots ,{{z}_{n}}))$ because $\tau (l)=\tau (m)$, and hence from Equation $($\ref{10.4}$)$, we can tell 
\[\vartheta (l)(\phi ({{z}_{1}}),\cdots ,\phi ({{z}_{n}}))=\vartheta (m)(\phi ({{z}_{1}}),\cdots ,\phi ({{z}_{n}})),\] a contradiction. 

Hence $\theta {{|}_{\operatorname{Im}\phi }}$ must be a map. Then $\forall n\in {{\mathbb{Z}}^{+}}$, $(a_1,\cdots ,a_n)\in {{S}^{n}}$ and $f\in \mathcal{F}$ of $n$ variables,

$\phi (\tau (f)(a_1,\cdots ,a_n))$ 

$=\phi (f(a_1,\cdots ,a_n))$ (by the definition of $\tau $)

$=\vartheta (f)(\phi ({{a}_{1}}),\cdots ,\phi ({{a}_{n}}))$ (by Equation (\ref{10.4}))

$={\tau }'(\vartheta (f))(\phi ({{a}_{1}}),\cdots ,\phi ({{a}_{n}}))$ (by the definition of ${\tau }'$)

$=\theta (\tau (f))(\phi ({{a}_{1}}),\cdots ,\phi ({{a}_{n}}))$ (because $(\tau (f),{\tau }'(\vartheta (f)))\in \theta $ and $\theta {{|}_{\operatorname{Im}\phi }}$ is a map). 

Since $\operatorname{Dom}\theta=\{\tau(f)\,|\,f\in \mathcal{F}\}$, by Definition \ref{9.4.2}, $\phi $ is a $\theta$-morphism.	
\end{proof}
	Lemma \ref{9.4.3} characterizes a $\theta$-morphism by formal functions. When operator gen-semigroups are induced from formal functions, the criterion for a $\theta$-morphism by Lemma \ref{9.4.3} is usually more convenient than that by Definition \ref{9.4.2}. Thus Lemma \ref{9.4.3} simplifies our proofs of the following several results.

The following generalizes both Propositions \ref{5.3.6} and \ref{9.3.4}.  
\begin{proposition} \label{9.4.4}
    Let $F/B$ $($resp. ${F}'/{B}'$$)$ be a field extension, let an operator gen-semigroup $T$ $($resp. ${T}'$$)$ be defined as in Example \ref{9.1.3} on $F$ $($resp. ${F}'$$)$ over $B$ $($resp. ${B}'$$)$, let $\varphi :B\to {B}'$ be a field isomorphism, and let 
    \[\vartheta :\bigcup\nolimits_{n\in {{\mathbb{Z}}^{+}}}{B[{{x}_{1}},\cdots ,{{x}_{n}}]}\to \bigcup\nolimits_{n\in {{\mathbb{Z}}^{+}}}{{B}'[{{x}_{1}},\cdots ,{{x}_{n}}]}\]
    be the map given by $f\mapsto {f}'$ where
\[f({{x}_{1}},\cdots ,{{x}_{n}})=\sum{{c_{p_1,\cdots ,p_n}}x_{1}^{{{p}_{1}}}\cdots x_{n}^{{{p}_{n}}}}\] and
\[{f}'({{x}_{1}},\cdots ,{{x}_{n}})=\sum{\varphi ({c_{p_1,\cdots ,p_n}})x_{1}^{{{p}_{1}}}\cdots x_{n}^{{{p}_{n}}}}.\]
Let $\theta =\{(\tau (f),{\tau }'(\vartheta (f)))\,|\,f\in B[{{x}_{1}},\cdots ,{{x}_{n}}],n\in {{\mathbb{Z}}^{+}}\}$, where $\tau $ and ${\tau }'$ are the maps inducing $T$ and ${T}'$, respectively, defined as in Example \ref{9.1.3}. 

Then $\forall (\emptyset \ne )U\subseteq F$ and $(\emptyset \ne )V\subseteq {F}'$, a map $\phi :{{\left\langle U \right\rangle }_{T}}\to {{\left\langle V \right\rangle }_{{{T}'}}}$ is a ring homomorphism extending $\varphi $ if and only if $\phi $ is a $\theta$-morphism.
\end{proposition}
\begin{proof}
    We may prove Proposition \ref{9.4.4} in a way analogous to the one for Proposition \ref{5.3.6}. However, now we employ Lemma \ref{9.4.3} instead. 

By Example \ref{9.1.3},
\begin{center}
  $T=\{f^*:{{F}^{n}}\to F$ given by $(a_1,\cdots ,a_n)\mapsto f(a_1,\cdots ,a_n)\,|\,f\in B[{{x}_{1}},\cdots ,{{x}_{n}}],n\in {{\mathbb{Z}}^{+}}\}$  
\end{center}
and
\begin{center}
   ${T}'=\{f^*:{{{F}'}^{n}}\to {F}'$ given by $(a_1,\cdots ,a_n)\mapsto f(a_1,\cdots ,a_n)\,|\,f\in {B}'[{{x}_{1}},\cdots ,{{x}_{n}}],n\in {{\mathbb{Z}}^{+}}\}$. 
\end{center}

Thus ${{\left\langle U \right\rangle }_{T}}$ and ${{\left\langle V \right\rangle }_{{{T}'}}}$ are the rings $B[U]$ and ${B}'[V]$, respectively. 

	Let $\mathcal{F}=\bigcup\nolimits_{n\in {{\mathbb{Z}}^{+}}}{B[{{x}_{1}},\cdots ,{{x}_{n}}]}$ and ${\mathcal{F}}'=\bigcup\nolimits_{n\in {{\mathbb{Z}}^{+}}}{{B}'[{{x}_{1}},\cdots ,{{x}_{n}}]}$. Then Lemma \ref{9.4.3} applies. Hence it suffices to show that $\phi :{{\left\langle U \right\rangle }_{T}}\to {{\left\langle V \right\rangle }_{{{T}'}}}$ is a ring homomorphism extending $\varphi $ if and only if statement (ii) in Lemma \ref{9.4.3} is true in this case. Specifically, we only need to show that the following statements are equivalent:
 \begin{enumerate}
     \item $\forall c\in B$, $\phi (c)=\varphi (c)$, and $\forall a,b\in {{\left\langle U \right\rangle }_{T}}$, $\phi (a+b)=\phi (a)+\phi (b)$ and $\phi (ab)=\phi (a)\phi (b)$.
     \item $\forall n\in {{\mathbb{Z}}^{+}}$, $(a_1,\cdots ,a_n)\in \left\langle U \right\rangle _{T}^{n}$ and $f\in B[{{x}_{1}},\cdots ,{{x}_{n}}]$,
     \begin{equation} \label{10.5}
         \phi (f(a_1,\cdots ,a_n))=\vartheta (f)(\phi ({{a}_{1}}),\cdots ,\phi ({{a}_{n}})).
     \end{equation}
 \end{enumerate}
(1) $\Rightarrow$ (2): $\forall n\in {{\mathbb{Z}}^{+}}$, $f({{x}_{1}},\cdots ,{{x}_{n}})=\sum{{c_{p_1,\cdots ,p_n}}x_{1}^{{{p}_{1}}}\cdots x_{n}^{{{p}_{n}}}}\in B[{{x}_{1}},\cdots ,{{x}_{n}}]$ and $(a_1,\cdots ,a_n)\in \left\langle U \right\rangle _{T}^{n}$,

$\phi (f(a_1,\cdots ,a_n))$ 

$=\phi (\sum{{c_{p_1,\cdots ,p_n}}a_{1}^{{{p}_{1}}}\cdots a_{n}^{{{p}_{n}}}})$ 

$=\sum{\phi ({c_{p_1,\cdots ,p_n}})(\phi (a_1))^{p_1}\cdots (\phi (a_n))^{p_n}}$ (because statement (1) is true)

$=\sum{\varphi ({c_{p_1,\cdots ,p_n}})(\phi (a_1))^{p_1}\cdots (\phi (a_n))^{p_n}}$ (because $\forall c\in B$, $\phi (c)=\varphi (c)$)

$=\vartheta (f)(\phi ({{a}_{1}}),\cdots ,\phi ({{a}_{n}}))$ (by the definition of $\vartheta $).

(2) $\Rightarrow$ (1): Considering the case where $f\in B[{{x}_{1}},\cdots ,{{x}_{n}}]$ is a constant polynomial, we can tell from Equation $($\ref{10.5}$)$ and the definition of $\vartheta$ that $\forall c\in B$, $\phi (c)=\varphi (c)$.

Let $a,b\in {{\left\langle U \right\rangle }_{T}}$. Let $g={{x}_{1}}+{{x}_{2}}$ and let $h={{x}_{1}}{{x}_{2}}$. Then

$\phi (a+b)$

$=\phi (g(a,b))$

$=\vartheta (g)(\phi (a),\phi (b))$ (by Equation $($\ref{10.5}$)$)

$=g(\phi (a),\phi (b))$ ($\vartheta (g)=g$ by the definition of $\vartheta $)

$=\phi (a)+\phi (b)$. 

And 

$\phi (ab)$

$=\phi (h(a,b))$

$=\vartheta (h)(\phi (a),\phi (b))$ (by Equation $($\ref{10.5}$)$)

$=h(\phi (a),\phi (b))$ ($\vartheta (h)=h$ by the definition of $\vartheta $)

$=\phi (a)\phi (b)$. 	
\end{proof}
The following is for ring homomorphisms between not necessarily commutative rings.
\begin{proposition} \label{9.4.5}
    Let $A$ and $R$ be rings with identity, let ${{T}_{A}}$ $($resp. ${{T}_{R}}$$)$ be the operator gen-semigroup on $A$ $($resp. $R$$)$ defined as in Example \ref{9.1.4}, and let $\theta =\{({{\tau }_{A}}(f),{{\tau }_{R}}(f))\,|\,f\in {{\mathbb{F}}_{2}}\left\langle {{x}_{1}},\cdots ,{{x}_{n}} \right\rangle ,n\in {{\mathbb{Z}}^{+}}\}$, where ${{\tau }_{A}}$ and ${{\tau }_{R}}$ are the maps inducing ${{T}_{A}}$ and ${{T}_{R}}$, respectively, defined as in Example \ref{9.1.4}. Then a map $\phi :A\to R$ is a ring homomorphism if and only if $\phi $ is a $\theta$-morphism.
\end{proposition}
\begin{proof}
 By Example \ref{9.1.4},
 \begin{center}
     ${{T}_{A}}=\{f^*:{{A}^{n}}\to A$ given by $(a_1,\cdots ,a_n)\mapsto f(a_1,\cdots ,a_n)\,|\,f\in {{\mathbb{F}}_{2}}\left\langle {{x}_{1}},\cdots ,{{x}_{n}} \right\rangle ,n\in {{\mathbb{Z}}^{+}}\}$,
 \end{center}
and
\begin{center}
    ${{T}_{R}}=\{f^*:{{R}^{n}}\to R$ given by $(a_1,\cdots ,a_n)\mapsto f(a_1,\cdots ,a_n)\,|\,f\in {{\mathbb{F}}_{2}}\left\langle {{x}_{1}},\cdots ,{{x}_{n}} \right\rangle ,n\in {{\mathbb{Z}}^{+}}\}$. 
\end{center}

Obviously ${{\left\langle A \right\rangle }_{{{T}_{A}}}}=A$ and ${{\left\langle R \right\rangle }_{{{T}_{R}}}}=R$. Hence $\phi :A\to R$ is $\phi :{{\left\langle A \right\rangle }_{{{T}_{A}}}}\to {{\left\langle R \right\rangle }_{{{T}_{R}}}}$. 

	Let $\mathcal{F}={\mathcal{F}}'=\bigcup\nolimits_{n\in {{\mathbb{Z}}^{+}}}{{{\mathbb{F}}_{2}}\left\langle {{x}_{1}},\cdots ,{{x}_{n}} \right\rangle }$ and let $\vartheta =\operatorname{Id}$. Then Lemma \ref{9.4.3} applies. By Lemma \ref{9.4.3}, we only need to show that $\phi :{{\left\langle A \right\rangle }_{{{T}_{A}}}}\to {{\left\langle R \right\rangle }_{{{T}_{R}}}}$ is a ring homomorphism if and only if statement (ii) in Lemma \ref{9.4.3} is true in this case. Specifically, it suffices to show that the following statements are equivalent:
\begin{enumerate}
    \item $\phi ({{1}_{A}})={{1}_{R}}$ and $\forall a,b\in A$, $\phi (a+b)=\phi (a)+\phi (b)$ and $\phi (ab)=\phi (a)\phi (b)$.
 \item 	$\forall n\in {{\mathbb{Z}}^{+}}$, $(a_1,\cdots ,a_n)\in {{A}^{n}}$ and $f\in {{\mathbb{F}}_{2}}\left\langle {{x}_{1}},\cdots ,{{x}_{n}} \right\rangle $,
 \begin{equation} \label{10.6}
     \phi (f(a_1,\cdots ,a_n))=f(\phi ({{a}_{1}}),\cdots ,\phi ({{a}_{n}})).
 \end{equation}
\end{enumerate}

(1) $\Rightarrow$ (2):  Since $f\in {{\mathbb{F}}_{2}}\left\langle {{x}_{1}},\cdots ,{{x}_{n}} \right\rangle $ is a polynomial, it is not hard to see that statement (2) is true.

(2) $\Rightarrow$ (1):  Considering the case where $f\in {{\mathbb{F}}_{2}}\left\langle {{x}_{1}},\cdots ,{{x}_{n}} \right\rangle $ is the constant polynomial 1, we can tell from Equation $($\ref{10.6}$)$ that $\phi ({{1}_{A}})={{1}_{R}}$. 

Let $a,b\in A$. Let $g={{x}_{1}}+{{x}_{2}}$ and $h={{x}_{1}}{{x}_{2}}$. Then

$\phi (a+b)$

$=\phi (g(a,b))$

$=g(\phi (a),\phi (b))$ (by Equation $($\ref{10.6}$)$)

$=\phi (a)+\phi (b)$. 

And 

$\phi (ab)$

$=\phi (h(a,b))$

$=h(\phi (a),\phi (b))$ (by Equation $($\ref{10.6}$)$)

$=\phi (a)\phi (b)$. 
\end{proof}
Proposition \ref{9.3.5} is generalized to
\begin{proposition} \label{9.4.6}
    Let $M$ and $N$ be modules over a commutative ring $R$ with identity, let ${{T}_{M}}$ $($resp. ${{T}_{N}}$$)$ be the operator gen-semigroup on $M$ $($resp. $N$$)$ defined as in Example \ref{9.1.5}, and let 
    \[\theta =\{({{\tau }_{M}}(f),{{\tau }_{N}}(f))\,|\,f\in \{\sum\nolimits_{i=1}^{n}{{{r}_{i}}{{x}_{i}}}\,|\,{{r}_{i}}\in R\},n\in {{\mathbb{Z}}^{+}}\},\] 
    where ${{\tau }_{M}}$ and ${{\tau }_{N}}$ are the maps inducing ${{T}_{M}}$ and ${{T}_{N}}$, respectively, defined as in Example \ref{9.1.5}. 
    
    If $X$ and $Y$ are nonempty subsets of $M$ and $N$, respectively, then a map $\phi :{{\left\langle X \right\rangle }_{{{T}_{M}}}}\to {{\left\langle Y \right\rangle }_{{{T}_{N}}}}$ is an R-homomorphism of R-modules if and only if $\phi $ is a $\theta$-morphism.
\end{proposition}
\begin{proof}
    Our proof is comparable with that of Proposition \ref{9.3.5}.
    
	By Example \ref{9.1.5}, 
 \begin{center}
     ${{T}_{M}}=\{f^*:{{M}^{n}}\to M$ given by $(a_1,\cdots ,a_n)\mapsto f(a_1,\cdots ,a_n)\,|\,f\in \{\sum\nolimits_{i=1}^{n}{{{r}_{i}}{{x}_{i}}}\,|\,{{r}_{i}}\in R\},n\in {{\mathbb{Z}}^{+}}\}$
 \end{center}
and
\begin{center}
    ${{T}_{N}}=\{f^*:{{N}^{n}}\to N$ given by $(a_1,\cdots ,a_n)\mapsto f(a_1,\cdots ,a_n)\,|\,f\in \{\sum\nolimits_{i=1}^{n}{{{r}_{i}}{{x}_{i}}}\,|\,{{r}_{i}}\in R\},n\in {{\mathbb{Z}}^{+}}\}$.
\end{center}

Clearly, ${{\left\langle X \right\rangle }_{{{T}_{M}}}}$ and ${{\left\langle Y \right\rangle }_{{{T}_{N}}}}$ are the submodules of $M$ and $N$ generated by $X$ and $Y$, respectively. 

	Let $\mathcal{F}={\mathcal{F}}'=\bigcup\nolimits_{n\in {{\mathbb{Z}}^{+}}}{\{\sum\nolimits_{i=1}^{n}{{{r}_{i}}{{x}_{i}}}\,|\,{{r}_{i}}\in R\}}$ and let $\vartheta =\operatorname{Id}$. Then Lemma \ref{9.4.3} applies. By Lemma \ref{9.4.3}, it suffices to show that $\phi :{{\left\langle X \right\rangle }_{{{T}_{M}}}}\to {{\left\langle Y \right\rangle }_{{{T}_{N}}}}$ is an $R$-homomorphism of $R$-modules if and only if statement (ii) in Lemma \ref{9.4.3} is true in this case. Specifically, we only need to show that the following statements are equivalent:
 \begin{enumerate}
     \item $\forall a,b\in {{\left\langle X \right\rangle }_{{{T}_{M}}}}$ and $r\in R$, $\phi (a+b)=\phi (a)+\phi (b)$ and $\phi (ra)=r\phi (a)$.
     \item $\forall n\in {{\mathbb{Z}}^{+}}$, $(a_1,\cdots ,a_n)\in \left\langle X \right\rangle _{{{T}_{M}}}^{n}$ and $f\in \{\sum\nolimits_{i=1}^{n}{{{r}_{i}}{{x}_{i}}}\,|\,{{r}_{i}}\in R\}$,
     \begin{equation} \label{10.7}
         \phi (f(a_1,\cdots ,a_n))=f(\phi ({{a}_{1}}),\cdots ,\phi ({{a}_{n}})).
     \end{equation}
 \end{enumerate}

(1) $\Rightarrow$ (2):  Because $f\in \{\sum\nolimits_{i=1}^{n}{{{r}_{i}}{{x}_{i}}}\,|\,{{r}_{i}}\in R\}$ is a linear polynomial whose constant term is zero, it is not hard to see that statement (2) is true.

(2) $\Rightarrow$ (1):  Let $a,b\in {{\left\langle X \right\rangle }_{{{T}_{M}}}}$ and let $r\in R$. Let $g={{x}_{1}}+{{x}_{2}}$ and let $h=r{{x}_{1}}$. Then

$\phi (a+b)$

$=\phi (g(a,b))$

$=g(\phi (a),\phi (b))$ (by Equation $($\ref{10.7}$)$)

$=\phi (a)+\phi (b)$. 

And 

$\phi (ra)$

$=\phi (h(a))$

$=h(\phi (a))$ (by Equation $($\ref{10.7}$)$)

$=r\phi (a)$.
\end{proof}
Proposition \ref{9.3.6} is generalized as follows.
\begin{proposition} \label{9.4.7}
    Let $G$ and $H$ be abelian groups, let ${{T}_{G}}$ $($resp. ${{T}_{H}}$$)$ be the operator gen-semigroup on $G$ $($resp. $H$$)$ defined as in Example \ref{9.1.6}, and let $\theta =\{({{\tau }_{G}}(f),{{\tau }_{H}}(f))\,|\,f\in \{\prod\nolimits_{i=1}^{n}{x_{i}^{{{p}_{i}}}}\,|\,{{p}_{i}}\in \mathbb{Z}\},n\in {{\mathbb{Z}}^{+}}\}$ , where ${{\tau }_{G}}$ and ${{\tau }_{H}}$ are the maps inducing ${{T}_{G}}$ and ${{T}_{H}}$, respectively, defined as in Example \ref{9.1.6}. 
    
    If $X$ and $Y$ are nonempty subsets of $G$ and $H$ respectively, then a map $\phi :{{\left\langle X \right\rangle }_{{{T}_{G}}}}\to {{\left\langle Y \right\rangle }_{{{T}_{H}}}}$ is a group homomorphism if and only if $\phi $ is a $\theta$-morphism.
\end{proposition}  
\begin{proof}
    By Example \ref{9.1.6}, 
    \begin{center}
        ${{T}_{G}}=\bigcup\nolimits_{n\in {{\mathbb{Z}}^{+}}}{\{f^*:{{G}^{n}}\to G}$ given by $(a_1,\cdots ,a_n)\mapsto f(a_1,\cdots ,a_n)\,|\,f\in \{\prod\nolimits_{i=1}^{n}{x_{i}^{{{p}_{i}}}}\,|\,{{p}_{i}}\in \mathbb{Z}\}$\}
    \end{center}
and
\begin{center}
    ${{T}_{H}}=\bigcup\nolimits_{n\in {{\mathbb{Z}}^{+}}}{\{f^*:{{H}^{n}}\to H}$ given by $(a_1,\cdots ,a_n)\mapsto f(a_1,\cdots ,a_n)\,|\,f\in \{\prod\nolimits_{i=1}^{n}{x_{i}^{{{p}_{i}}}}\,|\,{{p}_{i}}\in \mathbb{Z}\}$\}.
\end{center}

Besides, ${{\left\langle X \right\rangle }_{{{T}_{G}}}}$ and ${{\left\langle Y \right\rangle }_{{{T}_{H}}}}$ are the subgroups of $G$ and $H$ generated by $X$ and $Y$, respectively. 

	Let $\mathcal{F}={\mathcal{F}}'=\bigcup\nolimits_{n\in {{\mathbb{Z}}^{+}}}{\{\prod\nolimits_{i=1}^{n}{x_{i}^{{{p}_{i}}}}\,|\,{{p}_{i}}\in \mathbb{Z}\}}$ and let $\vartheta =\operatorname{Id}$. Then Lemma \ref{9.4.3} applies. By Lemma \ref{9.4.3}, it suffices to show that $\phi :{{\left\langle X \right\rangle }_{{{T}_{G}}}}\to {{\left\langle Y \right\rangle }_{{{T}_{H}}}}$ is a group homomorphism if and only if statement (ii) in Lemma \ref{9.4.3} is true in this case. Specifically, we only need to show that the following statements are equivalent:
 \begin{enumerate}
     \item $\forall a,b\in {{\left\langle X \right\rangle }_{{{T}_{G}}}}$, $\phi (ab)=\phi (a)\phi (b)$.
     \item $\forall n\in {{\mathbb{Z}}^{+}}$, $(a_1,\cdots ,a_n)\in \left\langle X \right\rangle _{{{T}_{G}}}^{n}$ and $f\in \{\prod\nolimits_{i=1}^{n}{x_{i}^{{{p}_{i}}}}\,|\,{{p}_{i}}\in \mathbb{Z}\}$,
\[\phi (f(a_1,\cdots ,a_n))=f(\phi ({{a}_{1}}),\cdots ,\phi ({{a}_{n}})).\]
 \end{enumerate}

It is not hard to see (1)$\Leftrightarrow $(2).	
\end{proof}
	Proposition \ref{9.3.7} is generalized to
 \begin{proposition} \label{9.4.8}
     Let $G$ and $H$ be groups, let ${{T}_{G}}$ $($resp. ${{T}_{H}}$$)$ be the operator gen-semigroup on $G$ $($resp. $H$$)$ defined as in Example \ref{9.1.7}, and let $\theta =\{({{\tau }_{G}}(f),{{\tau }_{H}}(f))\,|\,f\in \left\langle {{x}_{1}},\cdots ,{{x}_{n}},x_{1}^{-1},\cdots ,x_{n}^{-1} \right\rangle ,n\in {{\mathbb{Z}}^{+}}\}$, where ${{\tau }_{G}}$ and ${{\tau }_{H}}$ are the maps inducing ${{T}_{G}}$ and ${{T}_{H}}$, respectively, defined as in Example \ref{9.1.7}. 
     
     If $X$ and $Y$ are nonempty subsets of $G$ and $H$, respectively, then a map $\phi :{{\left\langle X \right\rangle }_{{{T}_{G}}}}\to {{\left\langle Y \right\rangle }_{{{T}_{H}}}}$ is a group homomorphism if and only if $\phi $ is a $\theta $-morphism.
 \end{proposition}
\begin{proof}
    Our proof is comparable with that of Proposition \ref{9.4.7}.
    
	By Example \ref{9.1.7},
 \begin{center}
    $T=\bigcup\nolimits_{n\in {{\mathbb{Z}}^{+}}}{\{f^*:{{G}^{n}}\to G}$ given by $(a_1,\cdots ,a_n)\mapsto f(a_1,\cdots ,a_n)\,|\,f\in \left\langle {{x}_{1}},\cdots ,{{x}_{n}},x_{1}^{-1},\cdots ,x_{n}^{-1} \right\rangle \}$. 
\end{center}
 and
 \begin{center}
    $T=\bigcup\nolimits_{n\in {{\mathbb{Z}}^{+}}}{\{f^*:{{H}^{n}}\to H}$ given by $(a_1,\cdots ,a_n)\mapsto f(a_1,\cdots ,a_n)\,|\,f\in \left\langle {{x}_{1}},\cdots ,{{x}_{n}},x_{1}^{-1},\cdots ,x_{n}^{-1} \right\rangle \}$. 
\end{center}

Clearly, ${{\left\langle X \right\rangle }_{{{T}_{G}}}}$ and ${{\left\langle Y \right\rangle }_{{{T}_{H}}}}$ are the subgroups of $G$ and $H$ generated by $X$ and $Y$, respectively. 

	Let $\mathcal{F}={\mathcal{F}}'=\bigcup\nolimits_{n\in {{\mathbb{Z}}^{+}}}{\left\langle {{x}_{1}},\cdots ,{{x}_{n}},x_{1}^{-1},\cdots ,x_{n}^{-1} \right\rangle }$ and let $\vartheta =\operatorname{Id}$. Then Lemma \ref{9.4.3} applies. By Lemma \ref{9.4.3}, we only need to show that $\phi :{{\left\langle X \right\rangle }_{{{T}_{G}}}}\to {{\left\langle Y \right\rangle }_{{{T}_{H}}}}$ is a group homomorphism if and only if statement (ii) in Lemma \ref{9.4.3} is true in this case. Specifically, it suffices to show that the following statements are equivalent:
 \begin{enumerate}
     \item $\forall a,b\in {{\left\langle X \right\rangle }_{{{T}_{G}}}}$, $\phi (ab)=\phi (a)\phi (b)$.
     \item $\forall n\in {{\mathbb{Z}}^{+}}$, $(a_1,\cdots ,a_n)\in \left\langle X \right\rangle _{{{T}_{G}}}^{n}$ and $f\in \left\langle {{x}_{1}},\cdots ,{{x}_{n}},x_{1}^{-1},\cdots ,x_{n}^{-1} \right\rangle $,
\[\phi (f(a_1,\cdots ,a_n))=f(\phi ({{a}_{1}}),\cdots ,\phi ({{a}_{n}})).\]
 \end{enumerate}
	
It is not hard to see (1)$\Leftrightarrow $(2).
\end{proof}

\subsection{$\theta$-isomorphisms} \label{10 $theta$-isomorphisms}
For $T_1$ and $T_2$ being operator gen-semigroups, Definition \ref{5.4.a} still applies: a $\theta$-morphism from a $T_1$-space $S_1$ to a $T_2$-space is a \emph{$\theta$-isomorphism} if it is bijective. To justify the definition, we need to show Lemma \ref{L6.2.2} and Proposition \ref{5.4.b} for the case of operator gen-semigroups as follows.
\begin{lemma} \label{L10.6.1} \emph{(Lemma \ref{L6.2.2})}
    Let $\phi $ be a $\theta $-isomorphism from a $T_1$-space $S_1$ to a $T_2$-space $S_2$. Then $\theta^{-1} {{|}_{S_1}}$ is a map, where $\theta^{-1}:=\{(g,f)\,|\,(f,g) \in \theta\}$.
\end{lemma}
\begin{proof}
    Let $(g,f_1),(g,f_2)\in \theta^{-1}$ with $g$ of $n$ variables. Then by Corollary \ref{9.4.c}, both $f_1$ and $f_2$ have $n$ variables.
    Let $(a_1, \cdots, a_n) \in S_1^n$. 
    Then to show that $\theta^{-1} {{|}_{S_1}}$ is a map, by Proposition \ref{9.4.b} and Definition \ref{9.4.a}, it suffices to show that
    $f_1(a_1, \cdots, a_n)=f_2(a_1, \cdots, a_n)$. Moreover, since $\phi $ is injective, we only need to show $\phi(f_1(a_1, \cdots, a_n))=\phi(f_2(a_1, \cdots, a_n))$: \\    
     $\phi(f_1(a_1, \cdots, a_n))$\\
     $=\theta(f_1)(\phi (a_1), \cdots, \phi(a_n))$  (by Definition  \ref{9.4.2})\\
    $=g(\phi (a_1), \cdots, \phi(a_n))$  (because $(f_1,g)\in \theta$  and $\theta {|}_{\operatorname{Im}\phi}$  is a map by Corollary \ref{9.4.c})\\ 
    $=\theta(f_2)(\phi (a_1), \cdots, \phi(a_n))$  (because $(f_2,g)\in \theta$  and $\theta {|}_{\operatorname{Im}\phi}$  is a map)\\
    $=\phi(f_2(a_1, \cdots, a_n))$  (by Definition  \ref{9.4.2}).
\end{proof}
\begin{proposition} \emph{(Proposition \ref{5.4.b})}
    Let $S_1$ and $S_2$ be a $T_1$-space and a $T_2$-space, respectively, let $\phi \in \operatorname{Iso}_{\theta}({{S}_{1}},{{S}_{2}})$ and let ${{\phi }^{-1}}$ be the inverse map of $\phi $. Then ${{\phi }^{-1}}\in \operatorname{Iso}_{\theta^{-1}}({{S}_{2}},{{S}_{1}})$.
\end{proposition}
\begin{proof}
       By Definition \ref{9.4.2}, $\forall n\in {{\mathbb{Z}}^{+}}$, $(a_1,\cdots ,a_n)\in {S_1^n}$ and $f\in \operatorname{Dom}\theta $ of $n$ variables,
 \[f(a_1,\cdots ,a_n)=\phi^{-1} (\theta (f)(\phi ({{a}_{1}}),\cdots ,\phi ({{a}_{n}}))).\] 
Hence $\forall n\in {{\mathbb{Z}}^{+}}, f\in \operatorname{Dom}\theta $ of $n$ variables and $ (b_1,\cdots ,b_n)\in {S_2^n}$,
\begin{equation} \label{9.a}
    f({\phi }^{-1}({b}_{1}),\cdots ,{\phi }^{-1}({b}_{n}))=\phi^{-1} (\theta (f)(b_1,\cdots , b_n)).
\end{equation}

 By Lemma \ref{L10.6.1}, $\theta^{-1} {{|}_{S_1}}$ is a map. Then by Proposition \ref{9.4.b}, $\forall g\in \operatorname{Dom}\theta^{-1} (=\operatorname{Im}\theta)$, $\exists n\in {{\mathbb{Z}}^{+}}$ such that $\forall (a_1, \cdots, a_n) \in S_1^n$, $\theta^{-1} (g)(a_1, \cdots, a_n)$ is well-defined, and from (\ref{9.a}), we can tell that the number of variables of $g$ is also $n$.
 
 Thus $\forall n\in {{\mathbb{Z}}^{+}}$, $g\in \operatorname{Dom}\theta^{-1}$ of $n$ variables and $ (b_1,\cdots ,b_n)\in {S_2^n}$, it follows from (\ref{9.a}) that 
   \[\theta^{-1}(g)({\phi }^{-1}({b}_{1}),\cdots ,{\phi }^{-1}({b}_{n}))={{\phi }^{-1}}(g(b_1,\cdots , b_n)).\]

Hence by Definition \ref{9.4.2}, $\phi^{-1} $ is a $\theta^{-1} $-morphism from $S_2$ to $S_1$.

Moreover, since $\phi^{-1} $ is bijective, by Definition \ref{5.4.a}, ${{\phi }^{-1}}\in \operatorname{Iso}_{\theta^{-1}}({{S}_{2}},{{S}_{1}})$.
\end{proof}

For operator gen-semigroups, we will generalize most results obtained in Sections \ref{Basic notions and properties} to \ref{Cons of $T$-mor and theta-mor}. Before doing this in Sections \ref{12 Basic properties} to \ref{II Cons of $T$-mor and}, however, we shall generalize the notion of operator gen-semigroup further in the next section.

\section{Basic notions for the case of (multivariable) partial functions} \label{Basic notions for partial}
In this section, we shall allow elements in $T$ to be partial functions. Then we will find that some more familiar concepts such as ring homomorphisms between fields and covariant functors between categories can be characterized in terms of $T$-morphisms or $\theta$-morphisms.

Subsection \ref{par-operator gen-semigroups} generalizes the notion of operator gen-semigroup. Then correspondingly, the notions of $T$-spaces, $T$-morphisms, $T$-isomorphisms, $\theta$-morphisms and $\theta$-isomorphisms are generalized in Sections \ref{11 $T$-spaces}, \ref{11 $T$-morphisms}, \ref{11 $T$-isomorphisms}, \ref{11 $theta$-morphisms} and \ref{11 $theta$-isomorphisms}, respectively.
\subsection{Partial-operator generalized-semigroup $T$} \label{par-operator gen-semigroups}

Recall the concept of partial functions as follows.
\begin{definition} \label{10.0}
    Let $X$ and $Y$ be sets. A \emph{partial function} $f$ of $n(\in {{\mathbb{Z}}^{+}})$ variables from the cartesian product ${X}^{n}$ to $Y$ is a (total) function from a subset $S$ of ${X}^{n}$ to $Y$, and $S$, which may be empty, is called the \emph{domain of definition} of $f$. 
    
    Moreover, let $(x_1, \cdots, x_m) \in X^m$, where $m\in {{\mathbb{Z}}^{+}}$. If $(x_1, \cdots, x_m)$ lies in the domain of definition of $f$, then $f(x_1, \cdots, x_m)$ is said to be \emph{well-defined}. 
\end{definition}
\begin{remark}
    By the definition, any function is a partial function.
\end{remark}
We are going to extend our study to partial functions. First, we generalize Definition \ref{9.1.1} to
\begin{definition} \label{10.1.1}
    Let $D$ be a set, let $n\in {{\mathbb{Z}}^{+}}$, and let $f$ be a partial function from the cartesian product ${{D}^{n}}$ to $D$. $\forall i=1,\cdots ,n$, let $n_i\in {{\mathbb{Z}}^{+}}$ and let ${{g}_{i}}$ be a partial function from the cartesian product ${{D}^{{{n}_{i}}}}$ to $D$. Let $m = \sum\nolimits_i {{n_i}} $. Then the \emph{composite} $f\circ ({{g}_{1}},\cdots ,{{g}_{n}})$ is the partial function from ${{D}^{m}}$ to $D$ given by 
    \[({{x}_{1}},\cdots ,{{x}_{m}})\mapsto f({{g}_{1}}({{\text{v}}_{1}}),\cdots ,{{g}_{n}}({{\text{v}}_{n}})),\] 
    where ${{\text{v}}_{i}}:=({{x}_{{{m}_{i}}+1}},\cdots , {{x}_{{{m}_{i}}+{{n}_{i}}}})\in {{D}^{{{n}_{i}}}},\forall i=1,\cdots ,n$, ${{m}_{1}}:=0$ and ${{m}_{i}}:=\sum\nolimits_{k=1}^{i-1}{{{n}_{k}}},\forall i=2,\cdots ,n$.
\end{definition}
\begin{remark} \begin{enumerate}
    \item Note that $({{x}_{1}},\cdots ,{{x}_{m}})$ is in the domain of definition of $f\circ ({{g}_{1}},\cdots ,{{g}_{n}})$ if and only if each ${{g}_{i}}({{\text{v}}_{i}})$ is well-defined and $f({{g}_{1}}({{\text{v}}_{1}}),\cdots ,{{g}_{n}}({{\text{v}}_{n}}))$ is well-defined.
    \item We may denote the $n$-tuple $(g,g,\cdots ,g)$ by ${{g}^{n}}$ (in e.g. Section \ref{II Cons of $T$-mor and}) for brevity.
    \end{enumerate}
\end{remark}
Before we generalize Definition \ref{9.1.2}, we need a terminology as follows.
\begin{terminology} \label{10.1.2}
Let $f$ and $g$ be partial functions from $A$ to $B$. If the domain of definition of $f$ contains the domain of definition of $g$, which is denoted by $D(g)$, and $f{{|}_{D(g)}}=g{{|}_{D(g)}}$,  then we say that $g$ is a \emph{restriction} of $f$. 
\end{terminology}
\begin{remark}
    Trivially, if $D(g)=\emptyset $, then $g$ is a restriction of any $f$.
\end{remark}
\begin{definition} \label{10.1.3}
    Let $D$ be a set and let $T$ be a subset of 
    \begin{center}
        \{all partial functions from the cartesian product ${{D}^{n}}$ to $D\, |\,n\in {{\mathbb{Z}}^{+}}$\}. 
    \end{center}
If $\forall n\in {{\mathbb{Z}}^{+}}$, $f\in T$ of $n$ variables, and $({{g}_{1}},\cdots ,{{g}_{n}})\in {{T}^{n}}$ (cartesian product), the composite $f\circ ({{g}_{1}},\cdots ,{{g}_{n}})$ is a restriction of some element of $T$, then we call $T$ a \emph{partial-operator generalized-semigroup}, abbreviated as \emph{par-operator gen-semigroup}, on $D$ and $D$ is called the \emph{domain} of $T$. 
\end{definition}
\begin{remark} \begin{enumerate}
    \item Trivially, $\emptyset $ is a par-operator gen-semigroup.
    \item Obviously, any operator gen-semigroup is a par-operator gen-semigroup.
\end{enumerate}
  \end{remark}
	Definition \ref{generated operator semigroup} is generalized to
 \begin{definition} \label{10.1.4}
     Let $D$ be a set and $G$ be a subset of 
     \begin{center}
       \{all partial functions from the cartesian product ${{D}^{n}}$ to $D\, |\,n\in {{\mathbb{Z}}^{+}}$\}.  
     \end{center}
Then we call the intersection of all par-operator gen-semigroups on $D$ containing $G$ the par-operator gen-semigroup on $D$ \emph{generated by} $G$, and denote it by $\left\langle G \right\rangle $.
 \end{definition}
\begin{remark}
    From Definition \ref{10.1.3}, we can tell that $\left\langle G \right\rangle $ is the smallest par-operator gen-semigroup on $D$ which contains $G$.
\end{remark} 
 \begin{example} \label{10.1.5}
    Let $B$ be a subfield of a field $F$. Let
    \begin{center}
   $T=\{$the partial function $f^*:{{F}^{n}}\to F$ given by $(a_1,\cdots ,a_n)\mapsto f(a_1,\cdots ,a_n)\, |\,f\in \text{Frac}(B[{{x}_{1}},\cdots ,{{x}_{n}}]),n\in {{\mathbb{Z}}^{+}}\}$.      
    \end{center}
Then it is not hard to see that $T$ is a par-operator gen-semigroup on $F$. The map which induces $T$, i.e. $\tau :\bigcup\nolimits_{n\in {{\mathbb{Z}}^{+}}}{\text{Frac}(B[{{x}_{1}},\cdots ,{{x}_{n}}])}\to T$ given by $f\mapsto f^*$, is surjective, but it is not necessarily injective. 
 \end{example}
\begin{example} \label{10.1.6}
    Let $B$ be a differential subfield of a differential field $F$; that is, field $F$ is endowed with a derivation and the derivation of $F$ restricts to the derivation of $B$. Let $T=\left\langle {{T}_{F}}\bigcup {{T}_{\partial }} \right\rangle $ (Definition \ref{10.1.4}), where ${{T}_{F}}$ is the par-operator gen-semigroup on $F$ defined as in Example \ref{10.1.5} and ${{T}_{\partial }}$ is the operator semigroup defined in Example \ref{T on diff ring--1}. Then by Definition \ref{10.1.4}, $T$ is a par-operator gen-semigroup on $F$.
\end{example}
\begin{example} \label{10.1.7}
    For this example, $D$ in Definitions \ref{10.1.1} and \ref{10.1.3} is generalized to be a class. Let $\mathcal{C}$ be a category and let $D$ be the collection of all morphisms in $\mathcal{C}$. $\forall n\in {{\mathbb{Z}}^{+}}$, let 
    \begin{center}
            ${{\mathcal{F}}_{n}}=\{$word $w$ on $\{x_1,\cdots ,x_n\}\,|$ each of $x_1,\cdots ,x_n$ arises in $w$\}.  
    \end{center}
    Let $T=\{$the partial function $f^*:{{D}^{n}}\to D$ given by 
\begin{center}
    $(a_1,\cdots ,a_n)\mapsto f(a_1,\cdots ,a_n)\, |\,f\in {{\mathcal{F}}_{n}},n\in {{\mathbb{Z}}^{+}}\}$,
\end{center}
where $f(a_1,\cdots ,a_n)$ is defined to be the composite ${{a}_{{{i}_{1}}}}\circ \cdots \circ {{a}_{{{i}_{k}}}}$ for $f={{x}_{{{i}_{1}}}}\cdots {{x}_{{{i}_{k}}}}$ (where ${{i}_{1}},\cdots ,{{i}_{k}}\in \{1,\cdots ,n\}$). 

In particular, for $f=x_1\in {{\mathcal{F}}_{1}}$, $f^*$ is the identity function on $D$.

    Thus, for $f\in {{\mathcal{F}}_{n}}$, $(a_1,\cdots ,a_n)\in D^n$ and ${{i}_{1}},\cdots ,{{i}_{k}}\in \{1,\cdots ,n\}$, the following four statements are equivalent:
\begin{enumerate}
    \item [(i)] $(a_1,\cdots ,a_n)$ is in the domain of definition of the $f^*$ induced by $f={{x}_{{{i}_{1}}}}\cdots {{x}_{{{i}_{k}}}}$.
    \item [(ii)] for $f={{x}_{{{i}_{1}}}}\cdots {{x}_{{{i}_{k}}}}$, $f(a_1,\cdots ,a_n)$ is well-defined.
    \item [(iii)] ${{a}_{{{i}_{1}}}}\circ \cdots \circ {{a}_{{{i}_{k}}}}$ is well-defined.
    \item[(iv)] $\forall j=1, \cdots, (k-1)$, the domain (object) of ${{a}_{{{i}_{j}}}}$ is the target (object) of ${{a}_{{{i}_{j+1}}}}$.
\end{enumerate}

Then it is not hard to see that $T$ is a par-operator gen-semigroup on $D$. Let	$\tau :\bigcup\nolimits_{n\in {{\mathbb{Z}}^{+}}}{{{\mathcal{F}}_{n}}}\to T$ be the map inducing $T$ given by $f\mapsto f^*$.
\end{example}

\subsection{$T$-spaces} \label{11 $T$-spaces}
\begin{definition} \label{10.2.1}
Let $T$ be a par-operator gen-semigroup on $D$ and let $U\subseteq D$. 

Then we call 
\begin{center}
    $S:=\{$well-defined $f({{u}_{1}},\cdots ,{{u}_{n}})\, |\,f\in T$, $({{u}_{1}},\cdots ,{{u}_{n}})\in {{U}^{n}}$, $n\in {{\mathbb{Z}}^{+}}\}$ 
\end{center}
the \emph{$T$-space} generated by $U$ and denote $S$ by ${{\left\langle U \right\rangle }_{T}}$.

Moreover, if $B\subseteq S$ is also a $T$-space, then we say that $B$ is a \emph{$T$-subspace} of $S$ and write $B\le S$ or $S\ge B$.
\end{definition}
\begin{example} \label{10.2.2}
 Let $B$ be a subfield of a field $F$. Let $T$ be the par-operator gen-semigroup on $F$ defined in Example \ref{10.1.5}. Then $\forall (\emptyset \ne )U\subseteq F$, we can tell that the $T$-space ${{\left\langle U \right\rangle }_{T}}$ is the field $B(U)$.

\end{example}
\begin{example} \label{10.2.3}
  Let $B$ be a differential subfield of a differential field $F$. Let $T(=\left\langle {{T}_{F}}\bigcup {{T}_{\partial }} \right\rangle )$ be the par-operator gen-semigroup on $F$ defined in Example \ref{10.1.6}. Then we claim that $\forall (\emptyset \ne )U\subseteq F$, the $T$-space ${{\left\langle U \right\rangle }_{T}}$ is the differential subfield of $F$ generated by $U$ over $B$. 

    To show the claim, let $K={{\left\langle U \right\rangle }_{T}}$. Then by Proposition \ref{11.1.1} (to be shown in Section \ref{12 Basic properties}), which generalizes Proposition \ref{<S> contained in S}, ${{\left\langle K \right\rangle }_{T}}\subseteq K$. And ${{\left\langle K \right\rangle }_{{{T}_{F}}}}\subseteq {{\left\langle K \right\rangle }_{T}}$ because ${{T}_{F}}\subseteq T$. Hence ${{\left\langle K \right\rangle }_{{{T}_{F}}}}\subseteq K$. On the other hand, ${{\left\langle K \right\rangle }_{{{T}_{F}}}}\supseteq K$ because $\operatorname{Id}\in {{T}_{F}}$. Therefore ${{\left\langle K \right\rangle }_{{{T}_{F}}}}=K$, and thus by Example \ref{10.2.2}, $K$ is a field. Analogously, we can show ${{\left\langle K \right\rangle }_{{{T}_{\partial }}}}=K$, and hence $K$ is a differential field.
\end{example}
\begin{example} \label{10.2.4}
    Let $\mathcal{C}$ be a category. Let $D$ and $T$ be defined as in Example \ref{10.1.7}. Then the $T$-space ${{\left\langle D \right\rangle }_{T}}=D$ since $\operatorname{Id}\in T$. Because there is a bijection between the objects $A$ in $\mathcal{C}$ and their identity morphisms $1_A$, we may view ${{\left\langle D \right\rangle }_{T}}$ as $\mathcal{C}$. 
\end{example}

\subsection{$T$-morphisms} \label{11 $T$-morphisms}
	Now we generalize Definition \ref{9.3.1} to
\begin{definition} \label{10.3.1}
Let $T$ be a par-operator gen-semigroup. Let $\sigma$ be a map from a $T$-space $S$ to a $T$-space. If $\forall n\in {{\mathbb{Z}}^{+}}$, $(a_1,\cdots ,a_n)\in {{S}^{n}}$ and $f\in T$, neither $f(a_1,\cdots ,a_n)$ nor $f(\sigma ({{a}_{1}}),\cdots , $ $\sigma ({{a}_{n}}))$ is well-defined or 
\[\sigma (f(a_1,\cdots ,a_n))=f(\sigma ({{a}_{1}}),\cdots ,\sigma ({{a}_{n}})),\] 
then we call $\sigma $ a \emph{$T$-morphism}.
\end{definition}
\begin{remark} 
  From Definitions \ref{10.1.3} and \ref{10.2.1} (or from Proposition \ref{11.1.1}), we can tell that if $f(a_1,\cdots ,a_n)$ is well-defined, then $f(a_1,\cdots ,a_n)\in S$, and so $\sigma (f(a_1,\cdots ,a_n))$ is well-defined. 
\end{remark}
	Note that in Examples \ref{10.1.5} and \ref{10.1.7}, par-operator gen-semigroups are induced from rational functions or words on a set. Hence we generalize Definition \ref{9.3.2} as follows.
 \begin{definition} \label{10.3.2}
    Let $f$ be an expression associated with a rule $R$. Let $n \in \mathbb{Z}^+$. 
 
    If all of $x_1,\cdots ,x_n$ arise as symbols in $f$, and there are sets $C$ and $D$ such that, by the rule $R$, $f$ yields nothing or a single element in $C$, which we denote by $f(a_1,\cdots ,a_n)(\in C)$, whenever every $x_1,\cdots ,x_n$ in $f$ is replaced by any $a_1,\cdots ,a_n \in D$, respectively, then we call the couple $(f,R)$ a ($C$-valued) \emph{formal partial function} of $n$ variables $x_1,\cdots ,x_n$ and $D$ is called a \emph{domain} of $(f,R)$. 
 
    Moreover, in the case where $a_1,\cdots ,a_n \in D$ and $f$ yields $f(a_1,\cdots ,a_n)$, we say that $f(a_1,\cdots ,a_n)$ is \emph{well-defined} or $(a_1,\cdots ,a_n)$ is in the \emph{domain of definition} of $f(x_1,\cdots ,x_n)$.
 
    Besides, we may call $f$ a \emph{formal partial function} if its associated rule $R$ is clear from the context.
 \end{definition}
 \begin{remark}
     Analogous to a formal function, a formal partial function may have more than one domain. And the notion of formal partial function generalizes the notion of partial function.
 \end{remark}
When par-operator gen-semigroups are induced from formal partial functions, the following lemma, which generalizes Lemma \ref{9.3.3}, will be useful. 
\begin{lemma} \label{10.3.3}
   Let $\mathcal{F}$ be a set of $D$-valued formal partial functions on a domain $D$ and let $T$ be a par-operator gen-semigroup on $D$ induced by $\mathcal{F}$; that is, there is a surjective map $\tau :\mathcal{F}\to T$ being given by $f\mapsto f^*$, where $f^*:{{D}^{n}}\to D$ is given by $(a_1,\cdots ,a_n)\mapsto f(a_1,\cdots ,a_n)$ for $f$ of $n$ variables, $\forall n\in {{\mathbb{Z}}^{+}}$. Let $\sigma $ be a map from a $T$-space $S$ to a $T$-space. Then the following statements are equivalent:
\begin{enumerate}
    \item [(i)] $\sigma $ is a $T$-morphism.
    \item [(ii)] $\forall n\in {{\mathbb{Z}}^{+}}$, $(a_1,\cdots ,a_n)\in {{S}^{n}}$ and $f\in \mathcal{F}$ of $n$ variables, neither $f(a_1,\cdots ,a_n)$ nor $f(\sigma ({{a}_{1}}),\cdots ,\sigma ({{a}_{n}}))$ is well-defined or
    \begin{equation} \label{11.1}
        \sigma (f(a_1,\cdots ,a_n))=f(\sigma ({{a}_{1}}),\cdots ,\sigma ({{a}_{n}})).
    \end{equation}
\end{enumerate}
\end{lemma}
\begin{proof}
(i) $\Rightarrow$ (ii): By Definition \ref{10.3.1}, $\forall n\in {{\mathbb{Z}}^{+}}$, $(a_1,\cdots ,a_n)\in {{S}^{n}}$ and $f\in \mathcal{F}$ of $n$ variables, neither $\tau (f)(a_1,\cdots ,a_n)$ nor $\tau (f)(\sigma ({{a}_{1}}),\cdots ,\sigma ({{a}_{n}}))$ is well-defined or $\sigma (\tau (f)(a_1,\cdots ,a_n))=\tau (f)(\sigma ({{a}_{1}}),\cdots ,\sigma ({{a}_{n}}))$.
	
 Then by the definition of $\tau $, (ii) must be true.

(ii) $\Rightarrow$ (i): 
Since $\tau $ is surjective, $\forall f^*\in T$, $\exists f\in \mathcal{F}$ such that $f^*=\tau (f)$. Then from statement (ii) and the definition of $\tau $, we can tell that $\forall n\in {{\mathbb{Z}}^{+}}$, $(a_1,\cdots ,a_n)\in {{S}^{n}}$ and $f^*\in T$ of $n$ variables, neither $f^*(a_1,\cdots ,a_n)$ nor $f^*(\sigma ({{a}_{1}}),\cdots ,\sigma ({{a}_{n}}))$ is well-defined or
\begin{center}
    $\sigma (f^*(a_1,\cdots ,a_n))=f^*(\sigma ({{a}_{1}}),\cdots ,\sigma ({{a}_{n}}))$.
\end{center}
     So by Definition \ref{10.3.1}, $\sigma $ must be a $T$-morphism.
\end{proof}
	Lemma \ref{10.3.3} characterizes a $T$-morphism by formal partial functions which induce $T$. When a par-operator gen-semigroup is induced from formal partial functions, the criterion for a $T$-morphism by Lemma \ref{10.3.3} is normally more convenient than that by Definition \ref{10.3.1}. 
	
 The following is comparable with Proposition \ref{9.3.4}, except that now $T$ is defined by Example \ref{10.1.5}, and hence ${{\left\langle U \right\rangle }_{T}}$ and ${{\left\langle V \right\rangle }_{T}}$ must be fields. 
\begin{proposition} \label{10.3.4}
Let $B$ be a subfield of a field $F$, let $T$ be the par-operator gen-semigroup on $F$ defined in Example \ref{10.1.5}, and let nonempty $U,V\subseteq F$. Then a map $\sigma :{{\left\langle U \right\rangle }_{T}}\to {{\left\langle V \right\rangle }_{T}}$ is a ring homomorphism with $B$ fixed pointwisely if and only if it is a $T$-morphism.
\end{proposition}
\begin{proof}
    By the definition of $T$ in Example \ref{10.1.5}, 
\begin{center}
    $T=\{$the partial function $f^*:{{F}^{n}}\to F$ given by $(a_1,\cdots ,a_n)\mapsto f(a_1,\cdots ,a_n)\, |\,f\in \text{Frac}(B[{{x}_{1}},\cdots ,{{x}_{n}}]),n\in {{\mathbb{Z}}^{+}}\}$.
\end{center}

Then ${{\left\langle U \right\rangle }_{T}}$ and ${{\left\langle V \right\rangle }_{T}}$ are the fields $B(U)$ and $B(V)$, respectively. Let $\mathcal{F}=\{f\in \text{Frac}(B[{{x}_{1}},\cdots ,{{x}_{n}}])\ |\ n\in {{\mathbb{Z}}^{+}}\}$. Then Lemma \ref{10.3.3} applies. Hence it suffices to show that $\sigma :{{\left\langle U \right\rangle }_{T}}\to {{\left\langle V \right\rangle }_{T}}$ is a ring homomorphism with $B$ fixed pointwisely if and only if statement (ii) in Lemma \ref{10.3.3} is true in this case. Specifically, we only need to show that the following statements are equivalent:
\begin{enumerate}
    \item $\forall c\in B$, $\sigma (c)=c$, and $\forall a,b\in {{\left\langle U \right\rangle }_{T}}(=B(U))$,
$\sigma (a+b)=\sigma (a)+\sigma (b)$ and $\sigma (ab)=\sigma (a)\sigma (b)$.
\item $\forall n\in {{\mathbb{Z}}^{+}}$, $(a_1,\cdots ,a_n)\in \left\langle U \right\rangle _{T}^{n}$ and $f\in \text{Frac}(B[{{x}_{1}},\cdots ,{{x}_{n}}])$, neither
$f(a_1,\cdots ,a_n)$ nor $f(\sigma ({{a}_{1}}),\cdots ,\sigma ({{a}_{n}}))$ is well-defined or
\begin{equation} \label{11.2}
    \sigma (f(a_1,\cdots ,a_n))=f(\sigma ({{a}_{1}}),\cdots ,\sigma ({{a}_{n}})). 
\end{equation}
\end{enumerate}
(1) $\Rightarrow$ (2): Let $f({{x}_{1}},\cdots ,{{x}_{n}})=g({{x}_{1}},\cdots ,{{x}_{n}})/h({{x}_{1}},\cdots ,{{x}_{n}})$, where $n\in {{\mathbb{Z}}^{+}}$ and $h,g\in B[{{x}_{1}},\cdots ,{{x}_{n}}]$. Let $(a_1,\cdots ,a_n)\in \left\langle U \right\rangle _{T}^{n}$. Then

$f(a_1,\cdots ,a_n)$ is well-defined;

$\Leftrightarrow h(a_1,\cdots ,a_n)\ne 0$; 

$\Leftrightarrow \sigma (h(a_1,\cdots ,a_n))\ne 0$ (by statement (1)); 

$\Leftrightarrow f(\sigma ({{a}_{1}}),\cdots ,\sigma ({{a}_{n}}))$ 

\hspace{0.4cm}$=g(\sigma ({{a}_{1}}),\cdots ,\sigma ({{a}_{n}}))/h(\sigma ({{a}_{1}}),\cdots ,\sigma ({{a}_{n}}))$ 

\hspace{0.4cm}$=\sigma (g(a_1,\cdots ,a_n))/\sigma (h(a_1,\cdots ,a_n))$ (because $h,g\in B[{{x}_{1}},\cdots ,{{x}_{n}}]$ and statement (1) is true) 

\hspace{0.4cm}is well-defined.

Moreover, in the case where any of the above equivalent conditions is true, 

$\sigma (f(a_1,\cdots ,a_n))$

$=\sigma (g(a_1,\cdots ,a_n)/h(a_1,\cdots ,a_n))$

$=\sigma (g(a_1,\cdots ,a_n))/\sigma (h(a_1,\cdots ,a_n))$.

Combining the above four equations, we have
\[\sigma (f(a_1,\cdots ,a_n))=f(\sigma ({{a}_{1}}),\cdots ,\sigma ({{a}_{n}})).\] 

(2) $\Rightarrow$ (1): 
Considering the case where $f\in \text{Frac}(B[{{x}_{1}},\cdots ,{{x}_{n}}])$ is a constant polynomial, we can tell from statement (2) that $\forall c\in B$, $\sigma (c)=c$. 

Let $a,b\in {{\left\langle U \right\rangle }_{T}}$. Let $g={{x}_{1}}+{{x}_{2}}$ and $h={{x}_{1}}{{x}_{2}}$.

Then 

$\sigma (a+b)$

$=\sigma (g(a,b))$

$=g(\sigma (a),\sigma (b))$ (by Equation $($\ref{11.2}$)$)

$=\sigma (a)+\sigma (b)$. 

And 

$\sigma (ab)$

$=\sigma (h(a,b))$

$=h(\sigma (a),\sigma (b))$ (by Equation $($\ref{11.2}$)$)

$=\sigma (a)\sigma (b)$.
\end{proof}
However, $\sigma $ in Proposition \ref{10.3.4} is between field extensions over the same field $B$. We will deal with ring homomorphisms between field extensions over different fields in Subsection \ref{11 $theta$-morphisms}.

To characterize differential ring homomorphisms between differential fields by $T$-morphisms, we need a lemma as follows.
\begin{lemma} \label{10.3.5}
    Let $T$ be a par-operator gen-semigroup, let $\sigma$ be a map from a $T$-space $S$ to a $T$-space, let $f\in T$ have $n$ variables, and $\forall i=1,\cdots,n$, let ${{g}_{i}}\in T$ have ${{n}_{i}}$ variables. Suppose that both of the following two conditions are satisfied,
    \begin{enumerate}
    \item [(i)]	$\forall (a_1,\cdots ,a_n)\in {{S}^{n}}$, neither $f(a_1,\cdots ,a_n)$ nor $f(\sigma ({{a}_{1}}),\cdots ,\sigma ({{a}_{n}}))$ is well-defined or $\sigma (f(a_1,\cdots ,a_n))=f(\sigma ({{a}_{1}}),\cdots ,\sigma ({{a}_{n}}))$;
    \item [(ii)] $\forall i=1,\cdots,n$ and $({{a}_{1}},\cdots ,{{a}_{{{n}_{i}}}})\in {{S}^{{{n}_{i}}}}$,\\
    neither ${{g}_{i}}({{a}_{1}},\cdots ,{{a}_{{{n}_{i}}}})$ nor ${{g}_{i}}(\sigma ({{a}_{1}}),\cdots ,\sigma ({{a}_{{{n}_{i}}}}))$ is well-defined or
    \[\sigma ({{g}_{i}}({{a}_{1}},\cdots ,{{a}_{{{n}_{i}}}}))={{g}_{i}}(\sigma ({{a}_{1}}),\cdots ,\sigma ({{a}_{{{n}_{i}}}})).\] 
    \end{enumerate}
    Then $\forall ({{a}_{1}},\cdots ,{{a}_{m}})\in {{S}^{m}}$, where $m=\sum\nolimits_{i}{{{n}_{i}}}$,\\  
      neither $(f\circ ({{g}_{1}},\cdots ,{{g}_{n}}))({{a}_{1}},\cdots ,{{a}_{m}})$ nor $(f\circ ({{g}_{1}},\cdots ,{{g}_{n}}))(\sigma ({{a}_{1}}),\cdots ,\sigma ({{a}_{m}}))$ is well-defined or
\[\sigma ((f\circ ({{g}_{1}},\cdots ,{{g}_{n}}))({{a}_{1}},\cdots ,{{a}_{m}}))=(f\circ ({{g}_{1}},\cdots ,{{g}_{n}}))(\sigma ({{a}_{1}}),\cdots ,\sigma ({{a}_{m}})).\]
\end{lemma}
\begin{proof}
    Let $({{a}_{1}},\cdots ,{{a}_{m}})\in {{S}^{m}}$, where $m=\sum\nolimits_{i}{{{n}_{i}}}$. Then
    
$(f\circ ({{g}_{1}},\cdots ,{{g}_{n}}))({{a}_{1}},\cdots ,{{a}_{m}})$ is well-defined;

$\Leftrightarrow f({{g}_{1}}({{\text{u}}_{1}}),\cdots ,{{g}_{n}}({{\text{u}}_{n}}))$ is well-defined, where ${{\text{u}}_{i}}=({{a}_{{{m}_{i}}+1}},\cdots ,{{a}_{{{m}_{i}}+{{n}_{i}}}})$, $\forall i=1,\cdots,n$, ${{m}_{1}}=0$, and ${{m}_{i}}=\sum\nolimits_{k=1}^{i-1}{{{n}_{k}}},\forall i=2, \cdots, n$; (by Definition \ref{10.1.1})

$\Leftrightarrow f(\sigma ({{g}_{1}}({{\text{u}}_{1}})),\cdots ,\sigma ({{g}_{n}}({{\text{u}}_{n}})))$ is well-defined; (by condition (i) and the fact that each ${{g}_{i}}({{\text{u}}_{i}})\in S$ (\textit{cf.} Proposition \ref{11.1.1} in Section \ref{12 Basic properties}))

$\Leftrightarrow f({{g}_{1}}(\sigma ({{\text{u}}_{1}})),\cdots ,{{g}_{n}}(\sigma ({{\text{u}}_{n}})))$ is well-defined, where $\forall i=1,\cdots,n$, $\sigma \text{(}{{\text{u}}_{i}}):=(\sigma \text{(}{{a}_{{{m}_{i}}+1}}),\cdots ,\sigma \text{(}{{a}_{{{m}_{i}}+{{n}_{i}}}}))$; (by condition (ii))

$\Leftrightarrow (f\circ ({{g}_{1}},\cdots ,{{g}_{n}}))(\sigma ({{a}_{1}}),\cdots ,\sigma ({{a}_{m}}))$ is well-defined. (by Definition \ref{10.1.1})

And in the case where any of the above five equivalent conditions is true,

$(f\circ ({{g}_{1}},\cdots ,{{g}_{n}}))({{a}_{1}},\cdots ,{{a}_{m}})\in S$ (by Proposition \ref{11.1.1}) 

and

$\sigma ((f\circ ({{g}_{1}},\cdots ,{{g}_{n}}))({{a}_{1}},\cdots ,{{a}_{m}}))$ 

$=\sigma (f({{g}_{1}}({{\text{u}}_{1}}),\cdots ,{{g}_{n}}({{\text{u}}_{n}})))$(by Definition \ref{10.1.1})

$=f(\sigma ({{g}_{1}}({{\text{u}}_{1}})),\cdots ,\sigma ({{g}_{n}}({{\text{u}}_{n}})))$ (by (i) and the fact that each ${{g}_{i}}({{\text{u}}_{i}})\in S$ (by Proposition \ref{11.1.1}))

$=f({{g}_{1}}(\sigma ({{\text{u}}_{1}})),\cdots ,{{g}_{n}}(\sigma ({{\text{u}}_{n}})))$ (by condition (ii))

$=(f\circ ({{g}_{1}},\cdots ,{{g}_{n}}))(\sigma ({{a}_{1}}),\cdots ,\sigma ({{a}_{m}}))$ (by Definition \ref{10.1.1})
\end{proof}
\begin{proposition} \label{10.3.6}
    Let $B$ be a differential subfield of a differential field $F$, let $T$ be the par-operator gen-semigroup on $F$ defined in Example \ref{10.1.6}, and let nonempty $U,V\subseteq F$. Then a map $\sigma :{{\left\langle U \right\rangle }_{T}}\to {{\left\langle V \right\rangle }_{T}}$ is a differential ring homomorphism with $B$ fixed pointwisely if and only if it is a $T$-morphism.
\end{proposition}
\begin{proof}
    Let $A={{\left\langle U \right\rangle }_{T}}$ and $R={{\left\langle V \right\rangle }_{T}}$. By Example \ref{10.2.3}, $A$ and $R$ are the differential subfields of $F$ generated by $U$ and $V$ over $B$, respectively.
    
    Recall that a ring homomorphism $\sigma :A\to R$ is a differential homomorphism if and only if $\sigma $ commutes with $\partial $. Then by Definition \ref{10.3.1}, we only need to show that the following statements are equivalent:
\begin{enumerate}
    \item $\sigma $ is a ring homomorphism with $B$ fixed pointwisely and $\forall a\in A$, $\sigma (\partial (a))=\partial (\sigma (a))$. 
     \item $\forall n\in {{\mathbb{Z}}^{+}}$, $(a_1,\cdots ,a_n)\in {{A}^{n}}$ and $f\in T$ of $n$ variables, neither $f(a_1,\cdots ,a_n)$ nor $f(\sigma ({{a}_{1}}),\cdots ,\sigma ({{a}_{n}}))$ is well-defined or
\[\sigma (f(a_1,\cdots ,a_n))=f(\sigma ({{a}_{1}}),\cdots ,\sigma ({{a}_{n}})).\]
\end{enumerate}
(1) $\Rightarrow$ (2): 
Let ${{T}_{F}}$ be the par-operator gen-semigroup on $F$ defined as in Example \ref{10.1.5} and let ${{T}_{\partial }}$ be the operator semigroup defined in Example \ref{T on diff ring--1}. 

Since $\sigma $ is a ring homomorphism with $B$ fixed pointwisely, by Proposition \ref{10.3.4}, $\sigma $ is a ${{T}_{F}}$-morphism. By Definition \ref{10.3.1}, $\forall n\in {{\mathbb{Z}}^{+}}$, $(a_1,\cdots ,a_n)\in {{A}^{n}}$, and $g\in {{T}_{F}}$ of $n$ variables, neither $g(a_1,\cdots ,a_n)$ nor $g(\sigma ({{a}_{1}}),\cdots ,\sigma ({{a}_{n}}))$ is well-defined or $\sigma (g(a_1,\cdots ,a_n))=g(\sigma ({{a}_{1}}),\cdots ,\sigma ({{a}_{n}}))$. Moreover, since $\sigma (\partial (a))=\partial (\sigma (a)),\forall a\in A$, we can tell that $\forall g\in {{T}_{\partial }}=\{{{\partial }^{n}}\,|\,n\in {{\mathbb{N}}_{0}}\}$ and $a\in A$, $\sigma (g(a))=g(\sigma (a))$. 

Therefore, $\forall n\in {{\mathbb{Z}}^{+}}$, $(a_1,\cdots ,a_n)\in {{A}^{n}}$, and $g\in {{T}_{F}}\bigcup {{T}_{\partial }}$ of $n$ variables, neither $g(a_1,\cdots ,a_n)$ nor $g(\sigma ({{a}_{1}}),\cdots ,\sigma ({{a}_{n}}))$ is well-defined or $\sigma (g(a_1,\cdots ,a_n))=g(\sigma ({{a}_{1}}),\cdots ,\sigma ({{a}_{n}}))$.

Let $f\in {{T}_{F}}\bigcup {{T}_{\partial }}$ have $n$ variables, let ${{g}_{i}}\in {{T}_{F}}\bigcup {{T}_{\partial }}$ have ${{n}_{i}}$ variables, $\forall i=1, \cdots, n$, and let $({{a}_{1}},\cdots ,{{a}_{m}})\in {{A}^{m}}$, where $m=\sum\nolimits_{i}{{{n}_{i}}}$. By the statement in the preceding paragraph, Lemma \ref{10.3.5} applies (with $T:= \left\langle {{T}_{F}}\bigcup {{T}_{\partial }} \right\rangle $). It follows that neither $(f\circ ({{g}_{1}},\cdots ,{{g}_{n}}))({{a}_{1}},\cdots ,{{a}_{m}})$ nor $(f\circ ({{g}_{1}},\cdots ,{{g}_{n}}))(\sigma ({{a}_{1}}),\cdots ,\sigma ({{a}_{m}}))$ is well-defined or
\[\sigma ((f\circ ({{g}_{1}},\cdots ,{{g}_{n}}))({{a}_{1}},\cdots ,{{a}_{m}}))=(f\circ ({{g}_{1}},\cdots ,{{g}_{n}}))(\sigma ({{a}_{1}}),\cdots ,\sigma ({{a}_{m}})).\]

Moreover, since $T=\left\langle {{T}_{F}}\bigcup {{T}_{\partial }} \right\rangle $, from Definitions \ref{10.1.3} and \ref{10.1.4}, we can tell that the above $f\circ ({{g}_{1}},\cdots ,{{g}_{n}})$ is a restriction of some element of $T$, and $\forall h\in T$, $h$ can be obtained by a finite number of compositions (defined by Definition \ref{10.1.1}) of elements of ${{T}_{F}}\bigcup {{T}_{\partial }}$ (or elements which “originally come from” ${{T}_{F}}\bigcup {{T}_{\partial }}$). Thus, by induction and application of Lemma \ref{10.3.5} as in the preceding paragraph, we can show that $\forall n\in {{\mathbb{Z}}^{+}}$, $h\in T$ of $n$ variables and $(a_1,\cdots ,a_n)\in {{A}^{n}}$, neither $h(a_1,\cdots ,a_n)$ nor $h(\sigma ({{a}_{1}}),\cdots ,\sigma ({{a}_{n}}))$ is well-defined or
\[\sigma (h(a_1,\cdots ,a_n))=h(\sigma ({{a}_{1}}),\cdots ,\sigma ({{a}_{n}})).\]

(2) $\Rightarrow$ (1): 
As shown in Example \ref{10.2.3}, ${{\left\langle A \right\rangle }_{{{T}_{F}}}}=A$. Thus $A$ is also a ${{T}_{F}}$-space. Then because ${{T}_{F}}\subseteq T$, by Definition \ref{10.3.1}, $\sigma $ is a ${{T}_{F}}$-morphism. Hence by Proposition \ref{10.3.4}, $\sigma $ is a ring homomorphism with $B$ fixed pointwisely. Moreover, if let $f=\partial (\in T)$, then by statement (2), $\sigma (\partial (a))=\partial (\sigma (a)),\forall a\in A$.	
\end{proof}

\subsection{$T$-isomorphisms} \label{11 $T$-isomorphisms}
To par-operator gen-semigroups, Definition \ref{$T$-isomorphisms} still applies: if a $T$-morphism is bijective, then we call it a \emph{$T$-isomorphism}. To justify the definition, we need to show Proposition \ref{1.3.a} with a different proof as follows.
 \begin{proposition} \label{11.2.a}\emph{(Proposition \ref{1.3.a})}
    Let $\sigma$ be a $T$-isomorphism from a $T$-space ${{S}_{1}}$ to a $T$-space ${{S}_{2}}$ and let ${{\sigma }^{-1}}$ be the inverse map of $\sigma $. Then ${{\sigma }^{-1}}\in \operatorname{Iso}_{T}({{S}_{2}},{{S}_{1}})$.
\end{proposition}
\begin{proof}
    By Definition \ref{10.3.1}, $\forall n\in {{\mathbb{Z}}^{+}}$, $f\in T$ of $n$ variables and $(a_1,\cdots ,a_n)\in {S_1^n}$, neither $f(a_1,\cdots ,a_n)$ nor $f(\sigma ({{a}_{1}}),\cdots , $ $\sigma ({{a}_{n}}))$ is well-defined or \[f(a_1,\cdots ,a_n)={\sigma }^{-1}(f(\sigma ({{a}_{1}}),\cdots ,\sigma ({{a}_{n}}))).\] Hence $\forall n\in {{\mathbb{Z}}^{+}}$, $f\in T$ of $n$ variables and $(b_1,\cdots ,b_n)\in {S_2^n}$, neither $f({b}_{1},\cdots ,{b}_{n})$ nor $f({\sigma }^{-1}({b}_{1}),\cdots ,{\sigma }^{-1}({b}_{n}))$ is well-defined or 
    \[f({\sigma }^{-1}({b}_{1}),\cdots ,{\sigma }^{-1}({b}_{n}))={\sigma }^{-1}(f(b_1,\cdots ,b_n)).\] 
    Then by Definition \ref{10.3.1}, ${\sigma }^{-1}$ is a $T$-morphism from $S_2$ to $S_1$, and hence ${{\sigma }^{-1}}\in \operatorname{Iso}_{T}({{S}_{2}},{{S}_{1}})$ by Definition \ref{$T$-isomorphisms}.
\end{proof}

\subsection{$\theta$-morphisms} \label{11 $theta$-morphisms}
We first generalize Notation \ref{9.4.1} to
\begin{notation} \label{10.4.1}
    Let $D_1$ (resp. $D_2$) be a set, let $M_1$ (resp. $M_2$) be a subset of 
    \{all partial functions from the cartesian product $D_{1}^{n}$ to ${D_1}\, |\,n\in {{\mathbb{Z}}^{+}}$\}
      (resp. a subset of  \{all partial functions from $D_{2}^{n}$ to ${{D}_{2}}\, |\,n\in {{\mathbb{Z}}^{+}}$\}),
     let $\theta \subseteq {{M}_{1}}\times {{M}_{2}}$, and let $A\subseteq {{D}_{2}}$. We denote by $\theta {{|}_{A}}$ the set $\{(f,g{{|}_{A}})\,|\,(f,g)\in \theta \}$, where $g{{|}_{A}}$ denotes the partial function obtained by restricting every variable of $g$ to $A$.
\end{notation}	
Then for the case of partial functions, Proposition \ref{9.4.b} needs to be modified as follows.
\begin{proposition} \label{10.4.b}
  Let $\theta$ and $A$ be defined as in Notation \ref{10.4.1}. Then the following statements are equivalent: \begin{enumerate}
        \item [(i)] $\theta {{|}_{A}}$ is a map.
        \item [(ii)] $\forall(f,g_1),(f,g_2)\in \theta$, $g_{1}{{|}_{A}}=g_2{{|}_{A}}$.
         \item [(iii)] $\forall(f,g_1),(f,g_2)\in \theta$, $n\in {{\mathbb{Z}}^{+}}$ and $(a_1, \cdots, a_n) \in A^n$, $g_{1}(a_1, \cdots, a_n)=g_2(a_1, \cdots, a_n)$ or neither $g_{1}(a_1, \cdots, a_n)$ nor $g_2(a_1, \cdots, a_n)$ is well-defined.
    \end{enumerate}
\end{proposition}
\begin{proof}
    Obvious.
\end{proof}
Definitions \ref{5.3.a} and \ref{9.4.a} are generalized to
\begin{definition} \label{10.4.a}
    Let $\theta$ and $A$ be defined as in Notation \ref{10.4.1}. Let $f\in \operatorname{Dom}\theta $, $n\in {{\mathbb{Z}}^{+}}$ and $(a_1, \cdots, a_n) \in A^n$. If $g_{1}(a_1, \cdots, a_n)=g_2(a_1, \cdots, a_n)$, $\forall(f,g_1),(f,g_2)\in \theta$, then we say that $\theta (f)(a_1, \cdots, a_n)$ is \emph{well-defined} and let $\theta (f)(a_1, \cdots, a_n)=g_{1}(a_1, \cdots, a_n)$. 
\end{definition}
\begin{remark}
    By Convention \ref{convention}, 
    \[g_{1}(a_1, \cdots, a_n)=g_2(a_1, \cdots, a_n), \forall(f,g_1),(f,g_2)\in \theta\]
    implies that $\forall(f,g)\in \theta$, $g$ has $n$ variables and $g(a_1, \cdots, a_n)$ is well-defined.
\end{remark}
Then by Proposition \ref{10.4.b}, we immediately get
\begin{corollary} \label{11.5.4} 
    Let $\theta$ and $A$ be defined as in Notation \ref{10.4.1}. Suppose that $\theta {{|}_{A}}$ is a map. Let $(f,g)\in \theta$. If $\exists n\in {{\mathbb{Z}}^{+}}$ and $(a_1, \cdots, a_n) \in A^n$ such that $g(a_1, \cdots, a_n)$ is well-defined, then $\theta(f)(a_1, \cdots, a_n)$ is well-defined and \[\theta(f)(a_1, \cdots, a_n)=g(a_1, \cdots, a_n).\]
\end{corollary}
Definition \ref{9.4.2} is generalized to
\begin{definition} \label{10.4.2}
    Let $T_1$ and $T_2$ be par-operator gen-semigroups and let $\theta \subseteq {{T}_{1}}\times {{T}_{2}}$. Let $\phi $ be a map from a $T_1$-space $S$ to a $T_2$-space. If 
    \begin{enumerate}
        \item [(i)] $\theta {{|}_{\operatorname{Im}\phi }}$ is a map and
         \item [(ii)] $\forall n\in {{\mathbb{Z}}^{+}}$, $(a_1,\cdots ,a_n)\in {{S}^{n}}$ and $f\in \operatorname{Dom}\theta $, neither $f(a_1,\cdots ,a_n)$ nor $\theta (f)(\phi ({{a}_{1}}),\cdots ,\phi ({{a}_{n}}))$ is well-defined or
    \[\phi (f(a_1,\cdots ,a_n))=\theta (f)(\phi ({{a}_{1}}),\cdots ,\phi ({{a}_{n}})),\]
    \end{enumerate}
    then $\phi $ is called a \emph{$\theta$-morphism}.
\end{definition}
\begin{remark}
     Suppose that $f(a_1,\cdots ,a_n)$ is well-defined, where $(a_1,\cdots ,a_n) \in S^n$ and $f\in \operatorname{Dom}\theta $. Then by condition (ii) in the definition and Convention \ref{convention}, $\theta (f)(\phi ({{a}_{1}}),\cdots ,\phi ({{a}_{n}}))$ is also well-defined. Hence $\forall (f,g)\in \theta $, by Definition \ref{10.4.a}, $g$ has $n$ variables.
\end{remark}
The following explains why, unlike Definition \ref{9.4.2}, we put condition (i) in Definition \ref{10.4.2} explicitly.
\begin{proposition} \label{10.4.c}
    Let $\theta$, $S$ and $\phi$ be given as in Definition \ref{10.4.2}. Suppose that condition \emph{(ii)} in Definition \ref{10.4.2} is satisfied. Then $\theta {{|}_{\operatorname{Im}\phi }}$ is possibly not a map.
 \end{proposition}
\begin{proof} 
By Proposition \ref{10.4.b}, $\theta {{|}_{\operatorname{Im}\phi }}$ is a map if and only if $\forall(f,g_1),(f,g_2)\in \theta$, $n\in {{\mathbb{Z}}^{+}}$ and $(a_1,\cdots ,a_n)\in {{S}^{n}}$, $g_1(\phi ({{a}_{1}}),\cdots ,\phi ({{a}_{n}}))=g_2(\phi ({{a}_{1}}),\cdots ,\phi ({{a}_{n}}))$ or neither $g_1(\phi ({{a}_{1}}),\cdots ,\phi ({{a}_{n}}))$ nor $g_2(\phi ({{a}_{1}}),\cdots ,\phi ({{a}_{n}}))$ is well-defined. We can tell that this condition for $\theta {{|}_{\operatorname{Im}\phi }}$ being a map is not guaranteed by condition (ii) in Definition \ref{10.4.2}. For example, condition (ii) in Definition \ref{10.4.2} cannot rule out the possibility where there exist $(f,g_1),(f,g_2)\in \theta$, $n\in {{\mathbb{Z}}^{+}}$ and $(a_1,\cdots ,a_n)\in {{S}^{n}}$ such that only one of $g_1(\phi ({{a}_{1}}),\cdots ,\phi ({{a}_{n}}))$ and $g_2(\phi ({{a}_{1}}),\cdots ,\phi ({{a}_{n}}))$ is well-defined (because in this case, it is possible that neither $f(a_1,\cdots ,a_n)$ nor $\theta (f)(\phi ({{a}_{1}}),\cdots ,\phi ({{a}_{n}}))$ is well-defined).
 \end{proof}

When par-operator gen-semigroups are induced from formal partial functions, the following lemma, which generalizes both Lemmas \ref{9.4.3} and \ref{10.3.3}, will be useful. 
\begin{lemma}\label{10.4.3}
    Let $\mathcal{F}$ $($resp. ${\mathcal{F}}'$$)$ be a set of D-valued $($resp. ${D}'$-valued$)$ formal partial functions on $D$ $($resp. ${D}'$$)$. Let $T$ $($resp. ${T}'$$)$ be a par-operator gen-semigroup on $D$ $($resp. ${D}'$$)$ such that there is a map $\tau :\mathcal{F}\to T$ being given by $f\mapsto f^*$, where $f^*:{{D}^{n}}\to D$ is given by $(a_1,\cdots ,a_n)\mapsto f(a_1,\cdots ,a_n)$ for $f$ of $n$ variables, $\forall n\in {{\mathbb{Z}}^{+}}$ $($resp. such that there is a map ${\tau }':{\mathcal{F}}'\to {T}'$ being given analogously$)$. Let $\vartheta :\mathcal{F}\to {\mathcal{F}}'$ be a map. Let $\theta =\{(\tau (f),{\tau }'(\vartheta (f)))\,|\,f\in \mathcal{F}\}$ and let $\phi $ be a map from a $T$-space $S$ to a ${T}'$-space. Then the following statements are equivalent:
\begin{enumerate}
    \item [(i)] $\phi $ is a $\theta$-morphism.
    \item[(ii)]	$\forall n\in {{\mathbb{Z}}^{+}}$, $(a_1,\cdots ,a_n)\in {{S}^{n}}$ and $f\in \mathcal{F}$ of $n$ variables, neither $f(a_1,\cdots ,a_n)$ nor $\vartheta (f)(\phi ({{a}_{1}}),\cdots ,\phi ({{a}_{n}}))$ is well-defined or 
\begin{equation} \label{11.3}
  \phi (f(a_1,\cdots ,a_n))=\vartheta (f)(\phi ({{a}_{1}}),\cdots ,\phi ({{a}_{n}})).  
\end{equation}
\end{enumerate}
\end{lemma}
\begin{proof}
Clearly, $\theta \subseteq T \times T'$.

(i) $\Rightarrow$ (ii): 
By Definition \ref{10.4.2}, $\theta {{|}_{\operatorname{Im}\phi }}$ is a map and $\forall n\in {{\mathbb{Z}}^{+}}$, $(a_1,\cdots ,a_n)\in {{S}^{n}}$ and $f\in \mathcal{F}$ of $n$ variables, neither $\tau (f)(a_1,\cdots ,a_n)$ nor $\theta (\tau (f))(\phi ({{a}_{\text{1}}}),\cdots ,\phi ({{a}_{n}}))$ is well-defined or 
 \[\phi (\tau (f)(a_1,\cdots ,a_n))=\theta (\tau (f))(\phi ({{a}_{\text{1}}}),\cdots ,\phi ({{a}_{n}})).\] 

Let $n\in {{\mathbb{Z}}^{+}}$, let $(a_1,\cdots ,a_n)\in {{S}^{n}}$ and let $f\in \mathcal{F}$ have $n$ variables. Then

$f(a_1,\cdots ,a_n)$ is well-defined;

$\Leftrightarrow \tau (f)(a_1,\cdots ,a_n)$ is well-defined; 

$\Leftrightarrow \theta (\tau (f))(\phi ({{a}_{\text{1}}}),\cdots ,\phi ({{a}_{n}}))$ is well-defined;

$\Leftrightarrow {\tau }'(\vartheta (f))(\phi ({{a}_{1}}),\cdots ,\phi ({{a}_{n}}))$ is well-defined (because $(\tau (f),{\tau }'(\vartheta (f)))\in \theta $ and $\theta {{|}_{\operatorname{Im}\phi }}$ is a map); 

$\Leftrightarrow \vartheta (f)(\phi ({{a}_{1}}),\cdots ,\phi ({{a}_{n}}))$ is well-defined. 

And in the case where the above equivalent conditions are true,

$\phi (f(a_1,\cdots ,a_n))$

$=\phi (\tau (f)(a_1,\cdots ,a_n))$  

$=\theta (\tau (f))(\phi ({{a}_{\text{1}}}),\cdots ,\phi ({{a}_{n}}))$ 

$={\tau }'(\vartheta (f))(\phi ({{a}_{\text{1}}}),\cdots ,\phi ({{a}_{n}}))$ (because $(\tau (f),{\tau }'(\vartheta (f)))\in \theta $ and $\theta {{|}_{\operatorname{Im}\phi }}$ is a map)

$=\vartheta (f)(\phi ({{a}_{\text{1}}}),\cdots ,\phi ({{a}_{n}}))$, as desired.

(ii) $\Rightarrow$ (i): 
By Definition \ref{10.4.2}, to show that $\phi $ is a $\theta$-morphism, we first need to show that $\theta {{|}_{\operatorname{Im}\phi }}$ is a map. 

Let $(h,g{{|}_{\operatorname{Im}\phi }}),(h,{g}'{{|}_{\operatorname{Im}\phi }})\in \theta {{|}_{\operatorname{Im}\phi }}$ with $(h,g),(h,{g}')\in \theta $. By the definition of $\theta $, $\exists l,m\in \mathcal{F}$ such that $(h,g)=(\tau (l),{\tau }'(\vartheta (l)))$ and $(h,{g}')=(\tau (m),{\tau }'(\vartheta (m)))$, and thus $\tau (l)=\tau (m)$. Hence $l$ and $m$ have the same number of variables, assumed $n$. Then by Equation (\ref{11.3}) and Convention \ref{convention}, $\vartheta (l)$ also has $n$ variables if $l$ is well-defined somewhere in $S^n$. Analogously, $\vartheta (m)$ also has $n$ variables if $m$ is well-defined somewhere in $S^n$.

Assume that $g{{|}_{\operatorname{Im}\phi }}\ne {g}'{{|}_{\operatorname{Im}\phi }}$; that is, ${\tau }'(\vartheta (l)){{|}_{\operatorname{Im}\phi }}\ne {\tau }'(\vartheta (m)){{|}_{\operatorname{Im}\phi }}$. Then $\exists k \in {{\mathbb{Z}}^{+}}$ and $({{z}_{1}},\cdots ,{{z}_{k}})\in {{S}^{k}}$ such that only one of $\vartheta (l)(\phi ({{z}_{1}}),\cdots ,\phi ({{z}_{k}}))$ and $\vartheta (m)(\phi ({{z}_{1}}),\cdots ,\phi ({{z}_{k}}))$ is well-defined or both of them are well-defined but
\[\vartheta (l)(\phi ({{z}_{1}}),\cdots ,\phi ({{z}_{k}}))\ne \vartheta (m)(\phi ({{z}_{1}}),\cdots ,\phi ({{z}_{k}})).\] 
By statement (ii), if $\vartheta (l)(\phi ({{z}_{1}}),\cdots ,\phi ({{z}_{k}}))$ (resp. $\vartheta (m)(\phi ({{z}_{1}}),\cdots ,\phi ({{z}_{k}}))$) is well-defined, then $l(z_1,\cdots ,z_k)$ (resp. $m(z_1,\cdots ,z_k))$) is also well-defined. Hence $k=n$.

However, $\tau (l)=\tau (m)$, and hence neither $l({{z}_{1}},\cdots ,{{z}_{n}})$ nor $m({{z}_{1}},\cdots ,{{z}_{n}})$ is well-defined or both of them are well-defined and $l({{z}_{1}},\cdots ,{{z}_{n}})=m({{z}_{1}},\cdots ,{{z}_{n}})$. It follows from (ii) that neither $\vartheta (l)(\phi ({{z}_{1}}),\cdots ,\phi ({{z}_{n}}))$ nor $\vartheta (m)(\phi ({{z}_{1}}),\cdots ,\phi ({{z}_{n}}))$ is well-defined or both of them are well-defined and
\[\vartheta (l)(\phi ({{z}_{1}}),\cdots ,\phi ({{z}_{n}}))=\vartheta (m)(\phi ({{z}_{1}}),\cdots ,\phi ({{z}_{n}})),\] 
a contradiction. 

Hence $\theta {{|}_{\operatorname{Im}\phi }}$ must be a map. 

Moreover, by statement (ii), $\forall n\in {{\mathbb{Z}}^{+}}$, $(a_1,\cdots ,a_n)\in {{S}^{n}}$ and $f\in \mathcal{F} $ of $n$ variables, there are only two cases as follows.

\emph{Case} (1): Neither $f(a_1,\cdots ,a_n))$ nor $\vartheta (f)(\phi ({{a}_{1}}),\cdots ,\phi ({{a}_{n}}))$ is well-defined, and hence neither $\tau (f)(a_1,\cdots ,a_n)$ nor $\tau'(\vartheta (f))(\phi ({{a}_{1}}),\cdots ,\phi ({{a}_{n}}))$ is well-defined. Because $(\tau (f),{\tau }'(\vartheta (f)))\in \theta $, by Definition \ref{10.4.a}, neither $\tau (f)(a_1,\cdots ,a_n)$ nor $\theta (\tau (f))(\phi ({{a}_{1}}),\cdots ,\phi ({{a}_{n}}))$ is well-defined.

\emph{Case} (2):

$\phi (\tau(f)(a_1,\cdots ,a_n))$

$=\phi (f(a_1,\cdots ,a_n))$

$=\vartheta (f)(\phi ({{a}_{1}}),\cdots ,\phi ({{a}_{n}}))$ (by Equation (\ref{11.3}))

$={\tau }'(\vartheta (f))(\phi ({{a}_{1}}),\cdots ,\phi ({{a}_{n}}))$

$=\theta (\tau (f))(\phi ({{a}_{1}}),\cdots ,\phi ({{a}_{n}}))$ (because $(\tau (f),{\tau }'(\vartheta (f)))\in \theta $ and $\theta {{|}_{\operatorname{Im}\phi }}$ is a map).

Therefore, $\forall n\in {{\mathbb{Z}}^{+}}$, $(a_1,\cdots ,a_n)\in {{S}^{n}}$ and $f\in \mathcal{F} $ of $n$ variables, neither $\tau (f)(a_1,\cdots ,a_n)$ nor $\theta (\tau (f))(\phi ({{a}_{1}}),\cdots ,\phi ({{a}_{n}}))$ is well-defined or
\[\phi (\tau(f)(a_1,\cdots ,a_n))=\theta (\tau (f))(\phi ({{a}_{1}}),\cdots ,\phi ({{a}_{n}})).\]
Since $\operatorname{Dom}\theta=\{\tau(f)\,|\,f\in \mathcal{F}\}$, by Definition \ref{10.4.2}, $\phi $ is a $\theta$-morphism.
\end{proof}
	Lemma \ref{10.4.3} characterizes a $\theta$-morphism by formal partial functions. When par-operator gen-semigroups are induced from formal partial functions, the criterion for a $\theta$-morphism by Lemma \ref{10.4.3} is more convenient than that by Definition \ref{10.4.2}. 

The following, which generalizes Proposition \ref{10.3.4}, is analogous to Proposition \ref{9.4.4}, except that now $T$ is defined by Example \ref{10.1.5}, and hence ${{\left\langle U \right\rangle }_{T}}$ and ${{\left\langle V \right\rangle }_{T}}$ must be fields. 
\begin{proposition} \label{10.4.4}
    Let $F/B$ (resp. ${F}'/{B}'$) be a field extension, let a par-operator gen-semigroup $T$ $($resp. ${T}'$$)$ be defined as in Example \ref{10.1.5} on $F$ $($resp. ${F}'$$)$ over $B$ $($resp. ${B}'$$)$, let $\varphi :B\to {B}'$ be a field isomorphism, and let 
    \[\vartheta :\bigcup\nolimits_{n\in {{\mathbb{Z}}^{+}}}{\emph{Frac}(B[{{x}_{1}},\cdots ,{{x}_{n}}])}\to \bigcup\nolimits_{n\in {{\mathbb{Z}}^{+}}}{\emph{Frac}({B}'[{{x}_{1}},\cdots ,{{x}_{n}}])}\]
    be the map given by $f\mapsto {f}'$ where
\begin{center}
    $f({{x}_{1}},\cdots ,{{x}_{n}})=(\sum{{c_{p_1,\cdots ,p_n}}x_{1}^{{{p}_{1}}}\cdots x_{n}^{{{p}_{n}}}})/(\sum{{c'_{p_1,\cdots ,p_n}}x_{1}^{{{p}_{1}}}\cdots x_{n}^{{{p}_{n}}}})$ 
 \end{center}
    and
 \begin{center}
 ${f}'({{x}_{1}},\cdots ,{{x}_{n}})=(\sum{\varphi ({c_{p_1,\cdots ,p_n}})x_{1}^{{{p}_{1}}}\cdots x_{n}^{{{p}_{n}}}})/(\sum{\varphi ({c'_{p_1,\cdots ,p_n}})x_{1}^{{{p}_{1}}}\cdots x_{n}^{{{p}_{n}}}})$.
 \end{center}
Let $\theta =\{(\tau (f),{\tau }'(\vartheta (f))\,|\,f\in \emph{Frac}(B[{{x}_{1}},\cdots ,{{x}_{n}}]),n\in {{\mathbb{Z}}^{+}}\}$, where $\tau $ and ${\tau }'$ are the maps inducing $T$ and ${T}'$, respectively, defined as in Example \ref{10.1.5}. 

Then $\forall (\emptyset \ne )U\subseteq F$ and $(\emptyset \ne )V\subseteq {F}'$, a map $\phi :{{\left\langle U \right\rangle }_{T}}\to {{\left\langle V \right\rangle }_{{{T}'}}}$ is a ring homomorphism extending $\varphi $ if and only if $\phi $ is a $\theta$-morphism.
\end{proposition}
\begin{proof}
    By the definition of $T$ in Example \ref{10.1.5}, 
    \begin{center}
        $T=\{$the partial function $f^*:{{F}^{n}}\to F$ given by
$(a_1,\cdots ,a_n)\mapsto f(a_1,\cdots ,a_n)\, |\,f\in \text{Frac}(B[{{x}_{1}},\cdots ,{{x}_{n}}]),n\in {{\mathbb{Z}}^{+}}\}$. 
    \end{center}
and
\begin{center}
    ${T}'=\{$the partial function $f^*:{{{F}'}^{n}}\to {F}'$ given by
$(a_1,\cdots ,a_n)\mapsto f(a_1,\cdots ,a_n)\, |\,f\in \text{Frac}({B}'[{{x}_{1}},\cdots ,{{x}_{n}}]),n\in {{\mathbb{Z}}^{+}}\}$. 
\end{center}

Thus ${{\left\langle U \right\rangle }_{T}}$ and ${{\left\langle V \right\rangle }_{{{T}'}}}$ are the fields $B(U)$ and ${B}'(V)$, respectively. 

	Let $\mathcal{F}=\bigcup\nolimits_{n\in {{\mathbb{Z}}^{+}}}{\text{Frac}(B[{{x}_{1}},\cdots ,{{x}_{n}}])}$ and ${\mathcal{F}}'=\bigcup\nolimits_{n\in {{\mathbb{Z}}^{+}}}{\text{Frac}({B}'[{{x}_{1}},\cdots ,{{x}_{n}}])}$. Then Lemma \ref{10.4.3} applies. Hence it suffices to show that $\phi :{{\left\langle U \right\rangle }_{T}}\to {{\left\langle V \right\rangle }_{{{T}'}}}$ is a ring homomorphism extending $\varphi $ if and only if statement (ii) in Lemma \ref{10.4.3} is true in this case. Specifically, it is sufficient to show that the following statements are equivalent:

\begin{enumerate}
    \item $\forall c\in B$, $\phi (c)=\varphi (c)$, and $\forall a,b\in {{\left\langle U \right\rangle }_{T}}$, $\phi (a+b)=\phi (a)+\phi (b)$ and $\phi (ab)=\phi (a)\phi (b)$.
\item $\forall n\in {{\mathbb{Z}}^{+}}$, $(a_1,\cdots ,a_n)\in \left\langle U \right\rangle _{T}^{n}$ and $f\in \text{Frac}(B[{{x}_{1}},\cdots ,{{x}_{n}}])$, neither $f(a_1,\cdots ,a_n)$ nor $\vartheta (f)(\phi ({{a}_{1}}),\cdots ,\phi ({{a}_{n}}))$ is well-defined or
\begin{equation} \label{11.4}
    \phi (f(a_1,\cdots ,a_n))=\vartheta (f)(\phi ({{a}_{1}}),\cdots ,\phi ({{a}_{n}})).
\end{equation}
\end{enumerate}
(1) $\Rightarrow$ (2): 
Let $f({{x}_{1}},\cdots ,{{x}_{n}})=g({{x}_{1}},\cdots ,{{x}_{n}})/h({{x}_{1}},\cdots ,{{x}_{n}})$, where $n\in {\mathbb{Z}}^{+}$ and $h,g\in B[{{x}_{1}},\cdots ,{{x}_{n}}]$. Let $(a_1,\cdots ,a_n)\in \left\langle U \right\rangle _{T}^{n}$. Then

$f(a_1,\cdots ,a_n)$ is well-defined; 

$\Leftrightarrow h(a_1,\cdots ,a_n)\ne 0$; 

$\Leftrightarrow \vartheta (h)(a_1,\cdots ,a_n)\ne 0$ (by the definition of $\vartheta $); 

$\Leftrightarrow \vartheta (f)(\phi ({{a}_{1}}),\cdots ,\phi ({{a}_{n}}))$ 

\hspace{0.5cm}$=\vartheta (g)(\phi ({{a}_{1}}),\cdots ,\phi ({{a}_{n}}))/\vartheta (h)(\phi ({{a}_{1}}),\cdots ,\phi ({{a}_{n}}))$ (by the definition of $\vartheta $) 

\hspace{0.5cm}is well-defined.

Moreover, in the case where the above equivalent conditions are true, 

$\phi (f(a_1,\cdots ,a_n))$ 

$=\phi (\sum{{c_{p_1,\cdots ,p_n}}a_{1}^{{{p}_{1}}}\cdots a_{n}^{{{p}_{n}}}}/\sum{{c'_{p_1,\cdots ,p_n}}a_{1}^{{{p}_{1}}}\cdots a_{n}^{{{p}_{n}}}})$ 

$=\sum{\phi ({c_{p_1,\cdots ,p_n}})(\phi (a_1))^{p_1}\cdots (\phi (a_n))^{p_n}}/\sum{\phi ({c'_{p_1,\cdots ,p_n}})(\phi (a_1))^{p_1}\cdots (\phi (a_n))^{p_n}}$ (because statement (1) is true)

$=\sum{\varphi ({c_{p_1,\cdots ,p_n}})(\phi (a_1))^{p_1}\cdots (\phi (a_n))^{p_n}}/\sum{\varphi ({c'_{p_1,\cdots ,p_n}})(\phi (a_1))^{p_1}\cdots (\phi (a_n))^{p_n}}$ (because $\forall c\in B$, $\phi (c)=\varphi (c)$)

$=\vartheta (f)(\phi ({{a}_{1}}),\cdots ,\phi ({{a}_{n}}))$ (by the definition of $\vartheta $),

as desired.

(2) $\Rightarrow$ (1): 
Considering the case where $f\in \text{Frac}(B[{{x}_{1}},\cdots ,{{x}_{n}}])$ is a constant polynomial, we can tell from Equation $($\ref{11.4}$)$ and the definition of $\vartheta$ that $\forall c\in B$, $\phi (c)=\varphi (c)$. 

Let $a,b\in {{\left\langle U \right\rangle }_{T}}$. Let $g={{x}_{1}}+{{x}_{2}}$ and $h={{x}_{1}}{{x}_{2}}$.

Then 

$\phi (a+b)$

$=\phi (g(a,b))$

$=\vartheta (g)(\phi (a),\phi (b))$ (by Equation $($\ref{11.4}$)$)

$=g(\phi (a),\phi (b))$ ($\vartheta (g)=g$ by the definition of $\vartheta $)

$=\phi (a)+\phi (b)$. 

And 

$\phi (ab)$

$=\phi (h(a,b))$

$=\vartheta (h)(\phi (a),\phi (b))$ (by Equation $($\ref{11.4}$)$)

$=h(\phi (a),\phi (b))$ ($\vartheta (h)=h$ by the definition of $\vartheta $)

$=\phi (a)\phi (b)$.
\end{proof}
	We are going to show that a covariant functor can be characterized by a $\theta$-morphism. Because in every category, there is a bijection between objects $A$ and their identity morphisms $1_A$, we view $D$ as $\mathcal{C}$ and ${D}'$ as ${\mathcal{C}}'$ in the following. 
\begin{proposition} \label{10.4.5}
    Let $\mathcal{C}$ (resp. ${\mathcal{C}}'$) be a category, let $D$ $($resp. ${D}'$$)$ be the collection of all morphisms in $\mathcal{C}$ $($resp. ${\mathcal{C}}'$$)$, and let ${{\mathcal{F}}_{n}}$, $\forall n\in {{\mathbb{Z}}^{+}}$, and the par-operator gen-semigroup $T$ $($resp. ${T}'$$)$ on $D$ $($resp. ${D}'$$)$ be defined as in Example \ref{10.1.7}. Let $\theta =\{(\tau (f),{\tau }'(f)\,|\,f\in \bigcup\nolimits_{n\in {{\mathbb{Z}}^{+}}}{{{\mathcal{F}}_{n}}}\}$, where $\tau $ and ${\tau }'$ are the maps inducing $T$ and ${T}'$, respectively, defined as in Example \ref{10.1.7}. 
    
    Suppose that a function $\phi :D\to {D}'$ maps each identity morphism in $D$ to an identity morphism in ${D}'$. Then $\phi $ is a covariant functor if and only if it is a $\theta$-morphism.
\end{proposition}
\begin{proof}
    By Example \ref{10.1.7}, 
    \begin{center}
        $T=\{f^*:{{D}^{n}}\to D$ given by $(a_1,\cdots ,a_n)\mapsto f(a_1,\cdots ,a_n)\, |\,f\in {{\mathcal{F}}_{n}},n\in {{\mathbb{Z}}^{+}}\}$,
    \end{center}
and
\begin{center}
   ${T}'=\{f^*:{{{D}'}^{n}}\to {D}'$ given by $(a_1,\cdots ,a_n)\mapsto f(a_1,\cdots ,a_n)\, |\,f\in {{\mathcal{F}}_{n}},n\in {{\mathbb{Z}}^{+}}\}$. 
\end{center}
 
    By Example \ref{10.2.4}, $D$ and $D'$ are $T$-space and $T'$-space, respectively. Let $\mathcal{F}={\mathcal{F}}'=\bigcup\nolimits_{n\in {{\mathbb{Z}}^{+}}}{{{\mathcal{F}}_{n}}}$ and let $\vartheta :\mathcal{F}\to {\mathcal{F}}'$ be the identity function. Then Lemma \ref{10.4.3} applies. Hence it suffices to show that $\phi $ is a covariant functor if and only if statement (ii) in Lemma \ref{10.4.3} is true in this case. Specifically, we only need to show that the following statements are equivalent:
\begin{enumerate}
    \item $\forall X\in \text{obj}(\mathcal{C})$, $\phi ({{1}_{X}})={{1}_{\phi (X)}}$; 
if $m:A\to B$ in $\mathcal{C}$, then $\phi (m):\phi (A)\to \phi (B)$ in ${\mathcal{C}}'$; and
if $A\xrightarrow{m}B\xrightarrow{h}C$ in $\mathcal{C}$, then $\phi (A)\xrightarrow{\phi (m)}\phi (B)\xrightarrow{\phi (h)}\phi (C)$ in ${\mathcal{C}}'$ and $\phi (h\circ m)=\phi (h)\circ \phi (m)$.
\item $\forall n\in {{\mathbb{Z}}^{+}}$, $(a_1,\cdots ,a_n)\in {{D}^{n}}$ and $f\in \mathcal{F}$ of $n$ variables, neither $f(a_1,\cdots ,a_n)$ nor $f(\phi ({{a}_{1}}),\cdots ,\phi ({{a}_{n}}))$ is well-defined or
\begin{equation} \label{11.5}
    \phi (f(a_1,\cdots ,a_n))=f(\phi ({{a}_{1}}),\cdots ,\phi ({{a}_{n}})).
\end{equation}
\end{enumerate}
(1) $\Rightarrow$ (2): 
Let $(a_1,\cdots ,a_n)\in {{D}^{n}}$ and let $f={{x}_{{{i}_{1}}}}\cdots {{x}_{{{i}_{k}}}}\in {{\mathcal{F}}_{n}}$ (\textit{cf.} Example \ref{10.1.7}). Then

$f(a_1,\cdots ,a_n)$ is well-defined;

$\Leftrightarrow {{a}_{{{i}_{1}}}}\circ \cdots \circ {{a}_{{{i}_{k}}}}$ is well-defined (by the definition of $f(a_1,\cdots ,a_n)$ in Example \ref{10.1.7});

$\Leftrightarrow \phi ({{a}_{{{i}_{1}}}})\circ \cdots \circ \phi ({{a}_{{{i}_{k}}}})$ is well-defined (explained below);

$\Leftrightarrow f(\phi ({{a}_{1}}),\cdots ,\phi ({{a}_{n}}))$ is well-defined (by the definition of $f(a_1,\cdots ,a_n)$ in Example \ref{10.1.7}).

	The second equivalence is explained as follows. $D$ (resp. ${D}'$) is the collection of all morphisms in category $\mathcal{C}$ (resp. ${\mathcal{C}}'$), and hence $\forall i=1,\cdots,n$, ${{a}_{i}}$ (resp. $\phi ({{a}_{i}})$) is a morphism in $\mathcal{C}$ (resp. ${\mathcal{C}}'$). Then ${{a}_{{{i}_{1}}}}\circ \cdots \circ {{a}_{{{i}_{k}}}}$ is well-defined if and only if $\forall j=1,\cdots,(k-1)$, the domain of ${{a}_{{{i}_{j}}}}$ is the target of ${{a}_{{{i}_{j+1}}}}$. And this occurs if and only if $\phi ({{a}_{{{i}_{1}}}})\circ \cdots \circ \phi ({{a}_{{{i}_{k}}}})$ is well-defined because $\phi $ is a covariant functor. 

	Therefore, $f(a_1,\cdots ,a_n)$ is well-defined if and only if $f(\phi ({{a}_{1}}),\cdots ,\phi ({{a}_{n}}))$ is well-defined.
 
Moreover, in the case where the above equivalent conditions are satisfied, since $\phi $ is a covariant functor,
\[\phi ({{a}_{{{i}_{1}}}}\circ \cdots \circ {{a}_{{{i}_{k}}}})=\phi ({{a}_{{{i}_{1}}}})\circ \cdots \circ \phi ({{a}_{{{i}_{k}}}});\] 
that is (by the definition of $f(a_1,\cdots ,a_n)$ in Example \ref{10.1.7}),
\[\phi (f(a_1,\cdots ,a_n))=f(\phi ({{a}_{1}}),\cdots ,\phi ({{a}_{n}})).\]

(2) $\Rightarrow$ (1): 
Firstly, we view $D$ (resp. ${D}'$) as $\mathcal{C}$ (resp. ${\mathcal{C}}'$). That is, $\forall X\in \text{obj}(\mathcal{C})$, ${{1}_{X}}\in D$ may be regarded as $X$, and vice versa, and so do ${\mathcal{C}}'$ and ${D}'$. It is supposed that $\phi $ maps each identity morphism in $D$ to an identity morphism in ${D}'$, and hence with a slight abuse of notation, $\forall X\in \text{obj}(\mathcal{C})$, $\phi (X)\in \text{obj}({\mathcal{C}}')$, more precisely, $\phi ({{1}_{X}})={{1}_{\phi (X)}}$. 

	Secondly, let $m:A\to B$ be a morphism in $\mathcal{C}$. Let $g\in \mathcal{F}$ be ${{x}_{1}}{{x}_{2}}{{x}_{3}}$. Then 

$\phi (m)$

$=\phi ({{1}_{B}}\circ m\circ {{1}_{A}})$ 

$=\phi (g({{1}_{B}},m,{{1}_{A}}))$ (by the definition of $f(a_1,\cdots ,a_n)$ in Example \ref{10.1.7})

$=g(\phi ({{1}_{B}}),\phi (m),\phi ({{1}_{A}}))$ (by Equation $($\ref{11.5}$)$)

$=\phi ({{1}_{B}})\circ \phi (m)\circ \phi ({{1}_{A}})$ (by the definition of $f(a_1,\cdots ,a_n)$ in Example \ref{10.1.7}).

Hence $\phi (m):\phi (A)\to \phi (B)$ in ${\mathcal{C}}'$ since $\phi ({{1}_{A}})={{1}_{\phi (A)}}$ and $\phi ({{1}_{B}})={{1}_{\phi (B)}}$.

Thirdly, if $A\xrightarrow{m}B\xrightarrow{h}C$ in $\mathcal{C}$, then again by the above argument, $\phi (A)\xrightarrow{\phi (m)}\phi (B)\xrightarrow{\phi (h)}\phi (C)$ in ${\mathcal{C}}'$. Moreover, let $f\in \mathcal{F}$ be ${{x}_{1}}{{x}_{2}}$. Then by the definition of $f(a_1,\cdots ,a_n)$ in Example \ref{10.1.7} and Equation $($\ref{11.5}$)$,
\[\phi (h\circ m)=\phi (f(h,m))=f(\phi (h),\phi (m))=\phi (h)\circ \phi (m).\]
\end{proof}

\subsection{$\theta$-isomorphisms} \label{11 $theta$-isomorphisms}
To $T_1$ and $T_2$ as par-operator gen-semigroups, Definition \ref{5.4.a} still applies: a $\theta$-morphism from a $T_1$-space $S_1$ to a $T_2$-space is a \emph{$\theta$-isomorphism} if it is bijective. To justify the definition, we need to show Lemma \ref{L6.2.2} and Proposition \ref{5.4.b} for the case of par-operator gen-semigroups as follows.
\begin{lemma} \label{L11.6.1} \emph{(Lemma \ref{L6.2.2})}
    Let $\phi $ be a $\theta $-isomorphism from a $T_1$-space $S_1$ to a $T_2$-space $S_2$. Then $\theta^{-1} {{|}_{S_1}}$ is a map, where $\theta^{-1}:=\{(g,f)\,|\,(f,g) \in \theta\}$.
\end{lemma}
\begin{proof}
    Let $(g,f_1),(g,f_2)\in \theta^{-1}$, $n \in \mathbb Z^+$ and $(a_1, \cdots, a_n) \in S_1^n$.
    Then to show that $\theta^{-1} {{|}_{S_1}}$ is a map, by Proposition \ref{10.4.b}, it suffices to show that neither $f_1(a_1, \cdots, a_n)$ nor $f_2(a_1, \cdots, a_n)$ is well-defined or
    $f_1(a_1, \cdots, a_n)=f_2(a_1, \cdots, a_n)$.
    
     $f_1(a_1, \cdots, a_n)$ is well-defined;\\
     $\Leftrightarrow$ $\theta(f_1)(\phi (a_1), \cdots, \phi(a_n))$ is well-defined (by Definition  \ref{10.4.2}) ;\\
    $\Leftrightarrow$ $g(\phi (a_1), \cdots, \phi(a_n))$ is well-defined (because $(f_1,g)\in \theta$  and  $\theta {|}_{\operatorname{Im}\phi}$  is a map by Definition  \ref{10.4.2}); \\ 
    $\Leftrightarrow$  $\theta(f_2)(\phi (a_1), \cdots, \phi(a_n))$ is well-defined  (because $(f_2,g)\in \theta$  and  $\theta {|}_{\operatorname{Im}\phi}$  is a map);\\
     $\Leftrightarrow$  $f_2(a_1, \cdots, a_n)$ is well-defined (by Definition  \ref{10.4.2}).
    
    And in the case where the above equivalent conditions are satisfied,\\    
     $\phi(f_1(a_1, \cdots, a_n))$\\
     $=\theta(f_1)(\phi (a_1), \cdots, \phi(a_n))$  (by Definition  \ref{10.4.2})\\
    $=g(\phi (a_1), \cdots, \phi(a_n))$ \\ 
    $=\theta(f_2)(\phi (a_1), \cdots, \phi(a_n))$ \\
    $=\phi(f_2(a_1, \cdots, a_n))$  (by Definition  \ref{10.4.2}),\\
    and hence  $f_1(a_1, \cdots, a_n)=f_2(a_1, \cdots, a_n)$ (because $\phi$ is injective), as desired.
\end{proof}
\begin{proposition} \emph{(Proposition \ref{5.4.b})}
    Let $S_1$ and $S_2$ be a $T_1$-space and a $T_2$-space, respectively, let $\phi \in \operatorname{Iso}_{\theta}({{S}_{1}},{{S}_{2}})$ and let ${{\phi }^{-1}}$ be the inverse map of $\phi $. Then ${{\phi }^{-1}}\in \operatorname{Iso}_{\theta^{-1}}({{S}_{2}},{{S}_{1}})$.
\end{proposition}
\begin{proof}
       By Definition \ref{10.4.2}, $\theta {{|}_{S_2}}$ is a map and $\forall n\in {{\mathbb{Z}}^{+}}$, $(a_1,\cdots ,a_n)\in {S_1^n}$ and $f\in \operatorname{Dom}\theta $, neither $f(a_1,\cdots ,a_n)$ nor $\theta (f)(\phi ({{a}_{1}}),\cdots ,\phi ({{a}_{n}}))$ is well-defined or
 \[f(a_1,\cdots ,a_n)=\phi^{-1} (\theta (f)(\phi ({{a}_{1}}),\cdots ,\phi ({{a}_{n}}))).\] 

Hence we have

\emph{Fact} (A): $\forall n\in {{\mathbb{Z}}^{+}}, f\in \operatorname{Dom}\theta $ and $ (b_1,\cdots ,b_n)\in {S_2^n}$,\\
neither $f({\phi }^{-1}({b}_{1}),\cdots ,{\phi }^{-1}({b}_{n}))$ nor $\theta (f)(b_1,\cdots , b_n)$ is well-defined or
\begin{equation} \label{10.a}
    f({\phi }^{-1}({b}_{1}),\cdots ,{\phi }^{-1}({b}_{n}))=\phi^{-1} (\theta (f)(b_1,\cdots , b_n)).
\end{equation}

Let $g \in \operatorname{Im} \theta$, $n\in {{\mathbb{Z}}^{+}}$ and $(b_1,\cdots ,b_n)\in {S_2^n}$. Then there are two cases as follows.

\emph{Case} (1): $g(b_1,\cdots , b_n)$ is well-defined. 

We are showing that $\theta^{-1}(g)({\phi }^{-1}({b}_{1}),\cdots ,{\phi }^{-1}({b}_{n}))$ is also well-defined and 
 \[\theta^{-1}(g)({\phi }^{-1}({b}_{1}),\cdots ,{\phi }^{-1}({b}_{n}))=\phi^{-1}(g(b_1,\cdots , b_n)).\]
 
Since $\theta {{|}_{S_2}}$ is a map, by Corollary \ref{11.5.4}, $\forall(f,g)\in \theta$, $\theta(f)(b_1,\cdots , b_n)$ is well-defined and \[\theta(f)(b_1,\cdots , b_n)=g(b_1,\cdots , b_n).\]
Then by Fact (A) and Equation (\ref{10.a}), $\forall(f,g)\in \theta$, \[f({\phi }^{-1}({b}_{1}),\cdots ,{\phi }^{-1}({b}_{n}))=\phi^{-1} (g(b_1,\cdots , b_n)).\]
Thus by (the $\theta^{-1}$ version of) Definition \ref{10.4.a}, $\theta^{-1}(g)({\phi }^{-1}({b}_{1}),\cdots ,{\phi }^{-1}({b}_{n}))$ is well-defined and 
 \[\theta^{-1}(g)({\phi }^{-1}({b}_{1}),\cdots ,{\phi }^{-1}({b}_{n}))=\phi^{-1}(g(b_1,\cdots , b_n)).\]

 \emph{Case} (2): $g(b_1,\cdots , b_n)$ is not well-defined. 
 
 We are showing that $\theta^{-1}(g)({\phi }^{-1}({b}_{1}),\cdots ,{\phi }^{-1}({b}_{n}))$ is not well-defined, either.
 
  Let $(f,g)\in \theta$. Assume that $\theta(f)(b_1,\cdots , b_n)$ is well-defined. Then by Definition \ref{10.4.a}, $g(b_1,\cdots , b_n)$ is also well-defined, a contradiction. Hence $\theta(f)(b_1,\cdots , b_n)$ is not well-defined. Thus by Fact (A), $f({\phi }^{-1}({b}_{1}),\cdots ,{\phi }^{-1}({b}_{n}))$ is not well-defined, either, and hence $\theta^{-1}(g)({\phi }^{-1}({b}_{1}),\cdots ,{\phi }^{-1}({b}_{n}))$ is not well-defined by (the $\theta^{-1}$ version of) Definition \ref{10.4.a}. 
 
 Combining Cases (1) and (2), we can tell that $\forall n\in {{\mathbb{Z}}^{+}}$, $g\in \operatorname{Dom}\theta^{-1}(=\operatorname{Im} \theta)$ and $(b_1,\cdots ,b_n)\in {S_2^n}$, 
   \[\theta^{-1}(g)({\phi }^{-1}({b}_{1}),\cdots ,{\phi }^{-1}({b}_{n}))={{\phi }^{-1}}(g(b_1,\cdots , b_n)),\]
or neither $\theta^{-1}(g)({\phi }^{-1}({b}_{1}),\cdots ,{\phi }^{-1}({b}_{n}))$ nor $g(b_1,\cdots , b_n)$ is well-defined.

Moreover, by Lemma \ref{L11.6.1}, $\theta^{-1}|_{S_1}$ is a map. Then by Definition \ref{10.4.2}, $\phi^{-1} $ is a $\theta^{-1} $-morphism from $S_2$ to $S_1$.

Furthermore, since $\phi^{-1} $ is bijective, by Definition \ref{5.4.a}, ${{\phi }^{-1}}\in \operatorname{Iso}_{\theta^{-1}}({{S}_{2}},{{S}_{1}})$.
\end{proof}

\section{Basic properties and more notions} \label{12 Basic properties}
In Sections \ref{12 Basic properties} to \ref{II Cons of $T$-mor and} as follows, our study will focus on par-operator gen-semigroups. The following statement is obvious by Definitions \ref{Operator semigroup}, \ref{9.1.2}, and \ref{10.1.3}:
Each operator semigroup is an operator gen-semigroup, and each operator gen-semigroup is a par-operator gen-semigroup.

	Thus our research on par-operator gen-semigroups also applies to operator gen-semigroups and operator semigroups. 
From now on, unless otherwise specified, $D$ always denotes a set and \textbf{$T$ always denotes a par-operator gen-semigroup on $D$ defined by Definition \ref{10.1.3}}. In Sections \ref{12 Basic properties} to \ref{II Cons of $T$-mor and}, unless otherwise specified, \textbf{$T$-morphisms and $\theta $-morphisms are always defined by Definitions \ref{10.3.1} and \ref{10.4.2}, respectively}.

For par-operator gen-semigroups, Sections \ref{12 Basic properties}, \ref{II Gal corr}, \ref{II latt struc}, \ref{II topologies employed} and \ref{II Cons of $T$-mor and} generalize results obtained in Sections \ref{Basic notions and properties}, \ref{I Galois corr}, \ref{I Lattice structures}, \ref{I Topologies employed} and \ref{Cons of $T$-mor and theta-mor}, respectively. Throughout Sections \ref{12 Basic properties} to \ref{II Cons of $T$-mor and}, by statements such as “Proposition (or Definition, etc.) x.x.x (which is the index) still holds (or still applies)” we mean that the proposition (or definition, etc.) still holds for (or still applies to) the case of par-operator gen-semigroups given specifically in the preceding paragraph. 

Roughly speaking, a result for an operator semigroup $T$ still holds for a par-operator gen-semigroup if the correctness of the proof does not depend on the number of the variables or the domain of definition of any element of $T$. However, if the correctness of the proof of a result does depend on the number of the variables or the domain of definition of an element of $T$, the result may still hold but need a different proof.

\subsection{Basic properties of $T$-spaces and quasi-$T$-spaces} \label{12.1}
In this subsection, we shall generalize properties of $T$-spaces or quasi-$T$-spaces obtained in Subsection \ref{$T$-spaces and quasi-$T$-spaces}.

Proposition \ref{<S> contained in S} is generalized to 
\begin{proposition} \label{11.1.1}
   Let $S$ be a $T$-space. Then $\forall n\in {{\mathbb{Z}}^{+}}$, $(a_1,\cdots ,a_n)\in {{S}^{n}}$ and $f\in T$ of $n$ variables, $f(a_1,\cdots ,a_n)\in S$ if $f(a_1,\cdots ,a_n)$ is well-defined. Hence $\forall A\subseteq S$, ${{\left\langle A \right\rangle }_{T}}\le S$. In particular, ${{\left\langle S \right\rangle }_{T}}\le S$.   
\end{proposition}
\begin{proof}
    Let $n\in {{\mathbb{Z}}^{+}}$, let $(a_1,\cdots ,a_n)\in {{S}^{n}}$ and let $f\in T$ have $n$ variables. By Definition \ref{10.2.1}, $\exists U\subseteq D$ such that ${{\left\langle U \right\rangle }_{T}}=S$. And $\forall i=1,\cdots,n$, $\exists {{g}_{i}}\in T$ of ${{n}_{i}}$ variables and ${{\text{u}}_{i}}\in {{U}^{{{n}_{i}}}}$ (cartesian product) such that ${{g}_{i}}({{\text{u}}_{i}})={{a}_{i}}.$ By Definition \ref{10.1.3}, the partial function $f\circ ({{g}_{1}},\cdots ,{{g}_{n}})$ defined by Definition \ref{10.1.1} is a restriction of some element of $T$. So $f(a_1,\cdots ,a_n)=f({{g}_{1}}({{\text{u}}_{1}}),\cdots ,{{g}_{n}}({{\text{u}}_{n}}))\in {{\left\langle U \right\rangle }_{T}}=S$ if $f(a_1,\cdots ,a_n)$ is well-defined.

Therefore, $\forall A\subseteq S$, $n\in {{\mathbb{Z}}^{+}}$, $(a_1,\cdots ,a_n)\in {{A}^{n}}$ and $f\in T$ of $n$ variables, $f(a_1,\cdots ,a_n)\in S$ if $f(a_1,\cdots ,a_n)$ is well-defined. Hence ${{\left\langle A \right\rangle }_{T}}\subseteq S$, and thus ${{\left\langle A \right\rangle }_{T}}\le S$ (by Definition \ref{10.2.1}).    
\end{proof}
Definition \ref{quasi-T-space} still applies. Then by Proposition \ref{11.1.1}, Proposition \ref{$T$-space is a quasi-$T$-space} still holds. Propositions \ref{1.2.9} and \ref{1.2.10} still hold because their proofs still apply. However, we could not generalize Propositions \ref{1.2.11} and \ref{1.2.12} because now the elements of $T$ may have more than one variable. Proposition \ref{1.2.14} and its proof still apply.

\subsection{Basic properties of $T$-morphisms and $T$-isomorphisms}
In this subsection, we shall generalize properties of $T$-morphisms and $T$-isomorphisms obtained in Subsections \ref{$T$-morphisms} and \ref{2 $T$-isomorphisms}.
	
 Proposition \ref{1.3.8} still holds, but now it requires a different proof as follows.
\begin{proposition} \label{11.2.1}
   \emph{(Proposition \ref{1.3.8})} Let $\sigma $ be a $T$-morphism from ${{S}_{1}}$ to ${{S}_{2}}$. Then $\operatorname{Im} \sigma  \le _q{S_2}$. Moreover, if $\operatorname{Id}\in T$ or more generally, $\operatorname{Im}\sigma \subseteq {{\left\langle \operatorname{Im}\sigma  \right\rangle }_{T}}$, then $\operatorname{Im}\sigma \le {{S}_{2}}$.
\end{proposition}
\begin{remark}
    We keep the original index of the proposition (i.e. \ref{1.3.8}) because we shall use it later (explicitly or implicitly).
\end{remark}
\begin{proof}
    By Definition \ref{10.3.1}, $\forall n\in {{\mathbb{Z}}^{+}}$, $(a_1,\cdots ,a_n)\in S_{1}^{n}$ and $f\in T$ of $n$ variables, neither $f(a_1,\cdots ,a_n)$ nor $f(\sigma ({{a}_{1}}),\cdots ,\sigma ({{a}_{n}}))$ is well-defined or 
    \[f(\sigma ({{a}_{1}}),\cdots ,\sigma ({{a}_{n}}))=\sigma (f(a_1,\cdots ,a_n))\in \operatorname{Im}\sigma .\]
    Hence ${{\left\langle \operatorname{Im}\sigma  \right\rangle }_{T}}\subseteq \operatorname{Im}\sigma $, and thus by Definition \ref{quasi-T-space}, $\operatorname{Im}\sigma $ is a quasi-$T$-subspace of $S_2$. 

Moreover, if $\operatorname{Id}\in T$ or $\operatorname{Im}\sigma \subseteq {{\left\langle \operatorname{Im}\sigma  \right\rangle }_{T}}$, then by Proposition \ref{1.2.9}, $\operatorname{Im}\sigma $ is a $T$-space, and hence $\operatorname{Im}\sigma \le {{S}_{2}}$.	
\end{proof}
By the way, we generalize Proposition \ref{11.2.1} to $\theta$-morphisms as follows, which generalizes Proposition \ref{image theta space} to the case of par-operator gen-semigroups.
\begin{proposition} \label{image theta space for multivariable}
    Let $\phi $ be a $\theta $-morphism from a $T_1$-space $S_1$ to a $T_2$-space $S_2$. Suppose $\operatorname{Im}\theta = T_2$. Then $ \operatorname{Im} \phi  \le _q{S_2}$. Moreover, if $\operatorname{Id}\in T_2$ or more generally, $\operatorname{Im}\phi \subseteq {{\left\langle \operatorname{Im}\phi  \right\rangle }_{T_2}}$, then $\operatorname{Im}\phi \le {{S}_{2}}$.
\end{proposition}
 \begin{proof}
     By Definition \ref{10.4.2}, $\forall n\in {{\mathbb{Z}}^{+}}$, $(a_1,\cdots ,a_n)\in {S_1^n}$ and $f\in \operatorname{Dom}\theta$, neither $f(a_1,\cdots ,a_n)$ nor $\theta (f)(\phi ({{a}_{1}}),\cdots ,\phi ({{a}_{n}}))$ is well-defined or
    \[\theta(f)(\phi ({{a}_{1}}),\cdots ,\phi ({{a}_{n}}))=\phi (f(a_1,\cdots ,a_n))\in \operatorname{Im}\phi .\]
     Since $\operatorname{Im}\theta = T_2$, we can tell ${{\left\langle \operatorname{Im}\phi  \right\rangle }_{T_2}}\subseteq \operatorname{Im}\phi $, and so by Definition \ref{quasi-T-space}, $\operatorname{Im}\phi $ is a quasi-$T_2$-subspace of $S_2$. 
     
Moreover, if $\operatorname{Id}\in T_2$ or $\operatorname{Im}\phi \subseteq {{\left\langle \operatorname{Im}\phi  \right\rangle }_{T_2}}$, then by Proposition \ref{1.2.9}, $\operatorname{Im}\phi $ is a $T_2$-space and hence $\operatorname{Im}\phi \le {{S}_{2}}$.
 \end{proof}
Definition \ref{1.3.2} still applies. And Proposition \ref{1.3.10} still holds as shown below.
\begin{proposition}
\label{11.2.2}
    \emph{(Proposition \ref{1.3.10})} Let $S$ be a $T$-space. Then $\operatorname{End}_{T}(S)$ defined by Definition \ref{1.3.2} constitutes a monoid, which we still denote by $\operatorname{End}_{T}(S)$, with composition of functions as the binary operation.
\end{proposition}
\begin{proof}
  The identity map on $S$ lies in $\operatorname{End}_{T}(S)$, and hence we only need to show that $\operatorname{End}_{T}(S)$ is a semigroup with composition of functions as the binary operation.

Let ${{\sigma }_{1}},{{\sigma }_{2}}\in \operatorname{End}_{T}(S)$, $f\in T$, $n\in {{\mathbb{Z}}^{+}}$ and $(a_1,\cdots ,a_n)\in {{S}^{n}}$. Then by Definition \ref{10.3.1},

$f(a_1,\cdots ,a_n)$ is well-defined;

$\Leftrightarrow f({{\sigma }_{2}}({{a}_{1}}),\cdots ,{{\sigma }_{2}}({{a}_{n}}))$ is well-defined; 

$\Leftrightarrow f({{\sigma }_{1}}({{\sigma }_{2}}({{a}_{1}})),\cdots ,{{\sigma }_{1}}({{\sigma }_{2}}({{a}_{n}})))$ is well-defined.

Moreover, in the case where the above equivalent conditions are satisfied,

$({{\sigma }_{1}}\circ {{\sigma }_{2}})(f(a_1,\cdots ,a_n))$

$={{\sigma }_{1}}({{\sigma }_{2}}(f(a_1,\cdots ,a_n)))$ 

$={{\sigma }_{1}}(f({{\sigma }_{2}}({{a}_{1}}),\cdots ,{{\sigma }_{2}}({{a}_{n}})))$ 

$=f({{\sigma }_{1}}({{\sigma }_{2}}({{a}_{1}})),\cdots ,{{\sigma }_{1}}({{\sigma }_{2}}({{a}_{n}})))$. 

Then again by Definition \ref{10.3.1}, ${{\sigma }_{1}}\circ {{\sigma }_{2}}\in \operatorname{End}_{T}(S)$. 

Composition of functions is associative. Therefore, $\operatorname{End}_{T}(S)$ is a monoid with composition of functions as the binary operation.
\end{proof}
Proposition \ref{1.3.12} is still obvious. 

As already said before, Definition \ref{$T$-isomorphisms} still applies to the case of par-operator gen-semigroups.
Proposition \ref{1.3.a} was generalized to $T$ being a par-operator gen-semigroup by Proposition \ref{11.2.a}.
Proposition \ref{1.3.11} still holds because its proof still applies. 
Proposition \ref{intersection of Aut} is still obvious. 
 
\subsection{Galois $T$-extensions}
Definition \ref{1.4.1} and the remark right after it still apply. Proposition \ref{1.4.2} still holds, but now it requires a different proof as follows.
\begin{proposition} \label{11.3.1}
     \emph{(Proposition \ref{1.4.2})} Let $S$ be a $T$-space and let $H$ be a subset of $\operatorname{End}_{T}(S)$. Then ${S^H} \le _qS$. Moreover, if $\operatorname{Id}\in T$ or ${{S}^{H}}\subseteq {{\left\langle {{S}^{H}} \right\rangle }_{T}}$, then ${{S}^{H}}\le S$.
\end{proposition}
\begin{proof}
    Let $n\in {{\mathbb{Z}}^{+}}$, $(a_1,\cdots ,a_n)\in {{({{S}^{H}})}^{n}}$ and $f\in T$ of $n$ variables. By Definition \ref{10.3.1}, $\forall \sigma \in H$, neither $f(a_1,\cdots ,a_n)$ nor $f(\sigma ({{a}_{1}}),\cdots ,\sigma ({{a}_{n}}))$ is well-defined, or 
 \[\sigma (f(a_1,\cdots ,a_n))=f(\sigma ({{a}_{1}}),\cdots ,\sigma ({{a}_{n}}))=f(a_1,\cdots ,a_n).\] 

Therefore, $f(a_1,\cdots ,a_n)\in {{S}^{H}}$ if $f(a_1,\cdots ,a_n)$ is well-defined. Thus ${{\left\langle {{S}^{H}} \right\rangle }_{T}}\subseteq {{S}^{H}}$. Hence ${{S}^{H}}$ is a quasi-$T$-space, and so ${S^H} \le _qS$. 

Moreover, if $\operatorname{Id}\in T$ or ${{S}^{H}}\subseteq {{\left\langle {{S}^{H}} \right\rangle }_{T}}$, then by Proposition \ref{1.2.9}, ${S^H}$ is a $T$-space, and hence ${{S}^{H}}\le S$.	
\end{proof}
Proposition \ref{1.4.3} is still obvious.
\subsection{Galois $T$-monoids and Galois $T$-groups}
Definition \ref{1.5.1} and the remark right after it still apply. 
Propositions \ref{1.5.2} and \ref{1.5.3} are still true because their proofs still apply. 
Propositions \ref{1.5.4}, \ref{1.5.5} and \ref{1.5.6} are still obvious. 

\subsection{Generated submonoids of $\operatorname{End}_{T}(S)$ and subgroups of $\operatorname{Aut}_{T}(S)$}
Definition \ref{1.6.1} and the remark right after it still apply. Proposition \ref{1.6.2} still holds because its proof still applies.

\section{Galois correspondences} \label{II Gal corr}
In this section, we shall find that all results obtained in Subsection \ref{I subsection Galois correspondences} still apply to the case where $T$ is a par-operator gen-semigroup. 

Notations \ref{2.1.1} (and the two remarks for it) and \ref{2.1.2} still apply.

Definition \ref{2.2.1} and the remark right after it still apply. Lemma \ref{2.2.2} still holds because its proof still applies. 

Definition \ref{2.2.3} and the remark right after it still apply. Lemma \ref{2.2.4} still holds because its proof still applies.

Corollaries \ref{2.2.5} and \ref{2.2.6} still hold because their proofs still apply.
	
Then by Propositions \ref{9.3.5}, \ref{9.3.6}, \ref{9.3.7}, \ref{9.4.5}, \ref{10.3.4}, \ref{10.3.6} and \ref{10.4.5}, we may apply Corollaries \ref{2.2.5} and \ref{2.2.6} to modules, abelian groups, non-abelian groups, rings, field, differential field, and categories, respectively. We can tell that Galois correspondences exist not only for Galois $T$-groups (by Corollary \ref{2.2.6}), but also for Galois $T$-monoids (by Corollary \ref{2.2.5}).
 
 Moreover, if we want to build up the Galois correspondences in terms of topology, then we may employ results in Section \ref{II topologies employed}, which generalize results in Section \ref{I Topologies employed}.
 
\section{Lattice structures of objects arising in Galois correspondences on a $T$-space $S$} \label{II latt struc}
	In this section, results obtained in Section \ref{I Lattice structures} are generalized for par-operator gen-semigroups.
\subsection{$\operatorname{Sub}_T(S)$, $\operatorname{SMn}(\operatorname{End}_{T}(S))$ and $\operatorname{SGr}(\operatorname{Aut}_{T}(S))$}
As said in Subsection \ref{12.1}, we could not generalize Proposition \ref{1.2.11} because now the elements of $T$ may have more than one variable. So to generalize Proposition \ref{3.1.1}, we need a notion as follows.
\begin{definition} \label{13.1.1}
    Let $S$ be a $T$-space. If $X\subseteq S$, then the intersection of all the quasi-$T$-subspaces of $S$ containing $X$, denoted by ${{\left\langle X \right\rangle }_{S}}$, is called the \emph{quasi-$T$-subspace of $S$ generated by $X$}. 
\end{definition}
\begin{remark}
    By Proposition \ref{1.2.10}, ${{\left\langle X \right\rangle }_{S}}$ is the smallest quasi-$T$-subspace of $S$ which contains $X$.
\end{remark}
Then Proposition \ref{3.1.1} is generalized to
\begin{proposition} \label{13.1.2}
    Let $S$ be a $T$-space. Then $\operatorname{Sub}_T(S)$ is a lattice with inclusion as the binary relation if $\forall {{S}_{1}},{{S}_{2}}\in \operatorname{Sub}_T(S)$, ${{S}_{1}}\wedge {{S}_{2}}:={{S}_{1}}\bigcap {{S}_{2}}$ and ${{S}_{1}}\vee {{S}_{2}}:={{\left\langle {{S}_{1}}\bigcup {{S}_{2}} \right\rangle }_{S}}$. Moreover, $\operatorname{Sub}_T(S)$ is a complete lattice if let $\wedge A=\bigcap{{{S}_{i}}}$ and let $\vee A={{\left\langle \bigcup{{{S}_{i}}} \right\rangle }_{S}}$, $\forall A=\{{{S}_{i}}\}\subseteq \operatorname{Sub}_T(S)$.
\end{proposition}
\begin{proof}
    Straightforward consequences of Proposition \ref{1.2.10} and Definition \ref{13.1.1}.
\end{proof}
Propositions \ref{3.1.2} and \ref{3.1.3} still hold because Propositions \ref{1.3.12} and \ref{intersection of Aut} still hold and Definition \ref{1.6.1} and the remark right after it still apply. Proposition \ref{3.1.4} still holds because Proposition \ref{1.6.2} still holds.

\subsection{$\operatorname{Int}_{T}^{\operatorname{End}}(S/B)$ and $\operatorname{Int}_{T}^{\operatorname{Aut}}(S/B)$}
	Lemma \ref{3.2.1} still holds because its proof still applies. Proposition \ref{3.2.2} would still hold if Proposition \ref{3.1.1} in it were replaced by Proposition \ref{13.1.2} as follows.
 \begin{proposition} \label{13.2.1}
     $\operatorname{Int}_{T}^{\operatorname{End}}(S/B)$ is a complete $\wedge$-sublattice of the complete lattice $\operatorname{Sub}_T(S)$ defined in Proposition \ref{13.1.2}. 
 \end{proposition}
\begin{proof}
    Almost the same as the proof of Proposition \ref{3.2.2} except that Proposition \ref{3.1.1} is replaced by Proposition \ref{13.1.2}.
\end{proof}
Analogously, Lemma \ref{3.2.4} still holds, and Proposition \ref{3.2.5} would still hold if Proposition \ref{3.1.1} in it were replaced by Proposition \ref{13.1.2} as follows.
\begin{proposition} \label{13.2.2}
      $\operatorname{Int}_{T}^{\operatorname{Aut}}(S/B)$ is a complete $\wedge$-sublattice of the complete lattice $\operatorname{Sub}_T(S)$ defined in Proposition \ref{13.1.2}.
\end{proposition}
\begin{proof}
    Almost the same as the proof of Proposition \ref{3.2.5} except that Proposition \ref{3.1.1} is replaced by Proposition \ref{13.1.2}.
\end{proof}

\subsection{$\operatorname{GSMn}_{T}(S/B)$ and
$\operatorname{GSGr}_{T}(S/B)$} Lemma \ref{3.3.1}, Proposition \ref{3.3.2}, Lemma \ref{3.3.4} and Proposition \ref{3.3.5} still hold because their proofs still apply.

\section{Topologies employed to construct Galois correspondences} \label{II topologies employed}
	In this section, results obtained in Section \ref{I Topologies employed} are generalized for par-operator gen-semigroups.
\subsection{Topologies on a $T$-space $S$ for $\operatorname{Int}_T^{\operatorname{End}}(S/B)$}
Lemma \ref{4.1.1} and its proof still apply.

However, we could not generalize a part of Theorem \ref{4.1.2} straightforwardly because we could not generalize Proposition \ref{1.2.11}. Instead, we have the following, which in a sense is analogous to Theorem \ref{4.3.2}. 
\begin{theorem} \label{14.1.1}
    Let $S$ be a $T$-space and let $B\subseteq S$. Let
    \[{{P}_{1}}=\{\text{all topologies on $S$ such that Equation $($\ref{5.1}$)$ is satisfied}\}\] and let
	\[{{Q}_{1}}=\{\text{all topologies on $S$ which are finer than ${{\mathcal{T}}_{1}}$}\},\] 
 where ${{\mathcal{T}}_{1}}$ is defined in Lemma \ref{4.1.1}. Then ${{P}_{1}}\subseteq {{Q}_{1}}$ and the following statements are equivalent:
 \begin{enumerate}
     \item [(i)] ${{P}_{1}}\ne \emptyset $.
      \item [(ii)] For any intersection $K$ of finite unions of elements of $\operatorname{Int}_{T}^{\operatorname{End}}(S/B)$, $K\in \operatorname{Sub}_T(S)$ implies $K\in \operatorname{Int}_{T}^{\operatorname{End}}(S/B)\bigcup \{\emptyset \}$.
      \item [(iii)] ${{\mathcal{T}}_{1}}\in {{P}_{1}}(\subseteq {{Q}_{1}})$; that is, ${{\mathcal{T}}_{1}}$ is the coarsest topology on $S$ such that Equation $($\ref{5.1}$)$ is satisfied.
 \end{enumerate}

Moreover, if $\operatorname{Int}_{T}^{\operatorname{End}}(S/B)$ is a complete $\vee$-sublattice of the complete lattice $\operatorname{Sub}_T(S)$ defined in Proposition \ref{13.1.2}, then the above three statements are true.
\end{theorem}
\begin{proof}
    By Lemma \ref{4.1.1}, ${{\mathcal{T}}_{1}}$ is the smallest topology on $S$ such that the collection of all closed sets contains $\operatorname{Int}_{T}^{\operatorname{End}}(S/B)$. Suppose $ \mathcal{T} \in {{P}_{1}}$. Then all elements of $\operatorname{Int}_{T}^{\operatorname{End}}(S/B)$ are closed in $\mathcal{T} $, and thus $\mathcal{T}_1 \subseteq \mathcal{T}$. Hence $\mathcal{T} \in {{Q}_{1}}$. Therefore ${{P}_{1}}\subseteq {{Q}_{1}}$. 

(i) $\Rightarrow$ (ii):
Let $\mathcal{T} \in {{P}_{1}}$. Let $K$ be any intersection of finite unions of elements of $\operatorname{Int}_{T}^{\operatorname{End}}(S/B)$. By the definition of $P_1$, each element of $\operatorname{Int}_{T}^{\operatorname{End}}(S/B)$ is closed in $\mathcal{T} $, and so is $K$. Suppose $K\in \operatorname{Sub}_T(S)$. Then $K\in \{{S}'\,|\,B\subseteq {S}'\le _qS\}$. Again by the definition of $P_1$, $K\in \operatorname{Int}_{T}^{\operatorname{End}}(S/B)\bigcup \{\emptyset \}$. Hence (ii) is true.

(ii) $\Rightarrow$ (iii):
By the definition of ${{\mathcal{T}}_{1}}$ in Lemma \ref{4.1.1}, all elements of $\operatorname{Int}_{T}^{\operatorname{End}}(S/B)$ are closed in ${{\mathcal{T}}_{1}}$. So by Definition \ref{2.2.1} and Proposition \ref{11.3.1},
\begin{center}
    $\operatorname{Int}_{T}^{\operatorname{End}}(S/B)\backslash \{\emptyset \}\subseteq $ \{$B\subseteq {S}'\le _qS\,|\,{S}'$ is nonempty and closed in ${{\mathcal{T}}_{1}}$\}.
\end{center}

On the other hand, we prove
\begin{center}
    $\operatorname{Int}_{T}^{\operatorname{End}}(S/B)\backslash \{\emptyset \}\supseteq $ \{$B\subseteq {S}'\le _qS\,|\,{S}'$ is nonempty and closed in ${{\mathcal{T}}_{1}}$\} 
\end{center}
as follows.

Let $K\in \{B\subseteq {S}'\le _qS\,|\,{S}'$ is nonempty and closed in ${{\mathcal{T}}_{1}}\}$. We need to show $K\in \operatorname{Int}_{T}^{\operatorname{End}}(S/B)\backslash \{\emptyset \}$.

Because ${{\mathcal{T}}_{1}}$ is generated by subbasis ${{\beta }_{1}}$, ${{\mathcal{T}}_{1}}$ is the collection of all unions of finite intersections of elements of ${{\beta }_{1}}=\{S\backslash A\,|\,A\in \operatorname{Int}_{T}^{\operatorname{End}}(S/B)\}\bigcup \{S\}$. Thus $S\backslash K$ is a union of finite intersections of elements of ${{\beta }_{1}}$. It follows from DeMorgan’s Laws that $K$ is an intersection of finite unions of elements of $\operatorname{Int}_{T}^{\operatorname{End}}(S/B)\bigcup \{\emptyset \}$. Because $K\ne \emptyset $ and $K\in \operatorname{Sub}_T(S)$, by (ii), $K\in \operatorname{Int}_{T}^{\operatorname{End}}(S/B)\backslash \{\emptyset \}$, as desired. Hence
\begin{center}
    $\operatorname{Int}_{T}^{\operatorname{End}}(S/B)\backslash \{\emptyset \}\supseteq $ \{$B\subseteq {S}'\le _qS\,|\,{S}'$ is nonempty and closed in ${{\mathcal{T}}_{1}}$\} 
\end{center}
Therefore, Equation $($\ref{5.1}$)$ is satisfied for ${{\mathcal{T}}_{1}}$.

Thus ${{\mathcal{T}}_{1}}\in {{P}_{1}}$. We showed ${{P}_{1}}\subseteq {{Q}_{1}}$ at the beginning of our proof, and thus by the definition of ${{Q}_{1}}$, ${\mathcal{T}}_{1}$ is the coarsest topology in $P_1$. Hence by the definition of ${{P}_{1}}$, ${\mathcal{T}}_{1}$ is the coarsest topology on $S$ such that Equation $($\ref{5.1}$)$ is satisfied. 

(iii) $\Rightarrow$ (i): Obvious.

Let $K$ be any intersection of finite unions of elements of $\operatorname{Int}_{T}^{\operatorname{End}}(S/B)$. By Proposition \ref{13.2.1}, any intersection of elements of $\operatorname{Int}_{T}^{\operatorname{End}}(S/B)$ lies in $\operatorname{Int}_{T}^{\operatorname{End}}(S/B)$. Then as a result of the first distributive law of set operations, $K$ is a union of elements of $\operatorname{Int}_{T}^{\operatorname{End}}(S/B)$. Suppose that $\operatorname{Int}_{T}^{\operatorname{End}}(S/B)$ is a complete $\vee$-sublattice of the complete lattice $\operatorname{Sub}_T(S)$ defined in Proposition \ref{13.1.2}. Then $K\in \operatorname{Sub}_T(S)$ implies $K\in \operatorname{Int}_{T}^{\operatorname{End}}(S/B)$, and hence (ii) is true (in this case).
\end{proof}
If $T$ is an operator semigroup, then by Propositions \ref{1.2.11} and \ref{1.4.2}, $\forall {{S}_{1}},{{S}_{2}}\in \operatorname{Int}_{T}^{\operatorname{End}}(S/B)$, ${{S}_{1}}\bigcup {{S}_{2}}\in \operatorname{Sub}_T(S)$, which implies that, in this case, statement (ii) in Theorem \ref{4.1.2} coincides with statement (ii) in Theorem \ref{14.1.1}. So Theorem \ref{14.1.1} generalizes Theorem \ref{4.1.2}.

\subsection{Topologies on a $T$-space $S$ for $\operatorname{Int}_{T}^{\operatorname{Aut}}(S/B) $}
Lemma \ref{4.2.1} and its proof still apply.

We could not generalize a part of Theorem \ref{4.2.2} straightforwardly because we could not generalize Proposition \ref{1.2.11}. Instead, we have the following, which is analogous to Theorem \ref{14.1.1}. 
\begin{theorem}   \label{14.2.1}
Let $S$ be a $T$-space and let $B\subseteq S$. Let
	\[{{P}_{2}}=\{\text{all topologies on $S$ such that Equation $($\ref{5.3}$)$ is satisfied}\}\] and let
	\[{{Q}_{2}}=\{\text{all topologies on $S$ which are finer $($larger$)$ than ${{\mathcal{T}}_{2}}$}\},\] 
 where ${{\mathcal{T}}_{2}}$ is defined in Lemma \ref{4.2.1}. Then ${{P}_{2}}\subseteq {{Q}_{2}}$ and the following statements are equivalent:
\begin{enumerate}
    \item[(i)] ${{P}_{2}}\ne \emptyset $.
    \item[(ii)] 	For any intersection $K$ of finite unions of elements of $\operatorname{Int}_{T}^{\operatorname{Aut}}(S/B)$, $K\in \operatorname{Sub}_T(S)$ implies $K\in \operatorname{Int}_{T}^{\operatorname{Aut}}(S/B)\bigcup \{\emptyset \}$.
  \item[(iii)]	${{\mathcal{T}}_{2}}\in {{P}_{2}}(\subseteq {{Q}_{2}})$; that is, ${{\mathcal{T}}_{2}}$ is the coarsest topology on $S$ such that Equation $($\ref{5.3}$)$ is satisfied.
\end{enumerate}

Moreover, if $\operatorname{Int}_{T}^{\operatorname{Aut}}(S/B)$ is a complete $\vee$-sublattice of the complete lattice $\operatorname{Sub}_T(S)$ defined in Proposition \ref{13.1.2}, then the above three statements are true.
\end{theorem}
\begin{proof}
    Almost the same as the proof of Theorem \ref{14.1.1} except that $\operatorname{Int}_{T}^{\operatorname{End}}(S/B)$, ${{\mathcal{T}}_{1}}$, ${{\beta }_{1}}$, ${{P}_{1}}$ and ${{Q}_{1}}$ are replaced by $\operatorname{Int}_{T}^{\operatorname{Aut}}(S/B)$, ${{\mathcal{T}}_{2}}$, ${{\beta }_{2}}$, ${{P}_{2}}$ and ${{Q}_{2}}$, respectively, and accordingly, Lemma \ref{4.1.1} and Proposition \ref{13.2.1} are replaced by Lemma \ref{4.2.1} and Proposition \ref{13.2.2}, respectively.	
\end{proof}
Just as Theorem \ref{14.1.1} generalizes Theorem \ref{4.1.2}, Theorem \ref{14.2.1} generalizes Theorem \ref{4.2.2}.

\subsection{Topologies on $\operatorname{End}_{T}(S)$}
Lemma \ref{4.3.1} and its proof still apply.
Theorem \ref{4.3.2} still holds because its proof still applies. 

\subsection{Topologies on $\operatorname{Aut}_{T}(S)$} 
Lemma \ref{4.4.1} and its proof still apply.
Theorem \ref{4.4.2} still holds because its proof still applies. 

\section{Constructions of the generalized morphisms and isomorphisms} \label{II Cons of $T$-mor and}
	In this section, results obtained in Section \ref{Cons of $T$-mor and theta-mor} are generalized for par-operator gen-semigroups.
\subsection{A construction of $T$-morphisms from ${{\left\langle u \right\rangle }_{T}}(u \in D)$} \label{A construction of T-morphisms}
 Although the contents of this subsection will be further generalized in the next subsection, we still keep this subsection because we will need it in Sections \ref{Transitive} and \ref{Two questions}.
 
 Notation \ref{5.1.1} is generalized for partial functions of multiple variables as follows.
 \begin{notation}  \label{15.1.1}
     Let $M \subseteq \bigcup\nolimits_{n\in {{\mathbb{Z}}^{+}}}{\{}$all partial functions from ${{D}^{n}}$ to $D$\}, where $D$ is a set. Let $u,v\in D$. 
    By $u\xrightarrow{M}v$ we mean that
    
    $\forall n,m\in {{\mathbb{Z}}^{+}}$, $f\in M$ of $n$ variables and $g\in M$ of $m$ variables,
    \begin{center}
        $f({{u}^{n}})=g({{u}^{m}})$ implies $f({{v}^{n}})=g({{v}^{m}})$,
    \end{center}
 where ${{x}^{k}}$ denotes the $k$-tuple $(x,\cdots ,x)$.

Besides, let ${{[u)}_{M}}$ denote the set $\{w\in D\,|\,u\xrightarrow{M}w\}$.

Moreover, by $u\overset{M}{\longleftrightarrow}v$ we mean that $u\xrightarrow{M}v$ and $v\xrightarrow{M}u$. 

Besides, let ${{[u]}_{M}}$ denote the set $\{w\in D\,|\,u\overset{M}{\longleftrightarrow}w\}$. 
 \end{notation}
\begin{remark} \begin{enumerate}
    \item Suppose $u\xrightarrow{M}v$. Then $\forall n\in {{\mathbb{Z}}^{+}}$ and $f\in M$ of $n$ variables, $f({{u}^{n}})=f({{u}^{n}})$ implies $f({{v}^{n}})=f({{v}^{n}})$, and hence by Convention \ref{convention}, that $f({{u}^{n}})$ is well-defined implies that $f({{v}^{n}})$ is well-defined.
    \item Apparently, $u\overset{M}{\longleftrightarrow}v$ defines an equivalence relation on $D$.
\end{enumerate} \end{remark}
Then Proposition \ref{5.1.3} still holds but its proof changes as follows.
\begin{proposition} \label{15.1.2}
    \emph{(Proposition \ref{5.1.3})} Let $\sigma$ be a $T$-morphism from a $T$-space $S$. Then $\forall a\in S$, $a\xrightarrow{T}\sigma (a)$.
\end{proposition}
 \begin{proof}
     Let $f\in T$ have $n$ variables, let $g\in T$ have $m$ variables and let $a\in S$.

If $f({{a}^{n}})=g({{a}^{m}})$,  where ${{x}^{k}}$ denotes the $k$-tuple $(x,\cdots ,x)$, then by Definition \ref{10.3.1}, both $f({{(\sigma (a))}^{n}})$ and $g({{(\sigma (a))}^{m}})$ are well-defined and
\[f({{(\sigma (a))}^{n}})=\sigma (f({{a}^{n}}))=\sigma (g({{a}^{m}}))=g({{(\sigma (a))}^{m}}).\] 

Therefore, $a\xrightarrow{T}\sigma (a)$.
 \end{proof}
 Analogously, it is not hard to tell that Proposition \ref{5.1.4} still holds.

Proposition \ref{5.1.6} can be generalized for par-operator gen-semigroups, but we omit it for brevity.

Proposition \ref{5.1.8} is generalized to
\begin{proposition} \label{15.1.5}
    Let $u,v\in D$. \begin{enumerate}
        \item [(a)] The following statements are equivalent:
    \begin{enumerate}
        \item [(i)] $u\xrightarrow{T}v$. 
        \item [(ii)] There exists a map $\sigma :{{\left\langle u \right\rangle }_{T}}\to {{\left\langle v \right\rangle }_{T}}$ given by $f({{u}^{n}})\mapsto f({{v}^{n}})$, $\forall n\in {{\mathbb{Z}}^{+}}$ and $f\in T$ of $n$ variables such that $f({{u}^{n}})$ is well-defined.
    \end{enumerate}
   \item [(b)] The map $\sigma $ given in \emph{(ii)} is a $T$-morphism.
  \end{enumerate}
\end{proposition}
\begin{proof}
For (a): 

$u\xrightarrow{T}v$;\\*
$\Leftrightarrow \forall n,m\in {{\mathbb{Z}}^{+}}$, $f\in T$ of $n$ variables and $g\in T$ of $m$ variables, $f({{u}^{n}})=g({{u}^{m}})$ implies $f({{v}^{n}})=g({{v}^{m}})$;\\* 
$\Leftrightarrow \sigma :{{\left\langle u \right\rangle }_{T}}\to {{\left\langle v \right\rangle }_{T}}$ given by $f({{u}^{n}})\mapsto f({{v}^{n}})$, $\forall n\in {{\mathbb{Z}}^{+}}$ and $f\in T$ of $n$ variables such that $f({{u}^{n}})$ is well-defined, is a well-defined map (\textit{cf.} Remark (1) right after Notation \ref{15.1.1}).

For (b): 

 Let $(a_1,\cdots ,a_n)\in \left\langle u \right\rangle _{T}^{n}$. Then $\forall i=1, \cdots, n$, $\exists {{g}_{i}}\in T$ of ${{n}_{i}}$ variables such that ${{g}_{i}}({{u}^{{{n}_{i}}}})={{a}_{i}}$.

Let $f\in T$ have $n$ variables. By Definition \ref{10.1.1}, $f\circ ({{g}_{1}},\cdots ,{{g}_{n}})$ is a partial function from ${{D}^{m}}$ to $D$ where $m=\sum\nolimits_{i}{{{n}_{i}}}$. By Definition \ref{10.1.3}, $f\circ ({{g}_{1}},\cdots ,{{g}_{n}})$ is a restriction of some element of $T$. Hence

$f(a_1,\cdots ,a_n)$ is well-defined;\\* $\Leftrightarrow \sigma (f(a_1,\cdots ,a_n))$ 

$=\sigma (f({{g}_{1}}({{u}^{{{n}_{1}}}}),\cdots ,{{g}_{n}}({{u}^{{{n}_{n}}}})))$

$=\sigma ((f\circ ({{g}_{1}},\cdots ,{{g}_{n}}))({{u}^{m}}))$ (by Definition \ref{10.1.1})

$=(f\circ ({{g}_{1}},\cdots ,{{g}_{n}}))({{v}^{m}})$  (by the definition of $\sigma $ in (ii))

$=f({{g}_{1}}({{v}^{{{n}_{1}}}}),\cdots ,{{g}_{n}}({{v}^{{{n}_{n}}}}))$ (by Definition \ref{10.1.1})

$=f(\sigma ({{g}_{1}}({{u}^{{{n}_{1}}}})),\cdots ,\sigma ({{g}_{n}}({{u}^{{{n}_{n}}}})))$ (by the definition of $\sigma $ in (ii)) 

$=f(\sigma ({{a}_{1}}),\cdots ,\sigma ({{a}_{n}}))$ 

is well-defined; \\* as desired for $\sigma $ to be a $T$-morphism (by Definition \ref{10.3.1}).	
\end{proof}
Corollary \ref{5.1.9} is generalized to the following, which is a straightforward result of Proposition \ref{15.1.5}.
\begin{corollary} \label{15.1.6}
    Let $u,v\in D$. \begin{enumerate}
        \item [(a)]  The following statements are equivalent:
    \begin{enumerate}
        \item [(i)] $u\overset{T}{\longleftrightarrow}v$. 
         \item [(ii)] There exists a bijective map $\sigma :{{\left\langle u \right\rangle }_{T}}\to {{\left\langle v \right\rangle }_{T}}$ given by $f({{u}^{n}})\mapsto f({{v}^{n}})$, $\forall n\in {{\mathbb{Z}}^{+}}$ and $f\in T$ of $n$ variables such that $f({{u}^{n}})$ is well-defined.
    \end{enumerate}
     \item [(b)] The bijective map $\sigma $ given in \emph{(ii)} is a $T$-isomorphism. 
    \end{enumerate}
 \end{corollary}

\subsection{A construction of $T$-morphisms from ${{\left\langle U \right\rangle }_{T}}(U \subseteq D)$} \label{16 cons of T-mor from U}
For partial functions of multiple variables, we generalize Notation \ref{5.2.1} as follows.
\begin{notation} \label{15.2.1}
    Let $D$ be a set, let $\alpha :U(\subseteq D)\to V(\subseteq D)$ be a map and let $M$ be a subset of $\bigcup\nolimits_{n\in {{\mathbb{Z}}^{+}}}{\{}$all partial functions from ${{D}^{n}}$ to $D$\}. 
    
    By $U\xrightarrow{M,\alpha }V$ we mean that
$\forall n,m\in {{\mathbb{Z}}^{+}}$, $f, g\in M$, $({{u}_{1}},\cdots ,{{u}_{n}})\in {{U}^{n}}$, and $({{w}_{1}},\cdots ,{{w}_{m}})\in {{U}^{m}}$,\\*$f({{u}_{1}},\cdots ,{{u}_{n}})=g({{w}_{1}},\cdots ,{{w}_{m}})\Rightarrow f(\alpha ({{u}_{1}}),\cdots ,\alpha ({{u}_{n}}))=g(\alpha ({{w}_{1}}),\cdots ,\alpha ({{w}_{m}}))$. 

Moreover, by $U\overset{M,\alpha }{\longleftrightarrow}V $ we mean that $\alpha $ is bijective, $U\xrightarrow{M,\alpha }V $ and $V \xrightarrow{M,{{\alpha }^{-1}}}U$; or equivalently, by $U\overset{M,\alpha }{\longleftrightarrow}V $ we mean that $\alpha $ is bijective and $\forall n,m\in {{\mathbb{Z}}^{+}}$, $f, g\in M$, $({{u}_{1}},\cdots ,{{u}_{n}})\in {{U}^{n}}$, and $({{w}_{1}},\cdots ,{{w}_{m}})\in {{U}^{m}}$,\\*$f({{u}_{1}},\cdots ,{{u}_{n}})=g(w_1,\cdots ,w_m)\Leftrightarrow f(\alpha ({{u}_{1}}),\cdots ,\alpha ({{u}_{n}}))=g(\alpha ({{w}_{1}}),\cdots ,\alpha ({{w}_{m}}))$.
\end{notation}
\begin{remark}
    Suppose $U\xrightarrow{M,\alpha }V $. Then $\forall f\in M$ of $n$ variables and $({{u}_{1}},\cdots ,{{u}_{n}})\in {{U}^{n}}$,\\*
$f({{u}_{1}},\cdots ,{{u}_{n}})=f({{u}_{1}},\cdots ,{{u}_{n}})$ $\Rightarrow $ $f(\alpha ({{u}_{1}}),\cdots ,\alpha ({{u}_{n}}))=f(\alpha ({{u}_{1}}),\cdots ,\alpha ({{u}_{n}}))$.\\*
Hence by Convention \ref{convention}, $f(\alpha ({{u}_{1}}),\cdots ,\alpha ({{u}_{n}}))$ is well-defined if $f({{u}_{1}},\cdots ,{{u}_{n}})$ is well-defined. 
\end{remark}
Then Proposition \ref{5.2.2} still holds but its proof changes as follows.
\begin{proposition} \label{15.2.2}
   \emph{ (Proposition \ref{5.2.2}) }  Let $\sigma$ be a $T$-morphism from a $T$-space $S_1$ to a $T$-space $S_2$. Then $\forall A\subseteq {{S}_{1}}$, $A\xrightarrow{T,\alpha }{S}_{2} $, where $\alpha :=\sigma {{|}_{A}}$. In particular, ${{S}_{1}}\xrightarrow{T,\sigma }{S}_{2}$. 
\end{proposition}
 \begin{proof}
    Let $n,m\in {{\mathbb{Z}}^{+}}$, let $f,g\in T$, let $({{a}_{1}},\cdots ,{{a}_{n}})\in {{A}^{n}}$, and let $({{b}_{1}},\cdots ,{{b}_{m}})\in {{A}^{m}}$. If $f({{a}_{1}},\cdots ,{{a}_{n}})=g({{b}_{1}},\cdots ,{{b}_{m}})$, by Definition \ref{10.3.1}, both $f(\sigma ({{a}_{1}}),\cdots ,\sigma ({{a}_{n}}))$ and $g(\sigma ({{b}_{1}}),\cdots ,\sigma ({{b}_{m}}))$ are well-defined and

    $f(\alpha ({{a}_{1}}),\cdots ,\alpha ({{a}_{n}}))$

$=f(\sigma ({{a}_{1}}),\cdots ,\sigma ({{a}_{n}}))$ (since $\alpha =\sigma {{|}_{A}}$)

$=\sigma (f({{a}_{1}},\cdots ,{{a}_{n}}))$ (by Definition \ref{10.3.1})

$=\sigma (g({{b}_{1}},\cdots ,{{b}_{m}}))$ (because $f({{a}_{1}},\cdots ,{{a}_{n}})=g({{b}_{1}},\cdots ,{{b}_{m}})$)

$=g(\sigma ({{b}_{1}}),\cdots ,\sigma ({{b}_{m}}))$ (by Definition \ref{10.3.1})

$=g(\alpha ({{b}_{1}}),\cdots ,\alpha ({{b}_{m}}))$ (because $\alpha =\sigma {{|}_{A}}$).

Hence by Notation \ref{15.2.1}, $A\xrightarrow{T,\alpha }S_2 $. 
 \end{proof}
 Proposition \ref{5.2.4} is generalized to the following, which also generalizes Proposition \ref{15.1.5}.
\begin{proposition} \label{15.2.3}
    Let $\alpha :U(\subseteq D)\to V(\subseteq D)$ be a map.  \begin{enumerate}
        \item [(a)] The following statements are equivalent:
    \begin{enumerate}
        \item [(i)] $U\xrightarrow{T,\alpha }V $.
        \item[(ii)] There exists a map $\sigma :{{\left\langle U \right\rangle }_{T}}\to {{\left\langle V  \right\rangle }_{T}}$ given by
$f({{u}_{1}},\cdots ,{{u}_{n}})\mapsto f(\alpha ({{u}_{1}}),\cdots ,\alpha ({{u}_{n}}))$, $\forall n\in {{\mathbb{Z}}^{+}}$, $({{u}_{1}},\cdots ,{{u}_{n}})\in {{U}^{n}}$ and $f\in T$ such that $f({{u}_{1}},\cdots ,{{u}_{n}})$ is well-defined.
    \end{enumerate}
    \item [(b)] The map $\sigma $ given in \emph{(ii)} is a $T$-morphism. 
    \end{enumerate}
 \end{proposition}
 \begin{proof} For (a):
 
     $U\xrightarrow{T,\alpha }V $;\\*
$\Leftrightarrow \forall n,m\in {{\mathbb{Z}}^{+}}$, $f\in T$ of $n$ variables, $g\in T$ of $m$ variables, $({{u}_{1}},\cdots ,{{u}_{n}})\in {{U}^{n}}$, and $({{w}_{1}},\cdots ,{{w}_{m}})\in {{U}^{m}}$,
\[f({{u}_{1}},\cdots ,{{u}_{n}})=g({{w}_{1}},\cdots ,{{w}_{m}})\] 
implies \[f(\alpha ({{u}_{1}}),\cdots ,\alpha ({{u}_{n}}))=g(\alpha ({{w}_{1}}),\cdots ,\alpha ({{w}_{m}}));\]
$\Leftrightarrow \sigma :{{\left\langle U \right\rangle }_{T}}\to {{\left\langle V  \right\rangle }_{T}}$ given by \[f({{u}_{1}},\cdots ,{{u}_{n}})\mapsto f(\alpha ({{u}_{1}}),\cdots ,\alpha ({{u}_{n}})),\] $\forall n\in {{\mathbb{Z}}^{+}}$, $({{u}_{1}},\cdots ,{{u}_{n}})\in {{U}^{n}}$ and $f\in T$ such that $f({{u}_{1}},\cdots ,{{u}_{n}})$ is well-defined, is a well-defined map. 

 For (b):

 Let $(a_1,\cdots ,a_n)\in \left\langle U \right\rangle _{T}^{n}$. Then $\forall i=1, \cdots, n$, $\exists {{g}_{i}}\in T$ of ${{n}_{i}}$ variables and ${{\text{z}}_{i}}\in {{U}^{{{n}_{i}}}}$ such that ${{g}_{i}}({{\text{z}}_{i}})={{a}_{i}}$. For ${{\text{z}}_{i}}:=({{u}_{1}},\cdots ,{{u}_{{{n}_{i}}}})$, we denote $(\alpha ({{u}_{1}}),\cdots ,\alpha ({{u}_{{{n}_{i}}}}))$ by $\alpha ({{\text{z}}_{i}})$ for brevity. 

Let $f\in T$ have $n$ variables. By Definition \ref{10.1.1}, $f\circ ({{g}_{1}},\cdots ,{{g}_{n}})$ is a partial function from ${{D}^{m}}$ to $D$ given by $({{\text{v}}_{1}},\cdots ,{{\text{v}}_{n}})\mapsto f({{g}_{1}}({{\text{v}}_{1}}),\cdots ,{{g}_{n}}({{\text{v}}_{n}}))$, where $({{\text{v}}_{1}},\cdots ,{{\text{v}}_{n}})=({{x}_{1}},\cdots ,{{x}_{m}})\in {{D}^{m}}$ and $m=\sum\nolimits_{i}{{{n}_{i}}}$. By Definition \ref{10.1.3}, $f\circ ({{g}_{1}},\cdots ,{{g}_{n}})$ is a restriction of some element of $T$. Hence

$f(a_1,\cdots ,a_n)$ is well-defined;\\* 
$\Leftrightarrow \sigma (f(a_1,\cdots ,a_n))$ 

$=\sigma (f({{g}_{1}}({{\text{z}}_{1}}),\cdots ,{{g}_{n}}({{\text{z}}_{n}})))$

$=\sigma ((f\circ ({{g}_{1}},\cdots ,{{g}_{n}}))({{\text{z}}_{1}},\cdots ,{{\text{z}}_{n}}))$ (by Definition \ref{10.1.1})

$=(f\circ ({{g}_{1}},\cdots ,{{g}_{n}}))(\alpha ({{\text{z}}_{1}}),\cdots ,\alpha ({{\text{z}}_{n}}))$ (by the definition of $\sigma $ in (ii))

$=f({{g}_{1}}(\alpha ({{\text{z}}_{1}})),\cdots ,{{g}_{n}}(\alpha ({{\text{z}}_{n}})))$ (by Definition \ref{10.1.1})

$=f(\sigma ({{g}_{1}}({{\text{z}}_{1}})),\cdots ,\sigma ({{g}_{n}}({{\text{z}}_{n}})))$ (by the definition of $\sigma $ in (ii))

$=f(\sigma ({{a}_{1}}),\cdots ,\sigma ({{a}_{n}}))$ 

is well-defined; \\*
as desired for $\sigma $ to be a $T$-morphism (by Definition \ref{10.3.1}).
 \end{proof}
Then Proposition \ref{5.2.5} is generalized to the following.
\begin{proposition} \label{15.2.4}
    Let $\alpha :U(\subseteq D)\to V(\subseteq D)$ be a map. 
     \begin{enumerate}
        \item [(a)] The following statements are equivalent: 
        \begin{enumerate}
        \item [(i)] $U\overset{T,\alpha }{\longleftrightarrow}V $.
        \item [(ii)] $\alpha $ is bijective and there exists a bijective map $\sigma :{{\left\langle U \right\rangle }_{T}}\to {{\left\langle V  \right\rangle }_{T}}$ given by $f({{u}_{1}},\cdots ,{{u}_{n}})\mapsto f(\alpha ({{u}_{1}}),\cdots ,\alpha ({{u}_{n}}))$, $\forall n\in {{\mathbb{Z}}^{+}}$, $({{u}_{1}},\cdots ,{{u}_{n}})\in {{U}^{n}}$ and $f\in T$ such that $f({{u}_{1}},\cdots ,{{u}_{n}})$ is well-defined.
    \end{enumerate}
     \item [(b)] The map $\sigma $ given in \emph{(ii)} is a $T$-isomorphism.  
 \end{enumerate}
\end{proposition}
\begin{proof}
    For (a): 
    
    $U\overset{T,\alpha }{\longleftrightarrow}V $; 
    
    $\Leftrightarrow$ $\alpha $ is bijective and $\forall n,m\in {{\mathbb{Z}}^{+}}$, $({{u}_{1}},\cdots ,{{u}_{n}})\in {{U}^{n}}$, $({{w}_{1}},\cdots ,{{w}_{m}})\in {{U}^{m}}$ and $f, g\in T$\\
    $f({{u}_{1}},\cdots ,{{u}_{n}})=g(w_1,\cdots ,w_m)\Leftrightarrow f(\alpha ({{u}_{1}}),\cdots ,\alpha ({{u}_{n}}))=g(\alpha ({{w}_{1}}),\cdots ,\alpha ({{w}_{m}}))$;
    
    $\Leftrightarrow$ $\alpha $ is bijective and $\sigma :{{\left\langle U \right\rangle }_{T}}\to {{\left\langle V  \right\rangle }_{T}}$ given by 
    \[f({{u}_{1}},\cdots ,{{u}_{n}})\mapsto f(\alpha ({{u}_{1}}),\cdots ,\alpha ({{u}_{n}})),\]
    $\forall n\in {{\mathbb{Z}}^{+}}$, $({{u}_{1}},\cdots ,{{u}_{n}})\in {{U}^{n}}$ and $f\in T$ such that $f({{u}_{1}},\cdots ,{{u}_{n}})$ is well-defined, is a well-defined bijective map.
    
    For (b): 
    
    By (b) in Proposition \ref{15.2.3}, $\sigma $ is a $T$-morphism. Since $\sigma $ is bijective, by Definition \ref{$T$-isomorphisms}, it is a $T$-isomorphism.
\end{proof}
Definition \ref{5.2.7} is generalized to
\begin{definition} \label{15.2.5}
    Let $T$ be a par-operator gen-semigroup on $D$, let $U\subseteq D$, and let $\sigma $ be a $T$-morphism from ${{\left\langle U \right\rangle }_{T}}$. If there exists a map $\alpha :U\to D$ such that $\sigma (f({{u}_{1}},\cdots ,{{u}_{n}}))=f(\alpha ({{u}_{1}}),\cdots ,\alpha ({{u}_{n}}))$, $\forall n\in {{\mathbb{Z}}^{+}}$, $({{u}_{1}},\cdots ,{{u}_{n}})\in {{U}^{n}}$ and $f\in T$ such that $f({{u}_{1}},\cdots ,{{u}_{n}})$ is well-defined, then we say that $\sigma $ is \emph{constructible by} $\alpha $, or just say that $\sigma $ is \emph{constructible} for brevity. 
\end{definition}
Corollary \ref{5.2.9} still holds, but its proof changes a little as follows.
\begin{corollary} \label{15.2.6}
    \emph{(Corollary \ref{5.2.9})} Let $U\subseteq D$. If $U\subseteq {{\left\langle U \right\rangle }_{T}}$, then any $T$-morphism $\sigma $ from ${{\left\langle U \right\rangle }_{T}}$ is constructible by $\alpha :=\sigma {{|}_{U}}$.
\end{corollary}
\begin{proof}
    Since $U\subseteq {{\left\langle U \right\rangle }_{T}}$, by Definition \ref{10.3.1}, 
\[\sigma (f({{u}_{1}},\cdots ,{{u}_{n}}))=f(\sigma ({{u}_{1}}),\cdots ,\sigma ({{u}_{n}}))=f(\alpha ({{u}_{1}}),\cdots ,\alpha ({{u}_{n}})),\] 
$\forall n\in {{\mathbb{Z}}^{+}}$, $({{u}_{1}},\cdots ,{{u}_{n}})\in {{U}^{n}}$ and $f\in T$ such that $f({{u}_{1}},\cdots ,{{u}_{n}})$ is well-defined.  
\end{proof}
However, we could not generalize Theorem \ref{5.2.10}. Indeed, in Theorem \ref{5.2.10}, if $T$ were replaced by a par-operator gen-semigroup $T=\left\langle g \right\rangle $ (generated by Definition \ref{10.1.4}, where $g$ is a partial function from some ${{D}^{n}}$ to $D$) and if we defined $\alpha $ in a way as in the proof of Theorem \ref{5.2.10}, then because any $f\in T$ may have more than one variable, we could not show that $\alpha $ is a well-defined map.
  
\subsection{A construction of $\theta $-morphisms}
In this subsection, unless otherwise specified, $\theta$-morphisms are defined by Definition \ref{10.4.2}, and ${{T}_{1}}$ and ${{T}_{2}}$ are par-operator gen-semigroups. 

Notation \ref{5.4.1} is generalized to the following, which also generalizes Notation \ref{15.2.1}.
\begin{notation} \label{15.4.1}
   Let $D_1$ (resp. $D_2$) be a set, let $M_1$ (resp. $M_2$) be a subset of $\bigcup\nolimits_{n\in {{\mathbb{Z}}^{+}}}{\{}$all partial functions from $D_{1}^{n}$ to $D_1$\} (resp. a subset of $\bigcup\nolimits_{n\in {{\mathbb{Z}}^{+}}}{\{}$all partial functions from $D_{2}^{n}$ to $D_2$\}), let $\theta \subseteq {{M}_{1}}\times {{M}_{2}}$, and let $\alpha :U(\subseteq {D_1})\to V(\subseteq {{D}_{2}})$ be a map. 
   
   By $U\xrightarrow{\theta ,\alpha }V $ we mean that
    $\forall f,g\in \operatorname{Dom}\theta $, $n,m\in {{\mathbb{Z}}^{+}} $, $({{u}_{1}},\cdots ,{{u}_{n}})\in {{U}^{n}}$, and $({{w}_{1}},\cdots ,{{w}_{m}})\in {{U}^{m}}$,
    $f({{u}_{1}},\cdots ,{{u}_{n}})=g({{w}_{1}},\cdots ,{{w}_{m}})$
    implies
    \[\theta (f)(\alpha ({{u}_{1}}),\cdots ,\alpha ({{u}_{n}}))=\theta (g)(\alpha ({{w}_{1}}),\cdots ,\alpha ({{w}_{m}})).\]

   Moreover, by $U\overset{\theta ,\alpha }{\longleftrightarrow}V $ we mean that $\alpha $ is bijective, $U\xrightarrow{\theta ,\alpha }V $ and $V \xrightarrow{{{\theta }^{-1}},{{\alpha }^{-1}}}U$, where ${{\theta }^{-1}}:=\{(h,f)\,|\,(f,h)\in \theta \}$.  
\end{notation}
\begin{remark}
    If ${{M}_{1}}={{M}_{2}}=:M$ and $\theta $ is the identity map on $M$, then $U\xrightarrow{\theta ,\alpha }V $ is equivalent to $U\xrightarrow{M,\alpha }V $ (defined by Notation \ref{15.2.1}).
\end{remark}
    Proposition \ref{5.4.2} would still hold if $T_1$ and $T_2$ in it were generalized to be par-operator gen-semigroups, but the proof would change as follows.
\begin{proposition} \label{15.4.2}
    Let $\phi $ be a $\theta $-morphism from a $T_1$-space $S_1$ to a $T_2$-space $S_2$. Then $\forall A\subseteq {{S}_{1}}$, $A\xrightarrow{\theta ,\alpha }S_2 $, where $\alpha :=\phi {{|}_{A}}$. In particular, ${{S}_{1}}\xrightarrow{\theta ,\phi }S_2 $.
\end{proposition} 
\begin{proof}
    By Definition \ref{10.4.2}, $\forall f,g\in \operatorname{Dom}\theta $, $(a_1,\cdots ,a_n)\in {{A}^{n}}$ and $({{b}_{1}},\cdots ,{{b}_{m}})\in {{A}^{m}}$, if $f(a_1,\cdots ,a_n)=g({{b}_{1}},\cdots ,{{b}_{m}})$, then

    $\theta (f)(\phi ({{a}_{1}}),\cdots ,\phi ({{a}_{n}}))$ 

$=\phi (f(a_1,\cdots ,a_n))$ (by Definition \ref{10.4.2})

$=\phi (g({{b}_{1}},\cdots ,{{b}_{m}}))$ (since $f(a_1,\cdots ,a_n)=g({{b}_{1}},\cdots ,{{b}_{m}})$)

$=\theta (g)(\phi ({{b}_{1}}),\cdots ,\phi ({{b}_{m}}))$ (by Definition \ref{10.4.2}),

and hence $\theta (f)(\alpha ({{a}_{1}}),\cdots ,\alpha ({{a}_{n}}))=\theta (g)(\alpha ({{b}_{1}}),\cdots ,\alpha ({{b}_{m}}))$. By Notation \ref{15.4.1}, $A\xrightarrow{\theta ,\alpha }S_2 $.
\end{proof}
	Definition \ref{5.3.9} is generalized as follows.
 \begin{definition} \label{15.3.2}
     Let $T_1$ and $T_2$ be par-operator gen-semigroups on $D_1$ and $D_2$, respectively, let $\theta \subseteq {{T}_{1}}\times {{T}_{2}}$, and let $A\subseteq {{D}_{2}}$. If  $\forall f,{{g}_{1}},\cdots ,{{g}_{n}}\in \operatorname{Dom}\theta $ and $\text{z}\in {{A}^{m}}$ such that $\theta (f\circ ({{g}_{1}},\cdots ,{{g}_{n}}))(\text{z})$ is well-defined,
			\[\theta (f\circ ({{g}_{1}},\cdots ,{{g}_{n}}))(\text{z})=(\theta (f)\circ (\theta ({{g}_{1}}),\cdots ,\theta ({{g}_{n}})))(\text{z}),\]
then we say that $\theta $ is \emph{distributive over} $A$.
 \end{definition}
  Propositions \ref{5.4.3} and \ref{15.2.3} are generalized as follows. 
 \begin{proposition} \label{15.4.3}
     Let $T_1$ and $T_2$ be par-operator gen-semigroups on $D_1$ and $D_2$, respectively, let $\theta \subseteq {{T}_{1}}\times {{T}_{2}}$ with $\operatorname{Dom}\theta ={{T}_{1}}$, and let $\alpha :U(\subseteq {D_1})\to V(\subseteq{{D}_{2}})$ be a map.   \begin{enumerate}
       \item [(a)] The following statements are equivalent:
     \begin{enumerate}
         \item [(i)]  $U\xrightarrow{\theta ,\alpha }V $. 
         \item[(ii)]  There exists a map $\phi :{{\left\langle U \right\rangle }_{{{T}_{1}}}}\to {{\left\langle V  \right\rangle }_{{{T}_{2}}}}$ given by
 \[f({{u}_{1}},\cdots ,{{u}_{n}})\mapsto \theta (f)(\alpha ({{u}_{1}}),\cdots ,\alpha ({{u}_{n}})),\] 
$\forall n\in {{\mathbb{Z}}^{+}}$, $({{u}_{1}},\cdots ,{{u}_{n}})\in {{U}^{n}}$ and $f\in {{T}_{1}}$ such that $f({{u}_{1}},\cdots ,{{u}_{n}})$ is well-defined.
     \end{enumerate}
  \item [(b)] Suppose that $\theta {{|}_{\operatorname{Im}\phi }}$ is a map and $\theta $ is distributive over $V $. Then the map $\phi $ in \emph{(ii)} is a $\theta $-morphism. 
 \end{enumerate}
 \end{proposition}
\begin{proof}
For (a): 
   
    $U\xrightarrow{\theta ,\alpha }V $;\\*
$\Leftrightarrow \forall f,g\in \operatorname{Dom}\theta (={{T}_{1}})$, $({{u}_{1}},\cdots ,{{u}_{n}})\in {{U}^{n}}$ and $(w_1,\cdots ,w_m)\in {{U}^{m}}$,  
\[f({{u}_{1}},\cdots ,{{u}_{n}})=g(w_1,\cdots ,w_m)\] implies 
\[\theta (f)(\alpha ({{u}_{1}}),\cdots ,\alpha ({{u}_{n}}))=\theta (g)(\alpha ({{v}_{1}}),\cdots ,\alpha ({{v}_{m}}));\]
$\Leftrightarrow \phi :{{\left\langle U \right\rangle }_{{{T}_{1}}}}\to {{\left\langle V  \right\rangle }_{{{T}_{2}}}}$ given by
\[f({{u}_{1}},\cdots ,{{u}_{n}})\mapsto \theta (f)(\alpha ({{u}_{1}}),\cdots ,\alpha ({{u}_{n}})),\] 
$\forall n\in {{\mathbb{Z}}^{+}}$, $({{u}_{1}},\cdots ,{{u}_{n}})\in {{U}^{n}}$ and $f\in {{T}_{1}}$ such that $f({{u}_{1}},\cdots ,{{u}_{n}})$ is well-defined, is a well-defined map. 

For (b):  

Since condition (i) in Definition \ref{10.4.2} is satisfied, it suffices show that condition (ii) in Definition \ref{10.4.2} is satisfied.

Let $f\in {{T}_{1}}$ have $n$ variables and let $(a_1,\cdots ,a_n)\in \left\langle U \right\rangle _{{{T}_{1}}}^{n}$. Then $\forall i=1, \cdots, n$, $\exists {{g}_{i}}\in {{T}_{1}}$ of ${{n}_{i}}$ variables and ${{\text{z}}_{i}}\in {{U}^{{{n}_{i}}}}$ such that ${{g}_{i}}({{\text{z}}_{i}})={{a}_{i}}$. For ${{\text{z}}_{i}}:=({{u}_{1}},\cdots ,{{u}_{{{n}_{i}}}})$, we denote $(\alpha ({{u}_{1}}),\cdots ,\alpha ({{u}_{{{n}_{i}}}}))$ by $\alpha ({{\text{z}}_{i}})$ for brevity.

By Definition \ref{10.1.1}, $f\circ ({{g}_{1}},\cdots ,{{g}_{n}})$ is a partial function from $D_{1}^{m}$ to ${D_1}$ given by $({{\text{v}}_{1}},\cdots ,{{\text{v}}_{n}})\mapsto f({{g}_{1}}({{\text{v}}_{1}}),\cdots ,{{g}_{n}}({{\text{v}}_{n}}))$, where $({{\text{v}}_{1}},\cdots ,{{\text{v}}_{n}})=({{x}_{1}},\cdots ,{{x}_{m}})\in D_{1}^{m}$ and $m=\sum\nolimits_{i}{{{n}_{i}}}$. By Definition \ref{10.1.3}, $f\circ ({{g}_{1}},\cdots ,{{g}_{n}})$ is a restriction of some element of ${{T}_{1}}$. Hence

$f(a_1,\cdots ,a_n)$ is well-defined; \\
$\Leftrightarrow \phi (f(a_1,\cdots ,a_n))$ 

$=\phi (f({{g}_{1}}({{\text{z}}_{1}}),\cdots ,{{g}_{n}}({{\text{z}}_{n}})))$

$=\phi ((f\circ ({{g}_{1}},\cdots ,{{g}_{n}}))({{\text{z}}_{1}},\cdots ,{{\text{z}}_{n}}))$ (by Definition \ref{10.1.1})

$=\theta (f\circ ({{g}_{1}},\cdots ,{{g}_{n}}))(\alpha ({{\text{z}}_{1}}),\cdots ,\alpha ({{\text{z}}_{n}}))$ (by the definition of $\phi $ in (ii))

$=(\theta (f)\circ (\theta ({{g}_{1}}),\cdots ,\theta ({{g}_{n}})))(\alpha ({{\text{z}}_{1}}),\cdots ,\alpha ({{\text{z}}_{n}}))$ (since $\theta $ is distributive over $V $)

$=\theta (f)(\theta ({{g}_{1}})(\alpha ({{\text{z}}_{1}})),\cdots ,\theta ({{g}_{n}})(\alpha ({{\text{z}}_{n}})))$ (by Definition \ref{10.1.1})

$=\theta (f)(\phi ({{g}_{1}}({{\text{z}}_{1}})),\cdots ,\phi ({{g}_{n}}({{\text{z}}_{n}})))$ (by the definition of $\phi $ in (ii))

$=\theta (f)(\phi ({{a}_{1}}),\cdots ,\phi ({{a}_{n}}))$ 

is well-defined; \\*
as desired for $\phi $ to be a $\theta$-morphism.	
\end{proof}
The following generalizes both Propositions \ref{5.4.4} and \ref{15.2.4}.
 \begin{proposition} \label{15.4.4}
    Let $T_1$, $T_2$, and $\alpha $ be defined as in Proposition \ref{15.4.3}. Let $\theta \subseteq {{T}_{1}}\times {{T}_{2}}$ such that $\operatorname{Dom}\theta ={{T}_{1}}$ and $\operatorname{Im}\theta ={{T}_{2}}$. \begin{enumerate}
     \item [(a)]  The following statements are equivalent:
     \begin{enumerate}
    \item[(i)] $U\overset{\theta ,\alpha }{\longleftrightarrow}V $. 
    \item[(ii)] $\alpha $ is bijective, there exists a bijective map $\phi :{{\left\langle U \right\rangle }_{{{T}_{1}}}}\to {{\left\langle V  \right\rangle }_{{{T}_{2}}}}$ given by 
    \[f({{u}_{1}},\cdots ,{{u}_{n}})\mapsto \theta (f)(\alpha ({{u}_{1}}),\cdots ,\alpha ({{u}_{n}})),\] 
    $\forall n\in {{\mathbb{Z}}^{+}}$, $({{u}_{1}},\cdots ,{{u}_{n}})\in {{U}^{n}}$ and $f\in {{T}_{1}}$ such that $f({{u}_{1}},\cdots ,{{u}_{n}})$ is well-defined, and its inverse $\phi^{-1}:{{\left\langle V  \right\rangle }_{{{T}_{2}}}}\to {{\left\langle U \right\rangle }_{{{T}_{1}}}}$ can be given by 
    \[g(v_1, \cdots, v_n)\mapsto \theta^{-1} (g)(\alpha^{-1} (v_1),\cdots,\alpha^{-1} (v_n)),\] 
    $\forall n\in {{\mathbb{Z}}^{+}}$, $(v_1, \cdots, v_n)\in V^n$ and $g\in {{T}_{2}}$ such that $g(v_1, \cdots, v_n)$ is well-defined.
   \end{enumerate}
    \item [(b)] Suppose that $\theta $ is distributive over $V $ and $\theta {{|}_{\operatorname{Im}\phi }}$ is a map. Then the bijective map $\phi $ given in \emph{(ii)} is a $\theta $-isomorphism.
   \end{enumerate}
 \end{proposition}
\begin{proof}
    For (a): 
    
$U\overset{\theta ,\alpha }{\longleftrightarrow}V $;

$\Leftrightarrow$ $\alpha $ is bijective, $U\xrightarrow{\theta ,\alpha }V $ and $ V \xrightarrow{{{\theta }^{-1}},{{\alpha }^{-1}}}U$ (by Notation \ref{15.4.1});

$\Leftrightarrow$ $\alpha $ is bijective, $\phi :{{\left\langle U \right\rangle }_{{{T}_{1}}}}\to {{\left\langle V  \right\rangle }_{{{T}_{2}}}}$ given by 
\[f({{u}_{1}},\cdots ,{{u}_{n}})\mapsto \theta (f)(\alpha ({{u}_{1}}),\cdots ,\alpha ({{u}_{n}})),\]
$\forall n\in {{\mathbb{Z}}^{+}}$, $({{u}_{1}},\cdots ,{{u}_{n}})\in {{U}^{n}}$ and $f\in {{T}_{1}}$ such that $f({{u}_{1}},\cdots ,{{u}_{n}})$ is well-defined, is a well-defined map, and $\phi' :{{\left\langle V  \right\rangle }_{{{T}_{2}}}}\to {{\left\langle U \right\rangle }_{{{T}_{1}}}}$ given by 
\[g(v_1, \cdots, v_n)\mapsto \theta^{-1} (g)(\alpha^{-1} (v_1),\cdots,\alpha^{-1} (v_n)),\]
$\forall n\in {{\mathbb{Z}}^{+}}$, $(v_1, \cdots, v_n)\in V^n$ and $g\in {{T}_{2}}$ such that $g(v_1, \cdots, v_n)$ is well-defined, is also a well-defined map (by Proposition \ref{15.4.3});

$\Leftrightarrow$ $\alpha $ is bijective, there is a bijective map $\phi :{{\left\langle U \right\rangle }_{{{T}_{1}}}}\to {{\left\langle V  \right\rangle }_{{{T}_{2}}}}$ given by 
\[f({{u}_{1}},\cdots ,{{u}_{n}})\mapsto \theta (f)(\alpha ({{u}_{1}}),\cdots ,\alpha ({{u}_{n}})),\]
$\forall n\in {{\mathbb{Z}}^{+}}$, $({{u}_{1}},\cdots ,{{u}_{n}})\in {{U}^{n}}$ and $f\in {{T}_{1}}$ such that $f({{u}_{1}},\cdots ,{{u}_{n}})$ is well-defined, and its inverse $\phi^{-1}:{{\left\langle V  \right\rangle }_{{{T}_{2}}}}\to {{\left\langle U \right\rangle }_{{{T}_{1}}}}$ can be given by 
\[g(v_1, \cdots, v_n)\mapsto \theta^{-1} (g)(\alpha^{-1} (v_1),\cdots,\alpha^{-1} (v_n)),\] 
$\forall n\in {{\mathbb{Z}}^{+}}$, $(v_1, \cdots, v_n)\in V^n$ and $g\in {{T}_{2}}$ such that $g(v_1, \cdots, v_n)$ is well-defined.

   The third equivalence relation is explained as follows. The sufficiency ($\Leftarrow$) is obvious, so we only show the necessity ($\Rightarrow$). For this purpose, we show that both $\phi'\circ\phi$ and $\phi\circ\phi'$ are the identity map.
   
   Let $(f,g)\in \theta$ and $(u_1, \cdots, u_n)\in U^n$ such that $f(u_1, \cdots, u_n)$ is well-defined. Then 
   
   $\phi'(\phi(f(u_1, \cdots, u_n)))$
   
   $=\phi'(\theta(f)(\alpha (u_1),\cdots,\alpha (u_n)))$ (by the definition of $\phi$) 

   $=\phi'(g(\alpha (u_1),\cdots,\alpha (u_n)))$ (by Definition \ref{10.4.a})
   
   $=\theta^{-1}(g)(u_1, \cdots, u_n)$ (by the definition of $\phi'$) 
   
   $=f(u_1, \cdots, u_n)$ (by (the $\theta^{-1}$ version of) Definition \ref{10.4.a})
   
   Then $\phi'\circ\phi$ is the identity map on ${{\left\langle U  \right\rangle }_{{{T}_{1}}}}$ (because $\operatorname{Dom}\theta ={{T}_{1}}$).
   
   Analogously, we can show that $\phi\circ\phi'$ is the identity map on ${{\left\langle V  \right\rangle }_{{{T}_{2}}}}$. Therefore, both $\phi$ and $\phi'$ are bijective and $\phi'$ is the inverse of $\phi$.
   
       For (b): 

    By (b) in Proposition \ref{15.4.3}, $\phi $ is a $\theta$-morphism. Since $\phi $ is bijective, $\phi $ is a $\theta$-isomorphism.
\end{proof}
  Definition \ref{5.4.6} is generalized as follows, which also generalizes Definition \ref{15.2.5}.
\begin{definition} \label{15.4.5}
    Let $T_1$ and $T_2$ be par-operator gen-semigroups on $D_1$ and $D_2$, respectively, let $\theta \subseteq {{T}_{1}}\times {{T}_{2}}$, let $U \subseteq D_1$ and let $\phi$ be a $\theta$-morphism from ${\left\langle U \right\rangle }_{{{T}_{1}}}$ to a $T_2$-space. If $\operatorname{Dom}\theta ={{T}_{1}}$ and there exists a map $\alpha :U\to {{D}_{2}}$ such that $\phi (f({{u}_{1}},\cdots ,{{u}_{n}}))=\theta (f)(\alpha ({{u}_{1}}),\cdots ,\alpha ({{u}_{n}}))$, $\forall n\in {{\mathbb{Z}}^{+}}$, $({{u}_{1}},\cdots ,{{u}_{n}})\in {{U}^{n}}$ and $f\in {{T}_{1}}$ such that $f({{u}_{1}},\cdots ,{{u}_{n}})$ is well-defined, then we say that $\phi $ is \emph{constructible by} $\alpha $, or just say that $\phi $ is \emph{constructible} for brevity.
\end{definition}  

Corollary \ref{5.4.8} would still hold if $T_1$ and $T_2$ in it were generalized to be par-operator gen-semigroups, as shown below. The following also generalizes Corollary \ref{15.2.6}. 
\begin{corollary} \label{15.4.6}
    Let $T_1$ and $T_2$ be par-operator gen-semigroups on $D_1$ and $D_2$, respectively, let $\theta \subseteq {{T}_{1}}\times {{T}_{2}}$ and let $U\subseteq {D_1}$. If $\operatorname{Dom}\theta ={{T}_{1}}$ and $U\subseteq {{\left\langle U \right\rangle }_{{{T}_{1}}}}$, then any $\theta$-morphism $\phi $ from ${{\left\langle U \right\rangle }_{{{T}_{1}}}}$ to a $T_2$-space is constructible by $\alpha :=\phi {{|}_{U}}$.
\end{corollary}
\begin{proof}
    Since $\operatorname{Dom}\theta ={{T}_{1}}$ and $U\subseteq {{\left\langle U \right\rangle }_{{{T}_{1}}}}$, by Definition \ref{10.4.2}, for all $n\in {{\mathbb{Z}}^{+}}$, $({{u}_{1}},\cdots ,{{u}_{n}})\in {{U}^{n}}$ and $f\in {{T}_{1}}$ such that $f({{u}_{1}},\cdots ,{{u}_{n}})$ is well-defined,
\[\phi (f({{u}_{1}},\cdots ,{{u}_{n}}))=\theta (f)(\phi ({{u}_{1}}),\cdots ,\phi ({{u}_{n}}))=\theta (f)(\alpha ({{u}_{1}}),\cdots ,\alpha ({{u}_{n}})).\]
\end{proof}
	As we said at the end of Subsection \ref{16 cons of T-mor from U}, we could not generalize Theorem \ref{5.2.10}. For the same reason, we could not generalize Theorem \ref{5.4.9}, either.

	For a par-operator gen-semigroup $T$, we could not generalize results in Subsection \ref{I Another construction of $T$-morphisms} because now any $f\in T$ may have more than one variable.

$ $

\addcontentsline{toc}{section}{Part III. \textbf{Solvability of equations}}

\begin{center}
Part III. SOLVABILITY OF EQUATIONS
\end{center}

	In Part III, we shall deviate from the topics of Parts I and II and study solvability of equations. A main goal of the classical Galois theory is to study the solvability by radicals of polynomial equations. In Section \ref{poly equ}, we shall introduce a new understanding of solvability of polynomial equations, which involves a composition series of the Galois group of the polynomial. Analogously, for homogeneous linear differential equations, in Section \ref{diff equ}, we shall introduce a solvability which involves a normal series of the differential Galois group of the differential equation. In Section \ref{equ sol}, we shall generalize our results to “general” equations in terms of the theory developed in Parts I and II. 
 
\section{A solvability of polynomial equations} \label{poly equ}
    In this section, we shall show that for any separable polynomial $p(x)$ over a field $B$ with a splitting field $E\supsetneq B$, there are $m\in {\mathbb{Z}}^{+}$ polynomials $p_1(x), \cdots, p_m(x)$ such that the Galois group of each ${{p}_{i}}(x)$ is simple and for any root of $p(x)$ in $E$, there is a formula which only involves rational functions over $B$ on the roots of $p_1(x),\cdots, p_m(x)$ in $E$.
 \subsection{Canonical extensions and canonical solvability}
\begin{lemma} \label{17.1.1}
    Let $E/B$ be an algebraic field extension. Suppose $B={{E}^{\operatorname{Gal}(E/B)}}$, where ${{E}^{\operatorname{Gal}(E/B)}}$ denotes the fixed field of the Galois group $\operatorname{Gal}(E/B)$, and $G:=\operatorname{Gal}(E/B)$ has a composition series $G={{G}_{0}}\,\triangleright\, {{G}_{1}}\,\triangleright\, \cdots \,\triangleright\, {{G}_{m}}=\{1\}$. Then there are intermediate fields ${{B}_{0}},{{B}_{1}},\cdots ,{{B}_{m}}$ such that $B={{B}_{0}}\subseteq {{B}_{1}}\subseteq \cdots \subseteq {{B}_{m}}=E$, ${{B}_{i}}/{{B}_{i-1}}$ is Galois, and $\operatorname{Gal}({{B}_{i}}/{{B}_{i-1}})\cong {{G}_{i-1}}/{{G}_{i}}$ is a simple group or the trivial group $\{1\}$, $\forall i=1, \cdots, m$.
\end{lemma}
\begin{proof}
    $\forall i=0, \cdots, m$, let ${{B}_{i}}={{E}^{{{G}_{i}}}}$. Then $B={{B}_{0}}\subseteq {{B}_{1}}\subseteq \cdots \subseteq {{B}_{m}}=E$ and $\forall i=0, \cdots, m$,
$\operatorname{Gal}(E/{{B}_{i}})=\operatorname{Gal}(E/{{E}^{{{G}_{i}}}})={{G}_{i}}$, and hence
\[G=\operatorname{Gal}(E/{{B}_{0}})\,\triangleright\, \operatorname{Gal}(E/{{B}_{1}})\,\triangleright\, \cdots \,\triangleright\, \operatorname{Gal}(E/{{B}_{m}})=\{1\}.\]

Since the normal series is a composition series, by the fundamental theorem of the classical Galois theory or infinite Galois theory, $\forall i=1, \cdots, m$, ${{B}_{i}}/{{B}_{i-1}}$ is Galois and
\[\operatorname{Gal}({{B}_{i}}/{{B}_{i-1}})\cong \operatorname{Gal}(E/{{B}_{i-1}})/\operatorname{Gal}(E/{{B}_{i}})={{G}_{i-1}}/{{G}_{i}}\] 
is a simple group or the trivial group.
\end{proof}
	Lemma \ref{17.1.1} implies that we may define a notion of solvability other than the solvability by radicals. In Subsection \ref{Formulas for the roots}, it will be clear why we introduce the notions as follows.
\begin{definition} \label{17.1.2}
    A field extension $F/B$ is called an \emph{irreducible Galois extension} if $F/B$ is Galois and $\operatorname{Gal}(F/B)$ is a simple group or the trivial group.  A field extension $E/B$ is called a \emph{canonical $($field$)$ extension} if there is a tower of fields $B={{B}_{0}}\subseteq {{B}_{1}}\subseteq \cdots \subseteq {{B}_{m}}=E$, where $m\in {{\mathbb{Z}}^{+}}$ and $\forall i=1, \cdots, m$, ${{B}_{i}}/{{B}_{i-1}}$ is an irreducible Galois extension.
\end{definition}
\begin{definition} \label{17.1.3}
    A polynomial $p(x)$ is said to be \emph{canonically solvable} if there is a canonical field extension $E/B$ such that $p(x)$ is over $B$, $p(x)$ is not over any proper subfield of $B$, and $E$ is a splitting field of $p(x)$.
\end{definition}
\begin{proposition} \label{17.1.4}
    Every separable polynomial is canonically solvable, and so is every polynomial over a field of characteristic \emph{0}.
\end{proposition}
\begin{proof}
    Let $p(x)$ be a separable polynomial over a field $B$ with a splitting field $E$ such that $p(x)$ is not over any proper subfield of $B$. Then $E/B$ is Galois and $B={{E}^{\operatorname{Gal}(E/B)}}$. Since $\operatorname{Gal}(E/B)$ is finite as the Galois group of a polynomial, it has a composition series. Then by Lemma \ref{17.1.1} and Definition \ref{17.1.3}, $p(x)$ is canonically solvable.
\end{proof}

\subsection{Formulas for roots of separable polynomials} \label{Formulas for the roots}
Now let’s explain why we introduced the notion of canonical solvability.

As is well-known, when we say that a polynomial over a field is solvable by radicals, we imply that there are formulas for the roots of the polynomial which may involve extraction of roots ($\sqrt[n]{a},n\in {{\mathbb{Z}}^{+}}$), rational functions, and no operation else. Thus actually we assume that solving equations in the form of ${X}^{n}=a$ $(n\in {{\mathbb{Z}}^{+}})$ is a basic operation. Analogous to this assumption, now we are assuming that solving a separable polynomial whose Galois group is simple is a basic operation. Said differently, we assume that the roots of any separable polynomial with a simple Galois group can be determined (by symbols, formulas, or values). Then theoretically, there exists a formula for any root of any separable polynomial, explained below by a process described in steps.
\begin{step} \label{17.2.1}
Given a separable polynomial $p(x)$ over a field $B$ with a splitting field $E$, by Proposition \ref{17.1.4} and Lemma \ref{17.1.1}, there exist intermediate fields ${{B}_{0}},{{B}_{1}},\cdots ,{{B}_{m}}$ such that $B={{B}_{0}}\subseteq {{B}_{1}}\subseteq \cdots \subseteq {{B}_{m}}=E$, ${{B}_{i}}/{{B}_{i-1}}$ is Galois, and $\operatorname{Gal}({{B}_{i}}/{{B}_{i-1}})\simeq {{G}_{i-1}}/{{G}_{i}}$ is a simple group or the trivial group, $\forall i=1, \cdots, m(\in {{\mathbb{Z}}^{+}})$. Obviously we only need to deal with the case where each  $\operatorname{Gal}({{B}_{i}}/{{B}_{i-1}})$ is a simple group, or equivalently, where each ${{B}_{i}}/{{B}_{i-1}}$ is an irreducible Galois extension (Definition \ref{17.1.2}) and ${{B}_{i}}\ne {{B}_{i-1}}$. 
\end{step}
\begin{step} \label{17.2.2}
    Since ${{B}_{1}}/{{B}_{0}}$ is an irreducible Galois extension and ${{B}_{1}}\ne {{B}_{0}}$, there is a separable polynomial ${{p}_{1}}(x)\in {{B}_{0}}[x]$ such that ${{B}_{1}}$ is the splitting field of ${{p}_{1}}(x)$ in $E$ and the Galois group of ${{p}_{1}}(x)$ is simple. Now we assume that every root of ${{p}_{1}}(x)$ can be determined (by a symbol, a formula, or a value). Then ${{B}_{1}}$, which is a splitting field of ${{p}_{1}}(x)$, can be generated over $B$ by the roots of ${{p}_{1}}(x)$ in $E$. Hence every element of ${{B}_{1}}$ can be determined. Said precisely, $\forall b\in {{B}_{1}}$, there is a formula for $b$ which only involves rational functions over $B$ on the roots of ${{p}_{1}}(x)$ in $E$.
\end{step}
\begin{step} \label{17.2.3}
    Similarly, because ${{B}_{2}}/{{B}_{1}}$ is an irreducible Galois extension and ${{B}_{2}}\ne {{B}_{1}}$, there is a separable polynomial ${{p}_{2}}(x)$ over ${{B}_{1}}$ such that ${{B}_{2}}$ is the splitting field of ${{p}_{2}}(x)$ in $E$ and the Galois group of ${{p}_{2}}(x)$ is simple. Again we assume that every root of ${{p}_{2}}(x)$ in $E$ can be determined. Then ${{B}_{2}}$, which is a splitting field of ${{p}_{2}}(x)$, can be generated over ${{B}_{1}}$ by the roots of ${{p}_{2}}(x)$ in $E$. And as was shown in Step \ref{17.2.2}, ${{B}_{1}}$ can be generated over $B$ by the roots of ${{p}_{1}}(x)$ in $E$. Hence by substitutions, ${{B}_{2}}$ can be generated over $B$ by the roots of ${{p}_{1}}(x)$ and ${{p}_{2}}(x)$ in $E$. Thus, $\forall b\in {{B}_{2}}$, there is a formula for $b$ which only involves rational functions over $B$ on the roots of ${{p}_{1}}(x)$ and ${{p}_{2}}(x)$ in $E$.
\end{step}
\begin{step} \label{17.2.4}
    By induction, eventually for each element of ${{B}_{m}}=E$, there is a formula which only involves rational functions over $B$ on the roots of $p_1(x),\cdots, p_m(x)$ in $E$, where each ${{p}_{i}}(x)$ is a separable polynomial over ${{B}_{i-1}}$ whose splitting field in $E$ is ${{B}_{i}}$ and whose Galois group is simple. Since $E$ is a splitting field of $p(x)$, for any root of $p(x)$ in $E$, there is a formula which only involves rational functions over $B$ on the roots of $p_1(x),\cdots, p_m(x)$ in $E$.
\end{step}
In fact, in the above process, the operation of solving equation $p(x)=0$ is “reduced” to the operations of solving $p_1(x)=0,\cdots, p_m(x)=0$, where the Galois group of each ${{p}_{i}}(x)$ is simple, which implies that the operation of solving ${{p}_{i}}(x)=0$ is “irreducible”. This is why we use the terminology “canonical”.

The above process applies to any polynomial which is canonically solvable. Therefore, the following is obvious, which also explains why we introduced the notion of canonical solvability.

\begin{corollary} \label{17.2.5}
    Let $p(x)$ be a polynomial over a field $B$ with a splitting field $E$ such that $E\supsetneq B$. If $p(x)$ is separable, or more generally, $p(x)$ is canonically solvable, then there exist $m(\in {{\mathbb{Z}}^{+}})$ polynomials $p_1(x),\cdots, p_m(x)$ such that the Galois group of each ${{p}_{i}}(x)$ is simple and for any root of $p(x)$ in $E$, there is a formula which only involves rational functions over $B$ on the roots of $p_1(x),\cdots, p_m(x)$ in $E$.
\end{corollary}

\section{A solvability of homogeneous linear differential equations} \label{diff equ}
	For homogeneous linear differential equations, we can do a similar thing as we just did for polynomial equations, although the process will be relatively complicated. 
\subsection{Normal Zariski-closed series and composition Zariski-closed series}
In differential Galois theory, the Picard-Vessiot extension of a homogeneous linear differential equation is the analogue of a splitting field of a polynomial (see e.g. \cite{1,4,9}). However, by the fundamental theorem of differential Galois theory, only Zariski-closed subgroups (not necessarily all subgroups) of a differential Galois group have a bijective correspondence with intermediate differential fields of a Picard-Vessiot extension. Hence we first give the following definition.
\begin{definition} \label{18.1.1}
    Let $G={{G}_{0}}\,\triangleright\, {{G}_{1}}\,\triangleright\, \cdots \,\triangleright\, {{G}_{m}}=\{1\}\,(m\in {{\mathbb{Z}}^{+}})$ be a normal series where every ${{G}_{i}}$ is a Zariski-closed subgroup of a general linear group over a field. Then we call it a \emph{normal Zariski-closed series} of $G$. Moreover, if any possible proper refinement of this normal series is not a normal Zariski-closed one, then we call it a \emph{composition Zariski-closed series} of $G$.
\end{definition}
\begin{remark} 
   We do not require that a composition Zariski-closed series to be a composition series.
\end{remark}
	The following is comparable with Lemma \ref{17.1.1}.
 \begin{lemma} \label{18.1.2}
     Let a differential field extension $B\subseteq E$ be a Picard-Vessiot extension such that the constant field of $B$ is algebraically closed. Suppose that the differential Galois group $\operatorname{Gal}(E/B)$ has a normal Zariski-closed series 
     \[\operatorname{Gal}(E/B)=:{G_0} \,\triangleright\, G_1 \,\triangleright\, \cdots \,\triangleright\, {{G}_{m}}=\{1\}.\] 
     Then there are intermediate differential fields ${{B}_{0}},{{B}_{1}},\cdots ,{{B}_{m}}$ such that 
     \[B={{B}_{0}}\subseteq {{B}_{1}}\subseteq \cdots \subseteq {{B}_{m}}=E,\]
     $\forall i=1,\cdots, m$, ${{B}_{i-1}}\subseteq {{B}_{i}}$ is a Picard-Vessiot extension and $\operatorname{Gal}({{B}_{i}}/{{B}_{i-1}})\cong {{G}_{i-1}}/{{G}_{i}}$, and $\forall i=0,\cdots, m$, $\operatorname{Gal}(E/{{B}_{i}})={{G}_{i}}$.
 \end{lemma}
\begin{proof}
    Let $C$ be the constant field of $B$, which is algebraically closed. Hence we can apply the fundamental theorem of differential Galois theory (see e.g. \cite{1,4,9}) for any Picard-Vessiot extension of $B$.

    $\forall i=0,\cdots, m$, let ${{B}_{i}}={{E}^{{{G}_{i}}}}$. Then 
    \[B={{E}^{\operatorname{Gal}(E/B)}}={{E}^{{{G}_{0}}}}={{B}_{0}}\subseteq {{B}_{1}}\subseteq \cdots \subseteq {{B}_{m}}=E\]
    and, by the fundamental theorem, each $\operatorname{Gal}(E/{{B}_{i}})=\operatorname{Gal}(E/{{E}^{{{G}_{i}}}})={{G}_{i}}$ because each ${{G}_{i}}$ is a Zariski-closed subgroup of $\operatorname{Gal}(E/B)$. Hence
\[\operatorname{Gal}(E/B)=\operatorname{Gal}(E/{{B}_{0}})\,\triangleright\, \operatorname{Gal}(E/{{B}_{1}})\,\triangleright\, \cdots \,\triangleright\, \operatorname{Gal}(E/{{B}_{m}})=\{1\}.\]

Moreover, $\forall i=0,\cdots, m$, ${{B}_{i}}={{E}^{{{G}_{i}}}}$ is a differential field with $B\subseteq {{B}_{i}}\subseteq E$. And since the constant field $C$ of $B={{B}_{0}}$ is algebraically closed, by the fundamental theorem, ${{B}_{0}}\subseteq {{B}_{1}}$ is a Picard-Vessiot extension because $\operatorname{Gal}(E/{{B}_{0}})\,\triangleright\, \operatorname{Gal}(E/{{B}_{1}})$. Hence by the definition of Picard-Vessiot extensions (see e.g. \cite{1}), $C$ is also the constant field of ${{B}_{1}}$. Again by the fundamental theorem, ${{B}_{1}}\subseteq {{B}_{2}}$ is also a Picard-Vessiot extension because $\operatorname{Gal}(E/{{B}_{1}})\,\triangleright\, \operatorname{Gal}(E/{{B}_{2}})$. By induction, eventually each ${{B}_{i-1}}\subseteq {{B}_{i}}$ is a Picard-Vessiot extension and each $\operatorname{Gal}({{B}_{i}}/{{B}_{i-1}})\cong \operatorname{Gal}(E/{{B}_{i-1}})/\operatorname{Gal}(E/{{B}_{i}})={{G}_{i-1}}/{{G}_{i}}$ by the fundamental theorem.			
\end{proof}
Because $\operatorname{Gal}(E/B)\,\triangleright\, \operatorname{Gal}(E/E)=\{1\}$, $\operatorname{Gal}(E/B)$ always has a normal Zariski-closed series. Thus we can employ Lemma \ref{18.1.2} to any Picard-Vessiot extension whenever the constant field is algebraically closed. 

\subsection{Canonical extensions and canonical solvability}
	Let $L(Y)=0$ be a homogeneous linear differential equation over $B$ for which $B\subseteq E$ is a Picard-Vessiot extension with an algebraically closed field of constants. Then Lemma \ref{18.1.2} implies that $L(Y)=0$ can be “decomposed” into homogeneous linear differential equations ${{L}_{1}}(Y)=0, \cdots, {{L}_{m}}(Y)=0$ such that $\forall i=1,\cdots,m$,  ${{B}_{i-1}}\subseteq {{B}_{i}}$ is a Picard-Vessiot extension for ${{L}_{i}}(Y)=0$, and hence we can solve $L(Y)=0$ by solving ${{L}_{1}}(Y)=0, \cdots, {{L}_{m}}(Y)=0$. We shall explain this process in Subsection \ref{Formulas for the solutions}. In light of this idea, we define a notion of solvability other than solvability by quadratures.
\begin{definition} \label{18.2.1}
    A Picard-Vessiot extension $B\subseteq E$ is said to be \emph{irreducible} if for any $K$ with $B\subsetneq K\subsetneq E$, $K$ is not a Picard-Vessiot extension of $B$. Otherwise it is said to be \emph{reducible}.
\end{definition}
\begin{remark}
    Trivially, $B=E$ is an irreducible Picard-Vessiot extension.
\end{remark}
	Then by the fundamental theorem of differential Galois theory, the following, which characterizes reducible Picard-Vessiot extensions, is obvious.
 \begin{corollary} \label{18.2.2}
     Let $B\subseteq E$ be a Picard-Vessiot extension with the constant field of $B$ being algebraically closed. Then $B\subseteq E$ is reducible if and only if $\operatorname{Gal}(E/B)$ has a proper nontrivial normal Zariski-closed subgroup.
 \end{corollary}
The following two definitions are comparable with Definitions \ref{17.1.2} and \ref{17.1.3}, respectively.
\begin{definition} \label{18.2.3}
A differential field extension $B\subseteq E$ is called a \emph{canonical extension} if there is a tower of differential fields $B={{B}_{0}}\subseteq {{B}_{1}}\subseteq \cdots \subseteq {{B}_{m}}=E$, where $m\in {{\mathbb{Z}}^{+}}$, such that $\forall i=1,\cdots,m$, ${{B}_{i-1}}\subseteq {{B}_{i}}$ is an irreducible Picard-Vessiot extension.
\end{definition}
\begin{definition} \label{18.2.4}
    Let $L(Y)=0$ be a homogeneous linear differential equation. If there is a Picard-Vessiot extension $B\subseteq E$ for $L(Y)=0$ such that
\begin{enumerate}
 \item [(i)] the constant field of $B$ is algebraically closed;
\item [(ii)] there does not exist any proper differential subfield ${B}'$ of $B$ such that $L(Y)=0$ is over ${B}'$ and the constant field of ${B}'$ is algebraically closed; and
\item [(iii)] $B\subseteq E$ is a canonical extension; 
\end{enumerate}
then $L(Y)=0$ is said to be \emph{canonically solvable}.
\end{definition}
	The following characterizes canonical solvability in terms of a composition Zariski-closed series of the Galois group.
 \begin{proposition} \label{18.2.5}
 Let $L(Y)=0$ be a homogeneous linear differential equation. Then $L(Y)=0$ is canonically solvable if and only if there is a Picard-Vessiot extension $B\subseteq E$ for $L(Y)=0$ such that conditions \emph{(i)} and \emph{(ii)} in Definition \ref{18.2.4} are satisfied and $\operatorname{Gal}(E/B)$ has a composition Zariski-closed series.
 \end{proposition}
\begin{proof}
Assume that there is a Picard-Vessiot extension $B\subseteq E$ for $L(Y)=0$ such that conditions (i) and (ii) in Definition \ref{18.2.4} are satisfied. Then by Definition \ref{18.2.4}, it is sufficient to show that the following two statements are equivalent.

(a) $B\subseteq E$ is a canonical extension.

(b) $\operatorname{Gal}(E/B)$ has a composition Zariski-closed series.\\
1. (a) $\Rightarrow $ (b):
    
Suppose that $B\subseteq E$ is a canonical extension. By Definition \ref{18.2.3}, there is a tower of differential fields $B={{B}_{0}}\subseteq {{B}_{1}}\subseteq \cdots \subseteq {{B}_{m}}=E$ where each ${{B}_{i-1}}\subseteq {{B}_{i}}$ is an irreducible Picard-Vessiot extension. Hence by the definition of Picard-Vessiot extensions, the constant field of every ${{B}_{i}}$ is the same, which we assume to be $C$. Because condition (i) in Definition \ref{18.2.4} is assumed to be satisfied, $C$ is algebraically closed, and hence we can apply the fundamental theorem of differential Galois theory to any Picard-Vessiot extension of any ${{B}_{i}}$ as follows.

$(\operatorname{Gal}(E/B)=)\operatorname{Gal}(E/{{B}_{0}})\,\triangleright\, \operatorname{Gal}(E/{{B}_{1}})$ because both ${{B}_{0}}\subseteq E$ and ${{B}_{0}}\subseteq {{B}_{1}}$ are Picard-Vessiot extensions. Analogously, since $L(Y)=0$ can be viewed as defined over ${{B}_{1}}$, ${{B}_{1}}\subseteq E$ is a Picard-Vessiot extension, and hence $\operatorname{Gal}(E/{{B}_{1}})\,\triangleright\, \operatorname{Gal}(E/{{B}_{2}})$ because ${{B}_{1}}\subseteq {{B}_{2}}$ is also Picard-Vessiot. Applying this argument for each ${{B}_{i-1}}\subseteq {{B}_{i}}$, eventually we obtain
\begin{equation}
    \label{16.2(1)}
    \operatorname{Gal}(E/B)=\operatorname{Gal}(E/{{B}_{0}})\,\triangleright\, \operatorname{Gal}(E/{{B}_{1}})\,\triangleright\, \cdots \,\triangleright\, \operatorname{Gal}(E/{{B}_{m}})=\{1\}.
\end{equation}

Because $C$ is an algebraically closed field, each $\operatorname{Gal}(E/{{B}_{i}})$ is a Zariski-closed subgroup of the general linear group $\operatorname{GL}(n,C)$, where $n$ is the order of $L(Y)$. Thus by Definition \ref{18.1.1}, (\ref{16.2(1)}) is a normal Zariski-closed series of $\operatorname{Gal}(E/B)$. 

Suppose that (\ref{16.2(1)}) has a proper refinement which is still normal Zariski-closed. Then $\exists i\in \{1,\cdots,m\}$ and a Zariski-closed subgroup $H$ of $\operatorname{GL}(n,C)$ such that $\operatorname{Gal}(E/{{B}_{i-1}})\,\triangleright\, H\,\triangleright\, \operatorname{Gal}(E/{{B}_{i}})$ with $\operatorname{Gal}(E/{{B}_{i-1}})\ne H\ne \operatorname{Gal}(E/{{B}_{i}})$. Hence by the fundamental theorem of differential Galois theory, $H=\operatorname{Gal}(E/E_{{}}^{H})$, ${{B}_{i-1}}\subsetneq {{E}^{H}}\subsetneq {{B}_{i}}$, and ${{E}^{H}}$ is a Picard-Vessiot extension of ${{B}_{i-1}}$. But ${{B}_{i-1}}\subseteq {{B}_{i}}$ is an irreducible Picard-Vessiot extension, a contradiction. 

Therefore, any possible proper refinement of (\ref{16.2(1)}) is no longer a normal Zariski-closed one. By Definition \ref{18.1.1}, (\ref{16.2(1)}) is a composition Zariski-closed series of $\operatorname{Gal}(E/B)$. \\
2. (b) $\Rightarrow $ (a):

Suppose that the differential Galois group $\operatorname{Gal}(E/B)$ has a composition Zariski-closed series: $\operatorname{Gal}(E/B)=:{{G}_{0}}\,\triangleright\, {{G}_{1}}\,\triangleright\, \cdots \,\triangleright\, {{G}_{m}}=\{1\}$. 

By Lemma \ref{18.1.2}, there are intermediate differential fields ${{B}_{0}},{{B}_{1}},\cdots ,{{B}_{m}}$ such that $B={{B}_{0}}\subseteq {{B}_{1}}\subseteq \cdots \subseteq {{B}_{m}}=E$, where ${{B}_{i-1}}\subseteq {{B}_{i}}$, $\forall i=1, \cdots, m$, is a Picard-Vessiot extension and $\forall i=0, \cdots, m$, $\operatorname{Gal}(E/{{B}_{i}})={{G}_{i}}$. Thus,
\begin{equation} \label{16.2(2)}
    \operatorname{Gal}(E/B)=\operatorname{Gal}(E/{{B}_{0}})\,\triangleright\, \operatorname{Gal}(E/{{B}_{1}})\,\triangleright\, \cdots \,\triangleright\, \operatorname{Gal}(E/{{B}_{m}})=\{1\}
\end{equation}
is a composition Zariski-closed series. 

Besides, by the definition of Picard-Vessiot extensions, all ${{B}_{i}}$ share the same constant field, which we assume to be $C$. Because condition (i) in Definition \ref{18.2.4} is assumed to be satisfied, $C$ is algebraically closed. Hence we can apply the fundamental theorem of differential Galois theory to any Picard-Vessiot extension of ${{B}_{i}}$, $\forall i=0,\cdots, m$.

Suppose that there exists a reducible Picard-Vessiot extension ${{B}_{i-1}}\subsetneq {{B}_{i}}$, where $i\in \{1,\cdots, m\}$. Then by Definition \ref{18.2.1}, $\exists K$ with ${{B}_{i-1}}\subsetneq K\subsetneq {{B}_{i}}$ such that $K$ is a Picard-Vessiot extension of ${{B}_{i-1}}$. Hence $C$ is also the constant field of $K$. By the fundamental theorem, $\operatorname{Gal}(E/{{B}_{i-1}})\,\triangleright\, \operatorname{Gal}(E/K)$, where $\operatorname{Gal}(E/K)$ is also Zariski-closed. Moreover, since ${{B}_{i-1}}\subsetneq {{B}_{i}}$ is a Picard-Vessiot extension, so is $K\subsetneq {{B}_{i}}$. Then again by the fundamental theorem, $\operatorname{Gal}(E/K)\,\triangleright\, \operatorname{Gal}(E/{{B}_{i}})$. Therefore, $\operatorname{Gal}(E/{{B}_{i-1}})\,\triangleright\, \operatorname{Gal}(E/K)\,\triangleright\, \operatorname{Gal}(E/{{B}_{i}})$. And $\operatorname{Gal}(E/{{B}_{i-1}})\ne \operatorname{Gal}(E/K)\ne \operatorname{Gal}(E/{{B}_{i}})$ because ${{B}_{i-1}}\ne K\ne {{B}_{i}}$. But (\ref{16.2(2)}) is a composition Zariski-closed series, a contradiction. 

Thus, every (Picard-Vessiot extension) ${{B}_{i-1}}\subseteq {{B}_{i}}$ of $B={{B}_{0}}\subseteq {{B}_{1}}\subseteq \cdots \subseteq {{B}_{m}}=E$ is irreducible. By Definition \ref{18.2.3}, $B\subseteq E$ is a canonical extension.
\end{proof}

\subsection{Formulas for solutions of homogeneous linear differential equations} \label{Formulas for the solutions}
	Now we show that in a sense we may develop formulas for the solutions of any homogeneous linear differential equation by the following process described in steps. To simplify our descriptions, we introduce
\begin{terminology} \label{18.3.0}
    Let $B\subseteq B(U)$ be a differential field extension generated by $U$. Then by \emph{differential field operations over $B$ on $U$} we mean rational functions over $B$ on $U$ and the derivation function of $B(U)$.
\end{terminology}
\begin{step} \label{18.3.1}
    Let $L(Y)=0$ be a homogeneous linear differential equation for which $B\subseteq E$ is a Picard-Vessiot extension such that conditions (i) and (ii) in Definition \ref{18.2.4} are satisfied. Let $C$ be the algebraically closed field of constants of $B\subseteq E$. 
    
    Because $\operatorname{Gal}(E/B)\,\triangleright\, \operatorname{Gal}(E/E)=\{1\}$, there must be a normal Zariski-closed series of $\operatorname{Gal}(E/B)$ in the form $\operatorname{Gal}(E/B)=:{{G}_{0}}\,\triangleright\, {{G}_{1}}\,\triangleright\, \cdots \,\triangleright\, {{G}_{m}}=\{1\}$. By Lemma \ref{18.1.2}, there are intermediate differential fields ${{B}_{0}},{{B}_{1}},\cdots ,{{B}_{m}}$ such that $B={{B}_{0}}\subseteq {{B}_{1}}\subseteq \cdots \subseteq {{B}_{m}}=E$ with every ${{B}_{i-1}}\subseteq {{B}_{i}}$ being a Picard-Vessiot extension. Thus there exist homogeneous linear differential equations ${{L}_{1}}(Y)=0,\cdots, {{L}_{m}}(Y)=0$ such that $\forall i=1,\cdots,m$, ${{B}_{i-1}}\subseteq {{B}_{i}}$ is a Picard-Vessiot extension for ${{L}_{i}}(Y)=0$. 
    
    Moreover, if $\operatorname{Gal}(E/B)$ has a composition Zariski-closed series, then by Proposition \ref{18.2.5}, $L(Y)=0$ is canonically solvable, and hence by Definitions \ref{18.2.3} and \ref{18.2.4}, in such case we may assume that $\forall i=1,\cdots,m$, ${{B}_{i-1}}\subseteq {{B}_{i}}$ is an irreducible Picard-Vessiot extension.
\end{step}
\begin{step} \label{18.3.2}
    Now ${{B}_{0}}\subseteq {{B}_{1}}$ is a Picard-Vessiot extension for ${{L}_{1}}(Y)=0$. Just as the solutions of ${Y}'=a$ are assumed to be $\int{a}dx$, now the solutions of ${{L}_{1}}(Y)=0$ in $E$ are assumed to be determined (by symbols, formulas, or values). Then because ${{B}_{1}}$ can be generated over ${{B}_{0}}=B$ by the solutions of ${{L}_{1}}(Y)=0$ in $E$, every element of ${{B}_{1}}$ can be determined. Said precisely, $\forall b\in {{B}_{1}}$, there is a formula for $b$ which only involves differential field operations over $B$ on the solutions of ${{L}_{1}}(Y)=0$ in $E$ (\textit{cf.} Terminology \ref{18.3.0}).
\end{step}
\begin{step} \label{18.3.3}
    Analogously, because ${{B}_{1}}\subseteq {{B}_{2}}$ is a Picard-Vessiot extension for ${{L}_{2}}(Y)=0$, ${{B}_{2}}$ can be generated over ${{B}_{1}}$ by the solutions of ${{L}_{2}}(Y)=0$ in $E$. Hence $\forall b\in {{B}_{2}}$, there is a formula for $b$ which only involves differential field operations over ${{B}_{1}}$ on the solutions of ${{L}_{2}}(Y)=0$ in $E$. As was shown in Step \ref{18.3.2}, every element of ${{B}_{1}}$ can be obtained by a formula which only involves differential field operations over $B$ on the solutions of ${{L}_{1}}(Y)=0$ in $E$. Hence by substitutions, every element of ${{B}_{2}}$ can be determined by a formula which only involves differential field operations over $B$ on the solutions of ${{L}_{1}}(Y)=0$ and ${{L}_{2}}(Y)=0$ in $E$.
\end{step}
\begin{step} \label{18.3.4}
    By induction, eventually every element of ${{B}_{m}}=E$ can be determined by a formula which only involves differential field operations over $B$ on the solutions of equations ${{L}_{1}}(Y)=0,\cdots, {{L}_{m}}(Y)=0$ in $E$. Therefore, for each solution of $L(Y)=0$ in $E$, there is a formula which only involves differential field operations over $B$ on the solutions of equations ${{L}_{1}}(Y)=0,\cdots, {{L}_{m}}(Y)=0$ in $E$.
\end{step}
    In fact, in the above process, the operation of solving equation $L(Y)=0$ is “decomposed” into the operations of solving ${{L}_{1}}(Y)=0,\cdots, {{L}_{m}}(Y)=0$. Moreover, Proposition \ref{18.2.5} tells us that in the case where $\operatorname{Gal}(E/B)$ has a composition Zariski-closed series, $\forall i=1,\cdots,m$, ${{B}_{i-1}}\subseteq {{B}_{i}}$ is an irreducible Picard-Vessiot extension, which implies that any further “decomposition” of any ${{L}_{i}}(Y)=0$, $\forall i=1,\cdots,m$, would be trivial. This explains why in this case we say that $L(Y)=0$ is canonically solvable.

The following summarizes this subsection.
\begin{corollary} \label{18.3.5}
    Let $L(Y)=0$ be a homogeneous linear differential equation for which $B\subseteq E$ is a Picard-Vessiot extension such that the constant field is algebraically closed. 
    
    Suppose that $\operatorname{Gal}(E/B)=:{{G}_{0}}\,\triangleright\, {{G}_{1}}\,\triangleright\, \cdots \,\triangleright\, {{G}_{m}}=\{1\}$, where $m\in {{\mathbb{Z}}^{+}}$. Then there exist intermediate differential fields $(B=){{B}_{0}},{{B}_{1}},\cdots ,{{B}_{m}}(=E)$ such that $\operatorname{Gal}({{B}_{i}}/{{B}_{i-1}})\cong {{G}_{i-1}}/{{G}_{i}}$, $\forall i=1,\cdots, m$, and homogeneous linear differential equations ${{L}_{1}}(Y)=0,\cdots , {{L}_{m}}(Y)=0$ such that $\forall i=1,\cdots, m$, ${{B}_{i-1}}\subseteq {{B}_{i}}$ is a Picard-Vessiot extension for ${{L}_{i}}(Y)=0$. And for any solution of $L(Y)=0$ in E, there is a formula which only involves differential field operations over $B$ on the solutions of equations ${{L}_{1}}(Y)=0,\cdots, {{L}_{m}}(Y)=0$ in $E$ $($cf. Terminology \ref{18.3.0}$)$. 
    
    Moreover, if $\operatorname{Gal}(E/B)\,\triangleright\, {{G}_{1}}\,\triangleright\, \cdots \,\triangleright\, {{G}_{m}}=\{1\}$ is a composition Zariski-closed series $($cf. Definition \ref{18.1.1}$)$, then $\forall i=1,\cdots, m$, the Picard-Vessiot extension ${{B}_{i-1}}\subseteq {{B}_{i}}$ for ${{L}_{i}}(Y)=0$ is irreducible $($cf. Definition \ref{18.2.1}$)$.
\end{corollary}
 
\section{A possible strategy for equation solving} \label{equ sol}
 Now for “general” equations (with unknowns), we shall generalize our results (in the two preceding sections) in terms of the theory developed in Parts I and II. Because this section is only a sketch of some ideas, it is not developed in a rigorous way.
 \subsection{Splitting $T$-spaces, Galois $T$-groups and Galois $T$-monoids of equations}
To simplify our descriptions, we introduce
\begin{terminology} \label{19.1.2}
    Let $T$ be a par-operator gen-semigroup on a set $D$ and let $B$ be a set. If $\forall U\subseteq D$ and $a\in {{\left\langle U \right\rangle }_{T}}$, there is a formula for $a$ which only involves some given operations on elements of $B$ and $U$, then we say that \emph{T is $($defined on $D)$ over $B$}. 
\end{terminology}
\begin{remark}
    Note that we do not require $B\subseteq D$.
\end{remark}
For instance, the par-operator gen-semigroups defined in Examples \ref{T on field--1}, \ref{9.1.3}, \ref{10.1.5} and \ref{10.1.6} are all defined over $B$. 
\begin{notation} \label{19.1.1}
    Let $Q$ be an equation (with unknowns) and let $D$ be a set. We denote by $U_{Q}^{D}$ the set of all solutions of $Q$ in $D$. 
\end{notation}
\begin{example} \label{19.1.3}
    Let $F/B$ be a field extension and let $T$ be the par-operator gen-semigroup on $F$ defined in Example \ref{10.1.5}. Let $p(x)\in B[x]$ such that $p$ splits over $F$ and let $U$ be the set of roots of $p$ in $F$. Then ${{\left\langle U_{Q}^{F} \right\rangle }_{T}}={{\left\langle U \right\rangle }_{T}}$ is a splitting field of $p$ over $B$, where $Q$ denotes the equation $p(x)=0$.
\end{example} 
 \begin{example} \label{19.1.4}
    Let $F/B$ be a differential field extension such that every constant of $F$ lies in $B$. Let $T$ be the par-operator gen-semigroup on $F$ defined in Example \ref{10.1.6}. Let $L(Y)=0$ be a homogeneous linear differential equation over $B$ with $U$ being a fundamental set of solutions of it in $F$. Then $B\subseteq {{\left\langle U_{Q}^{F} \right\rangle }_{T}}(={{\left\langle U \right\rangle }_{T}})$ is a Picard-Vessiot extension (for $L$), where $Q$ denotes the equation $L(Y)=0$.
\end{example}
 	Inspired by the above two examples, we generalize the concepts of splitting fields and Picard-Vessiot extensions as follows.
\begin{definition} \label{19.1.5}
    Let $Q$ be an equation (with unknowns) and let $T$ be a par-operator gen-semigroup on a set $D$. With $U_{Q}^{D}$ defined by Notation \ref{19.1.1}, we call ${{\left\langle U_{Q}^{D} \right\rangle }_{T}}$ the \emph{splitting $T$-space of $Q$}. Moreover, if $T$ is defined over $B$ (\textit{cf.} Terminology \ref{19.1.2}), then we call ${{\left\langle U_{Q}^{D} \right\rangle }_{T}}$ the \emph{splitting $T$-space of $Q$ over $B$}. 
\end{definition}
\begin{remark} \begin{enumerate}
    \item If $Q$ has no solution in $D$, then of course ${{\left\langle U_{Q}^{D} \right\rangle }_{T}}=U_{Q}^{D}=\emptyset $. 
    \item We use the terminology “splitting” to be consistent with the classical Galois theory although maybe nothing of $Q$ “split” like a polynomial.
    \item It is possible that $U_{Q}^{D}\nsubseteq {{\left\langle U_{Q}^{D} \right\rangle }_{T}}$ and/or $B\nsubseteq {{\left\langle U_{Q}^{D} \right\rangle }_{T}}$.
    \item $U_{Q}^{D}\subseteq {{\left\langle U_{Q}^{D} \right\rangle }_{T}}$ if $\operatorname{Id}\in T$.
\end{enumerate}
\end{remark}
Moreover, we generalize the notions of Galois groups of polynomials and differential Galois groups of homogeneous linear differential equations as follows.
\begin{definition} \label{19.1.6}
    Let $Q$ be an equation (with unknowns) and let $T$ be a par-operator gen-semigroup on a set $D$ over a set $B$ (\textit{cf.} Terminology \ref{19.1.2}). If $B\subseteq {{\left\langle U_{Q}^{D} \right\rangle }_{T}}$, then the \emph{Galois $T$-group of $Q$} and the \emph{Galois $T$-monoid of $Q$} are $\operatorname{GGr}_{T}({{\left\langle U_{Q}^{D} \right\rangle }_{T}}/B)$ and $\operatorname{GMn}_{T}({{\left\langle U_{Q}^{D} \right\rangle }_{T}}/B)$ (\textit{cf.} Definition \ref{1.5.1}), respectively.
\end{definition}
In Example \ref{19.1.3}, we can tell from Proposition \ref{10.3.4} that $\operatorname{GGr}_{T}({{\left\langle U_{Q}^{F} \right\rangle }_{T}}/B)=\operatorname{GMn}_{T}({{\left\langle U_{Q}^{F} \right\rangle }_{T}}/B)$ and they are the Galois group $\operatorname{Gal}({{\left\langle U \right\rangle }_{T}}/B)$ of $p(x)$ over $B$, where $U$ is the set of roots of $p$ in $F$.  
 
 And in Example \ref{19.1.4}, it is not hard to tell from Proposition \ref{10.3.6} that $\operatorname{GGr}_{T}({{\left\langle U_{Q}^{F} \right\rangle }_{T}}/B)=\operatorname{GMn}_{T}({{\left\langle U_{Q}^{F} \right\rangle }_{T}}/B)$ and they are the differential Galois group $\operatorname{Gal}({{\left\langle U \right\rangle }_{T}}/B)$ of $L(Y)=0$ over $B$, where $U$ is a fundamental set of solutions of $L(Y)=0$ in $F$.

 \subsection{Formulas for elements of splitting $T$-spaces of “general” equations}
Generalizing the strategies used in Subsections \ref{Formulas for the roots} and \ref{Formulas for the solutions}, the following shows that we can determine the splitting $T$-space of an equation provided that some conditions are satisfied. 
\begin{proposition} \label{19.2.1}
    Given an equation $Q$, we are required to determine ${{\left\langle U_{Q}^{D} \right\rangle }_{T}}$, where $T$ is a par-operator gen-semigroup on a set $D$ over a set $B$. 
    
    If there exist $m(\in {{\mathbb{Z}}^{+}})$ equations $Q_1,\cdots,Q_m$, where every $U_{{{Q}_{i}}}^{D}$ is nonempty and is assumed to be determined (by a symbol, a formula, or a value), and $m$ par-operator gen-semigroups $T_1,\cdots,T_m$ on $D$ such that the following two conditions are satisfied,
\begin{enumerate}
    \item [(i)] $\forall i=1,\cdots,m$, ${{T}_{i}}$ is defined over ${{B}_{i-1}}$, where ${{B}_{0}}:=B$ and $\forall i=1,\cdots,(m-1)$, ${{B}_{i}}:={{\left\langle U_{{{Q}_{i}}}^{D} \right\rangle }_{{{T}_{i}}}}$ and all given operations on $B_i$ are also given as operations on $B$ (\textit{cf.} Terminology \ref{19.1.2}); 
\item [(ii)] ${{\left\langle U_{Q}^{D} \right\rangle }_{T}}\subseteq {{\left\langle U_{{{Q}_{m}}}^{D} \right\rangle }_{{{T}_{m}}}}$;
\end{enumerate}
then for each element of ${{\left\langle U_{Q}^{D} \right\rangle }_{T}}$, there is a formula which only involves the given operations on $B$ and $U_{Q_1}^D,\cdots,U_{Q_m}^D$.
\end{proposition}
\begin{proof}
    Since ${{\left\langle U_{{{Q}_{1}}}^{D} \right\rangle }_{{{T}_{1}}}}={{B}_{1}}$ and ${{T}_{1}}$ is defined over ${{B}_{0}}=B$ (condition (i)), by Terminology \ref{19.1.2}, $\forall b\in {{B}_{1}}$, there is a formula for $b$ which only involves the given operations on $B$ and $U_{{{Q}_{1}}}^{D}$.

Analogously, because ${{\left\langle U_{{{Q}_{2}}}^{D} \right\rangle }_{{{T}_{2}}}}={{B}_{2}}$ and ${{T}_{2}}$ is defined over ${{B}_{1}}$, $\forall b\in {{B}_{2}}$, there is a formula for $b$ which only involves the given operations on ${{B}_{1}}$ and $U_{{{Q}_{2}}}^{D}$. And as was shown above, for each element of ${{B}_{1}}$, there is a formula which only involves the given operations on $B$ and $U_{{{Q}_{1}}}^{D}$. Hence by substitutions and the assumption that all given operations on $B_1$ are also given operations on $B$, any $b\in {{B}_{2}}$ can be determined by a formula which only involves the given operations on $B$, $U_{{{Q}_{1}}}^{D}$, and $U_{{{Q}_{2}}}^{D}$.

    By induction, eventually every element of ${{\left\langle U_{{{Q}_{m}}}^{D} \right\rangle }_{{{T}_{m}}}}$ can be determined by a formula which only involves the given operations on $B$ and $U_{Q_1}^D,\cdots,U_{Q_m}^D$. Thus, $\forall a\in {{\left\langle U_{Q}^{D} \right\rangle }_{T}}\subseteq {{\left\langle U_{{{Q}_{m}}}^{D} \right\rangle }_{{{T}_{m}}}}$(condition (ii)), there is a formula for $a$ which only involves the given operations on $B$ and $U_{Q_1}^D,\cdots,U_{Q_m}^D$.
\end{proof}
	The strategy described by Proposition \ref{19.2.1} is actually “decomposing” $Q$ into ${{Q}_{1}},\cdots ,{{Q}_{m}}$ each of which the solutions are assumed to be known. This is just the strategy used in Sections \ref{poly equ} and \ref{diff equ}. 

   Note that in Sections \ref{poly equ} and \ref{diff equ}, the “decomposition” of an equation $Q$ into ${{Q}_{1}},\cdots ,{{Q}_{m}}$ depends on a normal series of the Galois $T$-group of $Q$ (\textit{cf.} Lemmas \ref{17.1.1} and \ref{18.1.2}). However, for any other equation $Q$, we do not know whether a similar situation arises, i.e. whether there is a connection between the “decomposition” of $Q$ described in Proposition \ref{19.2.1} and a normal series of the Galois $T$-group of $Q$ (defined by Definition \ref{19.1.6}). 
   
   Moreover, we do not know whether there exists any case where the “decomposition” of $Q$ is related to the Galois $T$-monoid of $Q$ (defined by Definition \ref{19.1.6}) rather than the Galois $T$-group of $Q$.

Of course, the possible strategy for equation solving described in this section is just a very rough sketch and much more work is needed.

$ $

\addcontentsline{toc}{section}{Part IV. \textbf{Other topics and future research}}

\begin{center}
Part IV. OTHER TOPICS AND FUTURE RESEARCH
\end{center}

In Part IV, we shall briefly introduce some more results which we think to be important or deserve deeper study in the future.
\section{Dualities of operator semigroups} \label{Duality}
    Let $T$ be an operator semigroup and let $S$ be a $T$-space. Then $\forall f\in T$, $\sigma \in \operatorname{End}_{T}(S)$ and $a\in S $, $\sigma (f(a))=f(\sigma (a))$. The symmetry in this equation implies that $T$ and $\operatorname{End}_{T}(S)$ may in a way exchange their roles with each other. Indeed, there exists a duality between $\operatorname{End}_{T}(S)$ and the maximum operator semigroup on $D$ which “accommodates” $\operatorname{End}_{T}(S)$, as explained below. 

    In the following two results, for any set $F$ of functions from $D$ to $D$ and $S\subseteq D$, we may denote by $F{{|}_{S}}$ the set $\{(f{{|}_{S}})\,|\,f\in F\}$.
\begin{theorem}
\label{dual End}
    Let $T$ be an operator semigroup on a set $D$, let $S$ be a $T$-space, let $T^*=\operatorname{End}_{T}(S)$, and let
    \begin{center}
        ${{T}_{\operatorname{End}(S)}}=\{f:D\to D\, |\,\operatorname{Im}(f{{|}_{S}})\subseteq S, \sigma (f(a))=f(\sigma (a)), \forall a\in S$ and $\sigma \in \operatorname{End}_{T}(S)\}$.
    \end{center}
\begin{enumerate}
     \item [(a)] $T^*$ is an operator semigroup on $S$, $S$ is a $T^*$-space, and ${{T}_{\operatorname{End}(S)}}$ is an operator semigroup on $D$. 
    \item [(b)] $\operatorname{End}_{T^*}(S)\supseteq T{{|}_{S}}$, where the equation holds if and only if $T{{|}_{S}}={{T}_{\operatorname{End}(S)}}{{|}_{S}}$. 
\end{enumerate}
\begin{remark}
    \begin{enumerate}
        \item Obviously $T\subseteq {{T}_{\operatorname{End}(S)}}$, and intuitively, ${{T}_{\operatorname{End}(S)}}$ can be interpreted as the maximum operator semigroup on $D$ which “accommodates” $\operatorname{End}_{T}(S)$ (i.e. $\forall \sigma \in \operatorname{End}_{T}(S)$, $\sigma$ would still be a $T$-endomorphism of $S$ if $T$ were extended to ${{T}_{\operatorname{End}(S)}}$).
        \item Since $T^*=\operatorname{End}_{T}(S)$, if $T{{|}_{S}}={{T}_{\operatorname{End}(S)}}{{|}_{S}}$, then (b) implies $\operatorname{End}_{\operatorname{End}_{T}(S)}(S)=\operatorname{End}_{T^*}(S)=T|_S$. That is, $\operatorname{End}_{\operatorname{End}_{T}(S)}(S)=T|_S$ when $T$ reaches its “maximum” ${T}_{\operatorname{End}(S)}$. This is why we said that there exists a duality between $\operatorname{End}_{T}(S)$ and ${T}_{\operatorname{End}(S)}$.
    \end{enumerate}
\end{remark}
\end{theorem}
\begin{proof} For (a):

It follows from Proposition \ref{1.3.10} and Definition \ref{Operator semigroup} that $T^*=\operatorname{End}_{T}(S)$ is an operator semigroup on $S$, and hence ${{\left\langle S \right\rangle }_{T^*}}\subseteq S$. Since Id on $S$ lies in $\operatorname{End}_{T}(S)$, ${{\left\langle S \right\rangle }_{T^*}}=S$. Hence the $T$-space $S$ is also a $T^*$-space. 

To show that ${{T}_{\operatorname{End}(S)}}$ is an operator semigroup on $D$, we only need to show that it is closed under composition. Let $h,g\in {{T}_{\operatorname{End}(S)}}$, $\sigma \in \operatorname{End}_{T}(S)$ and $a\in S$. Then from the definition of ${{T}_{\operatorname{End}(S)}}$, we can tell that $\operatorname{Im}((h\circ g){{|}_{S}})\subseteq S$ and
\[\sigma ((h\circ g)(a))=\sigma (h(g(a)))=h(\sigma (g(a)))=h(g(\sigma (a)))=(h\circ g)(\sigma (a)).\] 
Hence $h\circ g\in {{T}_{\operatorname{End}(S)}}$, and thus ${{T}_{\operatorname{End}(S)}}$ is an operator semigroup on $D$. 

For (b):

By Definition \ref{1.3.1}, $\forall f\in T$ and $\sigma \in \operatorname{End}_{T}(S)=T^*$, $f$ commutes with $\sigma $ on $S$, i.e. $\forall a\in S, \sigma (f(a))=f(\sigma (a))$, and hence $f{{|}_{S}}$ is a $T^*$-endomorphism (\textit{cf.} Definition \ref{1.3.2}) of the $T^*$-space $S$ (since $\operatorname{Im}(f{{|}_{S}})\subseteq S$ by Proposition \ref{<S> contained in S}). Therefore, $T{{|}_{S}}\subseteq \operatorname{End}_{T^*}(S)$.

We are going to prove that $\operatorname{End}_{T^*}(S)=T{{|}_{S}}$ if and only if ${{T}_{\operatorname{End}(S)}}{{|}_{S}}=T{{|}_{S}}$.

Suppose $\operatorname{End}_{T^*}(S)=T{{|}_{S}}$. $T{{|}_{S}}\subseteq {{T}_{\operatorname{End}(S)}}{{|}_{S}}$ because $T\subseteq {{T}_{\operatorname{End}(S)}}$. To show ${{T}_{\operatorname{End}(S)}}{{|}_{S}}\subseteq T{{|}_{S}}$, let $f\in {{T}_{\operatorname{End}(S)}}$. Then by the definition of ${{T}_{\operatorname{End}(S)}}$, $\operatorname{Im}(f{{|}_{S}})\subseteq S$ and $\sigma (f(a))=f(\sigma (a))$, $\forall a\in S$ and $\sigma \in T^*=\operatorname{End}_{T}(S)$, and thus $f{{|}_{S}}\in \operatorname{End}_{T^*}(S)=T{{|}_{S}}$. Hence ${{T}_{\operatorname{End}(S)}}{{|}_{S}}\subseteq T{{|}_{S}}$.  Thus, ${{T}_{\operatorname{End}(S)}}{{|}_{S}}=T{{|}_{S}}$.

Conversely, suppose ${{T}_{\operatorname{End}(S)}}{{|}_{S}}=T{{|}_{S}}$. We already showed $T{{|}_{S}}\subseteq \operatorname{End}_{T^*}(S)$. To show $\operatorname{End}_{T^*}(S)\subseteq T{{|}_{S}}$, let $f:D\to D$ such that $f{{|}_{S}}\in \operatorname{End}_{T^*}(S)$. Then $\operatorname{Im}(f{{|}_{S}})\subseteq S$ and $\sigma (f(a))=f(\sigma (a)),\forall a\in S$ and $\sigma \in T^*=\operatorname{End}_{T}(S)$. It follows that $f\in {{T}_{\operatorname{End}(S)}}$, and thus $f{{|}_{S}}\in {{T}_{\operatorname{End}(S)}}{{|}_{S}}=T{{|}_{S}}$. Hence $\operatorname{End}_{T^*}(S)\subseteq T{{|}_{S}}$. Therefore, $\operatorname{End}_{T^*}(S)=T{{|}_{S}}$.
\end{proof}
Analogously, there exists a duality between $\operatorname{Aut}_{T}(S)$ and the maximum operator semigroup which “accommodates” $\operatorname{Aut}_{T}(S)$, though an additional condition is required.
\begin{theorem}
    \label{dual Aut}
    Let $T$ be an operator semigroup on a set $D$, let $S$ be a $T$-space, let ${{T}^{\#}}=\operatorname{Aut}_{T}(S)$, and let
    \begin{center}
        ${{T}_{\operatorname{Aut}(S)}}=\{f:D\to D\, |\,\operatorname{Im}(f{{|}_{S}})\subseteq S, \sigma (f(a))=f(\sigma (a))$, $\forall a\in S$ and $\sigma \in \operatorname{Aut}_{T}(S)\}$.
    \end{center}
 \begin{enumerate}
    \item [(a)]  ${{T}^{\#}}$ is an operator semigroup on $S$, $S$ is a ${{T}^{\#}}\text{-}$space and ${{T}_{\operatorname{Aut}(S)}}$ is an operator semigroup on $D$. 
    \item [(b)] If $\forall f\in T$, $f{{|}_{S}}$ is bijective onto $S$, then $\operatorname{Aut}_{{{T}^{\#}}}(S)\supseteq T{{|}_{S}}$
  \item [(c)] Suppose that $\forall f\in {{T}_{\operatorname{Aut}(S)}}$, $f{{|}_{S}}$ is bijective onto $S$. Then $\operatorname{Aut}_{{{T}^{\#}}}(S)\supseteq T{{|}_{S}}$ and the equality holds if and only if $T{{|}_{S}}={{T}_{\operatorname{Aut}(S)}}{{|}_{S}}$.
\end{enumerate}
\end{theorem}
\begin{remark}
    \begin{enumerate}
        \item Apparently $T\subseteq {{T}_{\operatorname{Aut}(S)}}$, and intuitively, ${{T}_{\operatorname{Aut}(S)}}$ can be interpreted as the maximum operator semigroup on $D$ which “accommodates” $\operatorname{Aut}_{T}(S)$ (i.e. $\forall \sigma \in \operatorname{Aut}_{T}(S)$, $\sigma$ would still be a $T$-automorphism of $S$ if $T$ were extended to ${{T}_{\operatorname{Aut}(S)}}$).
        \item Since ${{T}^{\#}}=\operatorname{Aut}_{T}(S)$, if $T{{|}_{S}}={{T}_{\operatorname{Aut}(S)}}{{|}_{S}}$ and $\forall f\in {{T}_{\operatorname{Aut}(S)}}$, $f{{|}_{S}}$ is bijective onto $S$, then (c) implies $\operatorname{Aut}_{\operatorname{Aut}_{T}(S)}(S)=\operatorname{Aut}_{{{T}^{\#}}}(S)=T|_S={{T}_{\operatorname{Aut}(S)}}{{|}_{S}}$. This is why we said that there exists a duality between $\operatorname{Aut}_{T}(S)$ and $(T\subseteq){{T}_{\operatorname{Aut}(S)}}$.
    \end{enumerate}
\end{remark}
\begin{proof} For (a):

It follows from Proposition \ref{1.3.11} and Definition \ref{Operator semigroup} that ${{T}^{\#}}=\operatorname{Aut}_{T}(S)$ is an operator semigroup on $S$, and hence ${{\left\langle S \right\rangle }_{{{T}^{\#}}}}\subseteq S$. Because Id on $S$ lies in $\operatorname{Aut}_{T}(S)$, ${{\left\langle S \right\rangle }_{{{T}^{\#}}}}=S$. Hence the $T$-space $S$ is also a ${{T}^{\#}}\text{-}$space. 

To show that ${{T}_{\operatorname{Aut}(S)}}$ is an operator semigroup on $D$, it suffices to show that it is closed under composition. Let $h,g\in {{T}_{\operatorname{Aut}(S)}}$, $\sigma \in \operatorname{Aut}_{T}(S)$ and $a\in S$. Then from the definition of ${{T}_{\operatorname{Aut}(S)}}$, we can tell that $\operatorname{Im}((h\circ g){{|}_{S}})\subseteq S$ and
	\[\sigma ((h\circ g)(a))=\sigma (h(g(a)))=h(\sigma (g(a)))=h(g(\sigma (a)))=(h\circ g)(\sigma (a)).\] 
Hence $h\circ g\in {{T}_{\operatorname{Aut}(S)}}$. Thus, ${{T}_{\operatorname{Aut}(S)}}$ is an operator semigroup on $D$. 

For (b):

Assume that $\forall f\in T$, $f{{|}_{S}}$ is bijective onto $S$. By Definition \ref{1.3.1}, $\forall f\in T$ and $\sigma \in \operatorname{Aut}_{T}(S)={{T}^{\#}}$, $f$ commutes with $\sigma $ on $S$, i.e. $\forall a\in S$, $\sigma (f(a))=f(\sigma (a))$, and hence $f{{|}_{S}}$ is a ${{T}^{\#}}\text{-}$automorphism of the ${{T}^{\#}}\text{-}$space $S$ (since $f{{|}_{S}}$ is bijective onto S). Therefore, $\operatorname{Aut}_{{{T}^{\#}}}(S)\supseteq T{{|}_{S}}$.

For (c): 

Assume that $\forall f\in {{T}_{\operatorname{Aut}(S)}}$, $f{{|}_{S}}$ is bijective onto $S$. 
	Since $T\subseteq {{T}_{\operatorname{Aut}(S)}}$, $\forall f\in T$, $f{{|}_{S}}$ is bijective onto $S$. Then by (b), $\operatorname{Aut}_{{{T}^{\#}}}(S)\supseteq T{{|}_{S}}$.

We are going to prove that $\operatorname{Aut}_{{{T}^{\#}}}(S)=T{{|}_{S}}$ if and only if ${{T}_{\operatorname{Aut}(S)}}{{|}_{S}}=T{{|}_{S}}$.

Suppose $\operatorname{Aut}_{{{T}^{\#}}}(S)=T{{|}_{S}}$. $T{{|}_{S}}\subseteq {{T}_{\operatorname{Aut}(S)}}{{|}_{S}}$ because $T\subseteq {{T}_{\operatorname{Aut}(S)}}$. To show ${{T}_{\operatorname{Aut}(S)}}{{|}_{S}}\subseteq T{{|}_{S}}$, let $f\in {{T}_{\operatorname{Aut}(S)}}$. Then by the definition of ${{T}_{\operatorname{Aut}(S)}}$, $\sigma (f(a))=f(\sigma (a))$, $\forall a\in S$ and $\sigma \in {{T}^{\#}}=\operatorname{Aut}_{T}(S)$, and hence $f{{|}_{S}}\in \operatorname{Aut}_{{{T}^{\#}}}(S)=T{{|}_{S}}$ (because $f{{|}_{S}}$ is bijective onto S). Hence ${{T}_{\operatorname{Aut}(S)}}{{|}_{S}}\subseteq T{{|}_{S}}$. Thus, ${{T}_{\operatorname{Aut}(S)}}{{|}_{S}}=T{{|}_{S}}$.

Conversely, suppose ${{T}_{\operatorname{Aut}(S)}}{{|}_{S}}=T{{|}_{S}}$. We already showed $\operatorname{Aut}_{{{T}^{\#}}}(S)\supseteq T{{|}_{S}}$. To show $\operatorname{Aut}_{{{T}^{\#}}}(S)\subseteq T{{|}_{S}}$, let $f:D\to D$ such that $f{{|}_{S}}\in \operatorname{Aut}_{{{T}^{\#}}}(S)$. Then $\operatorname{Im}(f{{|}_{S}})\subseteq S$ and $\sigma (f(a))=f(\sigma (a)),\forall a\in S$ and $\sigma \in {{T}^{\#}}=\operatorname{Aut}_{T}(S)$. It follows that $f\in {{T}_{\operatorname{Aut}(S)}}$, and hence $f{{|}_{S}}\in {{T}_{\operatorname{Aut}(S)}}{{|}_{S}}=T{{|}_{S}}$. Thus, $\operatorname{Aut}_{{{T}^{\#}}}(S)\subseteq T{{|}_{S}}$. Therefore, $\operatorname{Aut}_{{{T}^{\#}}}(S)=T{{|}_{S}}$.
\end{proof}

    In Theorem \ref{dual End}, if $T$ were generalized to be a (par-)operator gen-semigroup, then statement (a) would still hold (though generally it would be false that $T\subseteq {{T}_{\operatorname{End}(S)}}$ because ${T}_{\operatorname{End}(S)}$ is always an operator semigroup). However, since the elements of End$_{T^*}(S)$ are all unary functions but the elements of a (par-)operator gen-semigroup are possibly not, statement (b) in Theorem \ref{dual End} would no longer hold.

Analogously, in Theorem \ref{dual Aut}, if $T$ were generalized to be a (par-)operator gen-semigroup, then statement (a) would still hold (though generally it would be false that $T\subseteq {{T}_{\operatorname{Aut}(S)}}$ because ${T}_{\operatorname{Aut}(S)}$ is always an operator semigroup), but statements (b) and (c) would no longer hold.

\section{Fixed sets and transitive actions of $\operatorname{End}_{T}(S)$ and $\operatorname{Aut}_{T}(S)$} \label{Transitive}
In this section we develop, among other results, two analogues of the well-known fact that the Galois group of an irreducible polynomial acts transitively on its roots.

In this section, $T$ is always a par-operator gen-semigroup unless otherwise specified. 

In the classical Galois theory, the fixed field of $\operatorname{Aut}(F)$, where $F$ is a field, is denoted by ${{F}^{\operatorname{Aut}(F)}}$. To study ${{S}^{\operatorname{End}_{T}(S)}}$ and ${{S}^{\operatorname{Aut}_{T}(S)}}$, which generalize the notion of ${{F}^{\operatorname{Aut}(F)}}$, we first give
\begin{notation} \label{5.1.13}
    Let $A\subseteq D$, where $D$ is the domain of $T$. Let
\[C_{A}^{\operatorname{End}}=\{u\in A\,|\,{{[u)}_{T}}\bigcap A=\{u\}\},\] 
i.e. $C_{A}^{\operatorname{End}}$ is the set of $u\in A$ where $u$ is the only $v\in A$ such that $u\xrightarrow{T}v$. And let
\[C_{A}^{\operatorname{Aut}}=\{u\in A\,|\,{{[u]}_{T}}\bigcap A=\{u\}\},\]
i.e. $C_{A}^{\operatorname{Aut}}$ is the set of $u\in A$ where $u$ is the only $v\in A$ such that $u\overset{T}{\longleftrightarrow}v$.
\end{notation}
\begin{remark}
    For an operator semigroup $T$, see Notation \ref{5.1.1} for the notations involved. More generally, for a par-operator gen-semigroup $T$, see Notation \ref{15.1.1} for the notations involved.
\end{remark}  

For example, let $T$, $F$ and $B$ be defined as in Proposition \ref{5.1.2}, then it is not hard to tell that $C_{F}^{\operatorname{End}}=C_{F}^{\operatorname{Aut}}=B$.

\subsection{Fixed sets and transitive actions of $\operatorname{End}_{T}(S)$}
The following explains why we have “End” in notation $C_{A}^{\operatorname{End}}$.
\begin{lemma} \label{5.1.14}
Let $S$ be a $T$-space. Then $C_{S}^{\operatorname{End}}\subseteq {{S}^{\operatorname{End}_{T}(S)}}$, and hence $\operatorname{End}_{T}(S)=\operatorname{GMn}_{T}(S/C_{S}^{\operatorname{End}})$.
\end{lemma}
\begin{proof}
    Let $\sigma \in \operatorname{End}_{T}(S)$. By Proposition \ref{15.1.2}, $\forall u\in S$, $u\xrightarrow{T}\sigma (u)$. Let $u\in C_{S}^{\operatorname{End}}$. Then by Notation \ref{5.1.13}, $u=\sigma (u)$, and hence $u\in {{S}^{\sigma }}$. It follows that $u\in {{S}^{\operatorname{End}_{T}(S)}}$. Thus $C_{S}^{\operatorname{End}}\subseteq {{S}^{\operatorname{End}_{T}(S)}}$. 
    
    Hence by Propositions \ref{1.5.5} and \ref{1.5.6}, $\operatorname{End}_{T}(S)  \subseteq \operatorname{GMn}_{T}(S/{{S}^{\operatorname{End}_{T}(S)}}) \subseteq \operatorname{GMn}_{T}(S/C_{S}^{\operatorname{End}})$. Since $\operatorname{End}_{T}(S)  \supseteq \operatorname{GMn}_{T}(S/C_{S}^{\operatorname{End}})$, $\operatorname{End}_{T}(S)=\operatorname{GMn}_{T}(S/C_{S}^{\operatorname{End}})$.
\end{proof}
 Recall a well-known fact in the classical Galois theory: the Galois group of an irreducible polynomial acts transitively on its roots. By Proposition \ref{5.1.2}, the following is an analogue of this fact.
\begin{theorem} \label{5.1.15}
    Let $S$ be a $T$-space. Suppose that the identity function lies in $T$ and $\forall a\in S\backslash C_{S}^{\operatorname{End}}$, ${{\left\langle a \right\rangle }_{T}}=S$. Then $\forall u,v\in S$ with $u\xrightarrow{T}v$, $\exists \sigma \in \operatorname{GMn}_{T}(S/C_{S}^{\operatorname{End}})(=\operatorname{End}_{T}(S))$ such that $\sigma (u)=v$.
\end{theorem}
\begin{proof}
    Let $u,v\in S$ with $u\xrightarrow{T}v$. There are two cases as follows.

Suppose $u\in S\backslash C_{S}^{\operatorname{End}}$. It is assumed that ${{\left\langle a \right\rangle }_{T}}=S$, $\forall a\in S\backslash C_{S}^{\operatorname{End}}$. Hence ${{\left\langle u \right\rangle }_{T}}=S$. By Proposition \ref{15.1.5}, the map $\sigma :{{\left\langle u \right\rangle }_{T}}\to {{\left\langle v \right\rangle }_{T}}$ given by $f({{u}^{n}})\mapsto f({{v}^{n}})$, $\forall n\in {{\mathbb{Z}}^{+}}$ and $f\in T$ of $n$ variables such that $f({{u}^{n}})$ is well-defined, is a $T$-morphism from $S$ to itself (because $\left\langle v \right\rangle _T \subseteq S$ by Proposition \ref{11.1.1}). Then by Lemma \ref{5.1.14}, $\sigma \in \operatorname{End}_{T}(S)=\operatorname{GMn}_{T}(S/C_{S}^{\operatorname{End}})$. Moreover, if let $f=\operatorname{Id}$, then by the definition of $\sigma$, $\sigma (u)=\sigma (\operatorname{Id}(u))=\operatorname{Id}(v)=v$, as desired.

Suppose $u\in C_{S}^{\operatorname{End}}$. Then by Notation \ref{5.1.13}, $v=u$ because $u\xrightarrow{T}v$. If let $\sigma $ be the identity function on $S$, then apparently $\sigma \in \operatorname{GMn}_{T}(S/C_{S}^{\operatorname{End}})$ and $\sigma (u)=u=v$, as desired.
\end{proof}
	Combining the above two results, we have
\begin{corollary} \label{5.1.16}
    Let $S$ be a $T$-space. Suppose that the identity function lies in $T$ and $\forall a\in S\backslash C_{S}^{\operatorname{End}}$, ${{\left\langle a \right\rangle }_{T}}=S$. Then $C_{S}^{\operatorname{End}}={{S}^{\operatorname{End}_{T}(S)}}={{S}^{\operatorname{GMn}_{T}(S/C_{S}^{\operatorname{End}})}}$.
\end{corollary}
\begin{proof}
    Let $u\in {{S}^{\operatorname{End}_{T}(S)}}$. Assume $u\notin C_{S}^{\operatorname{End}}$, i.e. $\exists v\in S$ such that $v\ne u$ and $u\xrightarrow{T}v$. Then by Theorem \ref{5.1.15}, $\exists \sigma \in \operatorname{GMn}_{T}(S/C_{S}^{\operatorname{End}})$ such that $\sigma (u)=v$, which is contrary to $u\in {{S}^{\operatorname{End}_{T}(S)}}$. Hence ${{S}^{\operatorname{End}_{T}(S)}}\subseteq C_{S}^{\operatorname{End}}$.

    By Lemma \ref{5.1.14}, $C_{S}^{\operatorname{End}}\subseteq {{S}^{\operatorname{End}_{T}(S)}}={{S}^{\operatorname{GMn}_{T}(S/C_{S}^{\operatorname{End}})}}$. Therefore, $C_{S}^{\operatorname{End}}={{S}^{\operatorname{End}_{T}(S)}}={{S}^{\operatorname{GMn}_{T}(S/C_{S}^{\operatorname{End}})}}$.
\end{proof}
\subsection{Fixed sets and transitive actions of $\operatorname{Aut}_{T}(S)$}
	For $C_{S}^{\operatorname{Aut}}$ (\textit{cf.} Notation \ref{5.1.13}), the analogue of Lemma \ref{5.1.14} is
\begin{lemma} \label{5.1.17}
Let $S$ be a $T$-space. Then $C_{S}^{\operatorname{Aut}}\subseteq {{S}^{\operatorname{Aut}_{T}(S)}}$, and hence $\operatorname{Aut}_{T}(S)=\operatorname{GGr}_{T}(S/C_{S}^{\operatorname{Aut}})$.
\end{lemma}
\begin{proof}
    Recall that Proposition \ref{5.1.4} still holds for $T$ being a par-operator gen-semigroup (\textit{cf.} Subsection \ref{A construction of T-morphisms}). Let $\sigma \in \operatorname{Aut}_{T}(S)$. Then by Proposition \ref{5.1.4}, $\forall u\in S$, $u\overset{T}{\longleftrightarrow}\sigma (u)$. Let $u\in C_{S}^{\operatorname{Aut}}$. Then by Notation \ref{5.1.13}, $u=\sigma (u)$, and hence $u\in {{S}^{\sigma }}$. It follows that $u\in {{S}^{\operatorname{Aut}_{T}(S)}}$. Hence $C_{S}^{\operatorname{Aut}}\subseteq {{S}^{\operatorname{Aut}_{T}(S)}}$. 
    
    Then by Propositions \ref{1.5.5} and \ref{1.5.6}, $\operatorname{Aut}_{T}(S)  \subseteq \operatorname{GGr}_{T}(S/{{S}^{\operatorname{Aut}_{T}(S)}}) \subseteq \operatorname{GGr}_{T}(S/C_{S}^{\operatorname{Aut}})$. Since $\operatorname{Aut}_{T}(S)  \supseteq \operatorname{GGr}_{T}(S/C_{S}^{\operatorname{Aut}})$, $\operatorname{Aut}_{T}(S)=\operatorname{GGr}_{T}(S/C_{S}^{\operatorname{Aut}})$.
\end{proof}
And for $C_{S}^{\operatorname{Aut}}$, the analogue of Theorem \ref{5.1.15} is as follows, which, by Proposition \ref{5.1.2}, is also an analogue of the well-known fact that the Galois group of an irreducible polynomial acts transitively on its roots.
\begin{theorem} \label{5.1.18}
    Let $S$ be a $T$-space. Suppose that the identity function lies in $T$ and $\forall a\in S\backslash C_{S}^{\operatorname{Aut}}$, ${{\left\langle a \right\rangle }_{T}}=S$. Then $\forall u,v\in S$ with $u\overset{T}{\longleftrightarrow}v$, $\exists \sigma \in \operatorname{GGr}_{T}(S/C_{S}^{\operatorname{Aut}})(=\operatorname{Aut}_{T}(S))$ such that $\sigma (u)=v$.
\end{theorem}
\begin{proof}
    Let $u,v\in S$ with $u\overset{T}{\longleftrightarrow}v$. There are two cases as follows.
    
Suppose $u\in S\backslash C_{S}^{\operatorname{Aut}}$. It is assumed that ${{\left\langle a \right\rangle }_{T}}=S$, $\forall a\in S\backslash C_{S}^{\operatorname{Aut}}$. Hence ${{\left\langle u \right\rangle }_{T}}=S$. By Corollary \ref{15.1.6}, the map $\sigma :{{\left\langle u \right\rangle }_{T}}\to {{\left\langle v \right\rangle }_{T}}$ given by $f({{u}^{n}})\mapsto f({{v}^{n}})$, $\forall n\in {{\mathbb{Z}}^{+}}$ and $f\in T$ of $n$ variables such that $f({{u}^{n}})$ is well-defined, is a $T$-isomorphism from $S$ to itself (because $\left\langle v \right\rangle _T \subseteq S$ by Proposition \ref{11.1.1}). Then by Lemma \ref{5.1.17}, $\sigma \in \operatorname{Aut}_{T}(S)=\operatorname{GGr}_{T}(S/C_{S}^{\operatorname{Aut}})$. Moreover, if let $f=\operatorname{Id}$, then by the definition of $\sigma$, $\sigma (u)=\sigma (\operatorname{Id}(u))=\operatorname{Id}(v)=v$, as desired.

Suppose $u\in C_{S}^{\operatorname{Aut}}$. Then by Notation \ref{5.1.13}, $v=u$ because $u\overset{T}{\longleftrightarrow}v$. If let $\sigma $ be the identity function on $S$, then $\sigma \in \operatorname{GGr}_{T}(S/C_{S}^{\operatorname{Aut}})$ and $\sigma (u)=u=v$, as desired.
\end{proof}
	Combining the above two results, we have
 \begin{corollary} \label{5.1.19}
     Let $S$ be a $T$-space. Suppose that the identity function lies in $T$ and $\forall a\in S\backslash C_{S}^{\operatorname{Aut}}$, ${{\left\langle a \right\rangle }_{T}}=S$. Then $C_{S}^{\operatorname{Aut}}={{S}^{\operatorname{Aut}_{T}(S)}}={{S}^{\operatorname{GGr}_{T}(S/C_{S}^{\operatorname{Aut}})}}$.
 \end{corollary}
\begin{proof}
    Let $u\in {{S}^{\operatorname{Aut}_{T}(S)}}$. Assume $u\notin C_{S}^{\operatorname{Aut}}$, i.e. $\exists v\in S$ such that $v\ne u$ and $u\overset{T}{\longleftrightarrow}v$. Then by Theorem \ref{5.1.18}, $\exists \sigma \in \operatorname{GGr}_{T}(S/C_{S}^{\operatorname{Aut}})$ such that $\sigma (u)=v$, which is contrary to $u\in {{S}^{\operatorname{Aut}_{T}(S)}}$. Hence ${{S}^{\operatorname{Aut}_{T}(S)}}\subseteq C_{S}^{\operatorname{Aut}}$.

By Lemma \ref{5.1.17}, $C_{S}^{\operatorname{Aut}}\subseteq {{S}^{\operatorname{Aut}_{T}(S)}}={{S}^{\operatorname{GGr}_{T}(S/C_{S}^{\operatorname{Aut}})}}$. Therefore, $C_{S}^{\operatorname{Aut}}={{S}^{\operatorname{Aut}_{T}(S)}}={{S}^{\operatorname{GGr}_{T}(S/C_{S}^{\operatorname{Aut}})}}$.
\end{proof}

\section{Two questions on Galois $T$-extensions and normal subgroups of Galois $T$-groups} \label{Two questions}
In this section, $T$ is always a par-operator gen-semigroup unless otherwise specified.

As was promised at the end of Section \ref{I Galois corr}, this section focuses on the following.   
\begin{problem} \label{6.0}
Let $S$ be a $T$-space and let $B\subseteq K\subseteq S$. If $B={{S}^{\operatorname{GGr}_{T}(S/B)}}$ and $K={{S}^{\operatorname{GGr}_{T}(S/K)}}$, then 
\begin{enumerate}
    \item when does $B={{K}^{\operatorname{GGr}_{T}(K/B)}}$?
    \item when do $\operatorname{GGr}_{T}(S/K)\,\triangleleft \, \operatorname{GGr}_{T}(S/B)$ and 
    \[\operatorname{GGr}_{T}(S/B)/\operatorname{GGr}_{T}(S/K)\cong \operatorname{GGr}_{T}(K/B)?\]
\end{enumerate}
\end{problem}
In particular, if let $B\subseteq K\subseteq S$ be fields such that $S$ is a Galois extension of $B$, then, as is well-known, $K$ is Galois over $B$ (i.e. $B={{K}^{\operatorname{Gal}(K/B)}}$) if and only if $\operatorname{Gal}(S/K)$ is normal in $\operatorname{Gal}(S/B)$ (that is, (1) is true if and only if (2) is true). However, for a general $T$-space, the answer to Problem \ref{6.0} seems to be not so simple. We can only give some sufficient conditions for the two questions in Problem \ref{6.0}. Nevertheless, one still can see that some, if not all, results and notions introduced in this section have analogues in well-known Galois theories.

\subsection{$H$-stable subsets of a $T$-space}
\begin{definition} \label{6.2.1}
    Let $S$ be a $T$-space with $K\subseteq S$ and $H\subseteq \operatorname{End}_{T}(S)$. If $\forall \sigma \in H$ and $a\in K$, $\sigma (a)\in K, $ then we say that $K$ is \emph{stable under} $H$ or $K$ is \emph{H-stable}. 
\end{definition}
\begin{remark}
 Another choice of this terminology is $H$-invariant.
\end{remark}
\begin{lemma} \label{6.2.2}
    Let $S$ be a $T$-space with $B\subseteq S$. Let $H\,\triangleleft \, \operatorname{GGr}_{T}(S/B)$. Then ${{S}^{H}}$ is $\operatorname{GGr}_{T}(S/B)$-stable.
\end{lemma}
\begin{proof}
    Suppose that $\exists \sigma \in \operatorname{GGr}_{T}(S/B)$ and $a\in {{S}^{H}}$ such that $\sigma (a)\notin {{S}^{H}}$. Then $\sigma (a)\in S\backslash {{S}^{H}}$ and $a\in {{S}^{H}}\backslash B$ (otherwise $a \in B$ and $\sigma (a)=a \in S^H$). Assume that $\forall \tau \in H$, $\tau (\sigma (a))=\sigma (a)$, then $\sigma (a)\in {{S}^{H}}$, a contradiction. It follows that there is ${{\tau }_{a}}\in H$ such that ${{\tau }_{a}}(\sigma (a))\ne \sigma (a)$. However, since $H\,\triangleleft \, \operatorname{GGr}_{T}(S/B)$, ${{\sigma }^{-1}}{{\tau }_{a}}\sigma \in H$. Then ${{\sigma }^{-1}}({{\tau }_{a}}(\sigma (a)))=a$ because $a\in {{S}^{H}}$, and hence ${{\tau }_{a}}(\sigma (a))=\sigma (a)$, a contradiction again.
\end{proof}
\begin{lemma} \label{6.2.3}
    Let $S$ be a $T$-space with $B\subseteq K\subseteq S$. Suppose that $K={{S}^{\operatorname{GGr}_{T}(S/K)}}$, $K$ is $\operatorname{GGr}_{T}(S/B)$-stable, and every element of $\operatorname{GGr}_{T}(K/B)$ can be extended to an element of $\operatorname{GGr}_{T}(S/B)$. Then ${{S}^{\operatorname{GGr}_{T}(S/B)}}={{K}^{\operatorname{GGr}_{T}(K/B)}}$.
\end{lemma}
\begin{proof}
    $K={{S}^{\operatorname{GGr}_{T}(S/K)}}$ implies that no element of $S\backslash K$ is fixed under the action of $\operatorname{GGr}_{T}(S/K)(\subseteq \operatorname{GGr}_{T}(S/B))$. Then no element of $S\backslash K$ is fixed under the action of $\operatorname{GGr}_{T}(S/B)$, and hence ${{S}^{\operatorname{GGr}_{T}(S/B)}}={{K}^{\operatorname{GGr}_{T}(S/B)}}$ (\textit{cf.} Definition \ref{1.4.1} for ${{K}^{\operatorname{GGr}_{T}(S/B)}}$).

    $K$ is $\operatorname{GGr}_{T}(S/B)$-stable, and hence $\forall \sigma \in \operatorname{GGr}_{T}(S/B)(\subseteq \operatorname{Aut}_{T}S), $ $\sigma {{|}_{K}}\in \operatorname{GGr}_{T}(K/B)$. Hence $\operatorname{GGr}_{T}(S/B){{|}_{K}}\subseteq \operatorname{GGr}_{T}(K/B)$, and therefore ${{K}^{\operatorname{GGr}_{T}(S/B)}}\supseteq {{K}^{\operatorname{GGr}_{T}(K/B)}}$.

    On the other hand, since every element of $\operatorname{GGr}_{T}(K/B)$ can be extended to an element of $\operatorname{GGr}_{T}(S/B)$, ${{K}^{\operatorname{GGr}_{T}(K/B)}}\supseteq {{K}^{\operatorname{GGr}_{T}(S/B)}}$. Thus ${{K}^{\operatorname{GGr}_{T}(K/B)}}={{K}^{\operatorname{GGr}_{T}(S/B)}}$.
    
Therefore, ${{S}^{\operatorname{GGr}_{T}(S/B)}}={{K}^{\operatorname{GGr}_{T}(S/B)}}={{K}^{\operatorname{GGr}_{T}(K/B)}}$.	
\end{proof}
Now we obtain our first answer to the first question in Problem \ref{6.0} as follows, where an “extra” condition is that “every element of $\operatorname{GGr}_{T}(K/B)$ can be extended to an element of $\operatorname{GGr}_{T}(S/B)$”. 
\begin{corollary} \label{6.2.4}
    Let $S$ be a $T$-space with $B\subseteq K\subseteq S$. Suppose that $B={{S}^{\operatorname{GGr}_{T}(S/B)}}$, $K={{S}^{\operatorname{GGr}_{T}(S/K)}}$, $\operatorname{GGr}_{T}(S/K)\,\triangleleft \, \operatorname{GGr}_{T}(S/B)$, and every element of $\operatorname{GGr}_{T}(K/B)$ can be extended to an element of $\operatorname{GGr}_{T}(S/B)$. Then $B={{K}^{\operatorname{GGr}_{T}(K/B)}}$.
\end{corollary}
\begin{proof}
    By Lemma \ref{6.2.2}, $K$($={{S}^{\operatorname{GGr}_{T}(S/K)}}$) is $\operatorname{GGr}_{T}(S/B)$-stable. Then by Lemma \ref{6.2.3}, $B={{S}^{\operatorname{GGr}_{T}(S/B)}}={{K}^{\operatorname{GGr}_{T}(K/B)}}$.
\end{proof}
	A converse of Lemma \ref{6.2.2} is as follows, which is our first answer to the second question in Problem \ref{6.0}.
\begin{proposition} \label{6.2.5}
    Let $S$ be a $T$-space with $B\subseteq K\subseteq S$. Suppose that $K$ is $\operatorname{GGr}_{T}(S/B)$-stable. Then there exists a group homomorphism
$\gamma :\operatorname{GGr}_{T}(S/B)\to \operatorname{GGr}_{T}(K/B)$ such that $\ker \gamma =\operatorname{GGr}_{T}(S/K)$, and so
$\operatorname{GGr}_{T}(S/K)\,\triangleleft \, \operatorname{GGr}_{T}(S/B)$ and
$\operatorname{GGr}_{T}(S/B)/\operatorname{GGr}_{T}(S/K)\cong \operatorname{GGr}_{T}(K/B)$.
\end{proposition}
\begin{proof}
    Let $\gamma :\operatorname{GGr}_{T}(S/B)\to \operatorname{GGr}_{T}(K/B)$ be given by	$\gamma (\sigma )=\sigma {{|}_{K}}$.	

Since $K$ is $\operatorname{GGr}_{T}(S/B)\text{-}$stable, $\forall \sigma \in \operatorname{GGr}_{T}(S/B)$, $\sigma (K)\subseteq K$. And $\sigma (K)=K$ because $\sigma \in \operatorname{GGr}_{T}(S/B)\subseteq \operatorname{Aut}_{T}(S)$. Hence $\gamma (\sigma )=\sigma {{|}_{K}}\in \operatorname{GGr}_{T}(K/B)$, which implies that $\gamma $ is a well-defined map.

Moreover, if ${{\sigma }_{1}},{{\sigma }_{2}}\in \operatorname{GGr}_{T}(S/B)$, then $({{\sigma }_{1}}{{|}_{K}})\circ ({{\sigma }_{2}}{{|}_{K}})=({{\sigma }_{1}}\circ {{\sigma }_{2}}){{|}_{K}}$, and so
$\gamma ({{\sigma }_{1}})\circ \gamma ({{\sigma }_{2}})=\gamma ({{\sigma }_{1}}\circ {{\sigma }_{2}})$, which implies that $\gamma $ is a group homomorphism. 

Finally, because $\gamma $ is defined as the restriction to $K$, $\ker \gamma =\operatorname{GGr}_{T}(S/K)$:
\end{proof}
Combining Proposition \ref{6.2.5} with Corollary \ref{6.2.4}, we immediately obtain our second answer to the first question in Problem \ref{6.0}.
\begin{corollary} \label{6.2.6}
    Let $S$ be a $T$-space with $B\subseteq K\subseteq S$. If $B={{S}^{\operatorname{GGr}_{T}(S/B)}}$, $K={{S}^{\operatorname{GGr}_{T}(S/K)}}$, $K$ is $\operatorname{GGr}_{T}(S/B)$-stable, and each element of $\operatorname{GGr}_{T}(K/B)$ can be extended to an element of $\operatorname{GGr}_{T}(S/B)$, then $B={{K}^{\operatorname{GGr}_{T}(K/B)}}$.
\end{corollary} 

\subsection{Normal $T$-subsets of a $T$-space} 
\begin{definition} \label{normal $T$-subset}
Let $K$ be a subset of a $T$-space $S$. If $\forall a\in K$, ${{[a]}_{T}}\bigcap S\subseteq K$ (\textit{cf.} Notation \ref{15.1.1} for ${{[a]}_{T}}$), then we call $K$ a \emph{normal $T$-subset} of $S$ and write $K\,\triangleleft \, S$ or $S\,\triangleright\, K$.
\end{definition}
\begin{remark}
 We use this terminology because there is a correspondence between normal $T$-subsets of $S$ and normal subgroups of a Galois $T$-group of $S$, which will be shown in Corollary \ref{6.3.3} and Proposition \ref{6.3.4}.
\end{remark}
	Any normal $T$-subset of a $T$-space $S$ is $\operatorname{Aut}_{T}(S)$-stable, as shown below.
 \begin{lemma} \label{6.3.2}
     Let $S$ be a $T$-space. Suppose that $K$ is a normal $T$-subset of $S$. Then $K$ is stable under $\operatorname{Aut}_{T}(S)$.
 \end{lemma}
\begin{proof}
    Recall that Proposition \ref{5.1.4} still holds for $T$ being a par-operator gen-semigroup (\textit{cf.} Subsection \ref{A construction of T-morphisms}). Hence by Proposition \ref{5.1.4}, $a\overset{T}{\longleftrightarrow}\sigma (a)$, $\forall a\in S$ and $\sigma \in \operatorname{Aut}_{T}(S)$. 
    
    Let $a\in K$ and $\sigma \in \operatorname{Aut}_{T}(S)$. Then $a\overset{T}{\longleftrightarrow}\sigma (a)$, and thus $\sigma (a)\in {{[a]}_{T}}\bigcap S\subseteq K$ because $K\,\triangleleft \, S$. Hence by Definition \ref{6.2.1}, $K$ is stable under $\operatorname{Aut}_{T}(S)$.
\end{proof}
	The combination of Lemma \ref{6.3.2} and Proposition \ref{6.2.5} yields the following, which is obvious because $\operatorname{GGr}_{T}(S/B)\subseteq \operatorname{Aut}_{T}(S)$. It is our second answer to the second question in Problem \ref{6.0}.
\begin{corollary} \label{6.3.3}
    Let $S$ be a $T$-space with $B\subseteq K\subseteq S$. Suppose that $K$ is a normal $T$-subset of $S$. Then there exists a group homomorphism $\gamma :\operatorname{GGr}_{T}(S/B)\to \operatorname{GGr}_{T}(K/B)$ such that $\ker \gamma =\operatorname{GGr}_{T}(S/K)$, and hence $\operatorname{GGr}_{T}(S/K)\,\triangleleft \, \operatorname{GGr}_{T}(S/B)$ and $\operatorname{GGr}_{T}(S/B)/\operatorname{GGr}_{T}(S/K)\cong \operatorname{GGr}_{T}(K/B)$.
\end{corollary}
A converse of Corollary \ref{6.3.3} is as follows, which also explains the terminology “normal $T$-subsets”. 
\begin{proposition} \label{6.3.4}
    Let $S$ be a $T$-space. Suppose that the identity function lies in $T$, $\forall a\in S\backslash C_{S}^{\operatorname{Aut}}$, ${{\left\langle a \right\rangle }_{T}}=S$, and $H\,\triangleleft \, \operatorname{Aut}_{T}(S)$. Then ${{S}^{H}}\,\triangleleft \, S$.
\end{proposition}
\begin{proof}
    Let $ u\in {{S}^{H}}$ and $v\in {{[u]}_{T}}\bigcap S$. By Theorem \ref{5.1.18}, $\exists \sigma \in \operatorname{Aut}_{T}(S)$ such that $\sigma (u)=v$. By Lemma \ref{6.2.2} (with $B=\emptyset$), ${{S}^{H}}$ is $\operatorname{Aut}_{T}(S)$-stable, and hence $v\in {{S}^{H}}$. Thus $\forall u\in {{S}^{H}}$, ${{[u]}_{T}}\bigcap S\subseteq {{S}^{H}}$. Hence by Definition \ref{normal $T$-subset}, ${{S}^{H}}\,\triangleleft \, S$.
\end{proof}
Finally, combining Corollary \ref{6.3.3} with Corollary \ref{6.2.4}, we immediately obtain our third answer to the first question in Problem \ref{6.0}:
\begin{corollary} \label{6.3.5}
    Let $S$ be a $T$-space with $B\subseteq K\subseteq S$. If $B={{S}^{\operatorname{GGr}_{T}(S/B)}}$, $K={{S}^{\operatorname{GGr}_{T}(S/K)}}$, $K\,\triangleleft \, S$, and each element of $\operatorname{GGr}_{T}(K/B)$ can be extended to an element of $\operatorname{GGr}_{T}(S/B)$, then $B={{K}^{\operatorname{GGr}_{T}(K/B)}}$.
\end{corollary}

\section{$\mathcal{F}$-transcendental elements and $\mathcal{F}$-transcendental subsets} \label{transcendental}
   In this section, the algebraic notions of transcendental elements over a field and purely transcendental field extensions are generalized for formal partial functions (\textit{cf.} Definition \ref{10.3.2}). To simplify our descriptions, we introduce
\begin{definition} \label{20.1}
  Let $f$ and $g$ be formal partial functions of $n$ variables. If for each set $D$,
\begin{enumerate}
    \item [(i)] $D$ is a domain of $f$ if and only if $D$ is a domain of $g$; and
    \item [(ii)] $\forall ({{u}_{1}},\cdots ,{{u}_{n}})\in {{D}^{n}}$, neither $f({{u}_{1}},\cdots ,{{u}_{n}})$ nor $g({{u}_{1}},\cdots ,{{u}_{n}})$ is well-defined or $f({{u}_{1}},\cdots ,{{u}_{n}})=g({{u}_{1}},\cdots ,{{u}_{n}})$;
\end{enumerate}
then we say that $f$ is \emph{equivalent to} $g$ and write $f=g$.
 \end{definition}
  Moreover, throughout this section, unless otherwise specified, $D$ always stands for a set and $\mathcal{F}$ denotes a set of formal partial functions.
 \subsection{$\mathcal{F}$-transcendental elements}
To generalize the concept of transcendental elements in field theory, we give
\begin{definition} \label{20.2}
    Let $n\in {{\mathbb{Z}}^{+}}$ and let $({{u}_{1}},\cdots ,{{u}_{n}})\in {{D}^{n}}$. If $f({{u}_{1}},\cdots ,{{u}_{n}})=g({{u}_{1}},\cdots ,{{u}_{n}})$ implies $f=g$, $\forall f,g\in \mathcal{F}$, then we say that $({{u}_{1}},\cdots ,{{u}_{n}})$ is \emph{$\mathcal{F}$-transcendental}.
\end{definition}
  The above definition of transcendental elements coincides with that in field theory, as shown below.
  \begin{proposition} \label{20.3}
      Let $B$ be a subfield of a field $F$ and let $\mathcal{F}$ be the polynomial ring $B[x]$. Then $\forall a\in F$, $a$ is transcendental over $B$ if and only if $a$ is $\mathcal{F}$-transcendental.
  \end{proposition}
  \begin{proof}
      Let $a\in F$. Then

$a$ is algebraic over $B$;\\*
$\Leftrightarrow \exists $ nonzero $p(x)\in B[x]$ such that $p(a)=0$; \\*
$\Leftrightarrow \exists f,g\in \mathcal{F}(=B[x])$ such that $f\ne g$ and $f(a)=g(a)$;\\*
$\Leftrightarrow a$ is not $\mathcal{F}$-transcendental.
  \end{proof}
  \subsection{$\mathcal{F}$-transcendental subsets}
	Definition \ref{20.2} may be generalized in more than one way. One of them is as follows.
\begin{definition} \label{20.4}
    Let $U\subseteq D$. If $\forall n\in {{\mathbb{Z}}^{+}}$ and $\{{{u}_{1}},\cdots ,{{u}_{n}}\}\subseteq U$, $({{u}_{1}},\cdots ,{{u}_{n}})$ is $\mathcal{F}$-transcendental, i.e. $f({{u}_{1}},\cdots ,{{u}_{n}})=g({{u}_{1}},\cdots ,{{u}_{n}})$ implies $f=g$, $\forall f,g\in \mathcal{F}$, then we say that $U$ is \emph{$\mathcal{F}$-transcendental}.
\end{definition}   
\begin{remark}
    By $\{{{u}_{1}},\cdots ,{{u}_{n}}\}\subseteq U$ we imply that $\forall 1\le i < j \le n$, $u_i\ne u_j$.
\end{remark}
The following shows that the above definition of $\mathcal{F}$-transcendental subsets of $D$ coincides with the concept of purely transcendental field extension in field theory.
\begin{proposition} \label{20.5}
    Let $B$ be a subfield of a field $F$, let \[\mathcal{F}\text{=}\bigcup\nolimits_{n\in {{\mathbb{Z}}^{+}}}{\text{Frac}(B[{{x}_{1}},\cdots ,{{x}_{n}}])},\] let $T$ be the par-operator gen-semigroup on $F$ defined in Example \ref{10.1.5}, and let nonempty $U\subseteq F$. Then ${{\left\langle U \right\rangle }_{T}}$ is a purely transcendental field over $B$ if and only if ${{\left\langle U \right\rangle }_{T}}=B$ or $U$ is $\mathcal{F}$-transcendental.
\end{proposition}
\begin{proof}
    Since ${{\left\langle U \right\rangle }_{T}}=\bigcup\nolimits_{n\in {{\mathbb{Z}}^{+}}}{\{B({{u}_{1}},\cdots ,{{u}_{n}})|{{u}_{1}},\cdots ,{{u}_{n}}\in U\}}=B(U)$, the field ${{\left\langle U \right\rangle }_{T}}$ is purely transcendental over $B$ if and only if ${{\left\langle U \right\rangle }_{T}}=B$ or $U$ is algebraically independent over $B$. Hence we only need to show that $U$ is $\mathcal{F}$-transcendental if and only if $U$ is algebraically independent over $B$.

$U$ is $\mathcal{F}$-transcendental ;\\*
$\Leftrightarrow \forall n\in {{\mathbb{Z}}^{+}}$, $\{{{u}_{1}},\cdots ,{{u}_{n}}\}\subseteq U$ and $f,g\in \mathcal{F}$, 
\begin{center}
    $f({{u}_{1}},\cdots ,{{u}_{n}})=g({{u}_{1}},\cdots ,{{u}_{n}})$ implies $f=g$;
\end{center}
$\Leftrightarrow \forall n\in {{\mathbb{Z}}^{+}}$, $\{{{u}_{1}},\cdots ,{{u}_{n}}\}\subseteq U$ and $f,g\in \mathcal{F}$, 
\begin{center}
    $f({{u}_{1}},\cdots ,{{u}_{n}})-g({{u}_{1}},\cdots ,{{u}_{n}})=0$ implies that $f-g$ is a zero polynomial;
\end{center}
$\Leftrightarrow \forall n\in {{\mathbb{Z}}^{+}}$, $\{{{u}_{1}},\cdots ,{{u}_{n}}\}\subseteq U$ and $h\in \mathcal{F}$, 
\begin{center}
    $h({{u}_{1}},\cdots ,{{u}_{n}})=0$ implies that $h$ is a zero polynomial;
\end{center}
$\Leftrightarrow U$ is algebraically independent over $B$.						
\end{proof}
 
\section{A generalized first isomorphism theorem} \label{first iso}
We shall derive a theorem which generalizes the first isomorphism theorems for groups, rings, and modules. 
\begin{notation}
    In this section, a $\theta$-morphism is always defined by Definition \ref{10.4.2}. Moreover, $\phi$ is always a $\theta$-morphism from a $T$-space $S$ to a $T'$-space $S'$.
\end{notation}
\subsection{“Quotient space” $Q_\phi$}
\begin{notation} \label{24.2}
    For $\phi$, we denote by $Q_\phi$ the set $\{\phi^{-1}(z)\,|\,z \in \operatorname{Im}\phi \}$.  
\end{notation}
\begin{remark}
    Obviously, $Q_\phi$ is a partition of $S$. In fact, we may regard $Q_\phi$ as the “quotient space” induced by $\phi$ (\textit{cf.} Corollary \ref{quotient space} and Example \ref{example of quotient space}).
\end{remark}
To define a par-operator gen-semigroup (corresponding to $\phi$) on $Q_\phi$, which we shall denote by $T_{\phi}^*$ (Notation \ref{T_phi^*}), we define the elements $f_{\phi}^*$ (or $f^*$ for brevity) of $T_{\phi}^*$ as follows.
\begin{notation} \label{24.3}
    Let $n$-variable $f \in T$. Then by $f_{\phi}^*$ or $f^*$ ($\subseteq Q_\phi \times \cdots \times Q_\phi$) we denote the relation
    \begin{center}
        $\{(\phi^{-1}(\phi(x_1)),\cdots,\phi^{-1}(\phi(x_n)),\phi^{-1}(\phi(f(x_1,\cdots,x_n)))) \,|\, (x_1,\cdots,x_n)\in S^n$ and $f(x_1,\cdots,x_n)$ is well-defined\}.
    \end{center}
 \end{notation}
 \begin{remark}
     If $f(x_1,\cdots,x_n)$ is well-defined, by Proposition \ref{11.1.1}, $f(x_1,\cdots,x_n) \in S$, and hence $\phi(f(x_1,\cdots,x_n))$ is well-defined.
 \end{remark}
 To better understand the notation, we give
 \begin{example} \label{example of cosets}
     Let $\phi$ be a group homomorphism from a group $S$ to a group $S'$ (\textit{cf.} Proposition \ref{9.4.8}). Let $e$ be the identity element of $S$. Then $\phi^{-1}(\phi(e))$ is the kernel of $\phi$, $\forall x \in S$, $\phi^{-1}(\phi(x))$ is a coset of the kernel, and $Q_\phi$ is the set of all cosets of the kernel. Moreover, we shall see that $f_{\phi}^*$ is a group operation among elements of $Q_\phi$.
 \end{example}
 \begin{proposition} \label{24.4}
        Let $n$-variable $f \in T$. Then $f_{\phi}^*$ is a partial function from $Q_\phi^n$ to $Q_\phi$, where $Q_\phi^n$ denotes the cartesian product of $n$ copies of $Q_\phi$.
\end{proposition}
\begin{proof}
    It suffices to show that $\forall (C_1,\cdots,C_n)\in Q_\phi^n$, the set $\{C\,|\,(C_1,\cdots,C_n, C)\in f_{\phi}^*\}$ has at most one element. 

    For this purpose, suppose that $(a_1,\cdots,a_n),(b_1,\cdots,b_n) \in S^n$ such that both $f(a_1,\cdots,a_n)$ and $f(b_1,\cdots,b_n)$ are well-defined and $\phi(a_i)=\phi(b_i),\forall i=1,\cdots,n$. By Notation \ref{24.3}, it suffices to show $\phi(f(a_1,\cdots,a_n))=\phi(f(b_1,\cdots,b_n))$.
     
    Since $\phi$ is a $\theta$-morphism, by Definition \ref{10.4.2},
    \begin{align*}
        \phi(f(a_1,\cdots,a_n))&=\theta(f)(\phi(a_1),\cdots,\phi(a_n))\\
        &=\theta(f)(\phi(b_1),\cdots,\phi(b_n))\\
        &=\phi(f(b_1,\cdots,b_n)),
    \end{align*}
    as desired.
\end{proof}
Then the following is obvious. 
\begin{corollary} \label{24.5}
    Let $n$-variable $f \in T$. Then $f_{\phi}^*:Q_\phi^n \to Q_\phi$ given by 
    \[(\phi^{-1}(\phi(x_1)),\cdots,\phi^{-1}(\phi(x_n))) \mapsto \phi^{-1}(\phi(f(x_1,\cdots,x_n))),\]
    $\forall (x_1,\cdots,x_n) \in S^n$ such that $f(x_1,\cdots,x_n)$ is well-defined, is a well-defined partial function.
 \end{corollary}
 \begin{notation} \label{T_phi^*}
      We denote by $T_{\phi}^*$ the set $\{f_{\phi}^*\,|\, f \in T\}$, where $f_{\phi}^*:Q_\phi^n \to Q_\phi$ is given in Corollary \ref{24.5}.
 \end{notation}
\begin{proposition} \label{24.7}
    $T_{\phi}^*$ is a par-operator gen-semigroup on $Q_\phi$.
\end{proposition}
\begin{remark}
    If the proposition holds, then ${{\left\langle Q_\phi \right\rangle }_{T_{\phi}^*}} \subseteq Q_\phi $, and thus by Definition \ref{quasi-T-space}, $Q_\phi$ is a quasi-$T_{\phi}^*$-space.
\end{remark}
\begin{proof}
    By Corollary \ref{24.5}, $\forall f_{\phi}^* \in T_{\phi}^*$, $f_{\phi}^*$ is a partial function from some $Q_\phi^n$ to $Q_\phi$.  

    Let $f\in T$ have $n$ variables, and $\forall i=1, \cdots, n$, let ${g}_{i}\in T$ have $n_i$ variables. To show that $T_{\phi}^*$ is a par-operator gen-semigroup on $Q_\phi$, by Definition \ref{10.1.3}, we only need to show that the composite $f^*\circ ({{g}_{1}^*},\cdots ,{{g}_{n}^*})$ is a restriction of some element of $T_{\phi}^*$. For this purpose, since $f\circ ({{g}_{1}},\cdots ,{{g}_{n}}) \in T$ and hence $(f\circ ({{g}_{1}},\cdots ,{{g}_{n}}))^* \in T_{\phi}^*$, it suffices to show
    \begin{equation} \label{Eq24.1}
        f^*\circ ({{g}_{1}^*},\cdots ,{{g}_{n}^*})=(f\circ ({{g}_{1}},\cdots ,{{g}_{n}}))^*.
    \end{equation}
   
    Let $D$ be the domain of $T$. Then by Definition \ref{10.1.1}, $f\circ ({{g}_{1}},\cdots ,{{g}_{n}})$ is the partial function from ${{D}^{m}}$ to $D$ given by 
    \[({{x}_{1}},\cdots ,{{x}_{m}})\mapsto f({{g}_{1}}({{\text{v}}_{1}}),\cdots ,{{g}_{n}}({{\text{v}}_{n}})),\] 
    where $m:=\sum\nolimits_{i}{{{n}_{i}}}$, ${{\text{v}}_{i}}:=({{x}_{{{m}_{i}}+1}},\cdots , {{x}_{{{m}_{i}}+{{n}_{i}}}})\in {{D}^{{{n}_{i}}}},\forall i=1,\cdots ,n$, ${{m}_{1}}:=0$ and ${{m}_{i}}:=\sum\nolimits_{k=1}^{i-1}{{{n}_{k}}},\forall i=2,\cdots ,n$.
    
    Then by Corollary \ref{24.5}, $(f\circ ({{g}_{1}},\cdots ,{{g}_{n}}))^*$ is the partial function from $Q_\phi^m$ to $Q_\phi$ given by
    \[(\phi^{-1}(\phi(x_1)),\cdots ,\phi^{-1}(\phi(x_m)))\mapsto \phi^{-1}(\phi(f({{g}_{1}}({{\text{v}}_{1}}),\cdots ,{{g}_{n}}({{\text{v}}_{n}})))),\] 
    $\forall (x_1,\cdots,x_m) \in S^m$ such that $f({{g}_{1}}({{\text{v}}_{1}}),\cdots ,{{g}_{n}}({{\text{v}}_{n}}))$ is well-defined. 
    
    To prove Equation (\ref{Eq24.1}), now we are showing that $f^*\circ ({{g}_{1}^*},\cdots ,{{g}_{n}^*})$ can be given exactly the same as above.

    By Definition \ref{10.1.1}, $f^*\circ ({{g}_{1}^*},\cdots ,{{g}_{n}^*})$ is the partial function from $Q_\phi^m$ to $Q_\phi$ given by 
    \[({C_{1}},\cdots ,{C_{m}})\mapsto f^*({{g}_{1}^*}({{\text{V}}_{1}}),\cdots ,{{g}_{n}^*}({{\text{V}}_{n}})),\] 
    where ${{\text{V}}_{i}}:=({C_{{{m}_{i}}+1}},\cdots , {C_{{{m}_{i}}+{{n}_{i}}}})\in {{(Q_\phi)}^{n_i}},\forall i=1,\cdots ,n$.

    By Corollary \ref{24.5},  $\forall i=1,\cdots,n$, $g_i^*$ can be given by 
    \[(\phi^{-1}(\phi({{x}_{{{m}_{i}}+1}})),\cdots , \phi^{-1}(\phi({{x}_{{{m}_{i}}+{{n}_{i}}}}))) \mapsto \phi^{-1}(\phi(g_i(\text{v}_i))),\] 
    $\forall \text{v}_i \in S^{n_i}$ such that $g_i(\text{v}_i)$ is well-defined. Combining this with the preceding paragraph, we can tell that $f^*\circ ({{g}_{1}^*},\cdots ,{{g}_{n}^*})$ can be given by 
     \[(\phi^{-1}(\phi(x_1)),\cdots ,\phi^{-1}(\phi(x_m)))\mapsto f^*(\phi^{-1}(\phi(g_1(\text{v}_1))),\cdots ,\phi^{-1}(\phi(g_n(\text{v}_n)))),\]
     $\forall (x_1,\cdots,x_m) \in S^m$ such that $f^*(\phi^{-1}(\phi(g_1(\text{v}_1))),\cdots ,\phi^{-1}(\phi(g_n(\text{v}_n))))$ is well-defined.
    
    By Corollary \ref{24.5}, a well-defined $f^*(\phi^{-1}(\phi(g_1(\text{v}_1))),\cdots ,\phi^{-1}(\phi(g_n(\text{v}_n))))$ is given the value $\phi^{-1}(\phi(f(g_1(\text{v}_1),\cdots ,g_n(\text{v}_n))))$ with $f({{g}_{1}}({{\text{v}}_{1}}),\cdots ,{{g}_{n}}({{\text{v}}_{n}}))$ being well-defined.
    Combining this with the preceding paragraph, we see that $f^*\circ ({{g}_{1}^*},\cdots ,{{g}_{n}^*})$ can be given by 
     \[(\phi^{-1}(\phi(x_1)),\cdots ,\phi^{-1}(\phi(x_m)))\mapsto \phi^{-1}(\phi(f({{g}_{1}}({{\text{v}}_{1}}),\cdots ,{{g}_{n}}({{\text{v}}_{n}})))),\] 
    $\forall (x_1,\cdots,x_m) \in S^m$ such that $f({{g}_{1}}({{\text{v}}_{1}}),\cdots ,{{g}_{n}}({{\text{v}}_{n}}))$ is well-defined. 
    
    Since we showed that $(f\circ ({{g}_{1}},\cdots ,{{g}_{n}}))^*$ can be given exactly the same as above, Equation (\ref{Eq24.1}) holds, as desired. 
\end{proof}
\begin{notation}
   By $T_{\phi}^{\#}$ we denote the set $\{$Id on $Q_\phi\}\bigcup T_{\phi}^*$.
\end{notation}
\begin{corollary} \label{quotient space}
    $T_{\phi}^{\#}$ is a par-operator gen-semigroup on $Q_\phi$ and $Q_\phi$ is a $T_{\phi}^{\#}$-space. Specifically, $Q_\phi =  {{\left\langle Q_\phi \right\rangle }_{T_{\phi}^{\#}}}$.
\end{corollary}
\begin{proof}
    By Proposition \ref{24.7}, $T_{\phi}^*$ is a par-operator gen-semigroup on $Q_\phi$, so is $T_{\phi}^{\#}(=\{$Id on $Q_\phi\}\bigcup T_{\phi}^*)$. Hence $Q_\phi \supseteq  {{\left\langle Q_\phi \right\rangle }_{T_{\phi}^{\#}}}$. Since Id $\in T_{\phi}^{\#}$, $Q_\phi \subseteq  {{\left\langle Q_\phi \right\rangle }_{T_{\phi}^{\#}}}$. Thus, $Q_\phi =  {{\left\langle Q_\phi \right\rangle }_{T_{\phi}^{\#}}}$. 
\end{proof}
$Q_\phi$ in Corollary \ref{quotient space} is actually a “quotient space” of $S$ corresponding to $\phi$. To see this, we give
\begin{example} \label{example of quotient space}
Let $\phi$ be the same as in Example \ref{example of cosets}. Let $K$ be the kernel of $\phi$. Then the $T_{\phi}^{\#}$-space $Q_\phi$ is the quotient group $S/K$, where $T_{\phi}^{\#}$ is the set of all group operations among (a finite number of) elements of $S/K$. Moreover, $\phi^*$ defined as follows is the group isomorphism from $S/K$ to $\operatorname{Im} \phi$ induced by $\phi$.
\end{example}
\subsection{A generalized first isomorphism theorem}
\begin{notation}
    Let $\phi^*:Q_\phi \to \operatorname{Im}\phi$ be given by $\phi^{-1}(z) \mapsto z$ and let $\theta^*=\{(f^*,g)\,|\,(f,g)\in \theta\}$. 
\end{notation}
     Then $\theta^* \subseteq {T_{\phi}^{\#}}\times {T'}$. We are going to show that $\phi^*$  is a $\theta^*$-isomorphism from the $T_{\phi}^{\#}$-space $Q_\phi$ to $\operatorname{Im}\phi$ if $\operatorname{Im}\phi $ is a $T'$-space. But first, we need to show that $\phi^*$ is a $\theta^*$-morphism, which requires a lemma as follows.
\begin{lemma} \label{24.10}
    Suppose that $\theta {{|}_{\operatorname{Im}\phi }}$ is a map. Then $\theta^* {{|}_{\operatorname{Im}\phi^*}}$ is also a map. Moreover, $\theta^*(f^*)|_{\operatorname{Im}\phi^*}=\theta(f)|_{\operatorname{Im}\phi},\forall f \in \operatorname{Dom}\theta$.
\end{lemma}
\begin{proof}
      Apparently $\operatorname{Im}\phi^*=\operatorname{Im}\phi$. Thus, to show that $\theta^* {{|}_{\operatorname{Im}\phi^*}}$ is a map, it suffices to show that $\theta^* {{|}_{\operatorname{Im}\phi}}$ is a map.  
     
      For this purpose, let $f_1,f_2 \in T$ such that $f_1^*=f_2^*$. Let $(f_1,g_1),(f_2,g_2)\in \theta$, and hence $(f_1^*,g_1),(f_2^*,g_2)\in \theta^*$. Let $n\in {{\mathbb{Z}}^{+}}$ and $(a_1, \cdots, a_n),(b_1, \cdots, b_n) \in S^n$ such that $\phi(a_i)=\phi(b_i),\forall i=1,\cdots,n$.
      
      Then to show that $\theta^* {{|}_{\operatorname{Im}\phi}}$ is a map, by Proposition \ref{10.4.b}, it suffices to show that
     neither $g_{1}(\phi(a_1), \cdots, \phi(a_n))$ nor $g_2(\phi(b_1), \cdots, \phi(b_n))$ is well-defined or $g_{1}(\phi(a_1), \cdots, \phi(a_n))=g_2(\phi(b_1), \cdots, \phi(b_n))$:\\
    $g_{1}(\phi(a_1), \cdots, \phi(a_n))$ is well-defined;\\
    $\Leftrightarrow$ $\theta(f_1)(\phi(a_1), \cdots, \phi(a_n))$ is well-defined (since $(f_1,g_1)\in \theta$ and $\theta {{|}_{\operatorname{Im}\phi }}$ is a map);\\
    $\Leftrightarrow$ $f_1(a_1, \cdots, a_n)$ is well-defined (by (ii) in Definition \ref{10.4.2});\\
    $\Leftrightarrow$ $f_1^*(\phi^{-1}(\phi(a_1)), \cdots, \phi^{-1}(\phi(a_n)))$ is well-defined (by Corollary \ref{24.5});\\
    $\Leftrightarrow$ $f_2^*(\phi^{-1}(\phi(b_1)), \cdots, \phi^{-1}(\phi(b_n)))$ is well-defined (since $f_1^*=f_2^*$ and $\phi(a_i)=\phi(b_i),\forall i=1,\cdots,n$);\\
    $\Leftrightarrow$ $f_2(b_1, \cdots, b_n)$ is well-defined (by Corollary \ref{24.5});\\
    $\Leftrightarrow$ $\theta(f_2)(\phi(b_1), \cdots, \phi(b_n))$ is well-defined (by (ii) in Definition \ref{10.4.2});\\
    $\Leftrightarrow$ $g_{2}(\phi(b_1), \cdots, \phi(b_n))$ is well-defined (since $(f_2,g_2)\in \theta$ and $\theta {{|}_{\operatorname{Im}\phi }}$ is a map).

    And in the case where the above equivalent conditions are satisfied,
    \begin{align*}
        g_{1}(\phi(a_1), \cdots, \phi(a_n))&=\theta(f_1)(\phi(a_1), \cdots, \phi(a_n))\\
        (\text{by (ii) in Definition \ref{10.4.2}}) &=\phi(f_1(a_1, \cdots, a_n))\\
        (\text{by the definition of }\phi^*) &=\phi^*(\phi^{-1}(\phi(f_1(a_1, \cdots, a_n))))\\
        (\text{by Corollary \ref{24.5}}) &=\phi^*(f_1^*(\phi^{-1}(\phi(a_1)), \cdots, \phi^{-1}(\phi(a_n))))\\
        (\text{since } f_1^*=f_2^* \text{ and } \phi(a_i)=\phi(b_i),\forall i=1,\cdots,n) &=\phi^*(f_2^*(\phi^{-1}(\phi(b_1)), \cdots, \phi^{-1}(\phi(b_n))))\\
        (\text{by Corollary \ref{24.5}}) &=\phi^*(\phi^{-1}(\phi(f_2(b_1, \cdots, b_n))))\\
        (\text{by the definition of }\phi^*) &=\phi(f_2(b_1, \cdots, b_n))\\       
        (\text{by (ii) in Definition \ref{10.4.2}}) &=\theta(f_2)(\phi(b_1), \cdots, \phi(b_n))\\
        &=g_{2}(\phi(b_1), \cdots, \phi(b_n)),
    \end{align*}
    as desired. 
    
    Therefore, $\theta^* {{|}_{\operatorname{Im}\phi}}$ is a map, so is $\theta^* {{|}_{\operatorname{Im}\phi^*}}$.

    Moreover, let $(f,g)\in \theta$. Then $(f^*,g)\in \theta^*$. Since both $\theta^* {{|}_{\operatorname{Im}\phi^*}}$ and $\theta {{|}_{\operatorname{Im}\phi }}$ are maps,  
    \[ \theta^*(f^*)|_{\operatorname{Im}\phi^*}=g|_{\operatorname{Im}\phi^*}=g|_{\operatorname{Im}\phi}=\theta(f)|_{\operatorname{Im}\phi}.\]
\end{proof}
\begin{proposition}\label{24.11}
    $\phi^*$ is a $\theta^*$-morphism from the $T_{\phi}^{\#}$-space $Q_\phi$ to the $T'$-space $S'$.
\end{proposition}  
\begin{proof}
     Since $\phi$ is a $\theta$-morphism from $S$ to a $T'$-space, by Definition \ref{10.4.2},
      \begin{enumerate}
        \item [(i)] $\theta {{|}_{\operatorname{Im}\phi }}$ is a map and
         \item [(ii)] $\forall n\in {{\mathbb{Z}}^{+}}$, $(a_1,\cdots ,a_n)\in {{S}^{n}}$ and $f\in \operatorname{Dom}\theta $, neither $f(a_1,\cdots ,a_n)$ nor $\theta (f)(\phi ({{a}_{1}}),\cdots ,\phi ({{a}_{n}}))$ is well-defined or
    \[\phi (f(a_1,\cdots ,a_n))=\theta (f)(\phi ({{a}_{1}}),\cdots ,\phi ({{a}_{n}})).\]
    \end{enumerate}
    Then by Lemma \ref{24.10}, $\theta^* {{|}_{\operatorname{Im}\phi^*}}$ is a map and 
    $\theta^*(f^*)|_{\operatorname{Im}\phi^*}=\theta(f)|_{\operatorname{Im}\phi},\forall f \in \operatorname{Dom}\theta$. 
    Hence to show that $\phi^*$ is a $\theta^*$-morphism, again by Definition \ref{10.4.2}, we only need to show that      
       $\forall n\in {{\mathbb{Z}}^{+}}$, $(C_1,\cdots ,C_n)\in {Q_\phi^n}$ and $f^*\in \operatorname{Dom}\theta^* $, neither $f^*(C_1,\cdots ,C_n)$ nor $\theta^* (f^*)(\phi^* ({C_{1}}),\cdots ,\phi^* ({C_{n}}))$ is well-defined or
    \[\phi^* (f^*(C_1,\cdots ,C_n))=\theta^* (f^*)(\phi^* ({C_{1}}),\cdots ,\phi^* ({C_{n}})).\]

    For this purpose, let $n\in {{\mathbb{Z}}^{+}}$, $(a_1, \cdots, a_n) \in S^n$ and $f\in \operatorname{Dom}\theta$. 
    
    Then it suffices to show that neither $f^*(\phi^{-1}(\phi(a_1)), \cdots, \phi^{-1}(\phi(a_n)))$ nor $\theta^* (f^*)(\phi(a_1),\cdots ,\phi(a_n))$ is well-defined or
    \[\phi^* (f^*(\phi^{-1}(\phi(a_1)), \cdots, \phi^{-1}(\phi(a_n))))=\theta^* (f^*)(\phi(a_1),\cdots ,\phi(a_n)):\]
   $f^*(\phi^{-1}(\phi(a_1)), \cdots, \phi^{-1}(\phi(a_n)))$ is well-defined;\\
   $\Leftrightarrow$ $f(a_1, \cdots, a_n)$ is well-defined (by Corollary \ref{24.5});\\
   $\Leftrightarrow$  $\theta(f)(\phi(a_1), \cdots, \phi(a_n))$ is well-defined (by (ii) in Definition \ref{10.4.2});\\
   $\Leftrightarrow$ $\theta^* (f^*) (\phi(a_1), \cdots, \phi(a_n))$ is well-defined (since $\theta^*(f^*)|_{\operatorname{Im}\phi^*}=\theta(f)|_{\operatorname{Im}\phi}$).
   
   And in the case where the above equivalent conditions are satisfied,\\
   $\phi^*(f^*(\phi^{-1}(\phi(a_1)), \cdots, \phi^{-1}(\phi(a_n))))$\\
   $=\phi^*(\phi^{-1}(\phi(f(a_1, \cdots, a_n))))$ (by Corollary \ref{24.5}) \\
   $=\phi(f(a_1, \cdots, a_n))$ (by the definition of $\phi^*$) \\
   $=\theta(f)(\phi(a_1), \cdots, \phi(a_n))$ (by (ii) in Definition \ref{10.4.2}) \\
   $=\theta^* (f^*)(\phi(a_1), \cdots, \phi(a_n))$ (because $\theta^*(f^*)|_{\operatorname{Im}\phi^*}=\theta(f)|_{\operatorname{Im}\phi}$),\\ 
    as desired.
\end{proof}
   Then we have the main theorem of this section as follows.
\begin{theorem} \label{24main}
    Suppose that $\operatorname{Im}\phi $ is a $T'$-space. Then $\phi^*$ is a $\theta^*$-isomorphism from the $T_{\phi}^{\#}$-space $Q_\phi$ to the $T'$-space $\operatorname{Im}\phi$. 
\end{theorem}
\begin{proof}
     Since  $\operatorname{Im}\phi $ is a $T'$-space, $\phi$ is also a $\theta$-morphism to the $T'$-space $\operatorname{Im}\phi$. Then by Proposition \ref{24.11}, $\phi^*$ is a $\theta^*$-morphism from the $T_{\phi}^{\#}$-space $Q_\phi$ to the $T'$-space $\operatorname{Im}\phi$. Moreover, because $\phi^*$ is bijective, $\phi^*$ is a $\theta^*$-isomorphism from the $T_{\phi}^{\#}$-space $Q_\phi$ to the $T'$-space $\operatorname{Im}\phi$.
\end{proof}
Then by Proposition \ref{image theta space for multivariable}, the following is immediate.
\begin{corollary} \label{1st isomorphi}
    If $\operatorname{Im}\theta = T'$ and $\operatorname{Id} \in T'$, then $\operatorname{Im}\phi$ is a $T'$-space and $\phi^*$ is a $\theta^*$-isomorphism from the $T_{\phi}^{\#}$-space $Q_\phi$ to the $T'$-space $\operatorname{Im}\phi$. 
\end{corollary}
From the corollary together with Propositions \ref{9.4.8}, \ref{9.4.5} and \ref{9.4.6}, it is not hard to deduce the first isomorphism theorems for groups, rings, and modules (\textit{cf.} Example \ref{example of quotient space}).

\section{On topological spaces} \label{App to topo}
 In Subsection \ref{Topo Basic notions and properties}, we introduce some notions and properties of topological spaces which are related to our theory. To show an application, in Subsection \ref{“General” Galois correspondences}, we employ Corollaries \ref{2.2.5} and \ref{2.2.6} and obtain the Galois correspondences on topological spaces. 
\subsection{Basic notions and properties} \label{Topo Basic notions and properties}
Most results in this subsection will not be employed in Subsection \ref{“General” Galois correspondences}, but they may be useful for future research on topological spaces.

 To simplify our descriptions, we define some notations for Section \ref{App to topo} as follows.
\begin{notation} \label{7.1.1}
    Unless otherwise specified, $X$ and $Y$ always denote topological spaces. Let $\mathcal{P}(X)$ and $\mathcal{P}(Y)$ denote the power sets of $X$ and $Y$, respectively. Moreover,
\begin{center}
    $T:={{T}_{X}}:=\{\text{I}{{\text{d}}_{X}}$, $\text{C}{{\text{l}}_{X}}:\mathcal{P}(X)\to \mathcal{P}(X)$ given by ${{A}_{X}}\mapsto \overline{{{A}_{X}}}\}$
\end{center}
 and
 \begin{center}
     ${{T}_{Y}}:=\{\text{I}{{\text{d}}_{Y}}$, $\text{C}{{\text{l}}_{Y}}:\mathcal{P}(Y)\to \mathcal{P}(Y)$ given by ${{A}_{Y}}\mapsto \overline{{{A}_{Y}}}\}$,
 \end{center}
where $\text{I}{{\text{d}}_{X}}$ and $\text{I}{{\text{d}}_{Y}}$ denote the identity functions on $\mathcal{P}(X)$ and $\mathcal{P}(Y)$, respectively.
\end{notation}
  \begin{remark}
      By Definition \ref{Operator semigroup}, it is easy to see that $T(={{T}_{X}})$ and ${{T}_{Y}}$ are operator semigroups on $\mathcal{P}(X)$ and $\mathcal{P}(Y)$, respectively.
  \end{remark}
Then we have
\begin{proposition} \label{7.1.2}
    $\mathcal{P}(X)$ is a $T$-space and
\begin{center}
    $\operatorname{End}_{T}(\mathcal{P}(X))=\{\sigma :\mathcal{P}(X)\to \mathcal{P}(X)\,|\,\sigma (\overline{A})=\overline{\sigma (A)},\forall A\in \mathcal{P}(X)\}$.
\end{center}
\end{proposition}
\begin{proof}
    By Definition \ref{$T$-space}, $\mathcal{P}(X)$ is a $T$-space because ${{\left\langle \mathcal{P}(X) \right\rangle }_{T}}=P(X)$.

Since $T=\{\text{I}{{\text{d}}_{X}}$, $\text{C}{{\text{l}}_{X}}:\mathcal{P}(X)\to \mathcal{P}(X)$ given by ${{A}_{X}}\mapsto \overline{{{A}_{X}}}\}$, by Definition \ref{1.3.1}, a map $\sigma :\mathcal{P}(X)\to \mathcal{P}(X)$ is a $T$-morphism if and only if
\begin{center}
    $(\sigma (\text{C}{{\text{l}}_{X}}(A))=)\sigma (\overline{A})=\overline{\sigma (A)}(=\text{C}{{\text{l}}_{X}}(\sigma (A))),\forall A\in \mathcal{P}(X)$. 
\end{center}
Thus,
\[\operatorname{End}_{T}(\mathcal{P}(X))=\{\sigma :\mathcal{P}(X)\to \mathcal{P}(X)\,|\,\sigma (\overline{A})=\overline{\sigma (A)},\forall A\in \mathcal{P}(X)\}.\]
\end{proof}
For brevity, we take out a part from Proposition \ref{5.3.3} as follows.
\begin{definition} \label{7.1.3}
    Let $f:X\to Y$ be a map. Then we define its \emph{induced map} $f^*:\mathcal{P}(X)\to \mathcal{P}(Y)$ by 
\begin{enumerate}
    \item $\forall A\in \mathcal{P}(X)$ that is closed in $X$, let $f^*(A)=\overline{f(A)}$, and
    \item $\forall A\in \mathcal{P}(X)$ that is not closed in $X$, let $f^*(A)=f(A)$,
\end{enumerate}
where $f(A)$ denotes the set $\{f(x)\,|\,x\in A\}$ for convenience.
\end{definition}
\begin{remark}
    In Section \ref{App to topo}, unless otherwise specified, $f^*$ is defined by Definition \ref{7.1.3} whenever $f$ is a map between topological spaces.
\end{remark}
 With Definition \ref{7.1.3}, we can restate Proposition \ref{1.3.7} as follows.
\begin{proposition} \label{7.1.4}
   A map $f:X\to X$ is continuous if and only if its induced map $f^*:\mathcal{P}(X)\to \mathcal{P}(X)$ is a $T$-morphism.
\end{proposition}
	To facilitate our following discussion, we give
\begin{notation} \label{7.1.5}
    Let $F(X,Y)$ denote the set of all functions from $X$ to $Y$ and let $F^*(X,Y)$ denote the set $\{f^*\,|\,f\in F(X,Y)\}$. Let $F(X)=F(X,X)$ and let $F^*(X)=F^*(X,X)$. Moreover, let $C(X)$ denote the set of all continuous functions from $X$ to $X$, and let $C^*(X)$ denote the set $\{f^*\,|\,f\in C(X)\}$.
\end{notation}
Then by Proposition \ref{7.1.4}, the following is obvious.
\begin{corollary} \label{7.1.6}
    $F^*(X)\bigcap \operatorname{End}_{T}(\mathcal{P}(X))=C^*(X)$. 
\end{corollary}
	It would be desirable if $F^*(X)\bigcap \operatorname{End}_{T}(\mathcal{P}(X))=\operatorname{End}_{T}(\mathcal{P}(X))$. However, it is possible that $F^*(X)\bigcap \operatorname{End}_{T}(\mathcal{P}(X))\subsetneq \operatorname{End}_{T}(\mathcal{P}(X))$ (and so $C^*(X)\subsetneq \operatorname{End}_{T}(\mathcal{P}(X))$). That is, there may exist a $T$-morphism $\sigma $ from $\mathcal{P}(X)$ to $\mathcal{P}(X)$ such that $\sigma $ is not the induced map of any map from $X$ to $X$, as shown below.
 \begin{example} \label{7.1.7}
     Let $X$ be any topological space with the discrete topology. Then it follows from Proposition \ref{7.1.2} that any map from $\mathcal{P}(X)$ to $\mathcal{P}(X)$ is a $T$-morphism. Let $\sigma :\mathcal{P}(X)\to \mathcal{P}(X)$ be given by $A\mapsto \emptyset ,\forall A\in \mathcal{P}(X)$. If $X$ is nonempty, then by Definition \ref{7.1.3}, it is not hard to see that $\sigma \notin F^*(X)$.
 \end{example}
 
	Before going further, we define  
\begin{notation} \label{7.1.8}
    We denote by ${{I}_{X,Y}}$ the map from $F(X,Y)$ to $F^*(X,Y)$ given by $f\mapsto f^*$. Besides, let ${{I}_{X}}={{I}_{X,X}}$. 
\end{notation} 
	Obviously ${{I}_{X,Y}}$ is surjective. It would be desirable if ${{I}_{X,Y}}$ were always injective. However, there are cases where ${{I}_{X,Y}}$ is not injective, as shown below.
 \begin{example} \label{7.1.9}
     Let $X=\{1\}$ and $Y=\{2,3\}$. Let ${{\mathcal{T}}_{X}}=\{\emptyset ,X\}$ and ${{\mathcal{T}}_{Y}}=\{\emptyset ,Y\}$ be the (trivial) topologies on $X$ and $Y$, respectively. Let two maps $f,g:X\to Y$ be given by $f(1)=2$ and $g(1)=3$, respectively. Then $F(X,Y)=\{f,g\}$. And by Definition \ref{7.1.3}, $f^*(\emptyset )=\emptyset =g^*(\emptyset )$ and $f^*(\{1\})=\overline{\{f(1)\}}=\overline{\{2\}}=\{2,3\}=\overline{\{3\}}=\overline{\{g(1)\}}=g^*(\{1\})$. Thus $f^*=g^*$. Hence the map ${{I}_{X,Y}}$ from $F(X,Y)$ to $F^*(X,Y)=\{f^*(=g^*)\}$ is not injective.
 \end{example}
	Surprisingly, there are two important cases where ${{I}_{X,Y}}$ is injective, as shown by the following two results.
\begin{proposition} \label{7.1.10}
   Suppose that every finite point set in $Y$ is closed, i.e. $Y$ satisfies the ${{T}_{1}}$ axiom. Then ${{I}_{X,Y}}$ is injective.
\end{proposition}
\begin{proof}
    Assume that ${{I}_{X,Y}}$ is not injective. Then there are two different maps $f,g:X\to Y$ such that $f^*=g^*$. Hence $\exists x\in X$ such that $f(x)\ne g(x)$. 

From Definition \ref{7.1.3}, we can tell that $\{x\}$ is closed in $X$ (otherwise $f^*(\{x\})=\{f(x)\}\ne \{g(x)\}=g^*(\{x\})$, which is contrary to the assumption $f^*=g^*$). Then again by Definition \ref{7.1.3}, $\overline{\{f(x)\}}=f^*(\{x\})=g^*(\{x\})=\overline{\{g(x)\}}$. However, since every finite point set in $Y$ is closed, $\overline{\{f(x)\}}=\{f(x)\}$ and $\overline{\{g(x)\}}=\{g(x)\}$. Thus $f(x)=g(x)$, which is contrary to $f(x)\ne g(x)$.
\end{proof}
\begin{proposition} \label{7.1.11}
   ${{I}_{X}}$ is injective.
\end{proposition}
\begin{proof}
    Assume that ${{I}_{X}}$ is not injective. Then there are two different maps $f,g:X\to X$ such that $f^*=g^*$. Hence $\exists x\in X$ such that $f(x)\ne g(x)$. Then by the same argument as in the proof of Proposition \ref{7.1.10}, $\{x\}$ is closed in $X$ and $\overline{\{f(x)\}}=\overline{\{g(x)\}}$. 
    
    It follows that $\overline{\{f(x)\}}=\overline{\{g(x)\}}\supseteq \{f(x),g(x)\}$. Hence $\overline{\{f(x)\}}\ne \{f(x)\}$ and $\overline{\{g(x)\}}\ne \{g(x)\}$ (because $f(x)\ne g(x)$), and thus $f(x)\ne x$ (otherwise  $\overline{\{f(x)\}}=\overline{\{x\}}=\{x\}=\{f(x)\}$, a contradiction). Analogously, $g(x)\ne x$.

	Then $\overline{\{x,f(x)\}}\supseteq \{x,f(x),g(x)\}\supsetneq \{x,f(x)\}$ because $\overline{\{f(x)\}}\supseteq \{f(x),g(x)\}$ and $f(x)\ne g(x)\ne x$. Hence $\overline{\{x,f(x)\}}\ne \{x,f(x)\}$; that is, $\{x,f(x)\}$ is not closed in $X$. Hence by Definition \ref{7.1.3}, 
\begin{center}
     $f^*(\{x,f(x)\})=\{f(x),f(f(x))\}$ and $g^*(\{x,f(x)\})=\{g(x),g(f(x))\}$. 
\end{center}	

 Because it is assumed that $f^*=g^*$, $\{f(x),f(f(x))\}=\{g(x),g(f(x))\}$. 

 Since $f(x)\ne g(x)$, we can tell from the above equation that $f(f(x))=g(x)\ne f(x)=g(f(x))$. Then because it was shown that $\overline{\{f(x)\}}\ne \{f(x)\}$, by Definition \ref{7.1.3}, $f^*(\{f(x)\})=\{f(f(x))\}\ne \{g(f(x))\}=g^*(\{f(x)\})$, which is contrary to the assumption that $f^*=g^*$.
\end{proof}
	Moreover, there is an important case where ${{I}_{X}}$ commutes with composition of functions, as shown below.
 \begin{lemma} \label{7.1.12}
    Let $f$ and $g$ be continuous maps from $X$ to $X$. Suppose that $g$ is injective. Then $(f\circ g)^*=f^*\circ g^*$. 
 \end{lemma}
\begin{proof}
    It suffices to show that $(f\circ g)^*(A)=f^*(g^*(A)),\forall A\in \mathcal{P}(X)$. There are two cases as follows.

(1) $A=\overline{A}$:

	By Definition \ref{7.1.3}, 
     \[(f\circ g)^*(A)=\overline{(f\circ g)(A)}=\overline{f(g(A))}\] and 
     \[f^*(g^*(A))=f^*(\overline{g(A)})=\overline{f(\overline{g(A)})}.\]

   	Then it suffices to show $\overline{f(\overline{g(A)})}=\overline{f(g(A))}$. 
    
    Since $f$ is continuous, by e.g. \cite{6}, $f(\overline{g(A)})\subseteq \overline{f(g(A))}$. Hence $\overline{f(\overline{g(A)})}\subseteq \overline{f(g(A))}$. On the other hand, $f(\overline{g(A)})\supseteq f(g(A))$, and hence $\overline{f(\overline{g(A)})}\supseteq \overline{f(g(A))}$. Therefore, $\overline{f(\overline{g(A)})}=\overline{f(g(A))}$, as desired.

(2)  $A\ne \overline{A}$:

Assume that $g(A)$ is closed. Since $g$ is continuous, ${{g}^{-1}}(g(A))$ is also closed. However, since $g$ is injective, ${{g}^{-1}}(g(A))=A$. Hence ${{g}^{-1}}(g(A))$ is not closed (because $A\ne \overline{A}$), a contradiction. Thus $g(A)$ is not closed. Then by Definition \ref{7.1.3},
\[f^*(g^*(A))=f^*(g(A))=f(g(A))=(f\circ g)(A)=(f\circ g)^*(A),\] 
as desired.
\end{proof}

\begin{notation} \label{7.1.13}
    We denote by $\text{Hom}(X)$ the set of all homeomorphisms from $X$ to $X$. Let $\text{Hom}^*(X)=\{f^*\,|\,f\in \text{Hom}(X)\}$.
\end{notation}
	Then we immediately have
\begin{corollary} \label{7.1.14}
    $\emph{Hom}(X)$ constitutes a group, which we still denote by $\emph{Hom}(X)$, with composition of functions as the binary operation.
\end{corollary}   
$\text{Hom}^*(X)$ also constitutes a group, as shown below.
\begin{proposition} \label{7.1.15}
    With composition of functions as the binary operation, $\emph{Hom}^*(X)$ constitutes a subgroup, which we still denote by $\emph{Hom}^*(X)$, of $\operatorname{Aut}_{T}(\mathcal{P}(X))$ $($defined in Proposition \ref{1.3.11}$)$.
\end{proposition} 
\begin{remark}
    Recall that $T$ was defined in Notation \ref{7.1.1}.
\end{remark} 
\begin{proof}
    It follows from Lemma \ref{7.1.12} that $\text{Hom}^*(X)$ is closed under composition of functions as the binary operation; that is, $\forall f^*,g^*\in \text{Hom}^*(X)$, $f^*\circ g^*\in \text{Hom}^*(X)$. Composition of functions is associative. Moreover, by Definition \ref{7.1.3}, the identity map on $X$ induces the identity map on $\mathcal{P}(X)$, which we denote by $\operatorname{Id}^*(\in \text{Hom}^*(X))$. 
    
    By Definition \ref{7.1.3}, it is not hard to see that $\forall f^* \in \text{Hom}^*(X)$, $f^*$ is bijective. $\forall f\in \text{Hom}(X)$, let ${{(f^*)}^{-1}}=({{f}^{-1}})^*$. Then $\forall f\in \text{Hom}(X)$, by Lemma \ref{7.1.12}, ${{(f^*)}^{-1}}\circ f^*=({{f}^{-1}})^*\circ f^*=({{f}^{-1}}\circ f)^*=\operatorname{Id}^*$, and analogously, $f^*\circ {{(f^*)}^{-1}}=\operatorname{Id}^*$. 
    
    Therefore, $\text{Hom}^*(X)$ constitutes a group with composition of functions as the binary operation. 

 Then from Proposition \ref{7.1.4}, we can tell that $\text{Hom}^*(X)\subseteq \operatorname{Aut}_{T}(\mathcal{P}(X))$. By Proposition \ref{1.3.11}, $\operatorname{Aut}_{T}(\mathcal{P}(X))$ constitutes a group with composition of functions as the binary operation. Thus $\text{Hom}^*(X)$ constitutes a subgroup of $\operatorname{Aut}_{T}(\mathcal{P}(X))$.	
\end{proof}
 Moreover, $\text{Hom}^*(X)$ is isomorphic to $\text{Hom}(X)$, as shown below. 
\begin{proposition} \label{7.1.16}
    ${{I}_{X}}{{|}_{\emph{Hom}(X)}}$, which is the restriction of ${{I}_{X}}$ to $\emph{Hom}(X)$, is a group isomorphism from $\emph{Hom}(X)$ to $\emph{Hom}^*(X)$.
\end{proposition}   
\begin{remark}
   $I_X$ was defined by Notation \ref{7.1.8}.
\end{remark}
\begin{proof}
    By Proposition \ref{7.1.11}, ${{I}_{X}}{{|}_{\text{Hom}(X)}}$ is an injective map from $\text{Hom}(X)$ to $\text{Hom}^*(X)(=\{f^*\,|\,f\in \text{Hom}(X)\})$. It follows that ${{I}_{X}}{{|}_{\text{Hom}(X)}}$ is a bijective map between the two groups. And by Lemma \ref{7.1.12},
    \[{{I}_{X}}(f\circ g)={{I}_{X}}(f)\circ {{I}_{X}}(g),\forall f,g\in \text{Hom}(X).\]
Therefore, ${{I}_{X}}{{|}_{\text{Hom}(X)}}$ is a group isomorphism.	
\end{proof}

\subsection{Galois correspondences on topological spaces} \label{“General” Galois correspondences}
Firstly, we need to recall some notations. By Definitions \ref{1.4.1}, \ref{1.5.1}, \ref{2.2.1} and \ref{2.2.3}, we immediately have
\begin{proposition} \label{7.2.1}
    Let $B\subseteq \mathcal{P}(X)$ and let $H\subseteq \operatorname{End}_{T}(\mathcal{P}(X))$. Then
\[\mathcal{P}{{(X)}^{H}}=\{A\in \mathcal{P}(X)\,|\,\sigma (A)=A,\forall \sigma \in H\},\]
\[\operatorname{GMn}_{T}(\mathcal{P}(X)/B)=\{\sigma \in \operatorname{End}_{T}(\mathcal{P}(X))\,|\,\sigma (A)=A,\forall A\in B\},\]
\[\operatorname{Int}_{T}^{\operatorname{End}}(\mathcal{P}(X)/B)=\{\mathcal{P}{{(X)}^{L}}\supseteq B\,|\, L\subseteq \operatorname{End}_{T}(\mathcal{P}(X))\},\] and
\[\operatorname{GSMn}_{T}(\mathcal{P}(X)/B)=\{\operatorname{GMn}_{T}(\mathcal{P}(X)/K)\,|\,B\subseteq K\subseteq \mathcal{P}(X)\}.\]
\end{proposition}
\begin{remark}
    $\operatorname{End}_{T}(\mathcal{P}(X))$ was given in Proposition \ref{7.1.2}.
\end{remark}
Then by applying Corollary \ref{2.2.5} to a topological space $X$, we immediately get the following, which, roughly speaking, shows the Galois correspondence between the Galois $T$-monoids of $\mathcal{P}(X)$ and the fixed subsets of $\mathcal{P}(X)$ under the actions of the $T$-endomorphisms of $\mathcal{P}(X)$.
\begin{corollary} \label{7.2.2}
    Let $X$ be a topological space, let $\mathcal{P}(X)$ be the power set of $X$, and let $B\subseteq \mathcal{P}(X)$. Then the correspondences
    \begin{center}
   $\gamma :H\mapsto \mathcal{P}{{(X)}^{H}}$ and $\delta :K\mapsto \operatorname{GMn}_{T}(\mathcal{P}(X)/K)$ 
    \end{center}
	define inclusion-inverting mutually inverse bijective maps between $\operatorname{GSMn}_{T}(\mathcal{P}(X)/B)$ and $\operatorname{Int}_{T}^{\operatorname{End}}(\mathcal{P}(X)/B)$. 
\end{corollary}
Moreover, we may employ Theorem \ref{4.1.2} (resp. Theorem \ref{4.3.2}) to investigate whether we can endow $\mathcal{P}(X)$ with a topology such that 
\begin{center}
  $\operatorname{Int}_{T}^{\operatorname{End}}(\mathcal{P}(X)/B)\backslash \{\emptyset \}=$\{$B\subseteq {S}' \le _q \mathcal{P}(X)\,|\,{S}'$ is closed and nonempty\}   
\end{center}
(resp. whether we can define a topology on End$_{T}(\mathcal{P}(X))$ such that $\operatorname{GSMn}_{T}(\mathcal{P}(X)/B)$ is the set of all closed submonoids of $\operatorname{GMn}_{T}(\mathcal{P}(X)/B)$).

We will not go further for this subject.

$ $

Analogously, by Proposition \ref{7.1.2} and Definition \ref{$T$-isomorphisms}, we have
\begin{center}
    Aut$_{T}(\mathcal{P}(X))=\{$bijective $\sigma :\mathcal{P}(X)\to \mathcal{P}(X)\,|\,\sigma (\overline{A})=\overline{\sigma (A)},\forall A\in \mathcal{P}(X)\}$.
\end{center}
And by Definitions \ref{1.5.1}, \ref{2.2.1} and \ref{2.2.3}, we immediately have
\begin{proposition} \label{7.2.3}
    Let $B\subseteq \mathcal{P}(X)$. Then
	\[\operatorname{GGr}_{T}(\mathcal{P}(X)/B)=\{\sigma \in \operatorname{Aut}_{T}(\mathcal{P}(X))\,|\,\sigma (A)=A,\forall A\in B\},\] 
\[\operatorname{Int}_{T}^{\operatorname{Aut}}(\mathcal{P}(X)/B)=\{\mathcal{P}{{(X)}^{H}}\supseteq B\,|\, H\subseteq \operatorname{Aut}_{T}(\mathcal{P}(X))\},\] and
\[\operatorname{GSGr}_{T}(\mathcal{P}(X)/B)=\{\operatorname{GGr}_{T}(\mathcal{P}(X)/K)\,|\,B\subseteq K\subseteq \mathcal{P}(X)\}.\]
\end{proposition}
Then by applying Corollary \ref{2.2.6} to a topological space $X$, we immediately have the following, which shows the Galois correspondence between the Galois $T$-groups of $\mathcal{P}(X)$ and the fixed subsets of $\mathcal{P}(X)$ under the actions of the $T$-automorphisms of $\mathcal{P}(X)$.
\begin{corollary} \label{7.2.4}
    Let $X$ be a topological space, let $\mathcal{P}(X)$ be the power set of $X$, and let $B\subseteq \mathcal{P}(X)$. Then the correspondences
    \begin{center}
    	$\gamma :H\mapsto \mathcal{P}{{(X)}^{H}}$ and $\delta :K\mapsto \operatorname{GGr}_{T}(\mathcal{P}(X)/K)$  
    \end{center}
    define inclusion-inverting mutually inverse bijective maps between $\operatorname{GSGr}_{T}(\mathcal{P}(X)/B)$ and $\operatorname{Int}_{T}^{\operatorname{Aut}}(\mathcal{P}(X)/B)$.
\end{corollary}
Moreover, we may employ Theorem \ref{4.2.2} (resp. Theorem \ref{4.4.2}) to investigate whether we can endow $\mathcal{P}(X)$ with a topology such that 
\begin{center}
    $\operatorname{Int}_{T}^{\operatorname{Aut}}(\mathcal{P}(X)/B)\backslash \{\emptyset \}=$\{$B\subseteq {S}'\le _q \mathcal{P}(X)\,|\,{S}'$ is closed and nonempty\} 
\end{center} 
(resp. whether we can define a topology on Aut$_{T}(\mathcal{P}(X))$ such that
\begin{center}
 $\operatorname{GSGr}_{T}(\mathcal{P}(X)/B)=\{$all closed subgroups of $\operatorname{GGr}_{T}(\mathcal{P}(X)/B)\}$).   
\end{center}

\section{On dynamical systems} \label{App to dyn}
	In this section, we first show a close relation between our theory and dynamical systems. Then, to give an example of it, we obtain the Galois correspondences on dynamical systems.
\subsection{Basic notions and properties}
	Recall that a dynamical system can be defined as follows (see e.g. \cite{8}).
\begin{definition} \label{8.1.1}
    Let $M$ be a monoid with identity element $e$ acting on a set $S$. That is, there is a map $\Phi :M\times S\to S$ given by $(g,x)\mapsto {{\Phi }_{g}}(x)$ such that ${{\Phi }_{g}}\circ {{\Phi }_{h}}={{\Phi }_{g\circ h}}$ and ${{\Phi }_{e}}=\operatorname{Id}$, where ${{\Phi }_{f}},\forall f\in M$, is a map from $S$ to $S$. Then the triple $(M,S,\Phi )$ is called a \emph{dynamical system}.
\end{definition} 
\begin{proposition} \label{8.1.2}
     Let $(M,S,\Phi )$ be a dynamical system and let $T=\{{{\Phi }_{g}}\,|\,g\in M\}$. Then $T$ is an operator semigroup on $S$, $\operatorname{Id}\in T$ and $S$ is a $T$-space. 
\end{proposition}
\begin{proof}
    By Definition \ref{8.1.1}, $\forall {{\Phi }_{g}}\in T$, ${{\Phi }_{g}}$ is a map from $S$ to $S$. And $\forall {{\Phi }_{g}},{{\Phi }_{h}}\in T$, ${{\Phi }_{g}}\circ {{\Phi }_{h}}={{\Phi }_{g\circ h}}\in T$. Hence by Definition \ref{Operator semigroup}, $T$ is an operator semigroup on $S$. Then ${{\left\langle S \right\rangle }_{T}}\subseteq S$. Moreover, ${{\left\langle S \right\rangle }_{T}}\supseteq S$ since ${{\Phi }_{e}}=\operatorname{Id}\in T$. Thus ${{\left\langle S \right\rangle }_{T}}=S$ is a $T$-space.
\end{proof}
	Proposition \ref{8.1.2} implies that we can apply our theory for operator semigroups to dynamical systems.
 
	Conversely, for any operator semigroup $T$ with $\operatorname{Id}\in T$, each $T$-space induces a dynamical system, as shown below.
\begin{proposition} \label{8.1.4}
    Let $T$ be an operator semigroup on a set $D$ with $\operatorname{Id}\in T$, let $S$ be a $T$-space, and let a map $\Phi :T\times S\to S$ be given by $(f,x)\mapsto f(x)$. Then $\Phi $ is well-defined and $(T,S,\Phi )$ is a dynamical system.
\end{proposition}  
\begin{proof}
    By Proposition \ref{<S> contained in S}, ${{\left\langle S \right\rangle }_{T}}\subseteq S$, and hence $\Phi $ is well-defined. $\forall f\in T$, let ${{\Phi }_{f}}=f{{|}_{S}}$. Then $\Phi $ can be given equivalently by $(f,x)\mapsto {{\Phi }_{f}}(x)$. Hence from Definition \ref{8.1.1}, we can tell that $(T,S,\Phi )$ is a dynamical system. 
\end{proof}
Proposition \ref{8.1.4} implies that we can apply theories of dynamical systems to $T$-spaces with $T$ being an operator semigroup and $\operatorname{Id}\in T$.

\subsection{Galois correspondences on dynamical systems} \label{dyn Galois correspondences}
	In this subsection, unless otherwise specified, we use the following notation.
\begin{notation} \label{8.2.1}
    Let $M$, $S$, $\Phi $ and $(M,S,\Phi )$ be defined as in Definition \ref{8.1.1}. Let $T=\{{{\Phi }_{g}}\,|\,g\in M\}$.
\end{notation}
\begin{remark}
    By Proposition \ref{8.1.2}, $T$ is an operator semigroup on $S$, $\operatorname{Id}\in T$ and $S$ is a $T$-space.
\end{remark}
With Notation \ref{8.2.1} and Proposition \ref{8.1.2}, we may apply our theory for operator semigroups to a dynamical system $(M,S,\Phi )$ quite straightforwardly. We take Galois correspondences as our example.

$ $

Firstly, we need to recall some notations.

By Definitions \ref{1.3.1} and \ref{1.3.2}, we immediately have
\begin{proposition} \label{8.2.2}
$ $
    \begin{center}
        $\operatorname{End}_{T}(S)=\{\sigma :S\to S\,|\,\sigma (f(x))=f(\sigma (x)),\forall x\in S \,and\, f\in T\}$.
    \end{center}
\end{proposition} 	
	To get some more idea about $\operatorname{End}_{T}(S)$, we take the prototypical example of a discrete dynamical system as our example:
\begin{proposition} \label{8.2.3}
    Let $({{\mathbb{N}}_{0}},I,\Phi )$ be an iterated map $($as a dynamical system$)$ on interval $I\subseteq \mathbb{R}$, where ${{\Phi }_{0}}=\operatorname{Id}$, ${{\Phi }_{n}}$ is the composite of $n$ copies of $h$, $\forall n\in {{\mathbb{Z}}^{+}}$, and $h$ is a map from $I$ to itself. Let $T=\{{{\Phi }_{n}}\,|\,n\in {{\mathbb{N}}_{0}}\}$. Then $T\subseteq \operatorname{End}_{T}(I)$.
\end{proposition}
\begin{proof}
    By Proposition \ref{8.1.2}, $T$ is an operator semigroup on $I$, $\operatorname{Id}\in T$ and $I$ is a $T$-space. Since $T=\{\operatorname{Id}\}\bigcup \left\langle h \right\rangle $, $\sigma (f(x))=f(\sigma (x))$, $\forall x\in I$ and $\sigma ,f\in T$. Hence by Proposition \ref{8.2.2}, $T\subseteq \operatorname{End}_{T}(I)$.
\end{proof}
By Definitions \ref{1.4.1}, \ref{1.5.1}, \ref{2.2.1} and \ref{2.2.3}, we immediately have
\begin{proposition} \label{8.2.4}
    Let $B\subseteq S$ and let $H\subseteq \operatorname{End}_{T}(S)$. Then
\[{{S}^{H}}=\{x\in S\,|\,\sigma (x)=x,\forall \sigma \in H\},\]
\[\operatorname{GMn}_{T}(S/B)=\{\sigma \in \operatorname{End}_{T}(S)\,|\,\sigma (x)=x,\forall x\in B\},\]
\[\operatorname{Int}_{T}^{\operatorname{End}}(S/B)=\{{{S}^{L}}\supseteq B\,|\,L\subseteq \operatorname{End}_{T}(S)\},\] and
\[\operatorname{GSMn}_{T}(S/B)=\{\operatorname{GMn}_{T}(S/K)\,|\,B\subseteq K\subseteq S\}.\]
\end{proposition} 
	Then the following is a straightforward result by applying Corollary \ref{2.2.5} to a dynamical system.
\begin{corollary} \label{8.2.5}
    Let $(M,S,\Phi )$ be a dynamical system, let $T=\{{{\Phi }_{g}}\,|\,g\in M\}$, and let $B\subseteq S$. Then the correspondences
    				\[\gamma :H\mapsto {{S}^{H}}\, and \, \delta :K\mapsto \operatorname{GMn}_{T}(S/K) \]
define inclusion-inverting mutually inverse bijective maps between $\operatorname{GSMn}_{T}(S/B)$ and $\operatorname{Int}_{T}^{\operatorname{End}}(S/B)$.
\end{corollary}
Moreover, we may apply Theorem \ref{4.1.2} (resp. Theorem \ref{4.3.2}) to investigate whether we can endow $S$ with a topology such that
\begin{center}
    $\operatorname{Int}_{T}^{\operatorname{End}}(S/B)\backslash \{\emptyset \}=\{B\subseteq {S}'\le_q S\,|\,{S}'$ is closed and nonempty\}
\end{center}
(resp. whether we can endow $\operatorname{End}_{T}(S)$ with a topology such that $\operatorname{GSMn}_{T}(S/B)$ is the set of all closed submonoids of $\operatorname{GMn}_{T}(S/B)$).

$ $

Analogously, by Definitions \ref{1.3.1} and \ref{$T$-isomorphisms}, we have
\begin{proposition} \label{8.2.6}
$ $
\begin{center}
     $\operatorname{Aut}_{T}(S)=\{$bijective $\sigma :S\to S\, |\,\sigma (f(x))=f(\sigma (x)),\forall x\in S \,and\,f\in T\}$.
\end{center}
\end{proposition} 	
Then by Definitions \ref{1.5.1}, \ref{2.2.1} and \ref{2.2.3}, we immediately have
\begin{proposition} \label{8.2.7}
    Let $B\subseteq S$. Then
	\[\operatorname{GGr}_{T}(S/B)=\{\sigma \in \operatorname{Aut}_{T}(S)\, |\, \sigma (x)=x,\forall x\in B\},\] 
\[\operatorname{Int}_{T}^{\operatorname{Aut}}(S/B)=\{{{S}^{H}}\supseteq B\,|\,H\subseteq \operatorname{Aut}_{T}(S)\},\] and
\[\operatorname{GSGr}_{T}(S/B)=\{\operatorname{GGr}_{T}(S/K)\,|\,B\subseteq K\subseteq S\}.\]
\end{proposition}   
Then by applying Corollary \ref{2.2.6} to a dynamical system, we immediately get
\begin{corollary} \label{8.2.8}
    Let $(M,S,\Phi )$ be a dynamical system, let $T=\{{{\Phi }_{g}}\,|\,g\in M\}$, and let $B\subseteq S$. Then the correspondences
				\[\gamma :H\mapsto {{S}^{H}} \, and \, \delta :K\mapsto \operatorname{GGr}_{T}(S/K)\] 
define inclusion-inverting mutually inverse bijective maps between $\operatorname{GSGr}_{T}(S/B)$ and $\operatorname{Int}_{T}^{\operatorname{Aut}}(S/B)$. 
\end{corollary}
Moreover, we may employ Theorem \ref{4.2.2} (resp. Theorem \ref{4.4.2}) to investigate whether we can endow $S$ with a topology such that
\begin{center}
    $\operatorname{Int}_{T}^{\operatorname{Aut}}(S/B)\backslash \{\emptyset \}=$\{$B\subseteq {S}'\le_q S\,|\,{S}'$ is closed and nonempty\}
\end{center}
 (resp. whether we can endow Aut$_{T}(S)$ with a topology such that $\operatorname{GSGr}_{T}(S/B)$ is the set of all closed subgroups of $\operatorname{GGr}_{T}(S/B)$).

\section{Other topics for future research} \label{Other topics}
	Besides the subjects discussed in Sections \ref{Duality} to \ref{App to dyn}, future research may include the following. 
\begin{enumerate}
    \item More research on the structures of $T$-spaces and quasi-$T$-spaces and the properties of $T$-morphisms and $\theta$-morphisms is needed.
    \item Operator semigroups and par-operator gen-semigroups can be generated rather arbitrarily (\textit{cf.} Definitions \ref{generated operator semigroup} and \ref{10.1.4}). Therefore there must be much more examples of $T$-morphisms or $\theta$-morphisms than we gave in this paper. 
    \item Since $T$-morphisms and $\theta$-morphisms preserve structures of mathematical objects, more research on constructions of $T$-morphisms and $\theta$-morphisms are needed.
    \item More research on the subjects discussed in Sections \ref{poly equ} to \ref{equ sol} is needed.
    \item The unsolved problem at the end of Subsection \ref{I Topologies on Aut}.
    \item To incorporate the notion of contravariant functors, the notions of “contravariant $T$-morphisms” and “contravariant $\theta$-morphisms” are needed.
\end{enumerate}

$ $

\bibliographystyle{elsarticle-num}

\begin{thebibliography}{00}
\bibitem {1} T. Crespo and Z. Hajto, \textit{Algebraic groups and differential Galois theory},
vol. 122, Amer. Math. Soc., Providence, RI, 2011, pp. 124, 127--128.
\bibitem {2} G. Gratzer, \textit{Lattice theory: foundation},
Springer, 2011, pp. 5.
\bibitem {3} M. Jarden, \textit{Infinite Galois theory},
Handbook of algebra, vol. 1, 1996, pp. 271--319.
\bibitem {4} A. R. Magid, \textit{Lectures on differential Galois theory},
Amer. Math. Soc., Providence, RI, 1994.
\bibitem {5} P. Morandi, \textit{Field and Galois theory},
Springer Science and Business Media, 2012.
\bibitem {6} J. Munkres, \textit{Topology}, 2nd ed., 
Pearson Education, 2000, pp. 104, 83, 82.
\bibitem {7} J.J. Rotman, \textit{Advanced Modern Algebra}, 
Pearson Education, 2002, pp. 190.
\bibitem {8} G. Teschl, \textit{Ordinary differential equations and dynamical systems}, vol. 140,
Amer. Math. Soc., Providence, RI, 2012, pp. 187.
\bibitem {9} M. Van der Put and M. F. Singer, \textit{Galois theory of linear differential equations}, vol. 328,
Springer Science and Business Media, 2012.
\end{thebibliography}

\end{document}